\documentclass[9pt,reqno]{amsart}  
\usepackage{ifdraft}
\usepackage{tikz,pgfplots}
\usepackage{etoolbox,listings,xstring,wasysym}

\makeatletter

\tikzset{
  fl/.style = {path fading=fade l},
  fr/.style = {path fading=fade r},
  wh/.style = {draw=none,fill=none},
  Gc/.style = {draw=none,circle split, inner sep=0pt,minimum size=8pt,rotate=90,path picture={\draw[pattern=#1] (0,0.07) circle (1.5pt); }},
  Gd/.style = {draw=none,circle split, inner sep=0pt,minimum size=8pt,rotate=270,path picture={\draw[pattern=#1] (0,0.07) circle (1.5pt); }},
  Gf/.style={draw=none,circle split, inner sep=1pt,minimum size=8pt,rotate=90},
  G/.style={circle,draw,minimum size=8pt,inner sep=1pt,font=\tiny},
  xx/.style={circle,fill,draw,inner sep=0pt,minimum size=3pt},
  ab/.style={circle,fill,draw=none,inner sep=0pt,minimum size=0pt,as=},
  Gend/.style={inner sep=0pt,minimum size=3pt},
  r/.style={draw=red},
  o/.style={draw=black,fill=white},
  lb/.style args={#1}{label=below:#1},
  lt/.style args={#1}{label=above:#1},
  lr/.style args={#1}{label=right:#1},
  ll/.style args={#1}{label=left:#1},
  g/.style={postaction={decorate,decoration={
        markings,
        mark=at position .7 with {\arrow[#1]{Stealth[sep=-3pt]}}
      }}},
  s/.style={densely dashed,postaction={decorate,decoration={
        markings,
        mark=at position .7 with {\arrow[#1]{Stealth[sep=-3pt]}}
      }}},
  R/.style ={draw=lightgray, thick,densely dotted,edge node={node[above=-7pt] {\color{gray} $R$ }},anchor=south,pos=0.5,postaction={decoration={
        markings,
        mark=at position .7 with {\arrow[#1]{Stealth[sep=-3pt]}}
      },decorate}},
we/.style args ={#1}{draw=lightgray, thick,densely dotted,edge node={node[above=-7pt] {\color{gray} #1 }},anchor=south,pos=0.5},
IEg/.style ={draw=lightgray, thick,densely dotted},
  S/.style ={draw=lightgray, thick,densely dotted,edge node={node[above=-7pt] {\color{gray}$S$}},anchor=south,pos=0.5},  
  T/.style ={draw=lightgray, thick,densely dotted,edge node={node[above=-7pt] {\color{gray}$T$}},anchor=south,pos=0.5,postaction={decoration={
        markings,
        mark=at position .7 with {\arrow[#1]{Stealth[sep=-3pt]}}
      },decorate}},
  rpd/.style ={draw=lightgray,thick,densely dotted,edge node={node[above=-7pt] {\color{red}$\partial$}},anchor=south,pos=0.5,postaction={decoration={
        markings,
        mark=at position .7 with {\arrow[#1]{Stealth[sep=-3pt]}}
      },decorate}},
  pd/.style ={draw=lightgray,thick,densely dotted,edge node={node[above=-7pt] {\color{gray}$\partial$}},anchor=south,pos=0.5,postaction={decoration={
        markings,
        mark=at position .7 with {\arrow[#1]{Stealth[sep=-3pt]}}
      },decorate}},
  pdr/.style ={g,draw=lightgray,thick,densely dotted,edge node={node[above=-7pt] {\color{red}$\partial$}},anchor=south,pos=0.5},
  pdkr/.style args={#1}{g,draw=lightgray,thick,densely dotted,edge node={node[above=-7pt] { \color{gray}$\partial^{#1}$\color{red}$\partial$}},anchor=south,pos=0.5},
  pdh/.style ={g,draw=lightgray,thick,densely dotted,edge node={node[above=-7pt] {\color{gray}$\widehat\partial$}},anchor=south,pos=0.5},
  pdk/.style args={#1}{g,draw=lightgray,thick,densely dotted,edge node={node[above=-7pt] {\color{gray}$\partial^{#1}$}},anchor=south,pos=0.5},
  eq/.style = {double,postaction={decoration={name=none}}},
  inl/.style args={#1}{initial,initial where=left, initial text=#1,initial distance=10pt},
  pdn/.style args={#1}{g,draw=lightgray,thick,densely dotted,edge node={node[above=-7pt] {\color{gray}$\partial^{#1}$}},anchor=south,pos=0.5},
  inr/.style args={#1}{initial,initial where=right, initial text=#1,initial distance=10pt},
  int/.style args={#1}{initial,initial where=above, initial text=#1,initial distance=10pt},
  inb/.style args={#1}{initial,initial where=below, initial text=#1,initial distance=10pt},
  lpf/.style args={#1}{initial,initial where=left, initial text=$\vp\vf$,initial distance=10pt},
  tpf/.style args={#1}{initial,initial where=above, initial text=$\vp\vf$,initial distance=10pt},
  bpf/.style args={#1}{initial,initial where=below, initial text=$\vp\vf$,initial distance=10pt},
  rpf/.style args={#1}{initial,initial where=right, initial text=$\vp\vf$,initial distance=10pt},
  l1/.style args={#1}{initial,initial where=left, initial text=$\bm 1$,initial distance=10pt},
  t1/.style args={#1}{initial,initial where=above, initial text=$\bm 1$,initial distance=10pt},
  b1/.style args={#1}{initial,initial where=below, initial text=$\bm 1$,initial distance=10pt},
  r1/.style args={#1}{initial,initial where=right, initial text=$\bm 1$,initial distance=10pt},
  lm/.style args={#1}{initial,initial where=left, initial text=$\vm$,initial distance=10pt},
  tm/.style args={#1}{initial,initial where=above, initial text=$\vm$,initial distance=10pt},
  bm/.style args={#1}{initial,initial where=below, initial text=$\vm$,initial distance=10pt},
  rm/.style args={#1}{initial,initial where=right, initial text=$\vm$,initial distance=10pt},
  inr/.style args={#1}{initial,initial where=right, initial text=#1,initial distance=10pt},
  B/.style args={#1}{thick,draw=lightgray,decorate,decoration={snake,amplitude=.4mm,segment length=.8mm,post length=1.4mm},edge node={node[above=-7pt] {\color{gray} #1 }},anchor=south,pos=0.5},
  br/.style = {bend right},
  b0/.style = {bend left=0},
  bl/.style = {bend left},
  glb/.style = {looseness=20,in =220, out=320},
  gll/.style = {looseness=20,in =130, out=230},
  glr/.style = {looseness=20,in =310, out=50},
  glt/.style = {looseness=20,in =40, out=140},
  gm/.style={postaction={decorate,decoration={
        markings,
        mark=at position .3 with {\arrow[#1]{Diamond[open,sep=-3pt,width=5pt]}}
      }}},
}

\newcommand\sGraph[1]{
\begin{tikzpicture}[grow=right,baseline={([yshift=-2pt]current bounding box.center)},font=\footnotesize,>=Stealth]
  \graph[simple necklace layout,components go down left aligned,nodes={draw,circle,as=,minimum size=3pt,inner sep=0pt,fill},node distance = 40pt,component sep=25pt,edges=g]{ #1  };
  \end{tikzpicture}
}

\newcommand\ssGraph[1]{
\begin{tikzpicture}[subgraph text none,grow=right,baseline={([yshift=-2pt]current bounding box.center)},font=\footnotesize,>=Stealth]
  \graph[spring electrical layout,components go down left aligned,nodes={draw,circle,as=,minimum size=3pt,inner sep=0pt,fill},node distance = 15pt,component sep=15pt]{ #1  };
  \end{tikzpicture}
}

\tikzset{circle split part fill/.style  args={#1}{%
 alias=tmp@name, 
  postaction={%
    insert path={
     \pgfextra{%
     \pgfpointdiff{\pgfpointanchor{\pgf@node@name}{center}}%
                  {\pgfpointanchor{\pgf@node@name}{east}}%
     \pgfmathsetmacro\insiderad{\pgf@x}
      \fill[white,fill opacity=0] (\pgf@node@name.base) ([xshift=-\pgflinewidth]\pgf@node@name.east) arc
                          (0:180:\insiderad-\pgflinewidth)--cycle;
      \fill[fill=white,preaction={fill, white},pattern=#1] (\pgf@node@name.base) ([xshift=\pgflinewidth]\pgf@node@name.west)  arc
                           (180:360:\insiderad-\pgflinewidth)--cycle;   
      \draw[line width=0.4pt] (\pgf@node@name.base) ([xshift=\pgflinewidth]\pgf@node@name.west)  arc
                           (180:360:\insiderad-\pgflinewidth)--cycle;                                
         }}}}}  
\makeatother  
\makeatletter
\tikzset{my loop/.style =  {to path={
  \pgfextra{}
  [looseness=6,min distance=4mm]
  \tikz@to@curve@path},font=\sffamily\small
  }}  
\makeatletter

\csedef{pat0}{}
\csedef{pat1}{north east lines}
\csedef{pat2}{crosshatch dots}
\csedef{pat3}{crosshatch}
\csedef{pat4}{grid}

\definecolor{col0}{HTML}{FFFFFF}
\definecolor{col1}{HTML}{A2B969}
\definecolor{col2}{HTML}{EBCB38}
\definecolor{col3}{HTML}{0D95BC}
\definecolor{col4}{HTML}{063951}
\definecolor{col5}{HTML}{F36F13}
\definecolor{col6}{HTML}{C13018}
\definecolor{lightgray}{HTML}{CCCCCC}

\newcommand\csum[1]{%
\sum_{\forcsvlist{\createColorCircle@item}{#1}}
}

\newcommand\createColorCircle@item[1]{
\StrDel{#1}{h}[\colNum]
\newif\ifhalf
\IfSubStr{#1}{h}{\halftrue}{\halffalse}
{\color{col\colNum}\ifhalf\circ\else\bullet\fi}
}

\newcommand\plotLambda[1]{
  \def\mgraphspecs{}
  \foreach \ll [count = \countl] in {#1} {
    \csedef{firstCol}{col0}
    \def\graphspecs{}
    \StrCount{\ll}{,}[\graphlen]
    \ifnum\graphlen>1 
    \foreach \node [count = \g] in \ll {
      \StrDel{\node}{m}[\nodenumber]
      \StrDel{\nodenumber}{.}[\nodenumber]
      \global\csedef{tempCol}{col\nodenumber}
      \global\csdef{tempPat}{\col{1}}
      \ifnum\g=1
        \global\csedef{firstCol}{\tempCol}
        \IfSubStr{\node}{.}{\global\csedef{open}{y}}{\global\csedef{open}{n}}
        \IfStrEqCase{\open}{
          {y}{\xappto\graphspecs{\countl\g[as=,draw=none] }}
          {n}{\xappto\graphspecs{\countl\g[as=,preaction={fill, white},fill=\tempCol,pattern=\csname pat\nodenumber\endcsname] }}
        }
      \else
        \IfBeginWith{\node}{m}{\global\csedef{secondArrow}{{<[sep=-3pt,length=8pt]}}}{\global\csedef{secondArrow}{}}%
        \ifnum\g=2
          \IfStrEqCase{\open}{
            {y}{\xappto\graphspecs{ --[dotted,thick] \countl\g[as=,preaction={fill, white},fill=\tempCol,pattern=\csname pat\nodenumber\endcsname] }}
            {n}{\xappto\graphspecs{ --[\firstArrow-\secondArrow] \countl\g[as=,preaction={fill, white},fill=\tempCol,pattern=\csname pat\nodenumber\endcsname] }}
          }
        \else
          \xappto\graphspecs{ --[\firstArrow-\secondArrow] \countl\g[as=,preaction={fill, white},fill=\tempCol,pattern=\csname pat\nodenumber\endcsname] }%
        \fi
      \fi
      \IfEndWith{\node}{m}{\global\csedef{firstArrow}{{>[sep=-3pt,length=8pt]}}}{\global\csedef{firstArrow}{}}
      \global\csedef{prevTempCol}{\tempCol} 
    }
    \IfStrEqCase{\open}{
      {n}{\IfBeginWith{\ll}{m}{\global\csedef{secondArrow}{{<[sep=-3pt,length=8pt]}}}{\global\csedef{secondArrow}{}}
          \xappto\graphspecs{ --[\firstArrow-\secondArrow,decorate,decoration={snake,amplitude=.3mm,segment length=.6mm}] \countl1; } }
      {y}{\xappto\graphspecs{ --[dotted,thick] \countl0[draw=none,as=] ; } }
    }
    \else
      \ifnum\graphlen=0
        \StrDel{\ll}{m}[\nodenumber]
        \csedef{tempCol}{col\nodenumber}
        \IfSubStr{\ll}{m}{\csedef{firstArrow}{{>[sep=-3pt,length=8pt]}}}{\csedef{firstArrow}{}}
        \def\graphspecs{ \countl1[as=,preaction={fill, white},fill = \tempCol,pattern=\csname pat\nodenumber\endcsname] --[my loop, decorate,decoration={snake,amplitude=.3mm,segment length=.6mm},\firstArrow-] \countl1; }
      \else
        \IfSubStr{\ll}{.}{
          \StrDel{\ll}{.}[\nodenumber]
          \StrDel{\nodenumber}{,}[\nodenumber]
          \csedef{tempCol}{col\nodenumber}
          \def\graphspecs{ \countl0[draw=none,as=,orient = left] --[dotted,thick] \countl1[as=,fill = \tempCol,pattern=\csname pat\nodenumber\endcsname] --[dotted,thick] \countl2[draw=none,as=,nudge down=10pt]; }
        }{
          \StrBefore{\ll}{,}[\nodeOne]
          \StrBehind{\ll}{,}[\nodeTwo]
          \StrDel{\nodeOne}{m}[\nodeOneNumber]
          \StrDel{\nodeTwo}{m}[\nodeTwoNumber]
          \def\tempColOne{col\nodeOneNumber}
          \def\tempColTwo{col\nodeTwoNumber}
          \IfEndWith{\nodeOne}{m}{\global\csedef{firstArrowOne}{{>[sep=-3pt,length=8pt]}}}{\global\csedef{firstArrowOne}{}}
          \IfEndWith{\nodeTwo}{m}{\global\csedef{firstArrowTwo}{{<[sep=-3pt,length=8pt]}}}{\global\csedef{firstArrowTwo}{}}
          \IfBeginWith{\nodeOne}{m}{\global\csedef{secondArrowOne}{{>[sep=-3pt,length=8pt]}}}{\global\csedef{secondArrowOne}{}}
          \IfBeginWith{\nodeTwo}{m}{\global\csedef{secondArrowTwo}{{<[sep=-3pt,length=8pt]}}}{\global\csedef{secondArrowTwo}{}}
          \def\graphspecs{ \countl1[as=,preaction={fill, white},fill = \tempColOne,pattern=\csname pat\nodeOneNumber\endcsname] --[bend right,decorate,decoration={snake,amplitude=.3mm,segment length=.6mm}, \firstArrowOne-\secondArrowTwo ] \countl2[as=,preaction={fill, white},fill = \tempColTwo,pattern=\csname pat\nodeTwoNumber\endcsname]; \countl1 --[bend left,\secondArrowOne-\firstArrowTwo] \countl2;}
        }
      \fi
    \fi
    \xappto\mgraphspecs{ \graphspecs }
  }
  \xdef\mgraphspecs{\noexpand\graph[simple necklace layout,componentwise,component packing=skyline,components go right center aligned,orient=0,nodes=G]{ \mgraphspecs }}
  \begin{tikzpicture}[baseline={([yshift=-2pt]current bounding box.center)},font=\tiny,>=Stealth, node distance = 15pt,node sep=12pt,component sep=5pt]
    \mgraphspecs;
  \end{tikzpicture}%
}

\newcommand\pB[1]{
  \StrDel{#1}{h}[\patNum]
  \IfSubStr{#1}{h}{
  \begin{tikzpicture}[baseline={([yshift=-2pt]current bounding box.center)},font=\tiny,>=Stealth, node distance = 0pt,node sep=1pt,component sep=1pt]
   \graph[simple necklace layout,componentwise,component packing=skyline,components go right center aligned,orient=0,nodes=G] { 1[Gf,rotate=270,circle split part fill={\csname pat\patNum\endcsname},as=,minimum size=6pt]; };
  \end{tikzpicture}
  }{
  \begin{tikzpicture}[baseline={([yshift=-2pt]current bounding box.center)},font=\tiny,>=Stealth, node distance = 0pt,node sep=1pt,component sep=1pt]
   \graph[simple necklace layout,componentwise,component packing=skyline,components go right center aligned,orient=0,nodes=G] { 1[pattern=\csname pat\patNum\endcsname,as=,minimum size=6pt]; };
  \end{tikzpicture}
  }
}
\newcommand\plotlLambda[1]{
  \def\mgraphspecs{}
  \foreach \ll [count = \countl] in {#1} {
    \csedef{firstCol}{col0}
    \csedef{openEnd}{no}
    \def\graphspecs{}
    \StrCount{\ll}{,}[\graphlen]
    \foreach \node [count = \g] in \ll {
      \StrDel{\node}{m}[\nodenumber]
      \csedef{type}{n}
      \csedef{marking}{no}
      \IfSubStr{\node}{x}{\csedef{type}{x}}{}
      \IfSubStr{\node}{c}{\csedef{type}{c}}{}
      \IfSubStr{\node}{d}{\csedef{type}{d}}{}
      \IfSubStr{\node}{.}{\csedef{type}{open}}{}
      \IfBeginWith{\node}{m}{\csedef{firstMarking}{yes}\csedef{marking}{yes}}{\csedef{firstMarking}{no}}
      \IfEndWith{\node}{m}{\csedef{secondMarking}{yes}\csedef{marking}{yes}}{\csedef{secondMarking}{no}}
      \StrDel{\nodenumber}{.}[\nodenumber]
      \StrDel{\nodenumber}{c}[\nodenumber]
      \StrDel{\nodenumber}{d}[\nodenumber]
      \StrDel{\nodenumber}{x}[\nodenumber]
      \global\csedef{tempCol}{col\nodenumber}
      \ifnum\g=1
        \global\csedef{firstCol}{\tempCol}
        \IfStrEq{\type}{open}{
          \xappto\graphspecs{ \countl0[xx,as=,grow right,draw=none,fill=white] --[dotted,thick] \countl1[fill=white,preaction={fill, white},pattern=\csname pat\nodenumber\endcsname,as=]}
        }{
        \IfSubStr{\node}{x}{
          \xappto\graphspecs{ \countl0[xx,as=,grow right,fill=lightgray,draw=none] --[color=lightgray]\countl1[Gf,circle split part fill={\csname pat\nodenumber\endcsname},as=]}}
        {
          \IfStrEq{\marking}{yes}
          {
            \xappto\graphspecs{ \countl0[xx,as=,grow right] --[-{<[fill=lightgray,color=black,sep=-3pt,length=8pt]}] \countl1[fill=white,preaction={fill, white},pattern=\csname pat\nodenumber\endcsname,as=]}
          }{
            \xappto\graphspecs{ \countl0[xx,as=,grow right] -- \countl1[fill=white,preaction={fill, white},pattern=\csname pat\nodenumber\endcsname,as=]}
          }
        }}
      \else
        \IfStrEq{\firstMarking}{yes}{\csedef{secondArrow}{{<[sep=-3pt,length=8pt]}}}{\csedef{secondArrow}{}}
        \IfStrEqCase{\type}{
          {n}{\xappto\graphspecs{ --[\firstArrow-\secondArrow] \countl\g[fill=white,preaction={fill, white},pattern=\csname pat\nodenumber\endcsname,as=] }}%
          {c}{\xappto\graphspecs{ --[\firstArrow-\secondArrow] \countl\g[Gc=\csname pat\nodenumber\endcsname,rotate=180,circle split part fill={\csname pat\nodenumber\endcsname},as=] }}%
          {d}{\xappto\graphspecs{ --[\firstArrow-\secondArrow] \countl\g[Gd=\csname pat\nodenumber\endcsname,rotate=180,circle split part fill={\csname pat\nodenumber\endcsname},as=] }}%
          {open}{\global\csedef{openEnd}{yes}}%
        }
      \fi
      \IfStrEq{\secondMarking}{yes}{
        \global\csedef{firstArrow}{{>[sep=-3pt,length=8pt]}}%
      }{
        \global\csedef{firstArrow}{}%
      }
      \global\csedef{prevTempCol}{\tempCol} 
    }
    \IfStrEqCase{\openEnd}{
      {yes}{\xappto\mgraphspecs{ \graphspecs --[dotted,thick] \countl17[Gend,as=,draw=none]; }}%
      {no}{\xappto\mgraphspecs{ \graphspecs -- \countl17[Gend,as=]; }}%
    }%
  }
  \xdef\mgraphspecs{\noexpand\graph[tree layout,componentwise,component packing=skyline,components go down left aligned,nodes=G]{ \mgraphspecs }}
  \begin{tikzpicture}[grow=right,baseline={([yshift=-2pt]current bounding box.center)},font=\tiny,>=Stealth, node distance = 5pt,node sep=12pt,component sep=5pt]
  \mgraphspecs;
  \end{tikzpicture}%
}
  
\usepgfplotslibrary{external}
\tikzsetfigurename{tikz-figure}

\usepackage[english]{babel}
\usepackage[p,osf]{cochineal}
\usepackage[scale=.95,type1]{cabin}
\usepackage[zerostyle=c,scaled=.94]{newtxtt}
\usepackage[T1]{fontenc} 
\usepackage[utf8]{inputenc} 
\usepackage{amsthm}
\usepackage{amsmath}
\usepackage{amssymb}
\usepackage{amsfonts}
\usepackage{amsaddr}
\usepackage{xspace}
\usepackage{enumitem} 
\usepackage{csquotes}
\usepackage{xcolor}
\usepackage[colorlinks]{hyperref}
\definecolor{intcolor}{HTML}{CA0020}
\definecolor{extcolor}{HTML}{0571B0}
\hypersetup{breaklinks,linkcolor=intcolor,citecolor=intcolor,filecolor=black,urlcolor=extcolor,menucolor=intcolor,runcolor=extcolor}
\usepackage{graphicx}
\usepackage{mathtools}
\mathtoolsset{centercolon}
\usepackage{bm}
\usepackage{pgfkeys}
\usepackage{pgf}
\usepackage{xparse}
\ExplSyntaxOn
\NewDocumentCommand{\interval}{sO{}mm}
 {
  \IfBooleanTF{#1}
   {
    \yannick_interval:NNnnn \left \right { } { #3 } { #4 }
   }
   {
    \yannick_interval:NNnnn \mathopen \mathclose { #2 } { #3 } { #4 }
   }
 }

\cs_new_protected:Nn \yannick_interval:NNnnn
 {
  \str_case:nn { #4 }
   {
    {oo}{#1#3\c_yannick_left_open_tl #5 #2#3\c_yannick_right_open_tl}
    {co}{#1#3[                       #5 #2#3\c_yannick_right_open_tl}
    {oc}{#1#3\c_yannick_left_open_tl #5 #2#3]}
    {cc}{#1#3[                       #5 #2#3]}
   }
 }
\tl_const:Nn \c_yannick_left_open_tl { ( }
\tl_const:Nn \c_yannick_right_open_tl { ) }
\ExplSyntaxOff
\DeclareMathOperator{\diag}{diag}
\DeclareMathOperator{\Tr}{Tr}

\DeclareMathOperator{\Val}{Val}
\DeclareMathOperator{\E}{\mathbf{E}}
\DeclareMathOperator{\Prob}{\mathbf{P}}
\DeclareMathOperator{\WEst}{W-Est}
\DeclareMathOperator{\NEst}{N-Est}
\DeclareMathOperator{\Spec}{Spec}
\newcommand{\ov}{\overline}
\newcommand{\ii}{\mathrm{i}}

\ifdefined\C
\renewcommand{\C}{\mathbf{C}}
\else
\newcommand{\C}{\mathbf{C}}
\fi
\newcommand{\HC}{\mathbf{H}}
\renewcommand{\SS}{\mathcal{S}}
\newcommand{\un}{\underline}
\newcommand{\vx}{\bm{x}}
\newcommand{\vy}{\bm{y}}

\newcommand{\cN}{\mathcal{N}}
\newcommand{\wt}{\widetilde}
\newcommand{\wh}{\widehat}
\newcommand{\cE}{\mathcal{E}}
\newcommand{\R}{\mathbf{R}}
\newcommand{\N}{\mathbf{N}}
\newcommand{\Z}{\mathbf{Z}}
\newcommand{\sump}{\sideset{}{'}\sum}

\newcommand{\DD}{\mathbf{D}}

\newcommand{\cG}{\mathcal{G}}

\newcommand{\cO}{\mathcal{O}}
\newcommand{\co}{{\scriptstyle\mathcal{O}}}
\newcommand{\cB}{\mathcal{B}}

\newcommand{\dif}{\operatorname{d}\!{}}
\DeclarePairedDelimiter{\braket}{\langle}{\rangle}%
\DeclarePairedDelimiter{\abs}{\lvert}{\rvert}%
\DeclarePairedDelimiter{\norm}{\lVert}{\rVert}%
\providecommand\given{}
\newcommand\SetSymbol[1][]{\nonscript\:#1\vert\allowbreak\nonscript\:\mathopen{}}
\DeclarePairedDelimiterX{\tuple}[1](){\renewcommand\given{\SetSymbol[\delimsize]}#1}
\DeclarePairedDelimiterX{\set}[1]{\{}{\}}{\renewcommand\given{\SetSymbol[\delimsize]}#1}
\DeclarePairedDelimiterX{\Set}[1]\{\}{\renewcommand\given{\SetSymbol[\delimsize]}#1}
\DeclarePairedDelimiterXPP{\landauO}[1]{\cO}(){}{#1}
\DeclarePairedDelimiterXPP{\landauo}[1]{\co}(){}{#1}
\DeclarePairedDelimiterXPP{\landauOprec}[1]{\cO_\prec}(){}{#1}
\DeclarePairedDelimiterXPP{\Exp}[1]{\E}[]{}{\renewcommand\given{\SetSymbol[\delimsize]}#1}
\DeclareFontFamily{U}{mathx}{\hyphenchar\font45}
\DeclareFontShape{U}{mathx}{m}{n}{
      <5> <6> <7> <8> <9> <10>
      <10.95> <12> <14.4> <17.28> <20.74> <24.88>
      mathx10
      }{}
\DeclareSymbolFont{mathx}{U}{mathx}{m}{n}
\DeclareFontSubstitution{U}{mathx}{m}{n}
\DeclareMathAccent{\widecheck}{0}{mathx}{"71}
\usepackage[giveninits=true,url=false,doi=false,isbn=false,eprint=true,datamodel=mrnumber,sorting=nty,maxcitenames=4,maxbibnames=99,backref=false,block=space,backend=biber,bibstyle=phys]{biblatex} 
\AtEveryBibitem{\clearfield{month}}
\AtEveryCitekey{\clearfield{month}} 
\renewbibmacro{in:}{}
\DeclareFieldFormat[article]{title}{\emph{#1}} 
\DeclareFieldFormat{mrnumber}{\ifhyperref{\href{http://www.ams.org/mathscinet-getitem?mr=#1}{\nolinkurl{MR#1}}}{\nolinkurl{#1}}}
\DeclareFieldFormat{pmid}{\ifhyperref{\href{https://www.ncbi.nlm.nih.gov/pubmed/#1}{\nolinkurl{PMID#1}}}{\nolinkurl{#1}}}
\DeclareFieldFormat{eprint}{\ifhyperref{\href{https://arxiv.org/abs/#1}{\nolinkurl{arXiv:#1}}}{\nolinkurl{#1}}}
\renewbibmacro*{doi+eprint+url}{%
  \iftoggle{bbx:doi}{\printfield{doi}}{}
  \newunit\newblock%
  \printfield{mrnumber}%
  \newunit\newblock%
  \printfield{pmid}%
  \newunit\newblock%
  \printfield{eprint}%
  \iftoggle{bbx:url}{\usebibmacro{url+urldate}}{}}
\bibliography{bibclt}

\numberwithin{equation}{section}  
\newtheorem{theorem}{Theorem}[section]
\newtheorem{assumption}[theorem]{Assumption}
\newtheorem{lemma}[theorem]{Lemma}
\newtheorem{proposition}[theorem]{Proposition}
\newtheorem{definition}[theorem]{Definition}

\newtheorem{remark}[theorem]{Remark} 
\newtheorem{corollary}[theorem]{Corollary}
\newtheorem{convention}[theorem]{Convention}

\date{\today}
\author{Giorgio Cipolloni\(^{\dagger\ddagger}\) \and L\'aszl\'o Erd\H{o}s\(^{\dagger}\)}
\address{IST Austria, Am Campus 1, 3400 Klosterneuburg, Austria}
\author{Dominik Schr\"oder\(^{\ast}\)}
\address{Institute for Theoretical Studies, ETH Zurich, Clausiusstr.\ 47, 8092 Zurich, Switzerland}
\email{giorgio.cipolloni@ist.ac.at} 
\email{lerdos@ist.ac.at}
\email{dschroeder@ethz.ch}
\thanks{\(^\dagger\)Partially supported by ERC Advanced Grant No.~338804}
\thanks{\(^\ddagger\)This project has received funding from the European Union's Horizon 2020 research and innovation programme under the Marie Sk\l odowska-Curie Grant Agreement No. 665385.}
\thanks{\(^\ast\)Supported by Dr.\ Max R\"ossler, the Walter Haefner Foundation and the ETH Z\"urich Foundation}
\subjclass[2010]{60B20, 15B52} 
\keywords{Dyson Brownian Motion, Local Law, Girko's Formula, Linear Statistics, Central Limit Theorem}
\title[CLT for non-Hermitian random matrices]{Central limit theorem for linear eigenvalue statistics  of non-Hermitian random matrices}
\date{\today}
\begin{document}
\thispagestyle{empty}
\begin{abstract}
    We consider large non-Hermitian random matrices  \(X\) with complex, independent, identically distributed centred entries
    and show that the linear statistics of their eigenvalues are asymptotically Gaussian for test functions having \(2+\epsilon\) derivatives. Previously this result was known only for a few special cases; either the test functions were required to be analytic~\cite{MR2294978}, or the distribution of the matrix elements needed to be Gaussian~\cite{MR2361453}, or at least match the Gaussian up to the first four moments~\cite{MR3306005},~\cite{1510.02987}. We find the exact dependence of the limiting variance  on the fourth cumulant that was not known before.
    The proof relies on two novel ingredients: (i) a local law for  a product of \emph{two} resolvents of the Hermitisation of \(X\) with different spectral parameters and (ii) a coupling of several weakly dependent Dyson Brownian Motions. These methods are also the key inputs
    for our analogous results on the linear eigenvalue statistics of \emph{real} matrices \(X\)
    that are presented in the companion paper~\cite{MR4235475}.
\end{abstract}

\maketitle

\section{Introduction}\label{sec:INT}
Eigenvalues of random matrices form a strongly correlated point process. One manifestation of this fact is the unusually small fluctuation of their
linear statistics making the
eigenvalue process distinctly different from a Poisson point process.
Suppose that the \(n\times n\) random matrix \(X\) has i.i.d.\ entries of zero mean and variance \(1/n\).
The empirical density of the eigenvalues \(\{ \sigma_i\}_{i=1}^n\) converges to a limit distribution; it is the uniform distribution on the unit
disk in the non-Hermitian case \emph{(circular law)} and the semicircular density in the Hermitian case \emph{(Wigner semicircle law)}. For test functions \(f\) defined on the spectrum one may consider the fluctuation of the linear statistics and one expects that
\begin{equation}\label{eq:linst}
    L_n(f):= \sum_{i=1}^n f(\sigma_i)  - \E \sum_{i=1}^n f(\sigma_i)  \sim \cN (0, V_f)
\end{equation}
converges to a centred normal distribution as \(n\to \infty\).
The variance \(V_f\) is expected to depend only on the second and fourth moments of
the single entry distribution. Note that, unlike in the usual central limit theorem, there is no \(1/\sqrt{n}\) rescaling in~\eqref{eq:linst} which is a  quantitative indication
of a strong correlation. The main result of the current paper is the proof of~\eqref{eq:linst} for non-Hermitian  random matrices with
complex i.i.d.\ entries and for general test functions \(f\). We give an explicit
formula for \(V_f\) that involves the fourth cumulant of \(X\) as well, disproving a conjecture by Chafa{\"\i}~\cite{chefai}.  By polarisation,
from~\eqref{eq:linst} it also  follows  that the limiting joint distribution
of \((L_n(f_1), L_n(f_2), \ldots , L_n(f_k))\) for a fixed number of test functions is jointly Gaussian.

We remark that another manifestation of the strong eigenvalue correlation
is the  repulsion between neighbouring eigenvalues. For Gaussian ensembles the local repulsion
is directly seen from the well-known determinantal structure of the joint distribution of all eigenvalues; both in the non-Hermitian \emph{Ginibre} case and in the Hermitian \emph{GUE/GOE} case.
In the spirit of \emph{Wigner-Dyson-Mehta universality}  of the local correlation functions~\cite{MR0220494} level repulsion should also hold
for random matrices with general distributions. While for the
Hermitian case the universality has been rigorously established for a large class of random matrices
(see e.g.~\cite{MR3699468} for a recent monograph), the
analogous result for the non-Hermitian case is still open in the bulk spectrum (see, however,~\cite{MR4221653} for the edge
regime and~\cite{MR3306005} for entry distributions whose first four moments match the Gaussian).

These two manifestations of the eigenvalue correlations cannot be deduced from each other,  however
the proofs often share common tools.  For \(n\)-independent test functions \(f\),~\eqref{eq:linst} apparently
involves understanding the eigenvalues only on the macroscopic scales, while the level
repulsion is expressly a property on the microscopic scale of   individual eigenvalues.
However the suppression of the usual \(\sqrt{n}\) fluctuation is due to delicate correlations on all scales, so~\eqref{eq:linst} also requires understanding local scales.

Hermitian random matrices are much easier to handle, hence fluctuation results of the type~\eqref{eq:linst} have been gradually obtained for more and more general matrix ensembles as well as for broader classes of test functions, see,
e.g.~\cite{MR1487983,MR2189081,MR1411619,MR2561434,MR2829615} and~\cite{MR3116567} for the weakest regularity conditions on \(f\). Considering \(n\)-dependent test functions,
Gaussian fluctuations have been detected even on mesoscopic scales~\cite{MR1678012,MR1689027, MR3959983, MR3678478, MR4009708,  MR3852256, MR4187127,MR4255183,2001.07661}.

Non-Hermitian random matrices pose serious challenges, mainly
because their eigenvalues are potentially very unstable.
When \(X\) has i.i.d.\ centred Gaussian entries
with variance \(1/n\) (this is called the \emph{Ginibre ensemble}),
the explicit determinantal formulas for the correlation functions may be used
to compute the distribution of the linear statistics \(L_n(f)\). Forrester in~\cite{MR1687948}
proved~\eqref{eq:linst} for complex Ginibre ensemble and  radially symmetric \(f\) and obtained
the variance \(V_f = (4\pi)^{-1}\int_{\DD} \abs{\nabla f}^2 \dif^2z\)
where \(\DD \) is the unit disk. He also gave a heuristic argument based on  Coulomb gas theory for general \(f\) and his calculations predicted an additional boundary term \(\frac{1}{2}\norm{f}_{\dot{H}^{1/2}(\partial \DD)}^2\) in the variance \(V_f\). Rider  considered test functions \(f\) depending
only on  the angle~\cite{MR2095933} when \(f\not\in H^1(\DD)\) and accordingly \(V_f\) grows with \(\log n\) (similar growth is
proved for \(f=\log\) in~\cite{MR3161483}).  Finally, Rider and Vir\'ag in~\cite{MR2361453} have rigorously verified Forrester's prediction for general\(f\in C^1(\DD)\)
using a cumulant formula for determinantal processes found first  by Costin and Lebowitz~\cite{MR3155254}
and extended by Soshnikov~\cite{MR1894104}. They also presented a \emph{Gaussian free field (GFF)} interpretation of the result that we extend in Section~\ref{sec GFF}.

The first result beyond the explicitly computable  Gaussian case is due to Rider and Silverstein~\cite[Theorem 1.1]{MR2294978}
who proved~\eqref{eq:linst} for \(X\) with i.i.d.\ complex matrix elements
and for test functions \(f\) that are analytic on a large disk. Analyticity allowed them to use contour
integration and thus deduce the result from analysing the resolvent at spectral parameters far away
from the actual spectrum.
The domain of analyticity was optimized in~\cite{MR3540493}, where extensions to elliptic ensembles were also proven.
Polynomial test functions  via the alternative  moment method  were considered by Nourdin and Peccati in~\cite{MR2738319}.
The analytic method of~\cite{MR2294978}  was recently extended by Coston
and O'Rourke~\cite{MR4125967} to fluctuations of linear statistics for \emph{products} of i.i.d.\ matrices.
However, these method fail  for a larger class of test functions.

Since the first four moments of the matrix elements fully determine the limiting
eigenvalue statistics, Tao and Vu were able to compare the fluctuation of the local eigenvalue density
for a general non-Gaussian  \(X\)  with that of a Ginibre matrix~\cite[Corollary 10]{MR3306005}
assuming the first four moments of \(X\) match those of the complex Ginibre ensemble. This method
was extended by Kopel~\cite[Corollary 1]{1510.02987} to general smooth test functions
with an additional study on the real eigenvalues when \(X\) is real (see also the work of
Simm for polynomial statistics  of the real  eigenvalues~\cite{MR3612267}).

Our result removes the limitations of both previous approaches: we allow general test functions and
general distribution for the matrix elements without constraints on matching moments.
We remark that the dependence of the variance \(V_f\) on the fourth cumulant
of the single matrix entry escaped all previous works. The Ginibre ensemble with its vanishing fourth cumulant clearly
cannot catch this dependence.
Interestingly, even though the fourth cumulant in general  is not zero in
the work Rider and Silverstein~\cite{MR2294978}, it is multiplied by
a functional of \(f\) that happens to vanish for analytic functions (see~\eqref{eq:cov},~\eqref{eq:expv} and Remark~\ref{rem:compo} later). Hence this result did not detect the precise role of the fourth cumulant either. This may have motivated the conjecture~\cite{chefai}
that the variance does not depend on the fourth cumulant at all.

In order to focus on the main new ideas,
in this paper we consider the problem only for \(X\) with genuinely complex entries. Our method also works for  real matrices where
the real axis in the spectrum plays a special role that modifies the exact formula for the expectation and the variance \(V_f\) in~\eqref{eq:linst}. This leads to some additional technical complications that we have resolved in a separate work~\cite{MR4235475}
which contains the real version of our main Theorem~\ref{theo:CLT}.

Finally, we remark that the problem of fluctuations of linear statistics has been considered for
\(\beta\)-log-gases in one and two dimensions;
these are closely related to the eigenvalues
of the  Hermitian, resp.\ non-Hermitian Gaussian matrices for classical values \(\beta=1,2,4\) and for quadratic potential.
In  fact, in two dimensions the logarithmic
interaction also corresponds to the  Coulomb gas from statistical physics. Results analogous to~\eqref{eq:linst}
in one dimension were obtained e.g.\ in~\cite{MR1487983, MR3063494,  1303.1045, MR3865662, MR4021234,MR3885548, MR4009708, MR4168391}.
In two dimensions  similar results  have been established both in the macroscopic~\cite{MR3788208}
and  in the mesoscopic~\cite{MR4063572} regimes.

We now outline the main  ideas in our approach.
We use Girko's formula~\cite{MR773436} in the form given in~\cite{MR3306005}
to express linear eigenvalue statistics of \(X\) in terms
of  resolvents of a family of \(2n\times 2n\) Hermitian matrices
\begin{equation}\label{eq:linz1}
    H^z:= \begin{pmatrix}
        0                & X-z \\
        X^*-\overline{z} & 0
    \end{pmatrix}
\end{equation}
parametrized by \(z\in \C  \).  This formula asserts that
\begin{equation}\label{girko}
    \sum_{\sigma\in \Spec(X)} f(\sigma) = -\frac{1}{4\pi} \int_{\C } \Delta f(z)\int_0^\infty \Im \Tr  G^z(i\eta)\dif\eta \dif^2 z
\end{equation}
for any smooth, compactly supported test function \(f\) (the apparent divergence of the \(\eta\)-integral at infinity can easily be removed, see~\eqref{eq:GirkosplitA}). Here we set \(G^z(w):=  (H^z-w)^{-1}\) to be the resolvent of \(H^z\).
We have thus transformed our problem to a Hermitian one
and all tools and results
developed for Hermitian ensembles in the recent years are available.

Utilizing  Girko's formula requires a good understanding of the  resolvent
of \(H^z\) along the imaginary axis for all \(\eta>0\). On very small scales
\(\eta\ll n^{-1}\), there are no eigenvalues thus  \(\Im \Tr  G^z(\ii\eta)\) is negligible.
All other scales \(\eta\gtrsim n^{-1}\) need to be controlled carefully
since  \emph{a priori} they could  all contribute  to the fluctuation of \(L_n(f)\),
even though \emph{a posteriori} we find that the entire variance comes
from  scales \(\eta\sim 1\).

In the mesoscopic regime \(\eta\gg n^{-1}\), \emph{local laws} from~\cite{MR3770875, MR4408013}
accurately describe the leading order deterministic behaviour of \(\frac{1}{n}\Tr G^z(\ii \eta)\) and even
the matrix elements \(G^z_{ab}(\ii\eta)\); now we need to identify the
next order fluctuating term in the local law. In other words we need to prove a central limit theorem
for the traces of resolvents \(G^z\).  In fact, based upon~\eqref{girko}, for the higher \(k\)-th moments of  \(L_n(f)\)
we need the joint distribution of  \(\Tr  G^{z_l}(\ii \eta)\)  for different spectral parameters \(z_1, z_2, \ldots, z_k\).
This is one of our main technical achievements. Note that the asymptotic joint Gaussianity  of
traces of Wigner resolvents \(\Tr  (H-w_1)^{-1}, \Tr  (H-w_2)^{-1}, \ldots\) at different spectral parameters
has been obtained in~\cite{MR4095015, MR3678478}. However, the method of this result is not applicable since the role of the spectral parameter \(z\) in~\eqref{eq:linz1} is very different from \(w\); it is in an off-diagonal position thus these resolvents do not commute
and they are not
in the spectral resolution of a single matrix.

The microscopic regime, \(\eta\sim n^{-1}\), is much more involved than the mesoscopic one.  Local laws and their
fluctuations are not sufficient, we need
to trace the effect of the individual eigenvalues  \(0\le \lambda_1^z\le \lambda_2^z, \ldots\)
of \(H^z\)   near zero (the spectrum of  \(H^z\) is symmetric, we may focus on the positive eigenvalues).
Moreover, we need their \emph{joint} distribution for different \(z\) parameters
which, for arbitrary \(z\)'s, is not known even in the Ginibre case.
We prove, however, that  \(\lambda_1^z\) and \(\lambda_1^{z'}\) are asymptotically independent if \(z\) and \(z'\) are
far away, say \(\abs{z-z'}\ge n^{-1/100}\). A similar result holds simultaneously for several  small eigenvalues.
Notice that due to the \(z\)-integration in~\eqref{girko}, when the \(k\)-th moment of \(L_n(f)\) is computed,
the integration variables \(z_1, z_2, \ldots, z_k\) are typically far away from each other.
The resulting independence  of the spectra of \(H^{z_1}\), \(H^{z_2}, \ldots \) near zero
ensures that the microscopic regime eventually
does not contribute to the fluctuation of \(L_n(f)\).

The proof of the independence of \(\lambda_1^z\) and \(\lambda_1^{z'}\) relies on the analysis of the \emph{Dyson Brownian motion} (DBM) developed in the recent years~\cite{MR3699468} for the proof of the Wigner-Dyson-Mehta universality conjecture for Wigner matrices. The key mechanism is the fast local equilibration of the eigenvalues \({\bm \lambda}^z(t):= \{ \lambda_i^z(t)\}\) along
the stochastic flow generated by adding a small time-dependent Gaussian  component to the original matrix.
This Gaussian component can  then be removed by the \emph{Green function comparison theorem} (GFT).
One of the main technical results of~\cite{MR3916329} (motivated by the analogous analysis in~\cite{MR3914908}  for Wigner matrices that relied on coupling and homogenisation ideas introduced first in~\cite{MR3541852}) asserts that  for any fixed \(z\) the  DBM process \({\bm \lambda}^z(t)\) can be pathwise approximated by a similar DBM with a different initial condition by \emph{exactly} coupling the driving Brownian motions in their DBMs. We extend this idea to simultaneously trailing \({\bm \lambda}^z(t)\)
and \({\bm \lambda}^{z'}(t)\) by their independent Ginibre counterparts. The  evolutions of \({\bm  \lambda}^z(t)\) and \({\bm \lambda}^{z'}(t)\)  are not independent since their
driving Brownian motions are correlated;
the correlation is given by the eigenfunction overlap \(\braket{ u_i^z, u_j^{z'}}\braket{ v_j^{z'}, v_i^z}\)
where \(w_i^z = (u_i^z, v_i^z)\in \C^n\times \C^n\) denotes the eigenvector of \(H^z\) belonging to \(\lambda_i^z\).
However,  this overlap turns out to be small
if \(z\) and \(z'\) are far away and \(i\) is not too big. Thus the analysis of the microscopic regime
has two ingredients: (i) extending the coupling idea to driving Brownian motions whose
distributions are not identical but close to each other; and (ii) proving the smallness of the overlap.

While (i) can be achieved by relatively minor modifications to the  proofs  in~\cite{MR3916329}, (ii) requires
to develop a new type of local law. Indeed,
the overlap can be estimated in terms of traces of products of resolvents, \(\Tr G^z( \ii\eta) G^{z'}( \ii\eta')\) with \(\eta, \eta'\sim  n^{-1+\epsilon}\) in the mesoscopic regime.
Customary local laws, however, do not apply to a quantity involving \emph{products} of resolvents. In fact, even the leading deterministic
term needs to be identified by solving a new type of deterministic Dyson equation. We first show
the stability of this new equation using the lower bound on \(\abs{z-z'}\). Then we prove the necessary high probability  bound for the error term in the Dyson equation by a diagrammatic cumulant expansion adapted to the
new situation of product of resolvents. The key novelty is to extract the effect that \(G^z\) and \(G^{z'}\) are weakly correlated when
\(z\) and \(z'\) are far away from each other.

We close this section with an important  remark  concerning the proofs for  Hermitian versus non-Hermitian matrices.
Similarly to Girko's formula~\eqref{girko},  the linear eigenvalue statistics for Hermitian matrices are also
expressed by an integral of the resolvents over all spectral parameters.
However, in the corresponding Helffer-Sj\"ostrand formula, sufficient regularity of \(f\) directly neutralizes the potentially singular behaviour of the resolvent near the real axis, giving rise to CLT results even with suboptimal control on the resolvent in the mesoscopic regime. A similar trade-off in~\eqref{girko} is not apparent; it is unclear if and how the integration in \(z\) could help regularize the \(\eta\) integral.
This is a fundamental difference between CLTs for Hermitian and non-Hermitian ensembles that  explains the abundance of Hermitian results in contrast to the scarcity of available non-Hermitian CLTs.

\subsection*{Acknowledgement} L.E.\ would like to thank Nathana\"el Berestycki, and D.S.\ would like to thank Nina Holden for valuable discussions on the Gaussian free field. The authors are grateful to Peter Forrester for pointing out a missing term in \eqref{eq:expv} in the original manuscript. We thank Benjamin Landon for correcting a technical error  in the originally published version proof of Proposition~\ref{pro:ciala}: the BDG inequality
cannot be directly applied for the solution of~\eqref{difeq} in Duhamel form, instead the $\ell^2$-norm of
the solution can be controlled in a similar way, the statement of Proposition~\ref{pro:ciala}
is unchanged. The current arXiv version contains this correction.

\subsection*{Notations and conventions}
We introduce some notations we use throughout the paper. For integers \(k\in\N \) we use the notation \([k]:= \{1,\dots, k\}\). We write \(\HC \) for the upper half-plane \(\HC := \set{z\in\C \given \Im z>0}\), \(\DD\subset \C\) for the open unit disk, and for any \(z\in\C \) we use the notation \(\dif^2 z:= 2^{-1} \ii(\dif z\wedge \dif \overline{z})\) for the two dimensional volume form on \(\C \). For positive quantities \(f,g\) we write \(f\lesssim g\) and \(f\sim g\) if \(f \le C g\) or \(c g\le f\le Cg\), respectively, for some constants \(c,C>0\) which depend only on the constants appearing in~\eqref{eq:hmb}. For any two positive real numbers \(\omega_*,\omega^*\in\R _+\) by \(\omega_*\ll \omega^*\) we denote that \(\omega_*\le c \omega^*\) for some small constant \(0<c\le 1/100\). We denote vectors by bold-faced lower case Roman letters \({\bm x}, {\bm y}\in\C ^k\), for some \(k\in\N\). Vector and matrix norms, \(\norm{\vx}\) and \(\norm{A}\), indicate the usual Euclidean norm and the corresponding induced matrix norm. For any \(2n\times 2n\) matrix \(A\) we use the notation \(\braket{ A}:= (2n)^{-1}\Tr  A\) to denote the normalized trace of \(A\). Moreover, for vectors \({\bm x}, {\bm y}\in\C ^n\) and matrices \(A,B\in \C ^{2n\times 2n}\) we define
\[
    \braket{ {\bm x},{\bm y}}:= \sum \overline{x}_i y_i, \qquad \braket{ A,B}:= \braket{ A^*B}.
\]
We will use the concept of ``with very high probability'' meaning that for any fixed \(D>0\) the probability of the event is bigger than \(1-n^{-D}\) if \(n\ge n_0(D)\). Moreover, we use the convention that \(\xi>0\) denotes an arbitrary small constant which is independent of \(n\).

\section{Main results}\label{sec:MR}
We consider \emph{complex i.i.d.\ matrices} \(X\), i.e.\ \(n\times n\) matrices whose entries are independent and identically distributed as \smash{\(x_{ab}\stackrel{d}{=} n^{-1/2}\chi\)} for some complex random variable \(\chi\), satisfying the following:
\begin{assumption}\label{ass:1}
    We assume that \(\E \chi=\E \chi^2=0\) and \(\E \abs{\chi}^2=1\). In addition we assume the existence of high moments, i.e.\ that there exist constants \(C_p>0\), for any \(p\in\N \), such that
    \begin{equation}\label{eq:hmb}
        \E \abs{\chi}^p\le C_p.
    \end{equation}
\end{assumption}

The \emph{circular law}~\cite{MR1428519, MR863545, MR2663633, MR3813992, MR866352, MR773436, MR2575411, MR2409368} asserts that the empirical distribution of eigenvalues \(\{\sigma_i\}_{i=1}^n\) of a complex i.i.d.\ matrix \(X\) converges to the uniform distribution on the unit disk \(\DD\), i.e.
\begin{equation}\label{eq:circlaw}
    \lim_{n\to \infty}\frac{1}{n}\sum_{i=1}^n f(\sigma_i)=\frac{1}{\pi}\int_\DD f(z)\dif^2z,
\end{equation}
with very high probability for any continuous bounded function \(f\). Our main result is a central limit theorem for the centred \emph{linear statistics}
\begin{equation}\label{eq:linstatmainr}
    L_n(f):=\sum_{i=1}^n f(\sigma_i)-\E \sum_{i=1}^n f(\sigma_i)
\end{equation}
for general complex i.i.d.\ matrices and generic test functions \(f\).

In order to state the result we introduce some notations and certain Sobolev spaces. We fix some open bounded \(\Omega\subset\C\) containing the closed unit disk \(\ov{\DD}\subset \Omega\) and having a piecewise \(C^1\)-boundary, or, more generally, any boundary satisfying the \emph{cone property} (see e.g.~\cite[Section~8.7]{MR1817225}). We consider test functions \(f\in H^{2+\delta}_0(\Omega)\) in the Sobolev space
\(H^{2+\delta}_0(\Omega)\) which is defined as the completion of the smooth compactly supported functions \(C_c^\infty(\Omega)\) under the norm \[\norm{f}_{H^{2+\delta}(\Omega)}:= \norm{(1+\abs{\xi})^{2+\delta} \wh f(\xi)}_{L^2(\Omega)}\]
and we note that by Sobolev embedding such functions are continuously differentiable, and vanish at the boundary of \(\Omega\). For notational convenience we identify \(f\in H^{2+\delta}_0(\Omega)\) with its extension to all of \(\C\) obtained from setting \(f\equiv 0\) in \(\C\setminus\Omega\). We note that our results can trivially be extended to bounded test functions with non-compact support since due to~\cite[Theorem 2.1]{MR4408013}, with high probability, all eigenvalues satisfy \(\abs{\sigma_i}\le 1+\epsilon\) and therefore non-compactly supported test functions can simply be smoothly cut-off.  For \(h\) defined on the boundary of the unit disk \(\partial\DD \) we define its Fourier transform
\begin{equation}\label{eq:furtra}
    \wh{h}(k)=\frac{1}{2\pi}\int_0^{2\pi} h(e^{\ii\theta}) e^{-\ii \theta k}\dif\theta, \qquad k\in\Z .
\end{equation}
For \(f,g\in H_0^{2+\delta}(\Omega)\) we define the homogeneous semi-inner products
\begin{equation}\label{eq:h12norm}
    \begin{split}
        \braket{ g,f}_{\dot{H}^{1/2}(\partial\DD )}&:= \sum_{k\in\Z }\abs{k} \widehat{f}(k) \overline{\widehat{g}(k)}, \qquad \norm{ f}^2_{\dot{H}^{1/2}(\partial\mathbf{D})}:= \braket{ f,f}_{\dot{H}^{1/2}(\partial\DD )},
    \end{split}
\end{equation}
where, with a slight abuse of notation, we identified \(f\) and \(g\) with their restrictions to \(\partial \DD \).

\begin{theorem}[Central Limit Theorem for linear statistics]\label{theo:CLT}
    Let \(X\) be a complex \(n\times n\) i.i.d.\ matrix satisfying Assumption~\ref{ass:1} with eigenvalues \(\{\sigma_i\}_{i=1}^n\), and denote the fourth \emph{cumulant} of \(\chi\) by \(\kappa_4:= \E\abs{\chi}^4-2\). Fix \(\delta>0\), an open complex domain \(\Omega\) with \(\ov\DD\subset\Omega\subset\C\) and
    a complex valued test function \(f\in H_0^{2+\delta}(\Omega)\). Then the centred linear statistics \(L_n(f)\), defined in~\eqref{eq:linstatmainr}, converges
    \[L_n(f) \Longrightarrow L(f),  \]
    to a complex Gaussian random variable \(L(f)\) with expectation \(\E L(f)=0\) and variance \(\E \abs{L(f)}^2=C(f,f)=: V_f\) and \(\E L(f)^2 = C(\ov f, f)\), where
    \begin{equation}\label{eq:cov}
        \begin{split}
            C\left( g,f\right)&:= \frac{1}{4\pi}\braket{ \nabla g,\nabla f}_{L^2(\DD )}+\frac{1}{2}\braket{ g,f}_{\dot{H}^{1/2}(\partial\DD )} \\
            &\quad + \kappa_4 \left(\frac{1}{\pi}\int_\DD \overline{g(z)}\dif^2 z- \frac{1}{2\pi}\int_0^{2\pi} \overline{g(e^{\ii\theta})}\dif \theta\right) \\
            &\qquad\times\left(\frac{1}{\pi}\int_\DD f(z)\dif^2 z-\frac{1}{2\pi}\int_0^{2\pi} f(e^{\ii\theta})\dif\theta\right).
        \end{split}
    \end{equation}
    More precisely, any finite moment of \(L_n(f)\) converges at a rate \(n^{-c(k)}\), for some small \(c(k)>0\), i.e.
    \begin{equation}\label{eq moment convergence}
        \E L_n(f)^k \ov{L_n(f)}^l = \E L(f)^k \ov{ L(f)}^l +\mathcal{O}\left( n^{-c(k+l)}\right).
    \end{equation}
    Moreover, the expectation in~\eqref{eq:linstatmainr} is given by
    \begin{equation}\label{eq:expv}
        \E \sum_{i=1}^n f(\sigma_i)=\frac{n}{\pi}\int_\DD f(z)\dif^2 z+\frac{1}{8\pi}\int_\DD \Delta f(z)\, \dif^2z-\frac{\kappa_4}{\pi}\int_\DD f(z)(2\abs{z}^2-1)\dif^2 z+\mathcal{O}\left( n^{-c}\right)
    \end{equation}
    for some small constant \(c>0\). The implicit constants in the error terms in~\eqref{eq moment convergence}--\eqref{eq:expv} depend on the \(H^{2+\delta}\)-norm of \(f\) and  \(C_p\) from~\eqref{eq:hmb}.
\end{theorem}

\begin{remark}[\(V_f\) is strictly positive]\label{remark Vf pos}
    The variance \(V_f=\E \abs{L(f)}^2\) in Theorem~\ref{theo:CLT} is strictly positive. Indeed, by the Cauchy-Schwarz inequality it follows that
    \[
        \abs*{\frac{1}{\pi}\int_\mathbf{D} f(z)\, \dif ^2z-\frac{1}{2\pi}\int_0^{2\pi} f(e^{i\theta})\, d\theta}^2\le\frac{1}{8\pi} \int_\mathbf{D} \abs*{ \nabla f}^2\, \dif^2 z.
    \]
    Hence, since \(\kappa_4\ge -1\) in~\eqref{eq:cov}, this shows that
    \[
        V_f\ge\frac{1}{8\pi} \int_\mathbf{D} \abs*{ \nabla f}^2\, \dif^2 z + \frac{1}{2}\norm{ f}^2_{\dot{H}^{1/2}(\partial\mathbf{D})}>0.
    \]
\end{remark}

By polarisation, a multivariate Central Limit Theorem readily follows from Theorem~\ref{theo:CLT}:
\begin{corollary}\label{cor:multCLT}
    Let \(X\) be an \(n\times n\) i.i.d.\ complex matrix satisfying Assumption~\ref{ass:1}, and let \(L_n(f)\) be defined in~\eqref{eq:linstatmainr}. For a fixed open bounded complex domain \(\Omega\) with \(\ov\DD\subset\Omega\subset\C\), \(\delta>0\), \(p\in\N \) and for any finite collection of test functions \(f^{(1)},\dots,f^{(p)} \in H_0^{2+\delta}(\Omega)\) the vector
    \begin{equation}\label{eq:muclt}
        (L_n(f^{(1)}),\dots,L_n(f^{(p)}))\Longrightarrow (L(f^{(1)}),\dots,L(f^{(p)})),
    \end{equation}
    converges to a multivariate complex Gaussian of zero expectation \(\E L(f)=0\) and covariance \(\E L(f) \ov{L(g)}=\E L(f) L(\ov{g})=C(f,g) \) with \(C\) as in~\eqref{eq:cov}. Moreover, for any mixed \(k\)-moments we have an effective convergence rate of order \(n^{-c(k)}\), as in~\eqref{eq moment convergence}
\end{corollary}

\begin{remark}\label{rem:compo}
    We may compare Theorem~\ref{theo:CLT} with the previous results in~\cite[Eq. (5.18)]{MR2972857}, \cite[Theorem 1]{MR2361453} and~\cite[Theorem 1.1]{MR2294978}:
    \begin{enumerate}[label=(\roman*)]

        \item For $\kappa_4=0$ the expansion \eqref{eq:expv} agrees with the subleading order correction to the circular law from~\cite[Eq. (5.18)]{MR2972857} (see also \cite[Eq. (5.16)]{MR2240466} and \cite[Eq. (1.14)]{MR3735628}).
        \item Note that for a single \(f\colon\C \to \R \) in the Ginibre case, i.e.\ \(\kappa_4=0\), Theorem~\ref{theo:CLT} implies~\cite[Theorem 1]{MR2361453} with \(\sigma_f^2+ \widetilde{\sigma}_f^2=C(f,f)\), using the notation therein and with \(C(f,f)\) defined in~\eqref{eq:cov}.
        \item If additionally \(f\) is complex analytic in a neighbourhood of \(\overline{\DD }\), using the notation \(\partial:= \partial_z\), the expressions in~\eqref{eq:cov},\eqref{eq:expv} of Theorem~\ref{theo:CLT} simplify to
              \begin{equation}\label{eq:comr}
                  \E  \sum_{i=1}^n f(\sigma_i)=nf(0)+\mathcal{O}\left(n^{-\delta'}\right), \quad C\left(f,g\right)= \frac{1}{\pi}\int_\DD  \partial f(z) \overline{\partial g(z)}\dif^2 z,
              \end{equation}
              where we used that for any \(f,g\) complex analytic in a neighbourhood of \(\overline{\DD }\) we have
              \begin{equation}
                  \label{eq:chan}
                  \frac{1}{2\pi}\int_\DD  \braket{ \nabla g,\nabla f }\dif^2 z=\frac{1}{\pi}\int_{\DD } \partial f(z)\overline{\partial g(z)} \dif^2 z=\sum_{k\in \Z } \abs{k} \widehat{f\restriction_{\partial \DD }}(k) \overline{\widehat{g\restriction_{\partial \DD }}(k)},
              \end{equation}
              and that
              \[
                  \frac{1}{\pi}\int_\DD  f(z)\dif^2 z=\frac{1}{2\pi}\int_0^{2\pi} f(e^{i\theta})\dif \theta=f(0).
              \]
              The second equality in~\eqref{eq:chan} follows by writing \(f\) and \(g\) in Fourier series. The result in~\eqref{eq:comr} exactly agrees  with~\cite[Theorem 1.1]{MR2294978}.
    \end{enumerate}
\end{remark}

\begin{remark}[Mesoscopic regime]\label{rem:meso}
    We formulated our result for \emph{macroscopic} linear statistics, i.e.\ for test functions \(f\) that are independent of \(n\). One may also consider \emph{mesoscopic}
    linear statistics as well when \(f(\sigma)\) is replaced with  \( \varphi( n^a(\sigma-z_0))\) for some  fixed scale \(a>0\), reference point \(z_0\in \DD\)
    and function \(\varphi \in H^{2+\delta}(\C)\). Our proof  can directly  handle this situation as well for any  small \(a\le 1/500\)\footnote{The upper bound \(1/500\) for \(a\) is a crude overestimate, we did not optimise it along the proof. The actual value of \(a\) comes from the fact that it has to be smaller than \(\omega_d\) (see of Proposition~\ref{prop:indmr}) and from Lemma~\ref{lem:overb} (which is the main input of Proposition~\ref{prop:indmr}) it follows that \(\omega_d\le 1/100\).}, say,  since all our error terms
    are effective as a small power of \(1/n\). For \(a>0\) the leading term to the variance \(V_f\) comes solely from the \(\norm{\nabla f}^2\) term in~\eqref{eq:cov}, in particular the effect of the fourth cumulant is negligible.
\end{remark}

\subsection{Connection to the Gaussian free field}\label{sec GFF}
It has been observed in~\cite{MR2361453} that for the special case \(\kappa_4=0\) the limiting random field \(L(f)\) can be viewed as a variant of the \emph{Gaussian free field}~\cite{MR2322706}. The Gaussian free field on some bounded domain \(\Omega\subset\C\) can formally be defined as a \emph{Gaussian Hilbert space} of random variables \(h(f)\) indexed by functions in the homogeneous Sobolev space \(f\in \dot H_0^1(\Omega)\) such that the map \(f\mapsto h(f)\) is linear and
\begin{equation}
    \E h(f) = 0 ,\quad \E \ov{h(f)} h(g) = \braket{f,g}_{\dot H^1(\Omega)}.
\end{equation}
Here for \(\Omega\subset\C\) we defined the homogeneous Sobolev space \(\dot H_0^1(\Omega)\) as the completion of smooth compactly supported function \(C_c^\infty(\Omega)\) with respect to the semi-inner product
\[\braket{g,f}_{\dot{H}^{1}(\Omega)}:= \braket{\nabla g,\nabla f}_{L^2(\Omega)}, \qquad \norm{f}_{\dot{H}^{1}(\Omega)}^2:= \braket{f,f}_{\dot{H}^{1}(\Omega)}.\]
By the Poincar\'e inequality the space \(\dot H_0^1(\Omega)\) is in fact a Hilbert space and as a vector space coincides with the usual Sobolev space \(H_0^1(\Omega)\) with an equivalent norm but a different scalar product.

Since \(\ov{\DD}\subset\Omega\), the Sobolev space  \(\dot H^1_0(\Omega)\) can be orthogonally decomposed as
\[ \dot H_0^1(\Omega) = \dot H_{0}^1(\DD) \oplus  \dot H_{0}^1(\ov\DD^c)  \oplus  \dot H_0^1((\partial\DD)^c)^\perp,\]
where the complements are understood as the complements within \(\Omega\).
The orthogonal complement \smash{\(\dot H_0^1((\partial\DD)^c)^\perp\)} is (see e.g.~\cite[Thm.~2.17]{MR2322706}) given by the closed subspace of functions which are harmonic in \(\DD\cup\ov\DD^c=(\partial\DD)^c\), i.e.\ away from the unit circle.
For closed subspaces \(S\subset \dot H^1_0(\Omega)\) we denote the orthogonal projection onto \(S\) by \(P_S\).
Then by orthogonality and conformal symmetry it follows~\cite[Lemma 3.1]{MR2361453}\footnote{In Eq.~(3.1), and in the last displayed equation of the proof of Lemma 3.1 factors of \(2\) are missing. In the notation of~\cite{MR2361453} the correct equations read \[\frac{1}{2}\norm{P_{H}f}_{H^1(\C)}^2 = \norm{P_H f}_{H^1(\mathbf{U})}^2 = 2\pi \norm{f}_{H^{1/2}(\partial\mathbf U)}^2 \quad \text{and}\quad \braket{g_1,g_2}_{H^1(\mathbf U)}=2\pi\braket{g_1,g_2}_{H^{1/2}(\partial\mathbf U)}.\]} that
\begin{gather}
    \begin{aligned}
        \norm*{ P_{ \dot H_{0}^1(\DD)} f + P_{ \dot H_0^1((\partial\DD)^c)^\perp}  f}_{ \dot H^1(\Omega)}^2 & = \norm{ f }^2_{ \dot H^1(\DD)} + \norm{  P_{ \dot H_0^1((\partial\DD)^c)^\perp} f}_{ \dot H^1(\DD)}^2 \\
                                                                                                            & = \norm{ f }^2_{ \dot H^1(\DD)} + 2\pi \norm{f }_{\dot H^{1/2}(\partial\DD)}^2,
    \end{aligned}\label{eq f decomp}\raisetag{-4em}
\end{gather}
where we canonically identify \(f\in \dot H_0^1(\Omega)\) with its restriction to \(\DD\). If \(\kappa_4=0\), then the rhs.\ of~\eqref{eq f decomp} is precisely \(4\pi C(f,f)\) and therefore \(L(f)\) can be interpreted~\cite[Corollary 1.2]{MR2361453} as the projection
\begin{equation} \label{eq L h proj}
    L = (4\pi)^{-1/2} P h, \qquad P:= \Bigl(P_{ \dot H_{0}^1(\DD)} + P_{ \dot H_0^1((\partial\DD)^c)^\perp}  \Bigr)
\end{equation}
of the Gaussian free field \(h\) onto \( \dot H_{0}^1(\DD)  \oplus  \dot H_0^1((\partial\DD)^c)^\perp\), i.e.\ the Gaussian free field conditioned to be harmonic in \(\DD^c\). The projection~\eqref{eq L h proj} is defined via duality, i.e.\ \((Ph)(f) := h(Pf)\)
so that indeed
\[ \E\abs*{\left[\frac{1}{\sqrt{4\pi}} P h\right](f) }^2 = \frac{1}{4\pi} \Bigl( \norm{ f }^2_{ \dot H^1(\DD)} + 2\pi \norm{f }_{\dot H^{1/2}(\partial\DD)}^2 \Bigr) = C(f,f) = \E \abs{L(f)}^2.\]
If \(\kappa_4> 0\), then \(L\) can be interpreted as the sum
\begin{equation}\label{eq L kappa4 pos}
    L = \frac{1}{\sqrt{4\pi}} Ph + \sqrt{\kappa_4} \Bigl(\braket{\cdot}_{\DD}-\braket{\cdot}_{\partial\DD}\Bigr) \Xi
\end{equation}
of the Gaussian free field \(Ph\) conditioned to be harmonic in \(\DD^c\), and an independent standard real Gaussian \(\Xi\) multiplied by difference of the averaging functionals \(\braket{\cdot}_\DD\), \(\braket{\cdot}_{\partial\DD}\) on \(\DD\) and \(\partial\DD\). For \(\kappa_4<0\) there seems to be no direct interpretation of \(L\) similar to~\eqref{eq L kappa4 pos}.

\section{Proof strategy}
For the proof of Theorem~\ref{theo:CLT} we study the \(2n \times 2n\) matrix \(H^z\) defined in~\eqref{eq:linz1}, that is the Hermitisation of \(X-z\). Denote by \(\{\lambda^z_{\pm i}\}_{i=1}^n\) the eigenvalues of \(H^z\) labelled in an increasing order (we omit the index \(i=0\) for notational convenience). As a consequence of the block structure of \(H^z\) its spectrum is symmetric with respect to zero, i.e.\ \(\lambda^z_{-i}=-\lambda^z_i\) for any \(i\in [n]\).

Let \(G(w)=G^z(w):= (H^z-w)^{-1}\) denote the resolvent of \(H^z\) with \(\eta=\Im w\ne 0\). It is well known (e.g.\ see~\cite{MR3770875, MR4408013}) that \(G^z\) becomes approximately deterministic, as \(n\to \infty\), and its limit is expressed via the unique solution of the scalar equation
\begin{equation}\label{eq m}
    - \frac{1}{m^z} = w + m^z -\frac{\abs{z}^2}{w + m^z}, \quad \eta\Im m^z(w) >0,\quad \eta=\Im w\ne 0,
\end{equation}
which is a special case of the \emph{matrix Dyson equation} (MDE), see e.g.~\cite{MR3916109}. We note that on the imaginary axis \(m^z(\ii\eta)=\ii\Im m^z(\ii\eta)\). To find the limit of \(G^z\) we define a \(2n\times 2n\) block-matrix
\begin{equation}\label{eq M}
    M^z(w) := \begin{pmatrix}
        m^z(w) & -z u^z(w) \\ - \ov z u^z(w) & m^z(w)
    \end{pmatrix}, \quad u^z(w):=  \frac{m^z(w)}{w+ m^z(w)},
\end{equation}
where each block is understood to be a scalar multiple of the \(n\times n\) identity matrix. We note that \(m,u,M\) are uniformly bounded in \(z,w\), i.e.
\begin{equation}\label{eq M bound}
    \norm{M^z(w)}+\abs{m^z(w)}+\abs{u^z(w)}\lesssim 1.
\end{equation}
Indeed, taking the imaginary part of~\eqref{eq m} we have (dropping \(z, w\))
\begin{equation}\label{beta ast def}
    \beta_\ast \Im m = (1-\beta_\ast) \Im w,\qquad \beta_\ast := 1-\abs{m}^2-\abs{u}^2\abs{z}^2,
\end{equation}
which implies
\begin{equation}\label{mubound}
    \abs{m}^2 + \abs{u}^2\abs{z}^2< 1,
\end{equation}
as \(\Im m\) and \(\Im w\) have the same sign.  Note that~\eqref{mubound} saturates if \(\Im w\to 0\) and \(\Re w\) is in the support
of the \emph{self-consistent density of states}, \(\rho^z(E):=\pi^{-1} \Im m^z(E+\ii 0)\). Moreover,~\eqref{eq m} is equivalent to \(u= -m^2 + u^2 \abs{z}^2\), thus \(\abs{u}<1\) and~\eqref{eq M bound} follows.

For our analysis the derivative \(m'(w)\) in the \(w\)-variable plays a central role and we note that by taking the derivative of~\eqref{eq m} we obtain
\begin{equation}\label{beta def}
    m' = \frac{1-\beta}{\beta}, \qquad \beta:= 1-m^2-u^2 \abs{z}^2.
\end{equation}
On the imaginary axis, \(w=\ii\eta\), where by taking the real part of~\eqref{eq m} it follows that \(\Re m(\ii\eta)=0\), we can use~\cite[Eq.~(3.13)]{MR4408013}
\begin{equation}
    \label{eq:expm}
    \Im m(\ii\eta)\sim \begin{cases}
        \eta^{1/3}+\abs{1-\abs{z}^2}^{1/2}  & \text{if}\quad \abs{z}\le 1, \\
        \frac{\eta}{\abs{z}^2-1+\eta^{2/3}} & \text{if}\quad \abs{z}> 1,
    \end{cases},\qquad \eta\lesssim 1
\end{equation}
to obtain asymptotics for
\begin{equation}
    \label{eq:bbou}
    \beta_\ast \sim \frac{\eta}{\Im m}, \quad \beta = \beta_\ast + 2 (\Im m)^2, \qquad \eta\lesssim 1.
\end{equation}

The optimal local law from Theorem~\cite[Theorem 5.2]{MR3770875} and~\cite[Theorem 5.2]{MR4408013}\footnote{The local laws in~\cite[Theorem 5.2]{MR3770875} and~\cite[Theorem 5.2]{MR4408013} have been proven for \(\eta\ge \eta_f(z)\), with \(\eta_f(z)\) being the fluctuation scale defined in~\cite[Eq.~(5.2)]{MR4408013}, but they can be easily extend to any \(\eta>0\) by a standard argument, see~\cite[Appendix A]{MR4221653}.},
which for the application in Girko's formula~\eqref{girko} is only needed on the imaginary axis, asserts that \(G^z\approx M^z\) in the following sense:
\begin{theorem}[Optimal local law for \(G\)]\label{theo:Gll}
    The resolvent \(G^z\) is very well approximated by the deterministic matrix \(M^z\) in the sense
    \begin{equation}\label{single local law}
        \abs{\braket{(G^z(\ii \eta)-M^z(\ii \eta))A}} \le \frac{\norm{ A} n^\xi }{n\eta}, \qquad \abs{\braket{\vx,(G^z(\ii \eta)-M^z(\ii \eta))\vy}}\le \frac{\norm{\vx}\norm{\vy}n^\xi}{\sqrt{n\eta}},
    \end{equation}
    with very high probability, uniformly for \(\eta>0\) and for any deterministic matrices and vectors \(A,\vx,\vy\).
\end{theorem}

The matrix \(H^z\) can be related to the linear statistics of eigenvalues \(\sigma_i\) of \(X\) via the precise (regularised) version of Girko's Hermitisation formula~\eqref{girko}
\begin{gather}
    \begin{aligned}
        L_n(f) & =\frac{1}{4\pi} \int_\C  \Delta f(z) \Big[\log\abs{\det (H^z-\ii T)}-\E  \log \abs{\det (H^z-\ii T)}\Big]\dif^2 z                                                                                \\
               & \quad-\frac{n}{2\pi \ii} \int_\C  \Delta f \left[\left(\int_0^{\eta_0}+\int_{\eta_0}^{\eta_c}+\int_{\eta_c}^T \right) \bigl[\braket{ G^z(\ii\eta)-\E G^z(\ii\eta)}\bigr]\dif \eta\right] \dif^2z \\
               & =:  J_T+I_0^{\eta_0}+I_{\eta_0}^{\eta_c}+I_{\eta_c}^T,
    \end{aligned}\label{eq:GirkosplitA}\raisetag{-5em}
\end{gather}
for
\begin{equation}\label{eq:scales}
    \eta_0:= n^{-1-\delta_0}, \quad \eta_c:= n^{-1+\delta_1},
\end{equation}
and some very large \(T>0\), say \(T=n^{100}\). Note that in~\eqref{eq:GirkosplitA} we used that \(\braket{ G^z(\ii\eta)}=\ii\braket{ \Im G^z(\ii\eta)}\) by spectral symmetry. The test function \(f\colon\C \to \C \) is in \(H^{2+\delta}\) and it is compactly supported. \(J_T\) in~\eqref{eq:GirkosplitA} consists of the first line in the rhs., whilst \(I_0^{\eta_0},I_{\eta_0}^{\eta_c},I_{\eta_c}^T\) corresponds to the three different \(\eta\)-regimes in the second line of the rhs.\ of~\eqref{eq:GirkosplitA}.

\begin{remark}
    We remark that in~\eqref{eq:GirkosplitA} we split the \(\eta\)-regimes in a different way compared to~\cite[Eq.~(32)]{MR4221653}. We also use a different notation to identify the \(\eta\)-scales: here we use the notation \(J_T, I_0^{\eta_0}, I_{\eta_0}^{\eta_c}, I_{\eta_c}^T\), whilst in~\cite[Eq.~(32)]{MR4221653}
    we used the notation \(I_1, I_2, I_3, I_4\).
\end{remark}

The different regimes in~\eqref{eq:GirkosplitA} will be treated using different techniques. More precisely, the integral \(J_T\) is easily estimated as in~\cite[Proof of Theorem 2.3]{MR4408013}, which uses similar computations to~\cite[Proof of Theorem 2.5]{MR3770875}. The term \(I_0^{\eta_0}\) is estimated using the fact that with high probability there are no eigenvalues in the regime \([0,\eta_0]\); this follows by~\cite[Theorem 3.2]{MR2684367}. Alternatively (see Remark~\ref{rem:2ass} and Remark~\ref{rem:altern} later), the contribution of the regime \(I_0^{\eta_0}\) can be estimated without resorting to the quite sophisticated proof of~\cite[Theorem 3.2]{MR2684367} if the entries of \(X\) satisfy the additional assumption~\eqref{eq:addass}. More precisely, this can be achieved using~\cite[Proposition 5.7]{MR3770875} (which follows adapting the proof of~\cite[Lemma 4.12]{MR2908617}) to bound the very small regime \([0,n^{-l}]\), for some large \(l\in\N \), and then using~\cite[Corollary 4]{1908.01653} to bound the regime \([n^{-l},\eta_0]\).

The main novel work is done for the integrals \(I_{\eta_0}^{\eta_c}\) and \(I_{\eta_c}^T\). The main contribution to \(L_n(f)\) comes from the mesoscopic regime in \(I_{\eta_c}^T\), which is analysed using the following Central Limit Theorem for resolvents.

\begin{proposition}[CLT for resolvents]\label{prop:CLTresm}
    Let \(\epsilon,\xi>0\) be arbitrary. Then for \(z_1,\dots,z_p\in\C\) and \(\eta_1,\dots,\eta_p\ge n^{\xi-1} \max_{i\ne j}\abs{z_i-z_j}^{-2}\), denoting the pairings on \([p]\) by \(\Pi_p\), we have
    \begin{gather}
        \begin{aligned}
            \E\prod_{i\in[p]} \braket{G_i-\E G_i} & = \sum_{P\in \Pi_p}\prod_{\{i,j\}\in P} \E  \braket{G_i-\E G_i}\braket{G_j-\E G_j}  + \mathcal{O}\left(\Psi\right)     \\
                                                  & = \frac{1}{n^p}\sum_{P\in \Pi_p}\prod_{\{i,j\}\in P} \frac{V_{i,j}+\kappa_4 U_i U_j}{2}+ \mathcal{O}\left(\Psi\right),
        \end{aligned}\label{eq CLT resovlent}\raisetag{-5em}
    \end{gather}
    where \(G_i=G^{z_i}(\ii \eta_i)\),
    \begin{equation}\label{eq psi error}
        \Psi:= \frac{n^\epsilon}{(n\eta_*)^{1/2}}\frac{1}{\min_{i\ne j}\abs{z_i-z_j}^4}\prod_{i\in[p]}\frac{1}{\abs{1-\abs{z_i}}n\eta_i},
    \end{equation}
    \(\eta_*:= \min_i\eta_i\), and \(V_{i,j}=V_{i,j}(z_i,z_j,\eta_i,\eta_j)\) and \(U_i=U_i(z_i,\eta_i)\) are defined as
    \begin{equation}
        \label{eq:exder}
        \begin{split}
            V_{i,j}&:= \frac{1}{2}\partial_{\eta_i}\partial_{\eta_j} \log \bigl[ 1+(u_i u_j\abs{z_i}\abs{z_j})^2-m_i^2 m_j^2-2u_i u_j\Re z_i\overline{z_j}\bigr], \\
            U_i&:= \frac{\ii}{\sqrt{2}}\partial_{\eta_i} m_i^2,
        \end{split}
    \end{equation}
    with \(m_i=m^{z_i}(\ii\eta_i)\) and \(u_i=u^{z_i}(\ii\eta_i)\).

    Moreover, the expectation of \(G\) is given by
    \begin{equation}\label{prop clt exp}
        \braket{\E G}= \braket{M} - \frac{\ii\kappa_4}{4n}\partial_\eta(m^4) + \cO\Bigl(\frac{1}{\abs{1-\abs{z}}n^{3/2} (1+\eta)}+\frac{1}{\abs{1-\abs{z}}(n\eta)^2}\Bigr).
    \end{equation}
\end{proposition}

\begin{remark}
    In Section~\ref{sec:PCLT} we will apply this proposition in the regime where \(\min_{i\ne j}\abs{z_i-z_j}\) is quite large, i.e.\ it is at least \(n^{-\delta}\), for some small \(\delta>0\), hence we did not optimise the estimates for the opposite regime. However, using the more precise~\cite[Lemma 6.1]{MR4235475} instead of Lemma~\ref{lemma:betaM} within the proof, one can immediately strengthen Proposition~\ref{prop:CLTresm} on two accounts. First, the condition on \(\eta_*=\min\eta_i\) can be relaxed to
    \[\eta_*\gtrsim n^{\xi-1} \Bigl(\min_{i\ne j} \abs{z_i-z_j}^2 + \eta_*\Bigr)^{-1}.\]
    Second, the denominator \(\min_{i\ne j} \abs{z_i-z_j}^4\) in~\eqref{eq psi error} can be improved to
    \[\Bigl(\min_{i\ne j} \abs{z_i-z_j}^2 +\eta_*\Bigr)^2.\]
\end{remark}

In order to show that the contribution of \(I_{\eta_0}^{\eta_c}\) to \(L_n(f)\) is negligible, in Proposition~\ref{prop:indmr} we prove that \(\braket{  G^{z_1}(\ii\eta_1)}\) and \(\braket{  G^{z_2}(\ii\eta_2)}\) are asymptotically independent if \(z_1\), \(z_2\) are far enough from each other, they are well inside \(\mathbf{D}\), and \(\eta_0 \le \eta_1, \eta_2 \le \eta_c\).

\begin{proposition}[Independence of resolvents with small imaginary part]\label{prop:indmr}
    Fix \(p\in \mathbf{N}\). For any sufficiently small \(\omega_d,\omega_h,\omega_f>0\) such that \(\omega_h\ll \omega_f\), there exist \(\omega,\widehat{\omega}, \delta_0,\delta_1>0\) such that \(\omega_h\ll \delta_m\ll \widehat{\omega}\ll \omega\ll\omega_f\), for \(m=0,1\), such that for any \(\abs{z_l}\le 1-n^{-\omega_h}\), \(\abs{z_l-z_m}\ge n^{-\omega_d}\), with \(l,m \in [p]\), \(l\ne m\), it holds
    \begin{equation}
        \label{eq:indtrlm}
        \E \prod_{l=1}^p \braket{ G^{z_l}(\ii\eta_l)}=\prod_{l=1}^p\E  \braket{ G^{z_l}(\ii\eta_l)}+\mathcal{O}\left(\frac{n^{p(\omega_h+\delta_0)+\delta_1}}{n^{\omega}}+\frac{n^{\omega_f+3\delta_0}}{\sqrt{n}}\right),
    \end{equation}
    for any \(\eta_1,\dots,\eta_p\in [n^{-1-\delta_0},n^{-1+\delta_1}]\).
\end{proposition}

The paper is organised as follows: In Section~\ref{sec:PCLT} we conclude Theorem~\ref{theo:CLT} by combining Propositions~\ref{prop:CLTresm} and~\ref{prop:indmr}. In Section~\ref{sec local law G2} we prove a local law for \(G_1A G_2\), for a deterministic matrix \(A\). In Section~\ref{sec:CLTres}, using the result in Section~\ref{sec local law G2} as an input, we prove Proposition~\ref{prop:CLTresm}, the Central Limit Theorem for resolvents. In Section~\ref{sec:IND} we prove Proposition~\ref{prop:indmr} using the fact that the correlation among small eigenvalues of \(H^{z_1}\), \(H^{z_2}\) is ``small'', if \(z_1\), \(z_2\) are far from each other, as a consequence of the local law in Section~\ref{sec local law G2}.

\section{Central limit theorem for linear statistics}\label{sec:PCLT}
In this section, using Proposition~\ref{prop:CLTresm}--\ref{prop:indmr} as inputs, we prove our main result Theorem~\ref{theo:CLT}.

\subsection{Preliminary reductions in Girko's formula}
In this section we prove that the main contribution to \(L_n(f)\) in~\eqref{eq:GirkosplitA} comes from the regime \(I_{\eta_c}^T\). This is made rigorous in the following lemma.

\begin{lemma}\label{lem:i4}
    Fix \(p\in \N \) and some bounded open \(\ov\DD\subset\Omega\subset\C\), and for any \(l\in [p]\) let \(f^{(l)}\in H_0^{2+\delta}(\Omega)\). Then
    \begin{equation}
        \label{eq:allrsma}
        \E \prod_{l=1}^p L_n\bigl(f^{(l)}\bigr)= \E\prod_{l=1}^p I_{\eta_c}^T\bigl( f^{(l)}\bigr)+\mathcal{O}\left( n^{-c(p)}\right),
    \end{equation}
    for some small \(c(p)>0\), with \(L_n(f^{(l)})\) and \(I_{\eta_c}^T( f^{(l)})\) defined in~\eqref{eq:GirkosplitA}. The constant in \(\mathcal{O}(\cdot)\) may depend on \(p\) and on the \(L^2\)-norm of \(\Delta f^{(1)},\dots, \Delta f^{(p)}\).
\end{lemma}

\begin{remark}\label{rem:2ass}
    In the remainder of this section we need to ensure that with high probability the matrix \(H^z\), defined in~\eqref{eq:linz1}, does not have eigenvalues very close to zero, i.e.\ that
    \begin{equation}
        \label{eq:exversmall}
        \Prob \left(\Spec(H^z)\cap \left[ -n^{-l},n^{-l}\right] \ne \emptyset\right)\le C_l n^{-l/2},
    \end{equation}
    for any \(l\ge 2\) uniformly in \(\abs{z}\le 1\). The bound~\eqref{eq:exversmall} directly follows from~\cite[Theorem 3.2]{MR2684367}. Alternatively,~\eqref{eq:exversmall} follows by~\cite[Proposition 5.7]{MR3770875} (which follows adapting the proof of~\cite[Lemma 4.12]{MR2908617}), without recurring to the quite sophisticated proof of~\cite[Theorem 3.2]{MR2684367}, under the additional assumption that
    there exist \(\alpha, \beta>0\) such that the random variable \(\chi\) has a density \(g\colon\C \to \interval{co}{0,\infty}\) which satisfies
    \begin{equation}
        \label{eq:addass}
        g\in L^{1+\alpha}(\C ), \qquad \norm{ g}_{L^{1+\alpha}(\C )}\le n^\beta.
    \end{equation}
\end{remark}

We start proving \emph{a priori} bounds for the integrals defined in~\eqref{eq:GirkosplitA}.

\begin{lemma}\label{lem:aprior}
    Fix some bounded open \(\ov\DD\subset\Omega\subset \C\) and let \(f\in H_0^{2+\delta}(\Omega)\). Then for any \(\xi>0\) the bounds
    \begin{equation}\label{eq:apb}
        \abs{J_T}\le \frac{n^{1+\xi}\norm{ \Delta f}_{L^1(\Omega)}}{T^2}, \qquad \abs*{I_0^{\eta_0}}+\abs*{I_{\eta_0}^{\eta_c}}+\abs{I_{\eta_c}^T}\le n^\xi \norm{ \Delta f}_{L^2(\Omega)} \abs{\Omega}^{1/2},
    \end{equation}
    hold with very high probability, where \(\abs{\Omega}\) denotes the Lebesgue measure of the set \(\Omega\).
\end{lemma}
\begin{proof}
    The proof of the bound for \(J_T\) is identical to~\cite[Proof of Theorem 2.3]{MR4408013} and so omitted.

    The bound for \(I_0^{\eta_0}, I_{\eta_0}^{\eta_c}, I_{\eta_c}^T\) relies on the local law of Theorem~\ref{theo:Gll}. More precisely, by Theorem~\ref{theo:Gll} and~\eqref{prop clt exp} of Proposition~\ref{prop:CLTresm} it follows that
    \begin{equation}
        \label{eq:llexp}
        \abs*{\braket{ G^z-\E  G^z}}\le \frac{n^\xi}{n\eta},
    \end{equation}
    with very high probability uniformly in \(\eta>0\) and \(\abs{z}\le C\) for some large \(C>0\). First of all we remove the regime \([0,n^{-l}]\)  by~\cite[Theorem 3.2]{MR2684367}, i.e.\ its contribution is smaller than \(n^{-l}\), for some large \(l\in\N \), with very high probability. Alternatively, this can be achieved by~\cite[Proposition 5.7]{MR3770875} under the additional assumption~\eqref{eq:addass} in Remark~\ref{rem:2ass}. Then for any \(a,b\ge n^{-l}\), by~\eqref{eq:llexp}, we have
    \begin{equation}
        \label{eq:impbbfin}
        n\abs*{\int_\Omega \dif^2 z \Delta f(z)\int_a^b \dif \eta \bigl[\braket{ G(\ii\eta) -\E  G(\ii\eta)}\bigr] }\lesssim n^\xi \abs{\Omega}^{1/2} \norm{ \Delta f}_{L^2(\Omega)},
    \end{equation}
    with very high probability. This concludes the proof of the second bound in~\eqref{eq:apb}.
\end{proof}

We have a better bound for \(I_0^{\eta_0}\), \(I_{\eta_0}^{\eta_c}\) which holds true in expectation.

\begin{lemma}\label{lem:bbexp}
    Fix some bounded open \(\ov\DD\subset\Omega\subset \C\) and let \(f\in H_0^{2+\delta}(\Omega)\). Then there exists \(\delta'>0\) such that
    \begin{equation}
        \label{eq:impexb}
        \E \abs*{I_0^{\eta_0}}+\E \abs*{I_{\eta_0}^{\eta_c}}\le n^{-\delta'}\norm{ \Delta f}_{L^2(\Omega)}.
    \end{equation}
\end{lemma}

\begin{proof}[Proof of  Lemma~\ref{lem:i4}]
    Lemma~\ref{lem:i4} readily follows (see e.g.~\cite[Lemma 4.2]{MR4221653}) combining Lemma~\ref{lem:aprior} and Lemma~\ref{lem:bbexp}.
\end{proof}

We conclude this section with the proof of Lemma~\ref{lem:bbexp}.

\begin{proof}[Proof of Lemma~\ref{lem:bbexp}]
    The bound for \(\E \abs*{I_0^{\eta_0}}\) immediately follows by~\cite[Theorem 3.2]{MR2684367} (see also Remark~\ref{rem:altern} for an alternative proof).

    By the local law outside the spectrum, given in the second part of~\cite[Theorem 5.2]{MR4408013}, it follows that for \(0<\gamma<1/2\) we have
    \begin{equation}
        \label{eq:betll}
        \abs*{\braket{  G^z(\ii\eta)- M^z(\ii \eta)}} \le \frac{n^\xi}{n^{1+\gamma/3}\eta},
    \end{equation}
    uniformly for all \(\abs{z}^2\ge 1+(n^\gamma \eta)^{2/3}+n^{(\gamma-1)/2}\), \(\eta>0\), and \(\abs{z}\le 1+\tau^*\), for some \(\tau^*\sim 1\). We remark that the local law~\eqref{eq:betll} was initially proven only for \(\eta\) above the fluctuation scale \(\eta_f(z)\), which is defined in~\cite[Eq.~(5.2)]{MR4408013}, but it can be easily extend to any \(\eta>0\) using the monotonicity of the function \(\eta \mapsto \eta \braket{ \Im G(i\eta)}\) and the fact that
    \begin{equation}
        \label{eq:detb}
        \abs*{n^\xi \eta_f(z)\braket{  M^z(\ii n^\xi\eta_f(z)) }} +\abs*{ \eta \braket{ M^z(\ii \eta) }}\lesssim n^{2\xi} \frac{\eta_f(z)^2}{\abs{z}^2-1},
    \end{equation}
    uniformly in \(\eta>0\), since \(\Im M^z(\ii\eta)=\Im m^z(\ii\eta) I\) by~\eqref{eq M}, with \(I\) the \(2n\times 2n\) identity matrix, and \(\Im m^z(\ii\eta)\le \eta(\abs{z}^2-1)^{-1}\) by~\cite[Eq.~(3.13)]{MR4408013}. Note that we assumed the additional term \(n^{(\gamma-1)/2}\) in the lower bound for \(\abs{z}^2\) compared with~\cite[Theorem 5.2]{MR4408013} in order to ensure that the rhs.\ in~\eqref{eq:detb}, divided by \(\eta\), is smaller than the error term in~\eqref{eq:betll}.

    Next, in order to bound \(\E \abs{I_{\eta_0}^{\eta_c}}\), we consider
    \begin{align}
        \label{eq:bound2}
        \E & \abs{I_{\eta_0}^{\eta_c}}^2={}-\frac{n^2}{4\pi^2}\int_\C  \dif^2 z_1 (\Delta f)(z_1)\int_\C  \dif^2 z_2  (\Delta\overline{f})(z_2) \int_{\eta_0}^{\eta_c} \dif \eta_1 \int_{\eta_0}^{\eta_c} \dif \eta_2 \, F \\
        F  & =F(z_1,z_2,\eta_1,\eta_2):={} \E \Big[\braket{  G^{z_1}(\ii\eta_1) -\E  G^{z_1}(\ii\eta_1)}\braket{  G^{z_2}(\ii\eta_2) -\E  G^{z_2}(\ii\eta_2)}\Big].
    \end{align}
    By~\eqref{eq:impbbfin} it follows that the regimes \(1-n^{-2\omega_h}\le \abs{z_l}^2 \le 1+n^{-2\omega_h}\), with \(l=1,2\), and \(\abs{z_1-z_2}\le n^{-\omega_d}\) in~\eqref{eq:bound2}, with \(\omega_h, \omega_d\) defined in Proposition~\ref{prop:indmr}, are bounded by \(n^{-2\omega_h+\xi}\) and \(n^{-\omega_d/2+\xi}\), respectively. Moreover, the contribution from the regime \(\abs{z_l}\ge 1+n^{-2\omega_h}\) is also bounded by \(n^{-2\omega_h+\xi}\) using~\eqref{eq:betll} with \(\gamma\le 1-3\omega_h-2\delta_1\), say \(\gamma=1/4\). After collecting these error terms we conclude that
    \begin{equation}
        \label{eq:bound3}
        \begin{split}
            \E \abs{I_{\eta_0}^{\eta_c}}^2&=\frac{n^2}{4\pi^2}\int_{\abs{z_1}\le 1-n^{-\omega_h}} \dif^2 z_1 \Delta f(z_1)\int_{\substack{\abs{z_2}\le 1-n^{-\omega_h}, \\ \abs{z_2-z_1}\ge n^{-\omega_d}}} \dif^2 z_2 \Delta \overline{f(z_2)}  \\
            &\qquad\quad\times \int_{\eta_0}^{\eta_c} \dif \eta_1 \int_{\eta_0}^{\eta_c} \dif \eta_2 F+\cO\left(\frac{n^\xi}{n^{\omega_h}}+\frac{n^\xi}{n^{\omega_d/2}} \right).
        \end{split}
    \end{equation}
    We remark that the implicit constant in \(\cO(\cdot)\) in~\eqref{eq:bound3} and in the remainder of the proof may depend on \(\norm{ \Delta f}_{L^2(\Omega)}\).

    Then by Proposition~\ref{prop:indmr} it follows that
    \begin{equation}
        \label{eq:yetao}
        \E \Big[ \braket{  G^{z_1}(\ii\eta_1) -\E \langle  G^{z_1}(\ii\eta_1)}\braket{ G^{z_2}(\ii\eta_2) -\E   G^{z_2}(\ii\eta_2)}\Big]=\mathcal{O}\left(\frac{n^{c(\omega_h+\delta_0)+\delta_1}}{n^{\omega}}\right),
    \end{equation}
    with \(\omega_h\ll\delta_0\ll \omega\). Hence, plugging~\eqref{eq:yetao} into~\eqref{eq:bound3} it follows that
    \begin{equation}
        \label{eq:yetao2}
        \E \abs{I_{\eta_0}^{\eta_c}}^2=\mathcal{O}\left(\frac{n^{c(\omega_h+\delta_0)+2\delta_1}}{n^{\omega}}\right).
    \end{equation}
    This concludes the proof under the assumption \(\omega_h\ll\delta_m\ll \omega\), with \(m=0,1\), of Proposition~\ref{prop:indmr} (see Section~\ref{rem:s} later for a summary on all the scales involved in the proof of Proposition~\ref{prop:indmr}).
\end{proof}

\begin{remark}[Alternative proof of the bound for \(\E \abs{I_0^{\eta_0}}\)]\label{rem:altern}
    Under the additional assumption~\eqref{eq:addass} in Remark~\ref{rem:2ass}, we can prove the same bound for \(\E \abs{I_0^{\eta_0}}\) in~\eqref{eq:impexb} without relying on the fairly sophisticated proof of~\cite[Theorem 3.2]{MR2684367}.

    In order to bound \(\E \abs{I_0^{\eta_0}}\) we first remove the regime \(\eta\in [0,n^{-l}]\) as in the proof of Lemma~\ref{lem:aprior}. Then, using~\eqref{eq:impbbfin} to bound the integral over the regime \(\abs{1-\abs{z}^2}\le 1+n^{-2\omega_h}\), with \(\omega_h\) defined in Proposition~\ref{prop:indmr}, and~\eqref{eq:betll} for the regime \(\abs{z}^2\ge 1-n^{-2\omega_h}\), we conclude that
    \begin{equation}
        \label{eq:bound1}
        \E \abs{I_0^{\eta_0}} =\E \frac{n}{2\pi}\int_{\abs{z}\le 1-n^{-2\omega_h}} \abs*{ \Delta f} \abs*{\int_0^{\eta_0} \braket{ G^z -\E  G^z}\dif \eta }\dif^2 z+ \mathcal{O}\left( \frac{n^{\xi}}{n^{\omega_h}}\right).
    \end{equation}

    By universality of the smallest eigenvalue of \(H^z\) (which directly follows by Proposition~\ref{pro:ciala} for any fixed \(\abs{z}^2\le 1-n^{-2\omega_h}\); see also~\cite{MR3916329}), and the bound in~\cite[Corollary 2.4]{1908.01653} we have that
    \[
        \Prob \left(\lambda_1^z\le \eta_0 \right)\le n^{-\delta_0/4},
    \]
    with \(\eta_0=n^{-1-\delta_0}\) and \(\omega_h\ll \delta_0\). This concludes the bound in~\eqref{eq:impexb} for \(I_0^{\eta_0}\) following exactly the same proof of~\cite[Lemma 4.6]{MR4221653}, by~\eqref{eq:bound1}. We warn the reader that in~\cite[Corollary 2.4]{1908.01653} \(\lambda_1\) denotes the smallest eigenvalue of \((X-z)(X-z)^*\), whilst here \(\lambda_1^z\) denotes the smallest (positive) eigenvalue of \(H^z\).
\end{remark}

\subsection{Computation of the expectation in Theorem~\ref{theo:CLT}}\label{sec:exexex}

In this section we compute the expectation \(\E \sum_i f(\sigma_i)\) in~\eqref{eq:expv} using the computation of \(\E  \braket{ G }\) in~\eqref{prop clt exp} of Proposition~\ref{prop:CLTresm} as an input. More precisely, we prove the following lemma. Note that~\eqref{eq:exval} proves~\eqref{eq:expv} in Theorem~\ref{theo:CLT}.

\begin{lemma}\label{lem:compe}
    Fix some bounded open \(\ov\DD\subset\Omega\subset \C\) and let \(f\in H_0^{2+\delta}(\Omega)\), and let \(\kappa_4:= n^2[\E \abs{x_{11}}^4-2(\E \abs{x_{11}}^2)]\), then
    \begin{equation}
        \label{eq:exval}
        \E \sum_{i=1}^n f(\sigma_i)=\frac{n}{\pi}\int_\DD  f(z)\dif^2 z+\frac{1}{8\pi}\int_\DD \Delta f(z)\, \dif^2z-\frac{\kappa_4}{\pi}\int_\DD  f(z)(2\abs{z}^2-1)\dif^2 z+\mathcal{O}\left( n^{-\delta'}\right),
    \end{equation}
    for some small \(\delta'>0\).
\end{lemma}
\begin{proof}
    By estimating the regimes \(\eta<\eta_0\) and \(\eta > T\) using~\cite[Thm.\ 3.2]{MR2684367} and~\cite[Proof of Thm.\ 2.3]{MR4408013}, respectively, we have that
    \begin{equation}\label{lemma eq sum}
        \E\sum_i f(\sigma_i) =  -\frac{n}{2\pi\ii} \int_\C \Delta f(z) \biggl(\int_{\eta_0}^{\eta_c}+\int_{\eta_c}^T\biggr)\E\braket{G^z(\ii\eta)}\dif\eta\dif^2 z + \mathcal{O}(n^{-c})
    \end{equation}
    for some small \(c>0\).

    We now consider the regime $\eta\in [\eta_c,T]$.  Since the error term in~\eqref{prop clt exp} is not affordable for \(1-\abs{z}\) very close to zero, we remove the regime \(\abs{1-\abs{z}^2}\le n^{-2\nu}\) in the \(z\)-integral by the optimal local law at the expense of an error term \(n^{-\nu+\xi}\), for some very small \(\nu>0\) we will choose shortly. The regime \(\abs{1-\abs{z}^2}\ge n^{-2\nu}\), instead, is computed using~\eqref{prop clt exp}. Hence, collecting these error terms we conclude that there exists \(\delta'>0\) such that
    \begin{align}\label{eq:med2}
         & -\frac{n}{2\pi\ii} \int \dif^2 z \Delta f \int_{\eta_c}^T \dif \eta\, \E\braket{G}                                                                                                                                            \\\nonumber
         & \qquad=-\frac{n}{2\pi\ii} \int \dif^2 z \Delta f \int_{\eta_c}^T \dif \eta\, \Bigl(\braket{M}- \frac{\ii\kappa_4}{4n}\partial_\eta(m^4)\Bigr)+\mathcal{O}\left( n^{-\delta'}+ n^{-\nu+\xi} + \frac{n^{2\nu}}{n\eta_c} \right) \\\nonumber
         & \qquad=\frac{n}{\pi}\int_\DD f(z)\dif^2 z- \frac{\kappa_4}{\pi} \int_\DD  f(z) (2\abs{z}^2-1)\dif^2 z+\mathcal{O}\left( n^{-\delta'} + \frac{n^{2\nu}}{n\eta_c}+n^{2\nu}\eta_c+n^{-\nu+\xi}\right),
    \end{align}
    with \(\eta_c=n^{-1+\delta_1}\) defined in~\eqref{eq:scales}.  To add back the regimes \(\eta\in [0,\eta_c]\), \(\eta\ge T\), and the regime \(\abs{1-\abs{z}^2}\le n^{-2\nu}\) we used that \(\abs{\partial_\eta(m^4)}\lesssim n^{2\nu}\) and that using \(\abs{m}\le \eta^{-1}\) we have \(\abs{\partial_\eta(m^4)}\lesssim \eta^{-5}\) by~\eqref{beta def}. Choosing \(\nu, \delta'>0\) so that \(\nu\ll\delta_1 \ll \delta'\) we conclude
    \begin{equation}\label{lemma eq sum concl}
        \begin{split}
            \E \sum_i f(\sigma_i) &= \frac{n}{\pi}\int_\DD f(z)\dif^2 z- \frac{\kappa_4}{\pi} \int_\DD  f(z) (2\abs{z}^2-1)\dif^2 z \\
            &\quad - \frac{n}{2\pi\ii} \int_\C \Delta f(z) \int_{\eta_0}^{\eta_c}\E\braket{G^z(\ii\eta)}\dif\eta\dif^2 z + \mathcal{O}\left( n^{-c}\right)
        \end{split}
    \end{equation}
    from~\eqref{lemma eq sum}--\eqref{eq:med2}.

    Finally, we consider the regime $\eta\in [\eta_0,\eta_c]$. First we note that in the Ginibre case we have the expansion (see \cite[Eq. (5.18)]{MR2972857} and \cite[Eq. (1.14)]{MR3735628})
    \[
        \E \sum_i f(\widetilde{\sigma}_i)=\frac{n}{\pi}\int_\DD f(z)\,\dif^2z+\frac{1}{8\pi}\int_\DD \Delta f(z)\, \dif^2z,
    \]
    which, using~\eqref{lemma eq sum concl} with \(\kappa_4=0\), implies
    \begin{equation}\label{small eta Ginibre}
        - \frac{n}{2\pi\ii} \int_\C \Delta f(z) \int_{\eta_0}^{\eta_c}\E\braket{\widetilde G^z(\ii\eta)}\dif\eta\dif^2 z = \frac{1}{8\pi}\int_\DD \Delta f(z)\, \dif^2z + \mathcal{O}\left( n^{-c}\right),
    \end{equation}
    where \(\widetilde G\) denote the resolvent of the Hermitization of a complex Ginibre matrix. In order to compare \(\E\braket{G}\) and \(\E\braket{\widetilde G}\) we use that for \(\abs{1-\abs{z}^2}>n^{-2\nu}\) we have
    \begin{equation}\label{eta c G tilde G}
        \int_{\eta_0}^{\eta_c}\dif\eta\, \E\braket{G^z}=   \int_{\eta_0}^{\eta_c}\dif\eta\,\E\braket{\widetilde{G}^z}+\mathcal{O}\left(n^{-1-c}\right),
    \end{equation}
    for some small $c>0$, the complementary regime is negligible by its small volume. The relation~\eqref{eta c G tilde G} follows using Lemma~\ref{lem:firststepmason} and computations analogous (actually easier) to \eqref{eq:thirdstemason}. By combining~\eqref{lemma eq sum concl},~\eqref{small eta Ginibre} and~\eqref{eta c G tilde G} the proof of the lemma is concluded.
\end{proof}

\subsection{Computation of the second and higher moments in Theorem~\ref{theo:CLT}}\label{sec variance computation}
In this section we conclude the proof of Theorem~\ref{theo:CLT}, i.e.\ we compute
\begin{gather}
    \begin{aligned}
        \E \prod_{i\in[p]} L_n(f^{(i)}) & = \E\prod_{i\in[p]} I_{\eta_c}^T(f^{(i)}) +  \mathcal{O}(n^{-c(p)})                                                                                    \\
                                        & = \E\prod_{i\in[p]} \biggl[-\frac{n}{2\pi \ii}\int_\C \Delta f^{(i)}(z) \int_{\eta_c}^T \braket{G^z(\ii\eta)-\E G^z(\ii\eta)} \dif \eta\dif^2 z\biggr] \\
                                        & \quad+ \mathcal{O}(n^{-c(p)})
    \end{aligned}\label{eq: prod clt}\raisetag{-5em}
\end{gather}
to leading order using~\eqref{eq CLT resovlent}.
\begin{lemma}\label{lem:b}
    Let \(f^{(i)}\) be as in Theorem~\ref{theo:CLT} and set \(f^{(i)}=f\) or \(f^{(i)}=\overline{f}\) for any \(i\in [p]\), and recall that \(\Pi_p\) denotes the set of pairings on \([p]\). Then
    \begin{equation}
        \label{eq:bast}
        \begin{split}
            &\E\prod_{i\in[p]} \biggl[-\frac{n}{2\pi\ii}\int_\C \Delta f^{(i)}(z) \int_{\eta_c}^T \braket{G^{z}(\ii\eta)-\E G^z(\ii\eta)} \dif \eta\dif^2 z\biggr] \\
            &= \sum_{P\in \Pi_p} \prod_{\{i,j\}\in P} \biggl[-\int_\C \dif^2 z_i \Delta f^{(i)} \int_\C\dif^2 z_j\Delta f^{(j)} \int_{0}^\infty \dif\eta_i\int_{0}^\infty\dif\eta_j \frac{V_{i,j}+\kappa_4 U_i U_j}{8\pi^2}\biggr] \\
            &\qquad + \mathcal{O}(n^{-c(p)}),
        \end{split}
    \end{equation}
    for some small \(c(p)>0\), where \(V_{i,j}\) and \(U_i\) are as in~\eqref{eq:exder}. The implicit constant in \(\mathcal{O}(\cdot)\) may depend on \(p\).
\end{lemma}
\begin{proof}
    In order to prove the lemma we have to check that the integral of the error term in~\eqref{eq CLT resovlent} is at most of size \(n^{-c(p)}\), and that the integral of \(V_{i,j}+\kappa_4 U_i U_j\) for \(\eta_i\le \eta_c\) or \(\eta_i\ge T\) is similarly negligible. In the remainder of the proof we assume that \(p\) is even, since the terms with \(p\) odd are of lower order by~\eqref{eq CLT resovlent}.

    Note that by the explicit form of \(m_i, u_i\) in~\eqref{eq m}--\eqref{eq M}, by the definition of \(V_{i,j}\), \(U_i, U_j\) in~\eqref{eq:exder}, the fact that by \(-m_i^2+\abs{z_i}^2u_i^2=u_i\) we have
    \[
        V_{i,j}=\frac{1}{2}\partial_{\eta_i}\partial_{\eta_j}\log \left(1-u_i u_j\Big[ 1-\abs{z_i-z_j}^2+(1-u_i)\abs{z_i}^2+(1-u_j)\abs{z_j}^2\Big]\right),
    \]
    and using \(\abs{\partial_{\eta_i} m_i}\le [\Im m^{z_i}(\ii\eta_i)+\eta_i]^{-2}\) by~\eqref{beta def}--\eqref{eq:bbou}, we conclude (see also~\eqref{eq:bbst1}--\eqref{eq:bbst2} later)
    \begin{equation}
        \label{eq:VWbound}
        \abs{V_{i,j}}\lesssim \frac{[(\Im m^{z_i}(\ii\eta_i)+\eta_i)(\Im m^{z_j}(\ii\eta_j)+\eta_j)]^{-2}}{[\abs{z_i-z_j}^2+(\eta_i+\eta_j)(\min \{\Im m^{z_i}, \Im m^{z_j} \}^2)]^2},\,\, \abs{U_i}\lesssim \frac{1}{\Im m^{z_i}(\ii\eta_i)^2+\eta_i^3}.
    \end{equation}

    Using the bound~\eqref{eq:impbbfin} to remove the regime \(Z_i:=  \set{ \abs{1-\abs{z_i}^2}\le n^{-2\nu}}\) for any \(i\in[p]\), for some small \(\nu>0\), we conclude that the lhs.\ of~\eqref{eq:bast} is equal to
    \begin{equation}
        \label{eq:wickprod}
        \frac{(-n)^p}{(2\pi \ii)^p}\prod_{i\in [p]}\int_{Z_i^c}\dif^2 z_i \Delta f^{(i)}(z_i) \E \prod_{i\in[p]} \int_{\eta_c}^T \braket{G^{z_i}(\ii\eta_i)-\E G^{z_i}(\ii\eta_i)} \dif \eta_i + \mathcal{O}\left(\frac{n^{p\xi}}{n^\nu}\right),
    \end{equation}
    for any very small \(\xi>0\). Additionally, since the error term \(\Psi\) defined in~\eqref{eq psi error} behaves badly for small \(\abs{z_i-z_j}\), we remove the regime
    \[
        \widehat{Z}_i:= \bigcup_{j<i}\set{  z_j : \,\abs{z_i-z_j}\le n^{-2\nu}}
    \]
    in each \(z_i\)-integral in~\eqref{eq:wickprod} using~\eqref{eq:impbbfin}, and, denoting \(f^{(i)}=f^{(i)}(z_i)\), get
    \begin{equation}
        \label{eq:wickprod2}
        \frac{(-n)^p}{(2\pi \ii)^p}  \prod_{i\in[p]}\int_{Z_i^c \cap \widehat{Z}_i^c}\dif^2 z_i \Delta f^{(i)}\E \prod_{i\in[p]} \int_{\eta_c}^T \braket{G^{z_i}(\ii\eta_i)-\E G^{z_i}(\ii\eta_i)} \dif \eta_i+ \mathcal{O}\left(\frac{n^{p\xi}}{n^\nu}\right).
    \end{equation}
    Plugging~\eqref{eq CLT resovlent} into~\eqref{eq:wickprod2}, and using the first bound in~\eqref{eq:apb} to remove the regime \(\eta_i\ge T\) for the lhs.\ of~\eqref{eq:bast} we get
    \begin{gather}
        \begin{aligned}
             & \frac{1}{(2\pi \ii)^p}  \prod_{i\in[p]}\int_{Z_i^c \cap \widehat{Z}_i^c}\dif^2 z_i \Delta f^{(i)}\sum_{P\in \Pi_p} \prod_{\{i,j\}\in P} \int_0^\infty\int_0^\infty -\frac{V_{i,j}+\kappa_4 U_i U_j}{8\pi^2} \dif \eta_j\dif \eta_i \\
             & \qquad\qquad + \mathcal{O}\left(\frac{n^{p\xi}}{n^\nu}+\frac{n^{20\nu p+\delta_1}}{n}+\frac{n^{\xi p+2p\nu}}{n^{\delta_1/2}}\right),
        \end{aligned}\label{eq:wickprod3}\raisetag{-5em}
    \end{gather}
    where \(\eta_c=n^{-1+\delta_1}\), the second last error term comes from adding back the regimes \(\eta_i\in [0,\eta_c]\) using  that
    \[
        \abs{V_{i,j}}\le \frac{n^{20\nu}}{(1+\eta_i^2)(1+\eta_j^2)}, \qquad \abs{U_i}\le \frac{n^{4\nu}}{1+\eta_i^3},
    \]
    for \(z_i\in Z_i^c\cap \widehat{Z}_i^c\) and \(z_j\in Z_j^c\cap \widehat{Z}_j^c\)  by~\eqref{eq:VWbound}. The last error term in~\eqref{eq:wickprod3} comes from the integral of \(\Psi\), with \(\Psi\) defined in~\eqref{eq psi error}. Finally, we perform the \(\eta\)-integrations using the explicit formulas~\eqref{eq:expintV} and~\eqref{eq:expintW} below. After that, we add back the domains \(Z_i\) and  \(\widehat{Z_i}\)  for \(i\in [p]\) at a negligible error, since these domains have volume of order \(n^{-2\nu}\),  \(\Delta f^{(i)}\in L^2\), and the logarithmic singularities from~\eqref{eq:expintV} are integrable. This concludes~\eqref{eq:bast} choosing \(\nu\) so that \(\nu\ll \delta_1\ll 1\).
\end{proof}

In the next three sub-sections we compute the integrals in~\eqref{eq:bast} for any \(i,j\)'s. To make our notation simpler we use only the indices \(1,2\), i.e.\ we compute the integral of \(V_{1,2}\) and \(U_1U_2\).

\subsubsection{Computation of the \((\eta_1,\eta_2)\)-integrals}
Using the relations in~\eqref{eq:exder} we explicitly compute the \((\eta_1,\eta_2)\)-integral of \(V_{1,2}\):
\begin{gather}
    \begin{aligned}
         & -\int_0^\infty \int_0^\infty V_{1,2}\, \dif\eta_1 \dif \eta_2 =-\frac{1}{2}\log A\restriction_{\substack{\eta_1=0, \\ \eta_2=0}} \\
         & \quad=\Theta(z_1,z_2):=
        \frac{1}{2}\begin{cases}
                       -\log \abs{z_1-z_2}^2,                                  & \abs{z_1},\abs{z_2} \le 1,    \\
                       \log \abs{z_l}^{2}-\log\abs{z_1-z_2}^2,                 & \abs{z_m} \le 1, \abs{z_l}>1, \\
                       \log \abs{z_1 z_2}^{2}-\log\abs{1-z_1\overline{z}_2}^2, & \abs{z_1}, \abs{z_2}>1,
                   \end{cases}
    \end{aligned}\label{eq:expintV}\raisetag{-4em}
\end{gather}
with \(A(\eta_1,\eta_2,z_1,z_2)\) defined by
\[
    A(\eta_1,\eta_2,z_1,z_2):=1+(u_1u_2\abs{z_1}\abs{z_2})^2-m_1^2 m_2^2-2u_1u_2\Re z_1\overline{z}_2.
\]
Then the \(\eta_i\)-integral of \(U_i\), for \(i\in\{1,2\}\), is given by
\begin{equation}\label{eq:expintW}
    \int_0^\infty U_i\, \dif \eta_i=\frac{\ii}{\sqrt{2}}(1-\abs{z_i}^2).
\end{equation}
Before proceeding we rewrite \(\Theta(z_1,z_2)\) as
\[
    \begin{split}
        2\Theta(z_1,z_2)&=-\log\abs{z_1-z_2}^2+\log\abs{z_1}^{2} \bm1(\abs{z_1}> 1)+\log\abs{z_2}^{2} \bm1(\abs{z_2}> 1)\\
        &\quad +\left[\log\abs{z_1-z_2}^2-\log \abs{1-z_1\overline{z}_2}^2 \right]\bm1(\abs{z_1},\abs{z_2}> 1).
    \end{split}
\]

In the remainder of this section we use the notations
\[
    \dif z:= \dif z+\ii \dif y, \quad \dif\overline{z}:=\dif x-\ii \dif y, \quad  \quad\partial_z :=\frac{\partial_x-\ii\partial_y}{2}, \quad \partial_{\overline{z}} :=\frac{\partial_x+\ii\partial_y}{2},
\]
and \(\partial_l:= \partial_{z_l}\), \(\overline{\partial}_l:=\partial_{\overline{z}_l}\). With this notation \(\Delta_{z_l}=4\partial_{z_l}\partial_{\overline{z}_l}\).

We split the computation of the leading term in the rhs.\ of~\eqref{eq:bast} into two parts: the integral of \(V_{1,2}\), and  the integral of \(U_1U_2\).

\subsubsection{Computation of the \((z_1,z_2)\)-integral of \(V_{1,2}\)}
In this section we compute the integral of \(V_{1,2}\) in~\eqref{eq:bast}. To make our notation easier in the remainder of this section we use the notation \(f\) and \(g\),  instead of \(f^{(1)}\), \(f^{(2)}\), with \(f\) in Theorem~\ref{theo:CLT} and \(g=f\) or \(g=\overline{f}\).

\begin{lemma}\label{lem:vi}
    Let \(V_{1,2}\) be defined in~\eqref{eq:exder}, then
    \begin{equation}
        \label{eq:finV}
        \begin{split}
            &-\frac{1}{8\pi^2}\int_\C  \dif^2 z_1\int_\C  \dif^2 z_2 \Delta f(z_1) \Delta \overline{g(z_2)} \int_0^{\infty} \dif \eta_1\int_0^{\infty} \dif \eta_2  V_{1,2} \\
            &\qquad =\frac{1}{4\pi} \int_\DD  \braket{ \nabla g, \nabla f} \dif^2 z+\frac{1}{2}\sum_{m\in\Z } \abs{m} \widehat{f\restriction_{\partial \DD }}(m) \overline{\widehat{g\restriction_{\partial \DD }}}(m).
        \end{split}
    \end{equation}
\end{lemma}

Note that the rhs.\ of~\eqref{eq:finV} gives exactly the first two terms in~\eqref{eq:cov}.

Using the expression of \(V_{1,2}\) in~\eqref{eq:exder} and the computation of its \((\eta_1,\eta_2)\)-integral in~\eqref{eq:expintV}, we have that
\begin{equation}\label{eq:comp1}
    \begin{split}
        &-\frac{1}{8\pi^2}\int_\C  \dif^2 z_1\int_\C  \dif^2 z_2 \Delta f(z_1) \Delta \overline{g(z_2)} \int_0^{\infty} \dif \eta_1\int_0^{\infty} \dif \eta_2  V_{1,2} \\
        &\qquad = \frac{2}{\pi^2}\int_\C  \dif^2 z_1\int_\C  \dif^2 z_2 \partial_1\overline{\partial}_1 f(z_1)\partial_2\overline{\partial}_2 \overline{g(z_2)} \Theta(z_1,z_2),
    \end{split}
\end{equation}
with \(\Theta(z_1,z_2)\) is defined in the rhs.\ of~\eqref{eq:expintV}.

We compute the r.h.s.\ of~\eqref{eq:comp1} as stated in Lemma~\ref{lem:intbp}. The proof of this lemma is postponed to Appendix~\ref{sec:INTBP}.

\begin{lemma}\label{lem:intbp}
    Let \(\Theta(z_1,z_2)\) be defined in~\eqref{eq:expintV}, then we have that
    \begin{gather}
        \begin{aligned}
             & \frac{2}{\pi^2}\int_\C  \dif^2 z_1\int_\C  \dif^2 z_2 \partial_1\overline{\partial}_1 f(z_1)\partial_2\overline{\partial}_2 \overline{g(z_2)} \Theta(z_1,z_2) =\frac{1}{4\pi} \int_\DD  \braket{ \nabla g, \nabla f} \dif^2 z \\
             & \qquad\quad+ \lim_{\epsilon\to 0}\Bigg[ \frac{1}{2\pi^2} \int_{\abs{z_1}\ge 1} \dif^2 z_1 \int_{\substack{\abs{1-z_1\overline{z}_2}\ge \epsilon,                                                                              \\ \abs{z_2}\ge 1}} \dif^2 z_2 \, \partial_1 f(z_1) \overline{\partial}_2 \overline{g(z_2)}  \frac{1}{(1-\overline{z}_1z_2)^2} \\
             & \qquad\qquad\quad+ \frac{1}{2\pi^2} \int_{\abs{z_1}\ge 1} \dif^2 z_1 \int_{\substack{\abs{1-z_1\overline{z}_2}\ge \epsilon,                                                                                                   \\ \abs{z_2}\ge 1}} \dif^2 z_2\,  \overline{\partial}_1 f(z_1) \partial_2 \overline{g(z_2)} \frac{1}{(1-z_1\overline{z}_2)^2} \Bigg].
        \end{aligned}\label{eq:intbp}\raisetag{-5em}
    \end{gather}
\end{lemma}

\begin{proof}[Proof of Lemma~\ref{lem:vi}]

    By Lemma~\ref{lem:intbp} it follows that to prove Lemma~\ref{lem:vi} it is enough to compute the last two lines in the rhs.\ of~\eqref{eq:intbp}.

    Note that using the change of variables \(\overline{z}_1\to 1/{\overline{z}_1}\), \(z_2\to 1/z_2\) the integral in the rhs.\ of~\eqref{eq:intbp} is equal to the same integral on the domain \(\abs{z_1},\abs{z_2}\le 1\), \(\abs{1-z_1\overline{z}_2}\ge \epsilon\). By a standard density argument, using that \(f,g\in H_0^{2+\delta}\), it is enough to compute the limit in~\eqref{eq:intbp} only for polynomials, hence, from now on, we consider polynomials \(f\), \(g\) of the form
    \begin{equation}
        \label{eq:poly}
        f(z_1)=\sum_{k,l\ge 0} z_1^k \overline{z}_1^l a_{kl}, \qquad g(z_2)=\sum_{k,l\ge 0} z_2^k \overline{z}_2^l b_{kl},
    \end{equation}
    for some coefficients \(a_{kl}, b_{kl}\in \C \). We remark that the summations in~\eqref{eq:poly} are finite since \(f\) and \(g\) are polynomials. Then, using that
    \[
        \lim_{\epsilon\to 0}\int_{\abs{z_1}\le 1} \int_{\substack{\abs{1-z_1\overline{z}_2}\ge \epsilon, \\ \abs{z_2}\le 1}} z_1^\alpha \overline{z}_1^\beta z_2^{\alpha'} \overline{z}_2^{\beta'}\dif^2 z_1 \dif^2 z_2=\frac{\pi^2}{(\alpha+1)(\alpha'+1)}\delta_{\alpha,\beta}\delta_{\alpha',\beta'},
    \]
    we compute the limit in the rhs.\ of~\eqref{eq:intbp} as follows
    \begin{gather}
        \begin{aligned}
             & \lim_{\epsilon\to 0}\sum_{k,l,k',l', m\ge 0} \frac{1}{2\pi^2}\int_{\abs{z_1}\le 1}\int_{\substack{\abs{1-z_1\overline{z}_2}\ge \epsilon,                                         \\ \abs{z_2}\le 1}} \dif^2 z_1 \dif^2 z_2 \, m a_{kl} \overline{b_{k'l'}} \\
             & \qquad\qquad\qquad\times\Big[k k' z_1^{k-1}\overline{z}_1^{l+m-1} z_2^{l'+m-1} \overline{z}_2^{k'-1}+l l' z_1^{k+m-1} \overline{z}_1^{l-1}z_2^{k'+m-1}\overline{z}_2^{l'-1}\Big] \\
             & \qquad= \frac{1}{2}\sum_{\substack{k,l,k',l',                                                                                                                                    \\ m\ge0}}  m a_{kl} \overline{b_{k'l'}}\Big[\delta_{k,l+m}\delta_{k',l'+m} +\delta_{k,l-m}\delta_{k',l'-m}\Big] \\
             & \qquad= \frac{1}{2}\sum_{\substack{k,l,k',l'\ge0,                                                                                                                                \\ m\in\Z} }  \abs{m} a_{kl} \overline{b_{k'l'}}\delta_{k,l+m}\delta_{k',l'+m}.
        \end{aligned}\label{eq:comp3}\raisetag{-5em}
    \end{gather}
    On the other hand
    \begin{equation}
        \label{eq:h12}
        \sum_{m\in\Z } \abs{m} \widehat{f\restriction_{\partial \DD }}(m) \overline{\widehat{g\restriction_{\partial \DD }}}(m)= \sum_{m\in\Z }  \abs{m} \sum_{k,l,k',l'\ge 0} a_{kl}\overline{b_{k'l'}}\delta_{m,k-l} \delta_{m,k'-l'},
    \end{equation}
    where
    \[
        \widehat{f\restriction_{\partial \DD }}(k):=  \frac{1}{2\pi}\int_0^{2\pi} f\restriction_{\partial \DD }(e^{i\theta}) e^{-ik\theta}\dif \theta, \quad f\restriction_{\partial \DD }(e^{i\theta_j})=\sum_{k\in\Z } \widehat{f\restriction_{\partial \DD }}(k) e^{i\theta_j k}.
    \]
    Finally, combining~\eqref{eq:intbp} and~\eqref{eq:comp3}--\eqref{eq:h12}, we conclude the proof of~\eqref{eq:finV}.

\end{proof}

\subsubsection{Computation of the \((z_1,z_2)\)-integral of \(U_1 U_2\)}

In order to conclude the proof of Theorem~\ref{theo:CLT}, in this section we compute the integral of \(U_1 U_2\) in~\eqref{eq:bast}. Similarly to the previous section, we use the notation \(f\) and \(g\),  instead of \(f^{(1)}\), \(f^{(2)}\), with \(f\) in Theorem~\ref{theo:CLT} and \(g=f\) or \(g=\overline{f}\).

\begin{lemma}\label{lem:wi}
    Let \(\kappa_4=n^2[\E \abs{x_{11}}^2-2(\E \abs{x_{11}}^2)]\), and let \(U_1\), \(U_2\) be defined in~\eqref{eq:exder}, then
    \begin{gather}
        \begin{aligned}
             & -\frac{\kappa_4}{8\pi^2}\int_\C  \dif^2 z_1\int_\C  \dif^2 z_2 \Delta f(z_1) \Delta \overline{g(z_2)} \int_0^{\infty} \dif \eta_1\int_0^{\infty} \dif \eta_2  U_1 U_2                                                       \\
             & \qquad =\kappa_4 \left(\frac{1}{\pi}\int_\DD f(z)\dif^2 z-\widehat{f\restriction_{\partial\DD }}(0)\right)\left(\frac{1}{\pi}\int_\DD \overline{g(z)}\dif^2 z- \overline{\widehat{g\restriction_{\partial\DD }}}(0)\right).
        \end{aligned}\label{eq:finW}\raisetag{-3.5em}
    \end{gather}
\end{lemma}

\begin{proof}[Proof of Theorem~\ref{theo:CLT}]
    Theorem~\ref{theo:CLT} readily follows combining Lemma~\ref{lem:b}, Lemma~\ref{lem:vi} and Lemma~\ref{lem:wi}.
\end{proof}

\begin{proof}[Proof of Lemma~\ref{lem:wi}]
    First of all, we recall the following formulas of integration by parts
    \begin{equation}
        \label{eq:ibp}
        \int_\DD  \partial_z f(z,\overline{z}) \dif^2 z=\frac{\ii}{2}\int_{\partial \DD } f(z,\overline{z}) \dif\overline{z}, \quad \int_\DD  \partial_{\overline{z}} f(z,\overline{z})\dif^2 z=-\frac{\ii}{2}\int_{\partial\DD } f(z,\overline{z})  \dif z.
    \end{equation}
    Then, using the computation of the \(\eta\)-integral of \(U\) in~\eqref{eq:expintW}, and integration by parts~\eqref{eq:ibp} twice, we conclude that
    \[
        \begin{split}
            \int_\C    \Delta f\int_0^\infty U\, \dif \eta\dif^2 z &= \ii 2\sqrt{2} \int_{\DD }  \partial\overline{\partial} f(z) (1-\abs{z}^2) \dif^2 z=\ii 2\sqrt{2}  \int_{\DD } \overline{\partial} f(z) \overline{z}\dif^2z \\
            &=- \ii 2\sqrt{2} \left( \int_{\DD }   f(z)\,\dif^2z +\frac{\ii}{2}\int_{\partial(\DD )} f(z) \overline{z}\dif z\right) \\
            &=-\ii 2\sqrt{2} \left( \int_{\DD }   f(z)\dif^2 z -\pi \widehat{f\restriction_{\partial\DD }}(0)\right).
        \end{split}
    \]
    This concludes the proof of this lemma.
\end{proof}

\section{Local law for products of resolvents}\label{sec local law G2}
The main technical result of this section is a local law for \emph{products} of resolvents with different spectral parameters \(z_1\ne z_2\). Our goal is to find a deterministic
approximation to \(\braket{AG^{z_1} B G^{z_2} }\) for generic bounded deterministic matrices \(A,B\). Due to the correlation between the two resolvents the deterministic approximation to \(\braket{A G^{z_1}B G^{z_2}}\) is not simply \(\braket{A M^{z_1}B M^{z_2}}\).
In the context of linear statistics such local laws for products of resolvents have previously been obtained e.g.\ for Wigner matrices in~\cite{MR3805203} and for sample-covariance matrices in~\cite{MR4119592} albeit with  weaker error bounds. In the current non-Hermitian setting we need such local law  twice; for the resolvent CLT in Proposition~\ref{prop:CLTresm}, and
for the asymptotic independence of resolvents in Proposition~\ref{prop:indmr}. The key point
for the latter  is to obtain an improvement in the error term for mesoscopic  separation \(\abs{z_1-z_2}\sim n^{-\epsilon}\), a fine effect that has not been captured before.

Our proof  applies verbatim to both real and complex i.i.d.\ matrices, as well as to resolvents \(G^z(w)\) evaluated at an
arbitrary spectral parameter \(w\in \HC \).  We therefore work with this more general setup in this section, even
though for the application in the proofs of Propositions~\ref{prop:CLTresm}--\ref{prop:indmr} this generality is not necessary.

We recall from~\cite{MR4408013} that with the shorthand notations
\begin{equation}\label{G_i def}
    G_i:= G^{z_i}(w_i),\quad M_i :=  M^{z_i}(w_i),
\end{equation}
the deviation of \(G_i\) from \(M_i\) is computed from the identity
\begin{equation}\label{G deviation}  G_i = M_i -M_i \un{W G_i} + M_i \SS[G_i-M_i] G_i,\quad W := \begin{pmatrix}
        0 & X \\X^\ast &0
    \end{pmatrix}.\end{equation}
The relation~\eqref{G deviation} requires some definitions. First, the  linear \emph{ covariance} or \emph{self-energy operator} \(\SS\colon\C^{2n\times 2n}\to\C^{2n\times 2n}\) is given by
\begin{equation}\label{S def}
    \SS\biggl[\begin{pmatrix}
            A & B \\ C & D
        \end{pmatrix}\biggr] :=  \wt\E \wt W \begin{pmatrix}
        A & B \\ C & D
    \end{pmatrix} \wt W= \begin{pmatrix}
        \braket{D} & 0 \\ 0 & \braket{A}
    \end{pmatrix} ,\quad \wt W=\begin{pmatrix}
        0 & \wt X \\ \wt X^\ast & 0
    \end{pmatrix},
\end{equation}
where \(\wt X\sim\mathrm{Gin}_\C\), i.e.\ it averages the diagonal blocks and swaps them. Here \(\mathrm{Gin}_\C\) stands for the standard complex Ginibre ensemble. The ultimate equality in~\eqref{S def} follows directly from \(\E \wt x_{ab}^2=0\), \(\E\abs{\wt x_{ab}}^2=n^{-1}\). Second, underlining denotes, for any given function \(f\colon\C^{2n\times 2n}\to\C^{2n\times 2n}\), the \emph{self-renormalisation} \(\un{W f(W)}\) defined by
\begin{equation}\label{self renom}
    \un{W f(W)}:= W f(W) - \wt\E \wt W (\partial_{\wt W} f)(W),
\end{equation}
where \(\partial\) indicates a directional derivative in the direction \(\wt W\) and \(\wt W\) denotes an independent random matrix
as in~\eqref{S def} with \(\wt X\) a complex Ginibre matrix with expectation \(\wt\E\).
Note that we use complex Ginibre \(\wt X\) irrespective of the symmetry class of \(X\). Therefore, using the resolvent identity, it follows that
\[\un{W G} = WG + \wt\E \wt W G \wt W G = WG+\SS[G]G.\]
We now use~\eqref{G deviation} and~\eqref{self renom} to compute
\begin{gather}
    \begin{aligned}
        G_1 B G_2 & = M_1 B G_2 - M_1 \un{W G_1} B G_2 + M_1 \SS[G_1-M_1]G_1 B G_2               \\
                  & = M_1B M_2 + M_1 B(G_2-M_2) - M_1 \un{W G_1 B G_2} + M_1 \SS[ G_1 B G_2 ]M_2 \\
                  & \qquad + M_1 \SS[G_1 B G_2](G_2-M_2) + M_1 \SS[G_1-M_1]G_1 B G_2,
    \end{aligned}\label{G12 devition}\raisetag{-3em}
\end{gather}
where, in the second equality, we used
\[
    \begin{split}
        \un{W G_1 B G_2} &= W G_1 B G_2 + \SS[G_1] G_1 B G_2 + \SS[G_1 B G_2] G_2  \\
        &= \un{W G_1} B G_2 + \SS[G_1 B G_2] G_2.
    \end{split}
\]
Assuming that the self-renormalised terms and the ones involving \(G_i-M_i\) in~\eqref{G12 devition} are small,~\eqref{G12 devition} implies
\begin{equation}\label{G12 approx local law}
    G_1 B G_2 \approx M_B^{z_1,z_2},
\end{equation}
where
\begin{equation}\label{eq M12 def}
    M_B^{z_1,z_2}(w_1,w_2):= (1-M^{z_1}(w_1)\SS[\cdot]M^{z_2}(w_2))^{-1}[M^{z_1}( w_1)B M^{z_2}(w_2)].
\end{equation}
We define the corresponding \emph{\(2\)-body stability operator}
\begin{equation}\label{eq:stabop12}
    \wh\cB=\wh\cB_{12}=\wh\cB_{12}(z_1,z_2, w_1,w_2):= 1-M_1 \SS[\cdot]M_2,
\end{equation}
acting on the space of \(2n \times 2n\) matrices equipped with the usual Euclidean matrix norm which induces a natural norm for \(\wh\cB\).

Our main technical result of this section is making~\eqref{G12 approx local law} rigorous in the sense of Theorem~\ref{thm local law G2} below. To keep notations compact, we first introduce a commonly used (see, e.g.~\cite{MR3068390}) notion of high-probability bound.
\begin{definition}[Stochastic Domination]\label{def:stochDom}
    If \[X=\tuple*{ X^{(n)}(u) \given n\in\N, u\in U^{(n)} }\quad\text{and}\quad Y=\tuple*{ Y^{(n)}(u) \given n\in\N, u\in U^{(n)} }\] are families of non-negative random variables indexed by \(n\), and possibly some parameter \(u\), then we say that \(X\) is stochastically dominated by \(Y\), if for all \(\epsilon, D>0\) we have \[\sup_{u\in U^{(n)}} \Prob\left[X^{(n)}(u)>n^\epsilon  Y^{(n)}(u)\right]\leq n^{-D}\] for large enough \(n\geq n_0(\epsilon,D)\). In this case we use the notation \(X\prec Y\).
\end{definition}

\begin{theorem}\label{thm local law G2} Fix \(z_1,z_2\in\C\) and \(w_1,w_2\in\C\)
    with \(\abs{\eta_i}:=\abs{\Im w_i}\ge n^{-1}\) such that
    \[
        \eta_*:= \min\{\abs{\eta_1},\abs{\eta_2}\}\ge n^{-1+\epsilon}\norm{\wh\cB_{12}^{-1}}
    \]
    for some \(\epsilon>0\). Assume that \(G^{z_1}(w_1),G^{z_2}(w_2)\) satisfy the local laws  in the form
    \[ \abs{\braket{A(G^{z_i}-M^{z_i})}} \prec \frac{\norm{A}}{n\abs{\eta_i}}, \quad  \abs{\braket{\vx,(G^{z_i}-M^{z_i})\vy}} \prec \frac{\norm{\vx}\norm{\vy}}{\sqrt{n\abs{\eta_i}}}\]
    for any bounded deterministic matrix and vectors \(A,\vx,\vy\).
    Then, for any bounded deterministic matrix \(B\), with \(\norm{B}\lesssim 1\), the product of resolvents \( G^{z_1}B G^{z_2}=G^{z_1}(w_1)BG^{z_2}(w_2)\) is approximated by \(M_B^{z_1,z_2} =M_B^{z_1,z_2}(w_1,w_2)\) defined in~\eqref{eq M12 def}
    in the sense that
    \begin{gather}
        \begin{aligned}
            \abs{\braket{A (G^{z_1} BG^{z_2}-M_B^{z_1,z_2})}}     & \prec \frac{\norm{A}\norm{\wh\cB_{12}^{-1}}}{n\eta_\ast \abs{\eta_1\eta_2}^{1/2} }                                                                                    \\
                                                                  & \;\; \times\Bigl(\eta_\ast^{1/12}+\eta_\ast^{1/4}\norm{\wh\cB_{12}^{-1}} +\frac{1}{\sqrt{n\eta_\ast}}+\frac{\norm{\wh\cB_{12}^{-1}}^{1/4}}{(n\eta_\ast)^{1/4}}\Bigr), \\
            \abs{\braket{\vx,(G^{z_1}BG^{z_2}-M_B^{z_1,z_2})\vy}} & \prec
            \frac{\norm{\vx} \norm{\vy}\norm{\wh\cB_{12}^{-1}}}{(n\eta_\ast)^{1/2}|\eta_1\eta_2|^{1/2}}
        \end{aligned}\label{final local law}\raisetag{-5em}
    \end{gather}
    for any  deterministic \(A,\vx,\vy\).
\end{theorem}

The estimates in~\eqref{final local law} will be complemented by a upper bound on \(\norm{\wh\cB^{-1}}\) in Lemma~\ref{lemma:betaM}, where we will prove in particular that \(\norm{\wh\cB^{-1}}\lesssim n^{2\delta}\) whenever \(\abs{z_1-z_2}\gtrsim n^{-\delta}\), for some small fixed \(\delta>0\).

The proof of Theorem~\ref{thm local law G2} will follow from a
bootstrap argument once the main input, the following high-probability bound on \(\un{W G_1 B G_2}\) has been established.
\begin{proposition}\label{prop prob bound}Under the assumptions of Theorem~\ref{thm local law G2}, the following estimates hold uniformly in \(n^{-1}\lesssim \abs{\eta_1},\abs{\eta_2}\lesssim 1\).
    \begin{subequations}
        \begin{enumerate}[label=(\roman*)]
            \item We have the isotropic bound
                  \begin{equation}\label{prop iso bound}
                      \abs{\braket{\vx,\un{WG_1 B G_2} \vy}} \prec \frac{1}{( n\eta_\ast)^{1/2}\abs{\eta_1\eta_2}^{1/2}}
                  \end{equation}
                  uniformly for deterministic vectors and matrix \(\norm{\vx}+\norm{\vy}+\norm{B}\le1\).
            \item Assume that for some positive deterministic \(\theta=\theta(z_1,z_2,\eta_\ast)\) an a priori bound
                  \begin{equation}\label{eq a priori theta}\abs{\braket{A G_1 B G_2}}\prec\theta   \end{equation}
                  has already been established uniformly in deterministic matrices \(\norm{A}+\norm{B}\le 1\). Then we have the improved averaged bound
                  \begin{equation}\label{prop av bound}
                      \abs{\braket{\un{W G_1 B G_2 A}}} \prec \frac{1}{n\eta_\ast \abs{\eta_1\eta_2}^{1/2}}\Bigl((\theta\eta_\ast)^{1/4}+\frac{1}{\sqrt{n\eta_\ast}}+\eta_\ast^{1/12}\Bigr),
                  \end{equation}
                  again uniformly in deterministic matrices \(\norm{A}+\norm{B}\le 1\).
        \end{enumerate}
    \end{subequations}
\end{proposition}
\begin{proof}[Proof of Theorem~\ref{thm local law G2}]
    We note that from~\eqref{eq M12 def} and~\eqref{eq M bound} we have
    \begin{equation}\label{eq M12 bound}
        \norm{M_B^{z_1,z_2}}\lesssim \norm{\wh\cB^{-1}}
    \end{equation}
    and abbreviate \(G_{12}:= G_1 B G_2\), \(M_{12}:= M_B^{z_1,z_2}\). We now assume an a priori bound \(\abs{\braket{G_{12}A}}\prec \theta_1\), i.e.\ that~\eqref{eq a priori theta} holds with \(\theta=\theta_1\). In the first step we may take \(\theta_1=\abs{\eta_1\eta_2}^{-1/2}\) due to the local law for \(G_i\) from which it follows that \[
        \begin{split}
            \abs{\braket{A G_1 B G_2}}&\le \sqrt{\braket{A G_1 G_1^\ast A^\ast}}\sqrt{\braket{B G_2 G_2^\ast B^\ast}} \\
            &=\frac{1}{\sqrt{\abs{\eta_1}\abs{\eta_2}}}\sqrt{\braket{A \Im G_1 A^\ast}}\sqrt{\braket{B \Im G_2 B^\ast}} \prec\theta_1.
        \end{split}
    \]
    By~\eqref{G12 devition} and~\eqref{eq M12 def} we have
    \begin{gather}
        \begin{aligned}
            \wh\cB[G_{12}-M_{12}] & = M_1 B (G_2-M_2)-M_1\underline{WG_{12}}+M_1\SS[G_{12}](G_2-M_2) \\
                                  & \quad +M_1\SS[G_1-M_1]G_{12},
        \end{aligned}\label{eq B G-M}\raisetag{-3em}
    \end{gather}
    and from~\eqref{single local law} and~\eqref{prop av bound} we obtain
    \[\begin{split}
            \abs{\braket{A(G_{12}-M_{12})}} &= \abs{\braket{A^\ast, \wh\cB^{-1}\wh\cB[G_{12}-M_{12}] }} = \abs{\braket{(\wh\cB^{\ast})^{-1}[A^\ast]^\ast \wh\cB[G_{12}-M_{12}]}} \\
            & \prec \norm{\wh\cB^{-1}} \Bigl[ \frac{1}{n\eta_\ast} + \frac{(\theta_1\eta_\ast)^{1/4}+(\sqrt{n\eta_\ast})^{-1}+\eta_\ast^{1/12}}{n\eta_\ast\abs{\eta_1\eta_2}^{1/2}}+ \frac{\theta_1}{n\eta_\ast} \Bigr].
        \end{split}\]
    For the terms involving \(G_i-M_i\) we used that \(\SS[R]=\braket{R E_2}E_1 + \braket{R E_1} E_2\) with the \(2n\times 2n\) block matrices
    \begin{equation}\label{E1 E2 def}
        E_1=\begin{pmatrix}
            1 & 0 \\0&0
        \end{pmatrix},\qquad E_2=\begin{pmatrix}
            0 & 0 \\0&1
        \end{pmatrix},
    \end{equation}
    i.e.\ that \(\SS\) effectively acts as a trace, so that the averaged bounds are applicable. Therefore with~\eqref{eq M12 bound} it follows that
    \begin{equation}\label{eq G12 iter}
        \abs{\braket{G_{12}A}} \prec \theta_2 := \norm{\wh\cB^{-1}} \Bigl[1+ \frac{1}{n\eta_\ast} + \frac{(\theta_1\eta_\ast)^{1/4}+(\sqrt{n\eta_\ast})^{-1}+\eta_\ast^{1/12}}{n\eta_\ast\abs{\eta_1\eta_2}^{1/2}}+ \frac{\theta_1}{n\eta_\ast}\Bigr].
    \end{equation}
    By iterating~\eqref{eq G12 iter} we can use \(\abs{\braket{G_{12}A}}\prec\theta_2\ll\theta_1\) as new input in~\eqref{eq a priori theta} to obtain \(\abs{\braket{G_{12}A}}\prec\theta_3\ll\theta_2\) since \(n\eta_\ast\gg\norm{\wh\cB^{-1}}\). Here \(\theta_j\), for \(j=3,4,\dots\), is defined iteratively by replacing \(\theta_1\) with \(\theta_{j-1}\) in the rhs.\ of the defining equation for \(\theta_2\) in~\eqref{eq G12 iter}. This improvement continues until the fixed point of this iteration, i.e.\ until \(\theta_N^{3/4}\) approaches \(\norm{\wh\cB^{-1}}n^{-1}\eta_\ast^{-7/4}\). For any given \(\xi> 0\), after finitely many steps \(N=N(\xi)\) the iteration stabilizes to
    \[ \theta_\ast \lesssim n^\xi\biggl[ \norm{\wh\cB^{-1}} + \frac{ \norm{\wh\cB^{-1}}}{n\eta_\ast}\frac{\eta_\ast^{1/12}}{\abs{\eta_1\eta_2}^{1/2}} + \frac{1}{\eta_\ast}\Bigl(\frac{ \norm{\wh\cB^{-1}}}{n\eta_\ast}\Bigr)^{4/3}\biggr], \]
    from which
    \[ \abs{\braket{A(G_{12}-M_{12})}}\prec \frac{\norm{\wh\cB^{-1}}}{ n\eta_\ast\abs{\eta_1\eta_2}^{1/2}}\Bigl( \eta_\ast^{1/12}+\eta_\ast^{1/4}\norm{\wh\cB^{-1}} +\frac{1}{\sqrt{n\eta_\ast}}+\Bigl(\frac{\norm{\wh\cB^{-1}}}{n\eta_\ast}\Bigr)^{1/4}\Bigr), \]
    and therefore the averaged bound in~\eqref{final local law} follows.

    For the isotropic bound in~\eqref{final local law} note that
    \[\braket{\vx,(G_{12}-M_{12})\vy} = \Tr \bigl[(\wh\cB^\ast)^{-1}[\vx\vy^\ast]\bigr]^\ast\wh\cB[G_{12}-M_{12}]  \]
    and that due to the block-structure of \(\wh\cB\) we have
    \[(\wh\cB^\ast)^{-1}[\vx\vy^\ast] = \sum_{i=1}^4 \vx_i \vy_i^\ast,\qquad \norm{\vx_i}\norm{\vy_i}\lesssim \norm{\wh\cB^{-1}},\]
    for some vectors \(\vx_i,\vy_i\). The isotropic bound in~\eqref{final local law} thus follows in combination with the isotropic bound in~\eqref{single local law},~\eqref{eq B G-M} and~\eqref{prop iso bound} applied to the pairs of vectors \(\vx_i,\vy_i\). This completes the proof of the theorem modulo the proof of Proposition~\ref{prop prob bound}.
\end{proof}

\subsection{Probabilistic bound and the proof of Proposition~\ref{prop prob bound}}\label{section prob bound}
We follow the graphical expansion outlined in~\cite{MR3941370,MR4134946} adapted to the current setting. We focus on the case when \(X\) has complex entries  and additionally mention the few  changes required when \(X\) is a real matrix. We abbreviate \(G_{12}=G_1 B G_2\) and use iterated cumulant expansions to expand \(\E\abs{\braket{\vx,\un{WG_{12}}\vy}}^{2p}\) and \(\E\abs{\braket{\un{WG_{12}}A}}^{2p}\) in terms of polynomials in entries of \(G\). For the expansion of the first \(W\) we have in the complex case
\begin{equation}\label{1 cum exp}\begin{split}
        &\E \Tr(\un{W G_{12}}A) \Tr(\un{W G_{12}}A)^{p-1} \Tr(A^\ast\un{ G_{12}^\ast W})^p \\
        &\quad= \frac{1}{n}\E\sum_{ab} R_{ab} \Tr(\Delta^{ab} G_{12}A) \partial_{ba} \Bigl[ \Tr(\un{W G_{12}}A)^{p-1} \Tr(A^\ast \un{G_{12}^\ast W})^p  \Bigr] \\
        &\qquad + \sum_{k\ge 2}\sum_{ab}\sum_{\bm\alpha\in\{ab,ba\}^k} \frac{\kappa(ab,\bm\alpha)}{k!}  \\
        &\qquad\qquad\quad \times\E\partial_{\bm\alpha} \Bigl[ \Tr(\Delta^{ab} G_{12}A) \Tr(\un{W G_{12}}A)^{p-1} \Tr(A^\ast \un{G_{12}^\ast W})^p  \Bigr]
    \end{split}\end{equation}
and similarly for \(\braket{\vx,\un{WG_{12}}\vy}\), where unspecified summations \(\sum_a\) are understood to be over \(\sum_{a\in[2n]}\), and \((\Delta^{ab})_{cd}:= \delta_{ac}\delta_{bd}\).
Here we introduced the matrix \(R_{ab}:= \bm1(a\le n,b>n)+\bm 1(a>n,b\le n) \) which is the rescaled second order cumulant (variance), i.e.\ \(R_{ab}=n\kappa(ab,ba)\). For \(\bm\alpha=(\alpha_1,\dots,\alpha_k)\) we denote the joint cumulant of \(w_{ab},w_{\alpha_1},\dots,w_{\alpha_k}\) by \(\kappa(ab,\bm\alpha)\) which is non-zero only for \(\bm\alpha\in\{ab,ba\}^k\). The derivative \(\partial_{\bm\alpha}\) denotes the derivative with respect to \(w_{\alpha_1},\dots,w_{\alpha_k}\). Note that in~\eqref{1 cum exp} the \(k=1\) term differs from the \(k\ge 2\) terms in two aspects. First, we only consider the \(\partial_{ba}\) derivative since in the complex case we have \(\kappa(ab,ab)=0\). Second, the action of the derivative on the first trace is not present since it is cancelled by the \emph{self-renormalisation} of \(\un{WG_{12}}\).

In the real case~\eqref{1 cum exp} differs slightly. First, for the \(k=1\) terms both \(\partial_{ab}\) and \(\partial_{ba}\) have to be taken into account with the same weight \(R\) since \(\kappa(ab,ab)=\kappa(ab,ba)\). Second, we chose only to renormalise the effect of the \(\partial_{ba}\)-derivative and hence the \(\partial_{ab}\)-derivative acts on all traces. Thus in the real case, compared to~\eqref{1 cum exp} there is an additional term given by
\[\frac{1}{n}\E\sum_{ab}R_{ab} \partial_{ab}\Bigl[\Tr(\Delta^{ab} G_{12} A ) \Tr(\un{W G_{12}}A)^{p-1} \Tr(A^\ast \un{G_{12}^\ast W})^p \Bigr].\]

The main difference to~\cite[Section 4]{MR3941370} and~\cite[Section 4]{MR4134946} is that therein instead of \(\un{WG_{12}}\) the single-\(G\) renormalisation \(\un{WG}\) was considered. With respect to the action of the derivatives there is, however, little difference between the two since we have
\[ \partial_{ab}G=-G\Delta^{ab}G, \quad \partial_{ab}G_{12}=-G_1 \Delta^{ab}G_{12} - G_{12}\Delta^{ab}G_2.\]
Therefore after iterating the expansion~\eqref{1 cum exp} we structurally obtain the same polynomials as in~\cite{MR3941370,MR4134946}, except of the slightly different combinatorics and the fact that exactly \(2p\) of the \(G\)'s are \(G_{12}\)'s and the remaining \(G\)'s are either \(G_1\) or \(G_2\). Thus, using the local law for \(G_i\) in the form
\begin{gather}
    \begin{aligned}
        \abs{\braket{\vx,G_i\vy}}    & \prec 1,                                                                                                                                                          \\
        \abs{\braket{\vx,G_{12}\vy}} & \le \sqrt{\braket{\vx,G_{1} G_1^\ast\vx}}\sqrt{\braket{\vy,BG_2G_2^\ast B^\ast\vy}}                                                                               \\
                                     & =\frac{1}{\sqrt{\abs{\eta_1}\abs{\eta_2}}}\sqrt{\braket{\vx,(\Im G_1)\vx}}\sqrt{\braket{\vy,B(\Im G_2)B^\ast\vy}} \prec \frac{1}{\sqrt{\abs{\eta_1}\abs{\eta_2}}}
    \end{aligned}\label{eq G trivial bound}\raisetag{-6em}
\end{gather}
for \(\norm{\vx}+\norm{\vy}\lesssim 1\), we obtain exactly the same bound as in~\cite[Eq.~(23a)]{MR3941370} times a factor of \((\abs{\eta_1}\abs{\eta_2})^{-p}\) accounting for the \(2p\) exceptional \(G_{12}\) edges, i.e.
\begin{equation}\label{eq ward before}
    \E\abs{\braket{\vx,\un{WG_{12}}\vy}}^{2p} \lesssim \frac{n^\epsilon}{(n\eta_\ast)^{p}\abs{\eta_1}^p\abs{\eta_2}^p},\quad \E\abs{\braket{\un{WG_{12}}A}}^{2p} \lesssim \frac{n^\epsilon}{(n\eta_\ast)^{2p}\abs{\eta_1}^p\abs{\eta_2}^p}.
\end{equation}
The isotropic bound from~\eqref{eq ward before} completes the proof of~\eqref{prop iso bound}.

It remains to improve the averaged bound in~\eqref{eq ward before} in order to obtain~\eqref{prop av bound}. We first have to identify where the bound~\eqref{eq ward before} is suboptimal. By iterating the expansion~\eqref{1 cum exp} we obtain a complicated polynomial expression in terms of entries of \(G_{12},G_1,G_2\) which is most conveniently represented graphically as
\begin{equation}\label{eq graph expansion}
    \E\abs{\braket{\un{W G_{12}}A}}^{2p} = \sum_{\Gamma\in \mathrm{Graphs}(p)} c(\Gamma)\E\Val(\Gamma) + \cO\Bigl(n^{-2p}\Bigr)
\end{equation}
for some finite collection of \(\mathrm{Graphs}(p)\). Before we precisely define the \emph{value of \(\Gamma\)}, \(\Val(\Gamma)\), we first give two examples. Continuing~\eqref{1 cum exp} in case \(p=1\) we have
\begin{subequations}\label{eq graph ex}
    \begin{equation}\label{eq graph ex a2} \begin{split}
            &\E\Tr(\un{W G_{12}}A) \Tr(A^\ast \un{G_{12}^\ast W}) \\
            &= \sum_{ab} \frac{R_{ab}}{n}\E \Tr(\Delta^{ab}G_{12}A)\Tr(A^\ast G_{12}^\ast \Delta^{ba} ) \\
            &\quad -\sum_{ab}\frac{R_{ab}}{n}\E \Tr(\Delta^{ab}G_{12}A) \Tr(A^\ast\un{ G_2^\ast \Delta^{ba} G_{12}^\ast W}) \\
            &\quad - \sum_{ab} \frac{R'_{ab}}{2n^{3/2}} \E \Tr(\Delta^{ab}G_1\Delta^{ba}G_{12}A) \Tr(A^\ast G_{12}^\ast \Delta^{ba}) +\cdots
        \end{split}
    \end{equation}
    where, for illustration, we only kept two of the three Gaussian terms (the last being when \(W\) acts on \(G_1^\ast\)) and one non-Gaussian term. For the non-Guassian term we set \(R'_{ab}:=n^{3/2}\kappa(ab,ba,ba)\), \(\abs{R'_{ab}}\lesssim 1\). Note that in the case of i.i.d.\ matrices with \(\sqrt{n}x_{ab}\stackrel{\mathsf{d}}{=}x\), we have \(R'_{ab}=\kappa(x,\overline{x},\overline{x})\) for \(a\le n,b>n\) and \(R'_{ab}=\kappa(x,x,\overline{x})=\overline{\kappa(x,\overline{x},\overline{x})}\) for \(a>n,b\le n\). For our argument it is of no importance whether matrices representing cumulants of degree at least three like \(R'\) are block-constant. It is important, however, that the variance \(\kappa(ab,ba)\) represented by \(R\) is block-constant since later we will perform certain resummations. For the second term on the rhs.\ of~\eqref{eq graph ex a2} we then obtain by another cumulant expansion that
    \begin{equation}\label{eq graph ex a}
        \begin{split}
            &\sum_{ab}\frac{R_{ab}}{n}\E \Tr(\Delta^{ab}G_{12}A) \Tr(A^\ast\un{ G_2^\ast \Delta^{ba} G_{12}^\ast W})\\
            &= - \sum_{ab}\sum_{cd}\frac{R_{ab}R_{cd}}{n^2}\E (G_{12}\Delta^{dc}G_2A)_{ba} \Tr(A^\ast G_2^\ast \Delta^{ba}G_{12}^\ast \Delta^{cd})+\cdots\\
            & \;\;- \sum_{ab}\sum_{cd} \frac{R_{ab}R'_{cd}}{2!n^{5/2}} \E (G_{12}\Delta^{dc}G_2A)_{ba} \Tr (A^\ast G_2^\ast\Delta^{ba}G_{12}^\ast\Delta^{dc}G_{1}^\ast\Delta^{cd}),
        \end{split}
    \end{equation}
    where we kept one of the two Gaussian terms and one third order term.
    After writing out the traces,~\eqref{eq graph ex a2}--\eqref{eq graph ex a} become
    \begin{equation}
        \begin{split}
            &\sum_{ab}\frac{R_{ab}}{n} \E (G_{12}A)_{ba} (A^\ast G_{12}^\ast)_{ab}+\cdots \\
            &-  \sum_{ab} \frac{R'_{ab}}{n^{3/2}} \E (G_1)_{bb} (G_{12}A)_{aa} (A^\ast G_{12}^\ast)_{ab} \\
            & + \sum_{ab}\sum_{cd} \frac{R_{ab}R_{cd}}{n^2} \E (G_{12})_{bd} (G_2A)_{ca} (A^\ast G_2^\ast)_{d b} (G_{12}^\ast)_{ac}\\
            & +  \sum_{ab}\sum_{cd}\frac{R_{ab}R'_{cd}}{2!n^{5/2}} \E (G_{12})_{bd} (G_{2}A)_{ca} (G_1^\ast)_{cc} (A^\ast G_2^\ast)_{db} (G_{12}^\ast)_{ad}.
        \end{split}
    \end{equation}
\end{subequations}

If \(X\) is real, then in~\eqref{eq graph ex} some additional terms appear since \(\kappa(ab,ab)=\kappa(ab,ba)\) in the real case, while \(\kappa(ab,ab)=0\) in the complex case. In the first equality of~\eqref{eq graph ex} this results in additional terms like
\begin{equation}\label{eq real additional terms}
    \begin{split}
        \sum_{ab}\frac{R_{ab}}{n} \E \Bigl(& - \Tr(\Delta^{ab}G_1 \Delta^{ab}G_{12}A)\Tr(A^\ast \un{G_{12}^\ast W}) \\
        &+ \Tr( \Delta^{ab}G_{12}A)\Tr(A^\ast G_{12}^\ast \Delta^{ab}) \\
        &- \Tr( \Delta^{ab}G_{12}A)\Tr(A^\ast \un{G_2^\ast \Delta^{ab} G_{12}^\ast W}) +\dots \Bigr).
    \end{split}
\end{equation}
Out of the three terms in~\eqref{eq real additional terms}, however, only the first one is qualitatively different from the terms already considered in~\eqref{eq graph ex} since the other two are simply transpositions of already existing terms. After another expansion of the first term in~\eqref{eq real additional terms} we obtain terms like
\begin{equation}\label{eq real additional terms 2}
    \begin{split}
        &\sum_{ab}\frac{R_{ab}}{n} (G_{12}A)_{ba}(A^\ast G_{12}^\ast)_{ba} +\cdots \\
        & +  \sum_{ab}\sum_{cd}\frac{R_{ab}R_{cd}}{n^2} (G_1)_{ba}(G_{12}A)_{ba} (A^\ast G_2^\ast)_{dc} (G_{12}^\ast)_{dc}\\
        & + \sum_{ab}\sum_{cd}\frac{R_{ab}R_{cd}'}{2!n^{5/2}} (G_{12})_{bc}(G_2A)_{da} (A^\ast G_2^\ast)_{da} (G_{12}^\ast)_{bd} (G_2^\ast)_{cc}
    \end{split}
\end{equation}
specific to the real case.

Now we explain how to encode~\eqref{eq graph ex} in the graphical formalism~\eqref{eq graph expansion}. The summation labels \(a_i,b_i\) correspond to vertices, while matrix entries correspond to edges between respective labelled vertices.
We distinguish between the cumulant- or \(\kappa\)-edges \(E_\kappa\), like \(R,R'\) and \(G\)-edges \(E_G\), like \((A^\ast G_2^\ast)_{db}\) or \((G_{12}^\ast)_{ab}\), but do not graphically distinguish between \(G_1,G_{12}\), \(A^\ast G_2^\ast\), etc. The four terms from the rhs.\ of~\eqref{eq graph ex} would thus be represented as
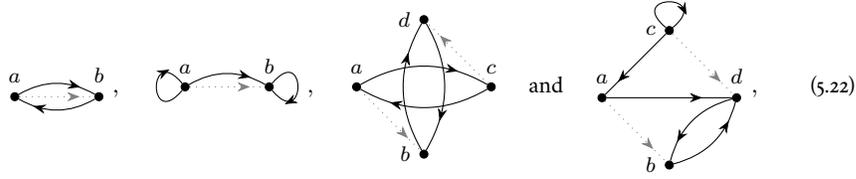
\begin{equation}\label{graph example}
    \ssGraph{a1[label=\(a\)] --[g,gray,dotted] b1[label=\(b\)]; a1 --[g,bl] b1; b1 --[g,bl] a1;},\quad \ssGraph{a1[label=\(a\)] --[g,gray,dotted] b1[label=\(b\)]; a1 --[g,bl] b1; a1 --[g,gll] a1; b1 --[g,glr] b1;}, \quad \sGraph{a1[label=\(a\)] --[gray, dotted] b1[label=left:\(b\)]; a2[label=above:\(c\)] --[gray, dotted] b2[label=left:\(d\)]; b1 --[bl] b2; a2 --[bl] a1; b2 --[bl] b1; a1 --[bl] a2; }\quad \text{and}\quad\sGraph{a1[label=above:\(a\)] --[gray, dotted] b1[label=left:\(b\)];b1 --[br] b2; a2[label=left:\(c\)] --[gray,dotted] b2[label=above:\(d\)]; a2 --[] a1; a2 --[glt] a2; b2 --[br] b1; a1 --[] b2; },
\end{equation}
where the edges from \(E_G\) are solid and those from \(E_\kappa\) dotted. Similarly, the three examples from~\eqref{eq real additional terms 2} would be represented as
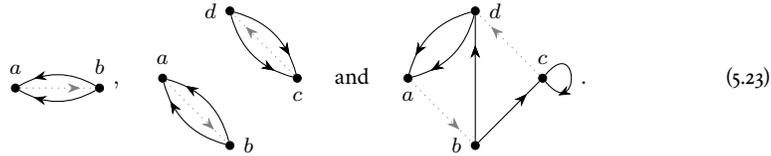
\begin{equation}\label{eq real graph ex}
    \ssGraph{a1[label=\(a\)] --[g,gray,dotted] b1[label=\(b\)]; b1 --[g,br] a1; b1 --[g,bl] a1;},\quad \sGraph{a1[label=above:\(a\)] --[gray, dotted] b1[label=right:\(b\)]; a2[label=below:\(c\)] --[gray, dotted] b2[label=left:\(d\)]; b1 --[g,bl] a1; b1 --[g,br] a1;  b2 --[g,bl] a2; b2 --[g,br] a2;  }\quad \text{and}\quad \sGraph{a1[label=below:\(a\)] --[gray, dotted] b1[label=left:\(b\)]; a2[label=above:\(c\)] --[gray, dotted] b2[label=right:\(d\)]; b1 --[g] a2; b2 --[g,bl] a1;  b2 --[g,br] a1; b1 --[g] b2; a2 --[glr] a2;   }.
\end{equation}

It is not hard to see that after iteratively performing cumulant expansions up to order \(4p\) for each remaining \(W\) we obtain a finite collection of polynomial expressions in \(R\) and \(G\) which correspond to graphs \(\Gamma\) from a certain set \(\mathrm{Graphs}(p)\) with the following properties. We consider a directed graph \(\Gamma = (V, E_\kappa\cup E_G)\) with an even number \(\abs{V}=2k\) of vertices, where \(k\) is the number of cumulant expansions along the iteration. The edge set is partitioned into two types of disjoint edges, the elements of \(E_\kappa\) are called \emph{cumulant edges} and the elements of \(E_G\) are called \emph{\(G\)-edges}. For \(u\in V\) we define the \(G\)-degree of \(u\) as
\[\begin{split}
        d_G(u):={}& d_G^\mathrm{out}(u) +d_G^\mathrm{in} (u),\\
        d_G^\mathrm{out}(u):={}& \abs{\set{v\in V \given (uv)\in E_G}},\quad d_G^\mathrm{in}(u):=\abs{\set{v\in V \given (vu)\in E_G}}.
    \end{split}\]
We now record some structural attributes.
\begin{enumerate}[label=(A\arabic*)]
    \item\label{perfect matching} The graph \((V,E_\kappa)\) is a perfect matching and in particular \(\abs{V}=2\abs{E_\kappa}\). For convenience we label the vertices by \(u_1,\dots,u_k,v_1,\dots,v_k\) with cumulant edges \((u_1v_1),\dots,(u_k v_k)\). The ordering of the elements of \(E_\kappa\) indicated by \(1,\dots,k\) is arbitrary and irrelevant.
    \item\label{number of kappa edges} The number of \(\kappa\)-edges is bounded by \(\abs{E_\kappa}\le 2p\) and therefore \(\abs{V}\le 4p\)
    \item\label{degree equal} For each \((u_i v_i)\in E_\kappa\), the \emph{\(G\)-degree}  of both vertices agrees, i.e.\ \(d_G(u_i)=d_G(v_i)=: d_G(i)\). Furthermore the \(G\)-degree satisfies \(2\le d_G(i)\le 4p\). Note that loops \((uu)\) contribute a value of \(2\) to the degree.
    \item\label{no loops} If \(d_G(i)=2\), then no loops are adjacent to either \(u_i\) or \(v_i\).
    \item\label{number of G edges} We distinguish two types of \(G\)-edges \(E_G=E_G^1\cup E_G^2\) whose numbers are given by \[\abs{E_G^2}=2p, \quad \abs{E_G^1}=\sum_i  d_G(i)-2p, \quad \abs{E_G}=\abs{E_G^1}+\abs{E_G^2}.\]
\end{enumerate}
Note that in the examples~\eqref{eq real graph ex} above we had \(\abs{E_\kappa}=1\) in the first and \(\abs{E_\kappa}=2\) in the other two cases. For the degrees we had \(d_G(1)=2\) in the first case, \(d_G(1)=d_G(2)=2\) in the second case, and \(d_G(1)=2, d_G(2)=3\) in the third case. The number of \(G\)-edges involving \(G_{12}\) is \(2\) in all cases, while the number of remaining \(G\)-edges is \(0\), \(2\) and \(3\), respectively, in agreement with~\ref{number of G edges}. We now explain how we relate the graphs to the polynomial expressions they represent.
\begin{enumerate}[label=(I\arabic*)]
    \item\label{vertex inter} Each vertex \(u\in V\) corresponds to a summation \(\sum_{a\in[2n]}\) with a label \(a\) assigned to the vertex \(u\).
    \item\label{edge inter} Each \(G\)-edge \((uv)\in E_G^1\) represents a matrix \(\cG^{(uv)}=A_1 G_i A_2\) or \(\cG^{(uv)}=A_1 G_i^\ast A_2\) for some norm-bounded deterministic matrices \(A_1,A_2\). Each \(G\)-edge \((uv)\in E_G^2\) represents a matrix \(\cG^{(uv)}=A_1 G_{12} A_2\) or \(\cG^{(uv)}=A_1 G_{12}^\ast A_2\) for norm bounded matrices \(A_1,A_2\). We denote the matrices \(\cG^{(uv)}\) with a calligraphic ``G'' to avoid confusion with the ordinary resolvent matrix \(G\).
    \item\label{kappa edge rule} Each \(\kappa\)-edge \((uv)\) represents the matrix \[ R_{ab}^{(uv)} = \kappa(\underbrace{\sqrt{n}w_{ab},\dots,\sqrt{n}w_{ab}}_{d_G^\mathrm{in}(u)}, \underbrace{\sqrt{n}\ov{w_{ab}},\dots,\sqrt{n}\ov{w_{ab}}}_{d_G^\mathrm{out}(u)} ),\]
    where \(d_G^\mathrm{in}(u)=d_G^\mathrm{out}(v)\) and \(d_G^\mathrm{out}(u)=d_G^\mathrm{in}(v)\) are the in- and out degrees of \(u,v\).
    \item\label{value inter} Given a graph \(\Gamma\) we define its value\footnote{In~\cite{MR4134946} we defined the value with an expectation so that~\eqref{eq graph expansion} holds without expectation. In the present paper we follow the convention of~\cite{MR3941370} and consider the value as a random variable.} as
    \begin{equation}\label{Val Gamma def} \Val(\Gamma) := n^{-2p} \prod_{(u_i v_i)\in E_\kappa} \biggl(\sum_{a_i,b_i\in[2n]} n^{-d_G(i)/2}R^{(u_i v_i)}_{a_i b_i}\biggr) \prod_{(u_i v_i)\in E_G} \cG^{(u_i v_i)}_{a_i b_i},\end{equation}
    where \(R^{(u_i v_i  )}\) is as in~\ref{kappa edge rule} and \(a_i,b_i\) are the summation indices associated with \(u_i,v_i\).
\end{enumerate}
\begin{proof}[Proof of~\eqref{eq graph expansion}]
    In order to prove~\eqref{eq graph expansion} we have to check that the graphs representing the polynomial expressions of the cumulant expansion up to order \(4p\) indeed have the attributes~\ref{perfect matching}--\ref{number of G edges}. Here~\ref{perfect matching}--\ref{degree equal} follow directly from the construction, with the lower bound \(d_G(i)\ge 2\) being a consequence of \(\E w_{ab}=0\) and the upper bound \(d_G(i)\le 4p\) being a consequence of the fact that we trivially truncate the expansion after the \(4p\)-th cumulant. The error terms from the truncation are estimated trivially using~\eqref{eq G trivial bound}. The fact~\ref{no loops} that no \(G\)-loops may be adjacent to degree two \(\kappa\)-edges follows since due to the self-renormalisation \(\un{WG_{12}}\) the  second cumulant of \(W\) can only act on some \(W\) or \(G\) in another trace, or if it acts on some \(G\) in its own trace then it generates a \(\kappa(ab,ab)\) factor (only possible when \(X\) is real). In the latter case one of the two vertices has two outgoing, and the other one two incoming \(G\)-edges, and in particular no loops are adjacent to either of them. The counting of \(G_{12}\)-edges in \(E_G^2\) in~\ref{number of G edges} is trivial since along the procedure no \(G_{12}\)-edges can be created or removed. For the counting of \(G_i\) edges in \(E_G^1\) note that the action of the \(k\)-th order cumulant in the expansion of \(\un{WG_{12}}\) may remove \(k_1\) \(W\)'s and may create additional \(k_2\) \(G_i\)'s with \(k=k_1+k_2\), \(k_1\ge 1\). Therefore, since the number of \(G_i\) edges is \(0\) in the beginning, and the number of \(W\)'s is reduced from \(2p\) to \(0\) the second equality in~\ref{number of G edges} follows.

    It now remains to check that with the interpretations~\ref{vertex inter}--\ref{value inter} the values of the constructed graphs are consistent in the sense of~\eqref{eq graph expansion}. The constant \(c(\Gamma)\sim 1\) accounts for combinatorial factors in the iterated cumulant expansions and the multiplicity of identical graphs. The factor \(n^{-2p}\) in~\ref{value inter} comes from the \(2p\) normalised traces. The relation~\ref{kappa edge rule} follows from the fact that the \(k\)-th order cumulant of \(k_1\) copies of \(w_{ab}\) and \(k_2\) copies of \(\overline{w_{ab}}=w_{ba}\) comes together with \(k_1\) copies of \(\Delta^{ab}\) and \(k_2\) copies of \(\Delta^{ba}\). Thus \(a\) is the first index of some \(G\) a total of \(k_2\) times, while the remaining \(k_1\) times the first index is \(b\), and for the second indices the roles are reversed.
\end{proof}
Having established the properties of the graphs and the formula~\eqref{eq graph expansion}, we now estimate the value of any individual graph.
\subsubsection*{Naive estimate}
We first introduce the so called \emph{naive estimate}, \(\NEst(\Gamma)\), of a graph \(\Gamma\) as the bound on its value obtained by estimating the factors in~\eqref{Val Gamma def} as \(\abs{\cG^e_{ab}}\prec 1\) for \(e\in E_G^1\) and \(\abs{\cG^e_{ab}}\prec (\abs{\eta_1}\abs{\eta_2})^{-1/2}\) for \(e\in E_G^2\), \(\abs{R^{e}_{ab}}\lesssim 1\) and estimating summations by their size. Thus, we obtain
\begin{equation}\label{eq Gamma naive est}
    \begin{split}
        \Val(\Gamma)\prec \NEst(\Gamma):&=\frac{1}{n^{2p}\abs{\eta_1}^p\abs{\eta_2}^{p}} \prod_i \Bigl(n^{2-d_G(i)/2}\Bigr) \\
        &\le \frac{n^{\abs{E_\kappa^2}} n^{\abs{E_\kappa^3}/2}}{n^{2p}\abs{\eta_1}^p\abs{\eta_2}^{p}}   \le \frac{1}{\abs{\eta_1}^p\abs{\eta_2}^{p}},
    \end{split}
\end{equation}
where \[E_\kappa^j:=\set{(u_i,v_i)\given d_G(i)=j}\] is the set of degree \(j\) \(\kappa\)-edges, and in the last inequality we used \(\abs{E_\kappa^2}+\abs{E_\kappa^3}\le \abs{E_\kappa}\le 2p\).

\subsubsection*{Ward estimate}
The first improvement over the naive estimate comes from the effect that sums of resolvent entries are typically smaller than the individual entries times the summation size. This effect can easily be seen from the \emph{Ward} or resolvent identity \(G^\ast G=\Im G/\eta=(G-G^\ast)/(2\ii\eta)\). Indeed, the naive estimate of \(\sum_a G_{ab}\) is \(n\) using \(\abs{G_{ab}}\prec 1\). However, using  the Ward identity we can improve this to
\[ \biggl\lvert\sum_a G_{ab}\biggr\rvert \le \sqrt{2n} \sqrt{\sum_{a}\abs{G_{ab}}^2} = \sqrt{2n} \sqrt{ (G^\ast G)_{bb} }  = \sqrt{\frac{2n}{\eta}} \sqrt{(\Im G)_{bb}} \prec n \frac{1}{\sqrt{n\eta}}, \]
i.e.\ by a factor of \((n\eta)^{-1/2}\). Similarly, we can gain two such factors if the summation index \(a\) appears in two \(G\)-factors off-diagonally, i.e.\
\[  \biggl\lvert\sum_a (G_1)_{ab} (G_2)_{ca}\biggr\rvert \le \sqrt{(G_1^\ast G_1)_{bb}}\sqrt{(G_2 G_2^\ast)_{cc}}\prec n\frac{1}{n\eta}. \]
However, it is impossible to gain more than two such factors per summation. We note that we have the same gain also for summations of \(G_{12}\). For example, the naive estimate on \(\sum_{a}(G_{12})_{ab}\) is \(n\abs{\eta_1\eta_2}^{-1/2}\) since \(\abs{(G_{12})_{ab}}\prec\abs{\eta_1\eta_2}^{-1/2}\). Using the Ward identity, we obtain an improved bound of
\[ \begin{split}
        \biggl\lvert\sum_a (G_{12})_{ab}\biggr\rvert &\le \sqrt{2n}\sqrt{(G_{12}^\ast G_{12})_{bb}} =\sqrt{\frac{2n}{\abs{\eta_1}}} \sqrt{(G_2^\ast B^\ast(\Im G_1)B G_2 )_{bb}} \\
        &\lesssim \sqrt{\frac{n}{\abs{\eta_1}^2}} \sqrt{(G_2^\ast G_2 )_{bb}} \prec \frac{\sqrt{n}}{\abs{\eta_1}\abs{\eta_2}^{1/2}}\le\frac{n}{\abs{\eta_1\eta_2}^{1/2}}\frac{1}{\sqrt{n\eta_\ast}},
    \end{split} \]
where we recall \(\eta_\ast=\min\{\abs{\eta_1},\abs{\eta_2}\}\). Each of these improvements is associated with a specific \(G\)-edge with the restriction that one cannot gain simultaneously from more than two edges adjacent to any given vertex \(u\in V\) while summing up the index \(a\) associated with \(u\). Note, however, that globally it is nevertheless possible to gain from arbitrarily many \(G\)-edges adjacent to any given vertex, as long as the summation order is chosen correctly. In order to count the number edges giving rise to such improvements we recall a basic definition~\cite{MR266812} from graph theory.
\begin{definition}
    For \(k\ge 1\) a graph \(\Gamma=(V,E)\) is called \emph{\(k\)-degenerate} if any induced subgraph has minimal degree at most \(k\).
\end{definition}
The relevance of this definition in the context of counting the number of gains of \((n\eta_\ast)^{-1/2}\) lies in the following equivalent characterisation~\cite{MR193025}.
\begin{lemma}\label{lemma equiv coloring deg}
    A graph \(\Gamma=(V,E)\) is \(k\)-degenerate if and only
    if there exists an ordering of vertices \(\{v_1,\dots,v_n\}=V\) such that for each \(m\in[n]\) it holds that
    \begin{equation}\deg_{\Gamma[\{v_1,\dots,v_m\}]}(v_m)\le k \label{vertex ordering}\end{equation}
    where for \(V'\subset V\), \(\Gamma[V']\) denotes the induced subgraph on the vertex set \(V'\).
\end{lemma}
We consider a subset of non-loop edges \(E_\mathrm{Ward}\subset E_G\setminus\set{(vv)\given v\in V}\) for which Ward improvements will be obtained. We claim that if \(\Gamma_\mathrm{Ward}=(V,E_\mathrm{Ward})\) is \(2\)-degenerate, then we may gain a factor of \((n\eta_\ast)^{-1/2}\) from each edge in \(E_\mathrm{Ward}\). Indeed, take the ordering \(\{v_1,\dots,v_{2\abs{E_\kappa}}\}\) guaranteed to exist in Lemma~\ref{lemma equiv coloring deg} and first sum up the index \(a_1\) associated with \(v_1\). Since \(\Gamma_\mathrm{Ward}\) is \(2\)-degenerate there are at most two edges from \(E_\mathrm{Ward}\) adjacent to \(v_1\) and we can gain a factor of \((n\eta_\ast)^{-1/2}\) for each of them. Next, we can sum up the index associated with vertex \(v_2\) and again gain the same factor for each edge in \(E_\mathrm{Ward}\) adjacent to \(v_2\). Continuing this way we see that in total we can gain a factor of \((n\eta_\ast)^{-\abs{E_\mathrm{Ward}}/2}\) over the naive bound~\eqref{eq Gamma naive est}.
\begin{definition}[Ward estimate]\label{def ward est}
    For a graph \(\Gamma\) with fixed subset \(E_\mathrm{Ward}\subset E_G\) of edges we define \[\WEst(\Gamma):=\frac{\NEst(\Gamma)}{(n\eta_\ast)^{\abs{E_\mathrm{Ward}}/2}}.\]
\end{definition}
By considering only \(G\)-edges adjacent to \(\kappa\)-edges of degrees \(2\) and \(3\) it is possible to find such a \(2\)-degenerate set with
\[\abs{E_\mathrm{Ward}} = \sum_{i} (4-d_G(i))_+\]
elements, cf.~\cite[Lemma 4.7]{MR4134946}. As a consequence, as compared with the first inequality in~\eqref{eq Gamma naive est}, we obtain an improved bound
\begin{gather}
    \begin{aligned}
        \Val(\Gamma) & \prec \WEst(\Gamma)={} \frac{1}{n^{2p}\abs{\eta_1\eta_2}^p} (n\eta_\ast)^{-\abs{E_\mathrm{Ward}}/2}\prod_i\Bigl(n^{2-d_G(i)/2}\Bigr)                                                                  \\
                     & ={}\frac{1}{n^{2p}\abs{\eta_1\eta_2}^p} \prod_{d_G(i)=2}\Bigl(\frac{n}{n\eta_\ast}\Bigr) \prod_{d_G(i)=3}\Bigl(\frac{\sqrt{n}}{\sqrt{n\eta_\ast}}\Bigr) \prod_{d_G(i)\ge 4}\Bigl(n^{2-d_G(i)/2}\Bigr) \\
                     & \lesssim{} \frac{1}{(n\eta_\ast)^{2p}\abs{\eta_1\eta_2}^p} \eta_\ast^{2p+\sum_i(d_G(i)/2-2)} \lesssim \frac{1}{(n\eta_\ast)^{2p}\abs{\eta_1\eta_2}^p},
    \end{aligned}\label{ward improvement}\raisetag{-5em}
\end{gather}
where in the penultimate inequality we used \(n^{-1}\le \eta_\ast\), and in the ultimate inequality that \(d_G(i)\ge 2\) and \(\abs{E_\kappa}\le 2p\) which implies that the exponent of \(\eta_\ast\) is non-negative and \(\eta_\ast\lesssim 1\). Thus we gained a factor of \((n\eta_\ast)^{-2p}\) over the naive estimate~\eqref{eq Gamma naive est}.

\subsubsection*{Resummation improvements}
The bound~\eqref{ward improvement} is optimal if \(z_1=z_2\) and if \(\eta_1,\eta_2\) have opposite signs. In the general case \(z_1\ne z_2\) we have to use two additional improvements which both rely on the fact that the summations \(\sum_{a_i,b_i}\) corresponding to \((u_i,v_i)\in E_\kappa^2\) can be written as matrix products since \(d_G(u_i)=d_G(v_i)=2\). Therefore we can sum up the \(G\)-edges adjacent to \((u_i v_i)\) as
\begin{subequations}\label{eq summing up deg 2}
    \begin{equation}\label{eq summing up deg 2a}
        \begin{split}
            &\sum_{a_i b_i} G_{xa_i}G_{a_i y} G_{zb_i}G_{b_i w} R_{a_i b_i} \\
            &\quad=\sum_{a_i b_i} G_{xa_i}G_{a_i y} G_{zb_i}G_{b_i w} \Bigl[\bm 1(a_i>n,b_i\le n)+\bm 1(a_i\le n,b_i>n)\Bigr]\\
            &\quad =(G E_1 G)_{xy} (GE_2 G)_{zw} + (G E_2 G)_{xy} (GE_1 G)_{zw},
        \end{split}
    \end{equation}
    where \(E_1\), \(E_2\) are defined in~\eqref{E1 E2 def}, in the case of four involved \(G\)'s and \(d_G^\mathrm{in}=d_G^\mathrm{out}=1\). If one vertex has two incoming, and the other two outgoing edges (which is only possible if \(X\) is real), then we similarly can sum up
    \begin{equation} \sum_{a b} G_{x a}G_{ya} G_{b z}G_{b w} R_{a b} = (GE_1G^t)_{xy} (G^t E_2 G)_{zw}+(GE_2G^t)_{xy} (G^t E_1 G)_{zw},\end{equation}
    so merely some \(G\) is replaced by its transpose \(G^t\) compared to~\eqref{eq summing up deg 2a} which will not change any estimate. In the remaining cases with two and three involved \(G\)'s we similarly have
    \begin{equation}
        \begin{split}
            \sum_{a b} G_{b a} G_{a b}R_{a b}& = \Tr G E_1 G E_2 + \Tr G E_2 GE_1\\
            \sum_{a b} G_{x a} G_{a b} G_{b y} R_{a b}& = (G E_1 G E_2 G)_{xy} + (G E_2 G E_1 G)_{xy}.
        \end{split}
    \end{equation}
\end{subequations}

By carrying out all available \emph{partial summations} at degree-\(2\) vertices as in~\eqref{eq summing up deg 2} for the value \( \Val(\Gamma)\) of some graph \(\Gamma\) we obtain a collection of \emph{reduced graphs}, in which cycles of \(G\)'s are contracted to the trace of their matrix product, and chains of \(G\)'s are contracted to single edges, also representing the matrix products with two \emph{external} indices. We denote generic cycle-subgraphs of \(k\) edges from \(E_G\) with vertices of degree two by \(\Gamma^\circ_k\), and generic chain-subgraphs of \(k\) edges from \(E_G\) with \emph{internal} vertices of degree two and external vertices of degree at least three by \(\Gamma^-_k\). With a slight abuse of notation we denote the \emph{value} of \(\Gamma_k^\circ\) by \(\Tr\Gamma_k^\circ\), and the \emph{value} of \(\Gamma_k^{-}\) with external indices \((a,b)\) by \((\Gamma_k^-)_{ab}\), where for a fixed choice of \(E_1,E_2\) in~\eqref{eq summing up deg 2} the internal indices are summed up. The actual choice of \(E_1, E_2\) is irrelevant for our analysis, hence we will omit it from the notation. The concept of the \emph{naive} and \emph{Ward} estimates of any graph \(\Gamma\) carry over naturally to these chain and cycle-subgraphs by setting
\begin{equation}\label{chain cycle ests}
    \begin{split}
        \NEst(\Gamma_k^\circ):={}&\frac{n^k}{\abs{\eta_1\eta_2}^{\abs{E_G^2(\Gamma_k^\circ)}/2}}, \quad \NEst(\Gamma_k^-):=\frac{n^{k-1}}{\abs{\eta_1\eta_2}^{\abs{E_G^2(\Gamma_k^-)}/2}}, \\
        \WEst(\Gamma_k^{\circ/-})={}&\frac{\NEst(\Gamma_k^{\circ/-})}{(n\eta_\ast)^{\abs{E_\mathrm{Ward}(\Gamma_k^{\circ/-})}/2}}, \,\,\, E_\mathrm{Ward}(\Gamma_k^{\circ/-})=E_G(\Gamma_k^{\circ/-})\cap E_\mathrm{Ward}(\Gamma).
    \end{split}
\end{equation}

After contracting the chain- and cycle-subgraphs we obtain \(2^{\abs{E_\kappa^2}}\) reduced graphs \(\Gamma_\mathrm{red}\) on the vertex set
\[ V(\Gamma_\mathrm{red}):= \set{v\in V(\Gamma) \given d_G(v)\ge 3}\]
with \(\kappa\)-edges
\[ E_\kappa(\Gamma_\mathrm{red}):= E_\kappa^{\ge 3}(\Gamma) \]
and \(G\)-edges
\[E_G(\Gamma_\mathrm{red}) := \set{(uv)\in E_G(\Gamma)\given \min\{d_G(u),d_G(v)\}\ge 3} \cup E_G^\mathrm{chain}(\Gamma_\mathrm{red}),\]
with additional \emph{chain-edges}
\[
    \begin{split}
        E_G^\mathrm{chain}(\Gamma_\mathrm{red}):&= \set*{(u_1u_{k+1})\given \parbox{18em}{\(k\ge 2\), \(u_1,u_{k+1}\in V(\Gamma_\mathrm{red})\), \(\exists \Gamma_k^-\subset \Gamma\), \(V(\Gamma_k^-)=(u_1,\dots, u_{k+1})\) }}.
    \end{split}
\]
The additional chain edges \((u_1u_{k+1})\in E_G^\mathrm{chain}\) naturally represent the matrices
\[\cG^{(u_1u_{k+1})} := \bigl((\Gamma_k^-)_{ab}\bigr)_{a,b\in[2n]}\]
whose entries are the values of the chain-subgraphs. Note that due to the presence of \(E_1,E_2\) in~\eqref{eq summing up deg 2} the matrices associated with some \(G\)-edges can be multiplied by \(E_1,E_2\). However, since in the definition~\ref{edge inter} of \(G\)-edges the multiplication with generic bounded deterministic matrices is implicitly allowed, this additional multiplication will not be visible in the notation. Note that the reduced graphs contain only vertices of at least degree three, and only \(\kappa\)-edges from \(E_\kappa^{\ge 3}\). The definition of value, naive estimate and Ward estimate naturally extend to the reduced graphs and we have
\begin{equation} \label{eq reduced graph cal}
    \Val(\Gamma)= \sum \Val(\Gamma_\mathrm{red}) \prod_{\Gamma_k^\circ\subset\Gamma} \Tr\Gamma_k^\circ
\end{equation}
and
\begin{equation}
    \begin{split}
        \NEst(\Gamma) &= \NEst(\Gamma_\mathrm{red}) \prod_{\Gamma_k^\circ\subset\Gamma} \NEst(\Gamma_k^\circ),  \\
        \WEst(\Gamma)&= \WEst(\Gamma_\mathrm{red}) \prod_{\Gamma_k^\circ\subset\Gamma} \WEst(\Gamma_k^\circ).
    \end{split}
\end{equation}
The irrelevant summation in~\eqref{eq reduced graph cal} of size \(2^{\abs{E_\kappa^2}}\) is due to the sums in~\eqref{eq summing up deg 2}.

Let us revisit the examples~\eqref{graph example} to illustrate the summation procedure. The first two graphs in~\eqref{graph example} only have degree-\(2\) vertices, so that the reduced graphs are empty with value \(n^{-2p}=n^{-2}\), hence
\[ \Val(\Gamma)=\frac{1}{n^2}\sum \Tr \Gamma_2^\circ \qquad \Val(\Gamma)=\frac{1}{n^2}\sum (\Tr\Gamma_2^\circ) (\Tr \Gamma_2^\circ),\]
where the summation is over two and, respectively, four terms. The third graph in~\eqref{graph example} results in no traces but in four reduced graphs
\[\Val(\Gamma) = \sum \Val(\hspace{-1em}\sGraph{a --[gray,dotted] b; a --[glt] a; b --[glt,double] b; a --[bl,double] b;}\hspace{-1em}),\]
where for convenience we highlighted the chain-edges \(E_G^\mathrm{chain}\) representing \(\Gamma_k^-\) by double lines (note that the two endpoints of a chain edge may coincide, but it is not interpreted as a cycle graph since this common vertex has degree more than two, so it is not summed up into a trace along the reduction process). Finally, to illustrate the reduction for a more complicated graph, we have
\[ \Val\left(\sGraph{ a1[label=left:\(a_1\)]--[gray,dotted]b1[label=left:\(b_1\)]; a2[label=right:\(a_2\)]--[gray,dotted]b2[label=right:\(b_2\)]; a3[label=below:\(a_3\)]--[gray,dotted]b3[label=left:\(b_3\)]; a3--a2--a1--[bl]b2--a1; a3--[glr] a3 --[bl] b3 --[glt] b3; b1 --[glb] b1 --[] b3; a4[label=right:\(a_4\)] --[gray,dotted] b4[label=left:\(b_4\)] --[bl] a4 --[bl] b4; }\right)= \sum (\Tr \Gamma_2^-) \Val\left(\sGraph{ a1[label=left:\(a_1\)]--[gray,dotted]b1[label=left:\(b_1\)];  a3[label=right:\(a_3\)]--[gray,dotted] b3[label=right:\(b_3\)]; a3--[double] a1--[glt,double]a1; a3--[glt] a3 --[bl] b3 --[glt] b3; b1 --[glr] b1 --[] b3;}\right) \]
where we labelled the vertices for convenience, and the summation on the rhs.\ is over four assignments of \(E_1,E_2\).

Since we have already established a bound on \(\Val(\Gamma)\prec\WEst(\Gamma)\) we only have to identify the additional gain from the resummation compared to the \emph{Ward-estimate}~\eqref{ward improvement}.

We will need to exploit two additional effects:
\begin{enumerate}[label=(\roman*)]
    \item\label{suboptimal Ward} The Ward-estimate is sub-optimal whenever, after resummation, we have some contracted cycle \(\Tr \Gamma_k^\circ\) or a reduced graph with a chain-edge \(\Gamma_k^-\) with \(k\ge 3\).
    \item\label{G12 Ward} When estimating \(\Tr \Gamma_k^\circ\), \(k\ge2\) with \(\Gamma_k^\circ\) containing some \(G_{12}\), then also the improved bound from~\ref{suboptimal Ward} is sub-optimal and there is an additional gain from using the a priori bound \(\abs{\braket{G_{12}A}}\prec \theta \).
\end{enumerate}
We now make the additional gains~\ref{suboptimal Ward}--\ref{G12 Ward} precise.
\begin{lemma}\label{gain lemma} For \(k\ge 2\) let \(\Gamma_k^\circ\) and \(\Gamma_k^-\) be some cycle and chain subgraphs.
    \begin{enumerate}[label=(\roman*)]
        \begin{subequations}
            \item We have
            \begin{equation}\label{long Gk gain}
                \abs{\Tr \Gamma_k^\circ} \prec (n\eta_\ast)^{-(k-2)/2} \WEst(\Gamma_k^\circ)
            \end{equation}
            and for all \(a,b\)
            \begin{equation}\label{long Gk iso gain}
                \abs{(\Gamma_k^-)_{ab}}\prec (n\eta_\ast)^{-(k-2)/2} \WEst(\Gamma_k^-).
            \end{equation}
            \item If \(\Gamma_k^\circ\) contains at least one \(G_{12}\) then we have a further improvement of \((\eta_\ast\theta)^{1/2}\), i.e.
            \begin{equation}\label{long Gk gain theta}
                \abs{\Tr \Gamma_k^\circ} \prec \sqrt{\eta_\ast\theta} (n\eta_\ast)^{-(k-2)/2} \WEst(\Gamma_k^\circ),
            \end{equation}
        \end{subequations}
        where \(\theta\) is as in~\eqref{eq a priori theta}.
    \end{enumerate}
\end{lemma}
The proof of Lemma~\ref{gain lemma} follows from the following optimal bound on general products \(G_{j_1\dots j_k}\) of resolvents and generic deterministic matrices.
\begin{lemma}\label{lemma general products}
    Let \(w_1,w_2,\dots\), \(z_1,z_2,\dots\) denote arbitrary spectral parameters with \(\eta_i=\Im w_i>0\). With \(G_j=G^{z_j}(w_{j})\) we then denote generic products of resolvents \(G_{j_1},\dots G_{j_k}\) or their adjoints/transpositions (in that order) with arbitrary bounded deterministic matrices in between by \(G_{j_1\dots j_k}\), e.g.\ \(G_{1i1}=A_1G_1A_2G_i A_3G_1A_4\).
    \begin{enumerate}[label=(\roman*)]
        \begin{subequations}
            \item For \(j_1,\dots j_k\) we have the isotropic bound
            \begin{equation}\label{eq general iso bound}
                \abs{\braket{\vx,G_{j_1\dots j_k}\vy}} \prec \norm{\vx}\norm{\vy}\sqrt{\eta_{j_1}\eta_{j_k}}\Bigl(\prod_{n=1}^k \eta_{j_n}\Bigr)^{-1}.
            \end{equation}
            \item For \(j_1,\dots,j_k\) and any \(1\le s< t\le k\) we have the averaged bound
            \begin{equation}\label{eq general av bound}
                \abs{\braket{G_{j_1\dots j_k}}} \prec \sqrt{\eta_{j_{s}}\eta_{j_{t}}}\Bigl(\prod_{n=1}^k \eta_{j_n}\Bigr)^{-1}.
            \end{equation}
        \end{subequations}
    \end{enumerate}
\end{lemma}
Lemma~\ref{lemma general products} for example implies \(\abs{(G_{1i})_{ab}}\prec (\eta_1\eta_i)^{-1/2}\) or \(\abs{(G_{i1i})_{ab}}\prec (\eta_1\eta_i)^{-1}\). Note that the averaged bound~\eqref{eq general av bound} can be applied more flexibly by choosing \(s,t\) freely, e.g.
\[\abs{\braket{G_{1i1i}}}\prec \min\{\eta_1^{-1}\eta_i^{-2},\eta_1^{-2}\eta_i^{-1}\},\]
while \(\abs{\braket{\vx,G_{1i1i}\vy}}\prec \norm{\vx}\norm{\vy} (\eta_1\eta_i)^{-3/2}\).
\begin{proof}[Proof of Lemma~\ref{lemma general products}]
    We begin with
    \[  \begin{split}
            &\abs{\braket{\vx,G_{j_1\dots j_k}\vy}} \\
            &\quad\le \sqrt{\braket{\vx,G_{j_1}G_{j_1}^\ast \vx}}\sqrt{\braket{\vy,G_{j_2 \dots j_k}^\ast G_{j_2 \dots j_k}\vy}} \prec \frac{\norm{\vx}}{\sqrt{\eta_{j_1}}} \sqrt{\braket{\vy,G_{j_2 \dots j_k}^\ast G_{j_2 \dots j_k}\vy}}\\
            & \quad\lesssim \frac{\norm{\vx}}{\sqrt{\eta_{j_1}}} \frac{1}{\eta_{j_2}} \sqrt{\braket{\vy,G_{j_3 \dots j_k}^\ast G_{j_3 \dots j_k}\vy}} \lesssim \dots  \\
            &\quad \lesssim \frac{\norm{\vx}}{\sqrt{\eta_{j_1}}} \frac{1}{\eta_{j_2}\dots \eta_{j_{k-1}}} \sqrt{\braket{\vy,G_{j_k}^\ast G_{j_k}\vy}} \prec \frac{\norm{\vx}\norm{\vy}}{\sqrt{\eta_{j_1}\eta_{j_k}}} \frac{1}{\eta_{j_2}\dots \eta_{j_{k-1}}},
        \end{split}
    \]
    where in each step we  estimated the middle \(G_{j_2}^\ast G_{j_2}, G_{j_3}^\ast G_{j_3},\dots\) terms trivially by \(1/\eta_{j_2}^2,1/\eta_{j_3}^2,\dots\), and in the last step we used Ward estimate. This proves~\eqref{eq general iso bound}. We now turn to~\eqref{eq general av bound} where by cyclicity without loss of generality we may assume \(s=1\). Thus
    \[\begin{split}
            \abs{\braket{G_{j_1\dots j_k}}} &\le \sqrt{\braket{G_{j_1\dots j_{t-1}}G_{j_1\dots j_{t-1}}^\ast}}\sqrt{\braket{ G_{j_{t}\dots j_{k}}^\ast G_{j_{t}\dots j_{k}}}} \\
            &=\sqrt{\braket{G_{j_1\dots j_{t-1}}G_{j_1\dots j_{t-1}}^\ast}}\sqrt{\braket{  G_{j_{t}\dots j_{k}}G_{j_{t}\dots j_{k}}^\ast}} \\
            &\lesssim \Bigl(\prod_{n\ne 1,t} \frac{1}{\eta_{j_n}}\Bigr) \sqrt{\braket{G_{j_1}G_{j_1}^\ast}}\sqrt{\braket{G_{j_{t}}G_{j_{t}}^\ast}} \prec \frac{1}{\sqrt{\eta_{j_{1}}\eta_{j_{t}}}}\Bigl(\prod_{n\ne 1,t} \frac{1}{\eta_{j_n}}\Bigr),
        \end{split}\]
    where in the second step we used cyclicity of the trace, the norm-estimate in the third step und the Ward-estimate in the last step.
\end{proof}
\begin{proof}[Proof of Lemma~\ref{gain lemma}]
    For the proof of~\eqref{long Gk gain} we recall from the definition of the Ward-estimate in~\eqref{chain cycle ests} that for a cycle \(\Gamma_k^\circ\) we have
    \[
        \WEst(\Gamma_k^\circ)\ge \frac{\NEst(\Gamma_k^\circ)}{(n\eta_\ast)^{k/2}} = \frac{n^{k/2}}{\abs{\eta_1\eta_2}^{\abs{E_G^2(\Gamma_k^\circ)}/2}} \frac{1}{\eta_\ast^{k/2}}
    \]
    since \(\abs{E_\mathrm{Ward}(\Gamma_k^\circ)}\le \abs{E_G(\Gamma_k^\circ)}\le k\). Thus, together with~\eqref{eq general av bound} and interpreting \(\Tr \Gamma_k^\circ\) as a trace of a product of \(k+\abs{E_G^2(\Gamma_k^\circ)}\) factors of \(G\)'s we conclude
    \begin{equation}\label{Tr Gamma circ}\abs{\Tr\Gamma_k^\circ} \prec \frac{n}{\abs{\eta_1\eta_2}^{\abs{E_G^2(\Gamma_k^\circ)}}\eta_\ast^{k-\abs{E_G^2(\Gamma_k^\circ)}-1}}\le \frac{n}{\abs{\eta_1\eta_2}^{\abs{E_G^2(\Gamma_k^\circ)}/2}\eta_\ast^{k-1}}\le \frac{\WEst(\Gamma_k^\circ)}{(n\eta_\ast)^{k/2-1}}.\end{equation}
    Note that Lemma~\ref{lemma general products} is applicable here even though therein (for convenience) it was assumed that all spectral parameters \(w_i\) have positive imaginary parts. However, the lemma also applies to spectral parameters with negative imaginary parts since it allows for adjoints and \(G^z(\ov w)=(G^z(w))^\ast\). The first inequality in~\eqref{Tr Gamma circ} elementarily follows from~\eqref{eq general av bound} by distinguishing the cases \(\abs{E_G^2}=k,k-1\) or \(\le k-2\), and always choosing \(s\) and \(t\) such that the \(\sqrt{\eta_{j_s} \eta_{j_t}}\) factor contains the highest possible \(\eta_\ast\) power. Similarly to~\eqref{Tr Gamma circ}, for~\eqref{long Gk iso gain} we have, using~\eqref{eq general iso bound},
    \begin{equation}\label{Tr Gamma chain} \abs{(\Gamma_k^-)_{ab}} \prec \frac{n^{k-1}}{\abs{\eta_1\eta_2}^{\abs{E_G^2(\Gamma_k^-)}/2}} \frac{1}{(n\eta_\ast)^{k/2}}\le  \frac{\WEst(\Gamma_k^-)}{(n\eta_\ast)^{k/2-1}}.
    \end{equation}

    For the proof of~\eqref{long Gk gain theta} we use a Cauchy-Schwarz estimate to isolate a single \(G_{12}\) factor from the remaining \(G\)'s in \(\Gamma_l^\circ\). We may represent the ``square'' of all the remaining factors by an appropriate cycle graph \(\Gamma_{2(k-1)}^\circ\) of length \(2(k-1)\) with \(\abs{E_G^2(\Gamma_{2(k-1)}^\circ)}=2(\abs{E_G^2(\Gamma_k^\circ)}-1)\). We obtain
    \[
        \begin{split}
            \abs{\Tr \Gamma_k^\circ} &\le \sqrt{\Tr(G_{12}G_{12}^\ast)}\sqrt{\abs{\Tr \Gamma_{2(k-1)}^\circ}} = \sqrt{\Tr G_1^\ast G_1 B G_2 G_2^\ast B^\ast }\sqrt{\abs{\Tr \Gamma_{2(k-1)}^\circ}} \\
            &= \frac{ \sqrt{\Tr (\Im G_1)B(\Im G_2)B^\ast} \sqrt{\abs{\Tr \Gamma^\circ_{2(k-1)}}}}{\sqrt{\abs{\eta_1\eta_2}}} \\
            &\prec \frac{\sqrt{\theta n}}{\sqrt{\abs{\eta_1\eta_2}}} \frac{\sqrt{n}}{\abs{\eta_1\eta_2}^{\abs{E_2^G(\Gamma_k^\circ)}/2-1/2 } \eta_\ast^{k-3/2} } \\
            &\le \sqrt{\eta_\ast\theta} (n\eta_\ast)^{-(k-2)/2} \WEst(\Gamma_k^\circ)
        \end{split}\]
    where in the penultimate step we wrote out \(\Im G=(G-G^\ast)/(2\ii)\) in order to use~\eqref{eq a priori theta}, and used~\eqref{Tr Gamma circ} for \(\Gamma_{2(k-1)}^\circ\).
\end{proof}
Now it remains to count the gains from applying Lemma~\ref{gain lemma} for each cycle- and chain subgraph of \(\Gamma\). We claim that
\begin{subequations}
    \begin{equation}\label{eq ward-estimate gain}
        \WEst(\Gamma) \le \bigl(\eta_\ast^{1/6}\bigr)^{d_{\ge3}} \frac{1}{(n\eta_\ast)^{2p}\abs{\eta_1\eta_2}^{p}}, \qquad d_{\ge3}:= \sum_{d_G(i)\ge 3 } d_G(i).
    \end{equation}
    Furthermore, suppose that \(\Gamma\) has \(c\) degree-\(2\) cycles \(\Gamma_k^\circ\) which according to~\ref{degree equal} has to satisfy \(0\le c':= \abs{E_\kappa^2}-c\le \abs{E_\kappa^2}\). Then we claim that
    \begin{equation}\label{eq value gain}
        \abs{\Val(\Gamma)} \prec \Bigl(\frac{1}{n\eta_\ast}\Bigr)^{(c'-d_{\ge3}/2)_+} \bigl(\sqrt{\eta_\ast\theta}\bigr)^{(p-c'-d_{\ge3}/2)_+} \WEst(\Gamma).
    \end{equation}
\end{subequations}
Assuming~\eqref{eq ward-estimate gain}--\eqref{eq value gain} it follows immediately that
\[ \abs{\Val(\Gamma)} \prec \frac{1}{(n\eta_\ast)^{2p}\abs{\eta_1\eta_2}^{p}} \Bigl(\sqrt{\eta_\ast\theta} + \frac{1}{n\eta_\ast} + \eta_\ast^{1/6}\Bigr)^p,\]
implying~\eqref{prop av bound}. In order to complete the proof of the Proposition~\ref{prop prob bound} it remains to verify~\eqref{eq ward-estimate gain} and~\eqref{eq value gain}.
\begin{proof}[Proof of~\eqref{eq ward-estimate gain}]
    This follows immediately from the penultimate inequality in~\eqref{ward improvement} and
    \[\eta_\ast^{2p+\sum_i(d_G(i)/2-2)}\le \eta_\ast^{\sum_i (d_G(i)/2-1)} = \eta_\ast^{\frac{1}{2}\sum_{d_G(i)\ge 3} (d_G(i)-2) }\le \eta_\ast^{\frac{1}{6}\sum_{d_G(i)\ge 3} d_G(i)},\]
    where we used~\ref{number of kappa edges} in the first inequality.
\end{proof}
\begin{proof}[Proof of~\eqref{eq value gain}]
    For cycles \(\Gamma_k^\circ\) or chain-edges \(\Gamma_k^-\) in the reduced graph we say that \(\Gamma_k^{\circ/-}\) has \((k-2)_+\) \emph{excess \(G\)-edges}. Note that for cycles \(\Gamma_k^\circ\) every additional \(G\) beyond the minimal number \(k\ge 2\) is counted as an excess \(G\)-edge, while for chain-edges \(\Gamma_k^-\) the first additional \(G\) beyond the minimal number \(k\ge 1\) is not counted as an excess \(G\)-edge. We claim that:
    \begin{enumerate}[label=(C\arabic*)]
        \item\label{count excess} The total number of excess \(G\)-edges is at least \(2c'-d_{\ge 3}\).
        \item\label{count 12} There are at least \(p-c'-d_{\ge3}/2\) cycles in \(\Gamma\) containing \(G_{12}\).
    \end{enumerate}
    Since the vertices of the reduced graph are \(u_i,v_i\) for \(d_G(i)\ge 3\), it follows that the reduced graph has \(\sum_{d_G(i)\ge 3} (d_G(u_i)+d_G(v_i))/2=d_{\ge 3}\) edges while the total number of \(G\)'s beyond the minimally required \(G\)'s (i.e.\ two for cycles and one for edges) is \(2c'\). Thus in the worst case there are at least \(2c'-d_{\ge3}\) excess \(G\)-edges, confirming~\ref{count excess}.

    The total number of \(G_{12}\)'s is \(2p\), while the total number of \(G_i\)'s is \(2 \abs{E_\kappa^2}+d_{\ge3}-2p\), according to~\ref{number of G edges}. For fixed \(c\) the number of cycles with \(G_{12}\)'s is minimised in the case when all \(G_i\)'s are in cycles of length \(2\) which results in \(\abs{E_\kappa^2}-p+\lfloor d_{\ge3}/2\rfloor\) cycles without \(G_{12}\)'s. Thus, there are at least
    \[ c - \Bigl(\abs{E_\kappa^2}-p+\lfloor d_{\ge3}/2\rfloor\Bigr)= p - c' - \lfloor d_{\ge 3}/2\rfloor \ge p - c' - d_{\ge 3}/2 \]
    cycles with some \(G_{12}\), confirming also~\ref{count 12}.

    The claim~\eqref{eq value gain} follows from~\ref{count excess}--\ref{count 12} in combination with Lemma~\ref{gain lemma}.
\end{proof}

\section{Central limit theorem for resolvents}\label{sec:CLTres}
The goal of this section is to prove the CLT for resolvents, as stated in~Proposition~\ref{prop:CLTresm}. We begin by analysing the \(2\)-body stability operator \(\wh\cB\) from~\eqref{eq:stabop12}, as well as its special case, the \(1\)-body stability operator
\begin{equation}\label{cB def}
    \cB:=\wh\cB(z,z,w,w)=1-M\SS[\cdot]M.
\end{equation}
Note that other than in the previous Section~\ref{sec local law G2}, all spectral parameters \(\eta,\eta_1,\dots,\eta_p\) considered in the present section are positive, or even, \(\eta,\eta_i\ge 1/n\).

\begin{lemma}\label{lemma:betaM}
    For \(w_1=\ii\eta_1,w_2=\ii\eta_2\in\ii\R\setminus\{0\}\) and \(z_1,z_2\in\C\) we have
    \begin{equation}\label{beta ast bound}
        \norm{\wh\cB^{-1}}^{-1} \gtrsim (\abs{\eta_1}+\abs{\eta_2})\min\set{(\Im m_1)^2, (\Im m_2)^2 } + \abs{z_1-z_2}^2.
    \end{equation}
    Moreover, for \(z_1=z_2=z\) and \(w_1=w_2=\ii\eta\) the operator \(\cB=\wh\cB\) has two non-trivial eigenvalues \(\beta,\beta_\ast\) with \(\beta,\beta_\ast\) as in~\eqref{beta ast def},~\eqref{beta def}, and the remaining eigenvalues being \(1\).
\end{lemma}
\begin{proof}
    Throughout the proof we assume that \(\eta_1,\eta_2>0\), all the other cases are completely analogous. With the shorthand notations \(m_i:= m^{z_i}(w_i),
    u_i:= u^{z_i}(w_i)\) and the partial trace \(\Tr_2\colon\C^{2n\times 2n}\to \C^4\) rearranged into a \(4\)-dimensional vector, the stability operator \(\wh\cB\), written as a \(4\times 4\) matrix is given by
    \begin{equation}\label{eq Bhat decomp}
        \wh\cB = 1- \Tr_2^{-1}\circ\begin{pmatrix}
            T_1 & 0 \\ T_2 & 0
        \end{pmatrix}\circ \Tr_2, \quad \Tr_2 \begin{pmatrix}
            R_{11} & R_{12} \\R_{21}&R_{22}
        \end{pmatrix}:= \begin{pmatrix}
            \braket{R_{11}} \\\braket{R_{22}}\\\braket{R_{12}}\\\braket{R_{21}}
        \end{pmatrix}.
    \end{equation}
    Here we defined
    \[
        T_1 := \begin{pmatrix}
            z_1 \ov{z_2} u_1 u_2 & m_1 m_2 \\ m_1 m_2 & \ov{z_1} z_2 u_1 u_2
        \end{pmatrix}, \quad T_2 := \begin{pmatrix}
            -z_1 u_1 m_2 & -z_2 u_2 m_1 \\ -\ov{z_2} u_2 m_1 & -\ov{z_1} u_1 m_2
        \end{pmatrix},
    \]
    and \(\Tr_2^{-1}\) is understood to map \(\C^4\) into \(\C^{2n\times 2n}\) in such a way that each \(n\times n\) block is a constant multiple of the identity matrix.
    From~\eqref{eq Bhat decomp} it follows that \(\wh\cB\) has eigenvalue \(1\) in the \(4(n^2-1)\)-dimensional kernel of \(\Tr_2\), and that the remaining four eigenvalues are \(1,1\) and the eigenvalues \(\wh\beta,\wh\beta_\ast\) of \(B_1:=1-T_1\), i.e.\
    \begin{equation}\label{B eigs}
        \wh\beta,\wh\beta_\ast:= 1 - u_1 u_2 \Re z_1 \ov{z_2} \pm \sqrt{ m_1^2 m_2^2 - u_1^2 u_2^2 (\Im z_1 \ov{z_2})^2  }.
    \end{equation}
    Thus the claim about the \(w_1=w_2\), \(z_1=z_2\) special case follows.
    The bound~\eqref{beta ast bound} follows directly from
    \begin{equation}
        \label{eq:lowbeta}
        \abs*{\widehat{\beta}\widehat{\beta}_*}\gtrsim (\eta_1+\eta_2)\min\{(\Im m_1)^2, (\Im m_2)^2 \}+\abs{z_1-z_2}^2,
    \end{equation}
    since \(\abs{\widehat{\beta}}, \abs{\widehat{\beta}_*}\lesssim 1\) and \(\norm{\wh\cB^{-1}}\lesssim \norm{B_1^{-1}} = (\min\set{\abs{\wh\beta},\abs{\wh\beta_\ast}})^{-1}\) due to \(B_1\) being normal.

    We now prove~\eqref{eq:lowbeta}. By~\eqref{B eigs}, using that \(u_i=-m_i^2+u_i^2 \abs{z_i}^2\) repeatedly, it follows that
    \begin{equation}
        \label{eq:bbst1}
        \begin{split}
            \widehat{\beta}\widehat{\beta}_* &=1-u_1u_2\Big[ 1-\abs{z_1-z_2}^2+(1-u_1)\abs{z_1}^2+(1-u_2)\abs{z_2}^2\Big] \\
            &=u_1u_2\abs{z_1-z_2}^2+(1-u_1)(1-u_2)-m_1^2u_2\left( \frac{1}{u_1}-1\right) \\
            &\quad-m_2^2u_1\left( \frac{1}{u_2}-1\right).
        \end{split}
    \end{equation}
    Then, using \(1-u_i=\eta_i/(\eta_i+\Im m_i)\gtrsim \eta_i/(\Im m_i)\), that \(m_i=\ii \Im m_i\), and assuming \(u_1,u_2\in [\delta,1]\), for some small fixed \(\delta>0\), we get that
    \begin{equation}
        \label{eq:bbst2}
        \begin{split}
            \abs*{\widehat{\beta}\widehat{\beta}_*}&\gtrsim \abs{z_1-z_2}^2+(\Im m_1)^2 (1-u_1)+(\Im m_2)^2 (1-u_2)\\
            &\gtrsim \abs{z_1-z_2}^2+\min\{(\Im m_1)^2, (\Im m_2)^2 \}(2-u_1-u_2) \\
            &\gtrsim \abs{z_1-z_2}^2+\min\{(\Im m_1)^2, (\Im m_2)^2 \} \left( \frac{\eta_1}{\Im m_1}+\frac{\eta_2}{\Im m_2}\right).
        \end{split}
    \end{equation}
    If instead at least one \(u_i\in [0,\delta]\) then, by the second equality in the display above, the bound~\eqref{eq:lowbeta} is trivial.
\end{proof}

We now turn to the computation of the expectation \(\E\braket{G^z(\ii\eta)}\) to higher precision beyond the approximation \(\braket{G}\approx\braket{M}\). Recall the definition of the \(1\)-body stability operator from~\eqref{cB def} with non-trivial eigenvalues \(\beta,\beta_\ast\) as in~\eqref{beta ast def},~\eqref{beta def}.
\begin{lemma}\label{lemma exp}
    For \(\kappa_4\ne0\) we have a correction of order \(n^{-1}\) to \(\E\braket{G}\) of the form
    \begin{subequations}
        \begin{equation}\label{G-M next order}
            \begin{split}
                \E\braket{G} = \braket{M} + \cE + \cO\Bigl(\frac{1}{\abs{\beta}}\Bigl(\frac{1}{n^{3/2} (1+\eta)}+\frac{1}{(n\eta)^2}\Bigr)\Bigr),
            \end{split}
        \end{equation}
        where
        \begin{equation}\label{beta bound}
            \frac{1}{\abs{\beta}} = \norm{(\cB^{\ast})^{-1}[1]}\lesssim  \frac{1}{\abs{1-\abs{z}^2}+\eta^{2/3}}
        \end{equation}
        and
        \begin{equation}\label{cE def}
            \cE := \frac{\kappa_4}{n} m^3\Bigl(\frac{1}{1-m^2-\abs{z}^2}-1\Bigr)=-\frac{\ii\kappa_4}{4n}\partial_\eta(m^4).
        \end{equation}
    \end{subequations}
\end{lemma}
\begin{proof}
    Using~\eqref{G deviation} we find
    \begin{gather}
        \begin{aligned}
            \braket{G-M} & = \braket{1,\cB^{-1}\cB[G-M]} = \braket{(\cB^\ast)^{-1}[1],\cB[G-M] }                                                \\
                         & = -\braket{M^\ast (\cB^\ast)^{-1}[1], \un{WG} } + \braket{M^\ast (\cB^\ast)^{-1}[1], \SS[G-M](G-M) }                 \\
                         & = -\braket{M^\ast (\cB^\ast)^{-1}[1],\un{WG}} + \cO_\prec \Bigl(\frac{\norm{(\cB^\ast)^{-1}[1]}}{(n\eta)^{2}}\Bigr).
        \end{aligned}\label{eq G-M exp}\raisetag{-3em}
    \end{gather}
    With
    \[A:=\big( (\cB^\ast)^{-1}[1]\big)^\ast M\]
    we find from the explicit formula for \(\cB\) given in~\eqref{eq Bhat decomp} and~\eqref{beta def} that
    \begin{equation}\label{MA eq}
        \braket{MA}= \frac{1-\beta}{\beta}  =\frac{1}{1-m^2-\abs{z}^2 u^2}-1=-\ii\partial_\eta m,
    \end{equation}
    and, using a cumulant expansion we find
    \begin{equation}\label{eq single WGA exp} \E\braket{\un{WG}A} = \sum_{k\ge 2}\sum_{ab}\sum_{\bm\alpha\in\{ab,ba\}^k} \frac{\kappa(ba,\bm \alpha)}{k!} \E \partial_{\bm\alpha}\braket{\Delta^{ba} G A}.\end{equation}
    We first consider \(k=2\) where by parity at least one \(G\) factor is off-diagonal, e.g.
    \[ \frac{1}{n^{5/2}}\sum_{a\le n}\sum_{b>n} \E G_{ab}G_{aa}(GA)_{bb}\]
    and similarly for \(a>n\), \(b\le n\).
    By writing \(G=M+G-M\) and using the isotropic structure of the local law~\eqref{single local law} we obtain
    \[
        \begin{split}
            &\frac{1}{n^{5/2}}\sum_{a\le n}\sum_{b>n} \E G_{ab}G_{aa}(GA)_{bb} \\
            &= \frac{1}{n^{5/2}}\E m (MA)_{n+1,n+1}\braket{E_1\bm 1,G E_2\bm 1} + \cO_\prec\Bigl( n^2 n^{-5/2} (n\eta)^{-3/2} \abs{\beta}^{-1}\Bigr)\\
            &=\cO_\prec\Bigl( \frac{1}{\abs{\beta}n^{3/2} (1+\eta)} + \frac{1}{\abs{\beta}n^2\eta^{3/2}}\Bigr),
        \end{split} \]
    where \(\bm1=(1,\dots,1)\) denotes the constant vector of norm \(\norm{\bm1}=\sqrt{2n}\). Thus we can bound all \(k=2\) terms by \(\abs{\beta}^{-1}\bigl(n^{-3/2} (1+\eta)^{-1}+n^{-2}\eta^{-3/2}\bigr)\).

    For \(k\ge 4\) we can afford bounding each \(G\) entrywise and obtain bounds of \(\abs{\beta}^{-1}n^{-3/2}\). Finally, for the \(k=3\) term there is an assignment \((\bm\alpha)=(ab,ba,ab)\) for which all \(G\)'s are diagonal and which contributes a leading order term given by
    \begin{equation} -\frac{\kappa_4}{2n^3}\sump_{ab}M_{aa}M_{bb}M_{aa}(MA)_{bb}= -\frac{\kappa_4 }{n}\braket{M}^3\braket{MA}, \label{eq psum def}\end{equation}
    where \[
        \sump_{ab}:= \sum_{a\le n}\sum_{b>n}+\sum_{a>n}\sum_{b\le n},
    \]
    and thus
    \begin{gather}
        \begin{aligned}
            \sum_{k\ge 2}\sum_{ab}\sum_{\bm\alpha\in\{ab,ba\}^k} \frac{\kappa(ba,\bm \alpha) }{k!} \partial_{\bm\alpha}\braket{\Delta^{ba}GA} & = -\frac{\kappa_4 }{n}\braket{M}^3\braket{MA}                                                       \\
                                                                                                                                              & \quad+ \cO\Bigl(\frac{1}{\abs{\beta} n^{3/2} (1+\eta)}+\frac{1}{\abs{\beta} n^{2}\eta^{3/2}}\Bigr),
        \end{aligned}\label{g-m single ref}\raisetag{-5em}
    \end{gather}
    concluding the proof.
\end{proof}

We now turn to the computation of higher moments which to leading order due to Lemma~\ref{lemma exp} is equivalent to computing
\[ \E\prod_{i\in[p]}\braket{G_i-M_i-\cE_i},\quad \cE_i:= \frac{\kappa_4}{n}\braket{M_i}^3\braket{M_i A_i},\quad
    A_i:= \big( (\cB_i^\ast)^{-1}[1]\big)^\ast M_i,\]
with \(G_i,M_i\) as in~\eqref{G_i def} for \(z_1,\dots,z_k\in\C\), \(\eta_1,\dots,\eta_k>1/n\). Using Lemma~\ref{lemma exp}, Eq.~\eqref{eq G-M exp}, \(\abs{\cE_i}\lesssim 1/n\) and the high-probability bound
\begin{equation}\label{eq a priori WGA}
    \abs{\braket{\un{WG_i A_i}}}\prec \frac{1}{\abs{\beta_i}n \eta_i}
\end{equation} we have
\begin{equation}
    \prod_{i\in[p]}\braket{G_i-\E G_i} = \prod_{i\in[p]}\braket{-\un{WG_i}A_i-\cE_i}  + \cO_\prec\Bigl(\frac{\psi}{n\eta}\Bigr), \quad \psi:= \prod_{i\in[p]}\frac{1}{\abs{\beta_i}n\abs{\eta_i}}.\label{eq G-M G-M reduction}
\end{equation}
In order to prove Proposition~\ref{prop:CLTresm} we need to compute the leading order term in the local law bound
\begin{equation}\label{eq prod WGA naive}
    \abs*{\prod_{i\in[p]}\braket{-\un{WG_i}A_i-\cE_i} }\prec \psi.
\end{equation}

\begin{proof}[Proof of Proposition~\ref{prop:CLTresm}]
    To simplify notations we will not carry the \(\beta_i\)-dependence within the proof because each \(A_i\) is of size \(\norm{A_i}\lesssim\abs{\beta_i}^{-1}\) and the whole estimate is linear in each \(\abs{\beta_i}^{-1}\). We first perform a cumulant expansion in \(\un{WG_1}\) to compute
    \begin{gather}
        \begin{aligned}
             & \E \prod_{i\in[p]}\braket{-\un{W G_i}A_i-\cE_i}                                                                                                                                                           \\
             & \quad= -\braket{\cE_1}\E\prod_{i\ne 1}\braket{-\un{W G_i}A_i-\cE_i}                                                                                                                                       \\
             & \qquad+ \sum_{i\ne1}\E\wt\E \braket{-\wt W G_1 A_1}\braket{-\wt W G_i A_i+\un{WG_i \wt W G_i} A_i}\prod_{j\ne 1,i}\braket{-\un{W G_j}A_j-\cE_j}                                                           \\
             & \qquad +\sum_{k\ge 2}\sum_{ab}\sum_{\bm\alpha\in\{ab,ba\}^k} \frac{\kappa(ba,\bm \alpha)}{k!} \E \partial_{\bm\alpha}\Bigl[\braket{-\Delta^{ba}G_1A_1}\prod_{i\ne 1}\braket{-\un{W G_i}A_i-\cE_i} \Bigr],
        \end{aligned}\label{eq g-m 2 first exp}\raisetag{-6em}
    \end{gather}
    where \(\wt W\) denotes an independent copy of \(W\) with expectation \(\wt\E\), and the underline is understood with respect to \(W\) and not \(\wt W\). We now consider the terms of~\eqref{eq g-m 2 first exp} one by one. For the second term on the rhs.\ we use the identity
    \begin{equation}\label{eq tr W tr W}
        \E\braket{W A}\braket{W B} = \frac{1}{2n^2} \braket{A E_1 B E_2 + A E_2 B E_1 }  = \frac{\braket{A E B E'}}{2n^2},
    \end{equation}
    where we recall the block matrix definition from~\eqref{E1 E2 def} and follow the convention that \(E,E'\) are summed over both choices \((E,E')=(E_1,E_2),(E_2,E_1)\). Thus we obtain
    \begin{equation}\label{eq Wick Gauss term}
        \begin{split}
            & \wt\E \braket{-\wt W G_1 A_1}\braket{-\wt W G_i A_i+\un{WG_i \wt W G_i} A_i} \\
            &\quad= \frac{1}{2n^2}\braket{G_1 A_1 E G_i A_i E'- G_1 A_1 E \un{G_i A_i W G_i E'}}\\
            &\quad = \frac{1}{2n^2}\braket{G_1 A_1 E G_i A_i E' + G_1\SS[G_1 A_1 E G_i A_i ] G_i E'  - \un{G_1 A_1 E G_i A_i W G_i E'}}.
        \end{split}
    \end{equation}
    Here the self-renormalisation in the last term is defined analogously to~\eqref{self renom}, i.e.
    \[ \un{f(W)Wg(W)}:= f(W)Wg(W)-\wt\E (\partial_{\wt W}f)(W)\wt W g(W) - \wt\E f(W)\wt W (\partial_{\wt W}g)(W),\]
    which is only well-defined if it is clear to which \(W\) the action is associated, i.e.\ \(\un{WWf(W)}\) would be ambiguous. However, we only use the self-renormalisation notation for \(f(W),g(W)\) being (products of) resolvents and deterministic matrices, so no ambiguities should arise. For the first two terms in~\eqref{eq Wick Gauss term} we use \(\norm{M_{AE_1}^{z_1,z_i}}\lesssim \norm{\wh\cB_{1i}^{-1}}\lesssim \abs{z_1-z_i}^{-2}\) due to~\eqref{beta ast bound} and the first bound in~\eqref{final local law} from Theorem~\ref{thm local law G2} (estimating the big bracket by \(1\)) to obtain
    \begin{equation}\label{eq GG local law application}
        \begin{split}
            &\braket{G_1 A_1 E G_i A_i E' + G_1\SS[G_1 A_1 E G_i A_i ] G_i E'} \\
            &\qquad\qquad\qquad = \braket{M^{z_1,z_i}_{A_1E} A_i E'+ M^{z_i,z_1}_{E'} \SS[M^{z_1,z_i}_{A_1E}A_i]} \\
            &\qquad\qquad\qquad\quad+ \cO_\prec\Bigl(\frac{1}{n\abs{z_1-z_i}^4\eta_\ast^{1i} \abs{\eta_1\eta_i}^{1/2}}+\frac{1}{n^2\abs{z_1-z_i}^4(\eta_\ast^{1i})^2 \abs{\eta_1\eta_i}}\Bigr),
        \end{split}
    \end{equation}
    where \(\eta_\ast^{1i}:=\min\{\eta_1,\eta_i\}\). For the last term in~\eqref{eq Wick Gauss term} we claim that
    \begin{equation}\label{var trace bound}
        \E\abs{\braket{\un{G_1 A_1 E G_i A_i W G_i E'}}}^2\lesssim \Bigl(\frac{1}{n\eta_1\eta_i\eta_\ast^{1i}}\Bigr)^2,
    \end{equation}
    the proof of which we present after concluding the proof of the proposition.
    Thus, using~\eqref{var trace bound} together with~\eqref{eq a priori WGA},
    \[
        \begin{split}
            &\abs*{n^{-2}\E \braket{\un{G_1 A_1 E G_i A_i W G_i E'}}\prod_{j\ne 1,i} \braket{-\un{W G_j}A_j-\cE_i}}\\
            &\qquad \lesssim \frac{n^\epsilon}{n^2} \biggl[\prod_{j\ne 1,i}\frac{1}{n\eta_j}\biggr]\Bigl(\E\abs{\braket{\un{G_1 A_1 E G_i A_i W G_i E'}}}^2\Bigr)^{1/2} \\
            &\qquad\lesssim \frac{n^\epsilon}{n\eta_\ast^{1i}} \prod_{j}\frac{1}{n\eta_j} \le \frac{n^\epsilon\psi}{n\eta_\ast}.
        \end{split}\]
    Together with~\eqref{eq prod WGA naive} and~\eqref{eq Wick Gauss term}--\eqref{eq GG local law application} we obtain
    \begin{equation}\label{eq Wick Gauss term 2}
        \begin{split}
            &\E\wt\E \braket{-\wt W G_1 A_1}\braket{-\wt W G_i A_i+\un{WG_i \wt W G_i} A_i}\prod_{j\ne 1,i}\braket{-\un{W G_j}A_j-\cE_j}\\
            &\qquad\quad= \frac{V_{1,i}}{2n^2} \E \prod_{j\ne 1,i}\braket{-\un{W G_j}A_j-\cE_j} \\
            &\qquad\quad\quad+ \landauO*{\psi n^\epsilon \Bigl(\frac{1}{n\eta_\ast}+\frac{\abs{\eta_1\eta_i}^{1/2}}{n\eta_\ast^{1i}\abs{z_1-z_i}^4}+\frac{1}{(n\eta_\ast^{1i})^2\abs{z_1-z_i}^4}\Bigr)}
        \end{split}
    \end{equation}
    since, by an explicit computation the rhs.\ of~\eqref{eq GG local law application} is given by \(V_{1,i}\) as defined in~\eqref{eq:exder}. Indeed, from the explicit formula for \(\cB\) it follows that main term on the rhs.\ of~\eqref{eq GG local law application} can be written as \(\wt V_{1,i}\), where
    \begin{equation}
        \label{eq:vw}
        \begin{split}
            \wt V_{i,j}:&= \frac{2m_i m_j\bigl[2u_i u_j \Re z_i\overline{z_j} + (u_i u_j \abs{z_i} \abs{z_j})^2\bigl[s_i s_j-4\bigr]\bigr]}{ t_i t_j\bigl[1+(u_i u_j \abs{z_i} \abs{z_j})^2-m_i^2m_j^2-2u_i u_j\Re z_i\overline{z_j}\bigr]^2 } \\
            & +\frac{2m_i m_j(m_i^2+u_i^2\abs{z_i}^2)(m_j^2+u_j^2\abs{z_j}^2) }{t_i t_j\bigl[1+(u_i u_j \abs{z_i} \abs{z_j})^2-m_i^2m_j^2-2u_i u_j\Re z_i\overline{z_j}\bigr]^2 },
        \end{split}
    \end{equation}
    using the notations \(t_i:=1-m_i^2-u_i^2\abs{z_1}^2\), \(s_i:=m_i^2-u_i^2\abs{z_i}^2\).
    By an explicit computation using the equation~\eqref{eq m} for \(m_i,m_j\) it can be checked that \(\wt V_{i,j}\) can be written as a derivative and is given by \(\wt V_{i,j}=V_{i,j}\) with \(V_{i,j}\) from~\eqref{eq:exder}.

    Next, we consider the third term on the rhs.\ of~\eqref{eq g-m 2 first exp} for \(k=2\) and \(k\ge3\) separately. We first claim the auxiliary bound
    \begin{equation}\label{eq aux var iso bound}
        \abs{\braket{\vx,\un{G B W G}\vy }} \prec \frac{\norm{\vx}\norm{\vy}\norm{B}}{n^{1/2}\eta^{3/2}}.
    \end{equation}
    Note that~\eqref{eq aux var iso bound} is very similar to~\eqref{prop iso bound} except that in~\eqref{eq aux var iso bound} both \(G\)'s have the same spectral parameters \(z,\eta\) and the order of \(W\) and \(G\) is interchanged. The proof of~\eqref{eq aux var iso bound} is, however, very similar and we leave details to the reader.

    After performing the \(\bm\alpha\)-derivative in~\eqref{eq g-m 2 first exp} via the Leibniz rule, we obtain a product of \(t\ge 1\) traces of the types \(\braket{(\Delta G_i)^{k_i}A_i}\) and \(\braket{\un{W(G_i \Delta)^{k_i}G_i}A_i}\) with \(k_i\ge0\), \(\sum k_i=k+1\), and \(p-t\) traces of the type \( \braket{\un{W G_i A_i}+\cE_i} \). For the term with multiple self-renormalised \(G\)'s, i.e.\ \(\braket{\un{W(G_i \Delta)^{k_i}G_i}A_i}\) with \(k_i\ge 1\) we rewrite
    \begin{gather}
        \begin{aligned}
            \braket{\un{W(G \Delta)^{k}G}A} & = \braket{\un{GAW(G \Delta)^{k}}}                                                                                   \\
                                            & = \braket{\un{GAWG}\Delta (G \Delta)^{k-1}}+\sum_{j=1}^{k-1}\braket{GA\SS[(G\Delta)^j G] (G \Delta)^{k-j}}          \\
                                            & = \braket{\un{GAWG}\Delta (G \Delta)^{k-1}}+\sum_{j=1}^{k-1}\braket{GA E (G \Delta)^{k-j}}\braket{GE'(G\Delta)^j }.
        \end{aligned}\label{long WG rewrite}\raisetag{-4em}
    \end{gather}

    \subsubsection*{Case \(k=2\), \(t=1\).} In this case the only possible term is given by \(\braket{\Delta G_1\Delta G_1 \Delta G_1 A_1}\) where by parity at least one \(G=G_1\) is off-diagonal and in the worst case (only one off-diagonal factor) we estimate
    \[
        \begin{split}
            n^{-1-3/2} \sum_{a\le n}\sum_{b>n} G_{aa} G_{bb} (GA)_{ab} &= \frac{m^2}{n^{5/2}}\braket{E_1\bm 1,GAE_2\bm1} + \cO_\prec\Bigl(\frac{1}{n^{1/2}}\frac{1}{(n\eta_1)^{3/2}}\Bigr)\\
            &=\cO_\prec\Bigl(\frac{1}{n^{3/2}}+\frac{1}{n^2\eta_1^{3/2}}\Bigr),
        \end{split}\]
    after replacing \(G_{aa}=m+(G-M)_{aa}\) and using the isotropic structure of the local law in~\eqref{single local law}, and similarly for \(\sum_{a>n}\sum_{b\le n}\).

    \subsubsection*{Case \(k=2\), \(t=2\).} In this case there are \(2+2\) possible terms
    \[
        \begin{split}
            &\braket{\Delta G_1 \Delta G_1 A_1} \braket{\Delta G_i A_i+\un{W G_i \Delta G_i}A_i}\\
            &\qquad+\braket{\Delta G_1 A_1} \braket{\Delta G_i\Delta G_i A_i+\un{W G_i \Delta G_i\Delta G_i }A_i}.
        \end{split}\]
    For the first two, in the worst case, we have the estimate
    \[
        \begin{split}
            &\frac{1}{n^{7/2}} \sump_{ab} (G_1)_{aa} (G_{1}A_1)_{bb} \Bigl((G_i A_i)_{ab}+(\un{G_i A_i W G_i})_{ab}\Bigr) \\
            &\qquad\qquad = \cO_\prec\Bigl(\frac{1}{n^{5/2}}+\frac{1}{n^3\eta_1\eta_i^{3/2}}\Bigr)
        \end{split}
    \]
    using~\eqref{eq aux var iso bound}, where we recall the definition of \(\sum'\) from~\eqref{eq psum def}. Similarly, using~\eqref{long WG rewrite} and~\eqref{eq aux var iso bound} for the ultimate two terms, we have the bound
    \[
        \begin{split}
            &\frac{1}{n^{7/2}} \E\sump_{ab}  (G_{1}A_1)_{ab} \Bigl((\un{G_i A_i WG_i})_{aa} (G_i)_{bb}  +\frac{(G_i A_i E G_i)_{ab} (G_i E' G_i)_{ab}}{n}\Bigr) \\
            &\qquad= \cO_\prec\Bigl(\frac{1}{n^{3}\eta_1^{1/2}\eta_i^2}\Bigr).
        \end{split}
    \]
    \subsubsection*{Case \(k=2\), \(t=3\).} In this final \(k=2\) case we have to consider four terms
    \[\braket{\Delta G_1A_1}\braket{\Delta G_i A_i +\un{W G_i \Delta G_i A_i} }\braket{\Delta G_j A_j +\un{W G_j \Delta G_j A_j} },\]
    which, using~\eqref{eq aux var iso bound}, we estimate by
    \[
        \begin{split}
            &\frac{1}{n^{9/2}}\sump_{ab} (G_1A_1)_{ab}\Bigl((G_i A_i)_{ab}+(\un{G_i A_i W G_i})_{ab}\Bigr)\Bigl((G_j A_j)_{ab}+(\un{G_j A_j W G_j})_{ab}\Bigr)\\
            &\quad = \cO_\prec\Bigl(\frac{1}{n^{4} \eta_1^{1/2}\eta_i^{3/2}\eta_j^{3/2}}  \Bigr).
        \end{split}\]
    By inserting the above estimates back into~\eqref{eq g-m 2 first exp}, after estimating all untouched traces by \(n^\epsilon/(n\eta_i)\) in high probability using~\eqref{eq a priori WGA}, we obtain
    \begin{equation}\label{k=2 est}
        \begin{split}
            &\sum_{k= 2}\sum_{ab}\sum_{\bm\alpha\in\{ab,ba\}^k} \frac{\kappa(ba,\bm \alpha)}{k!} \E \partial_{\bm\alpha}\Bigl[\braket{-\Delta^{ba}G_1A_1}\prod_{i\ne 1}\braket{-\un{W G_i}A_i-\cE_i} \Bigr]\\
            &\qquad = \cO\Bigl(\frac{\psi n^\epsilon}{\sqrt{n\eta_\ast}}\Bigr).
        \end{split}
    \end{equation}

    \subsubsection*{Case \(k\ge 3\).} In case \(k\ge 3\) after the action of the derivative in~\eqref{eq g-m 2 first exp} there are \(1\le t\le k+1\) traces involving some \(\Delta\). By writing the normalised traces involving \(\Delta\) as matrix entries we obtain a prefactor of \(n^{-t-(k+1)/2}\) and a \(\sum_{ab}\)-summation over entries of \(k+1\) matrices of the type \(G\), \(GA\), \(\un{GAWG}\) such that each summation index appears exactly \(k+1\) times. There are some additional terms from the last sum in~\eqref{long WG rewrite} which are smaller by a factor \((n\eta)^{-1}\) and which can be bounded exactly as in the \(k=2\) case. If there are only diagonal \(G\) or \(GA\)-terms, then we have a naive bound of \(n^{-t-(k-3)/2}\) and therefore potentially some leading-order contribution in case \(k=3\). If, however, \(k>3\), or there are some off-diagonal \(G,GA\) or some \(\un{GAWG}\) terms, then, using~\eqref{eq aux var iso bound} we obtain an improvement of at least \((n\eta)^{-1/2}\) over the naive bound~\eqref{eq prod WGA naive}. For \(k=3\), by parity, the only possibility of having four diagonal \(G,GA\) factors, is distributing the four \(\Delta\)'s either into a single trace or two traces with two \(\Delta\)'s each. Thus the relevant terms are
    \[ \braket{\Delta G_1 \Delta G_1 \Delta G_1 \Delta G_1 A_1},\quad \braket{\Delta G_1 \Delta G_1 A_1}\braket{\Delta G_i \Delta G_i A_i}. \]
    For the first one we recall from~\eqref{g-m single ref} for \(k=3\) that
    \begin{equation}\label{eq kappa4 1}
        \sum_{ab}\sum_{\bm\alpha} \kappa(ba,\bm\alpha) \braket{\Delta^{ba} G_1 \Delta^{\alpha_1} G_1 \Delta^{\alpha_2} G_1 \Delta^{\alpha_3} G_1 A_1}=\cE_1 + \cO_\prec\Bigl(\frac{1}{n^{3/2}}+\frac{1}{n^2 \eta_1^{3/2}}\Bigr).
    \end{equation}
    For the second one we note that only choosing \(\bm\alpha=(ab,ab,ba),(ab,ba,ab)\) gives four diagonal factors, while any other choice gives at least two off-diagonal factors. Thus
    \begin{equation} \label{eq kappa4 2}
        \begin{split}
            &\sum_{ab}\sum_{\bm\alpha} \kappa(ba,\bm\alpha) \braket{\Delta^{ba} G_1 \Delta^{\alpha_1} G_1}\braket{\Delta^{\alpha_2} G_i \Delta^{\alpha_3} G_i A_i} \\
            & = \frac{\kappa_4}{n^{2}}\sump_{ab} \braket{\Delta^{ba} G_1 \Delta^{ab} G_1 A_1}\bigl[\braket{\Delta^{ab} G_i \Delta^{ba} G_i A_i}+\braket{\Delta^{ba} G_i \Delta^{ab} G_i A_i}\bigl]+ \cO_\prec(\mathcal{E})\\
            & = \frac{\kappa_4}{4n^4}\sump_{ab} (G_1)_{aa} (G_1 A_1)_{bb} \bigl[  (G_i)_{bb} (G_i A_i)_{aa}+ (G_i)_{aa}(G_i A_i)_{bb}\bigl]+ \cO_\prec(\mathcal{E})\\
            & = \frac{\kappa_4}{4n^4}\sump_{ab} m_1 m_i (M_1 A_1)_{bb} \bigl[  (M_i A_i)_{aa}+ (M_i A_i)_{bb}\bigl]+ \cO_\prec\left(\sqrt{n\eta_*}\mathcal{E}\right)\\
            & = \frac{\kappa_4}{n^2} \braket{M_1} \braket{M_i} \braket{M_1 A_1} \braket{ M_i A_i} + \cO_\prec\left(\frac{1}{n^{5/2}\eta_*^{1/2}}\right),
        \end{split}
    \end{equation}
    where \(\mathcal{E}:=(n^3\eta_*)^{-1}\).
    We recall from~\eqref{MA eq} that
    \[ \braket{M_1} \braket{M_i} \braket{M_1 A_1} \braket{ M_i A_i}= \frac{1}{2}U_1 U_i \]
    with \(U_{i}\) defined in~\eqref{eq:exder}.
    Thus, we can conclude for the \(k\ge3\) terms in~\eqref{eq g-m 2 first exp} that
    \begin{gather}
        \begin{aligned}
             & \sum_{k \ge 3}\sum_{ab}\sum_{\bm\alpha\in\{ab,ba\}^k} \frac{\kappa(ba,\bm \alpha)}{k!} \E \partial_{\bm\alpha}\Bigl[\braket{-\Delta^{ba}G_1A_1}\prod_{i\ne 1}\braket{-\un{W G_i}A_i-\cE_i} \Bigr] \\
             & \qquad = \braket{\cE_1} \E\prod_{i\ne 1}\braket{-\un{W G_i}A_i-\cE_i} + \sum_{i\ne 1} \frac{\kappa_4 U_{1} U_i }{2n^2} \E\prod_{j\ne 1,i} \braket{-\un{W G_j}A_j-\cE_j}                           \\
             & \qquad\quad+ \cO\Bigl(\frac{\psi n^\epsilon}{(n\eta_\ast)^{1/2}}\Bigr).
        \end{aligned}\label{eq kappa ge 3 conclusion}\raisetag{-5em}
    \end{gather}

    By combining~\eqref{eq g-m 2 first exp} with~\eqref{eq Wick Gauss term 2},~\eqref{k=2 est} and~\eqref{eq kappa ge 3 conclusion} we obtain
    \begin{gather}
        \begin{aligned}
            \E\prod_{i}\braket{-\un{W G_i}A_i-\cE_i} & = \sum_{i\ne 1} \frac{V_{1,i}+\kappa_4 U_{1}U_i}{2n^2} \E\prod_{j\ne 1,i} \braket{-\un{W G_j}A_j-\cE_j}                                                                        \\
                                                     & \quad + \landauO*{\frac{\psi n^\epsilon }{\sqrt{n\eta_\ast}}+\frac{\psi n^\epsilon }{n\eta_\ast^{1/2}\abs{z_1-z_i}^4}+\frac{\psi n^\epsilon }{(n\eta_\ast)^2\abs{z_1-z_i}^4}},
        \end{aligned}\raisetag{-5em}
    \end{gather}
    and thus by induction
    \begin{gather}
        \begin{aligned}
            \E\prod_{i}\braket{-\un{W G_i}A_i-\cE_i} & = \frac{1}{n^{p}}\sum_{P\in\Pi_p}\prod_{\{i,j\}\in P} \frac{V_{i,j}+\kappa_4 U_i U_j}{2}                                                                                       \\
                                                     & \quad + \landauO*{\frac{\psi n^\epsilon }{\sqrt{n\eta_\ast}}+\frac{\psi n^\epsilon }{n\eta_\ast^{1/2}\abs{z_1-z_i}^4}+\frac{\psi n^\epsilon }{(n\eta_\ast)^2\abs{z_1-z_i}^4}},
        \end{aligned}\raisetag{-5em}
    \end{gather}
    from which the equality \(\E \prod_i \braket{G_i-\E G_i}\) and the second line of~\eqref{eq CLT resovlent} follows, modulo the proof of~\eqref{var trace bound}. The remaining equality then follows from applying the very same equality for each element of the pairing. Finally,~\eqref{prop clt exp} follows directly from Lemma~\ref{lemma exp}.
\end{proof}

\begin{proof}[Proof of~\eqref{var trace bound}]
    Using the notation of Lemma~\ref{lemma general products}, our goal is to prove that
    \begin{equation}   \label{eq G1Gii claim}
        \E\abs{\braket{\un{W G_{i1i} }}}^2 \lesssim \Bigl(\frac{1}{n\eta_1\eta_i\eta_\ast^{1i}}\Bigr)^2.
    \end{equation}
    Since only \(\eta_1,\eta_i\) play a role within the proof of~\eqref{var trace bound}, we drop the indices from \(\eta_\ast^{1i}\) and simply write \(\eta_\ast=\eta_\ast^{1i}\). Using a cumulant expansion we compute
    \begin{equation}\label{G1Gii first}
        \begin{split}
            &\E\abs{\braket{\un{W G_{i1i}}}}^2 \\
            &= \E\wt\E \braket{\wt W G_{i1i} } \Bigl(\braket{\wt W G_{i1i} } + \braket{\un{W G_i \wt W G_{i1i}  }+\un{W G_{i1}  \wt W G_{1i}  }+\un{W G_{i1i} \wt W G_i  }}\Bigr)\\
            &\quad + \sum_{k\ge 2}\cO\Bigl(\frac{1}{n^{(k+1)/2}}\Bigr)\sump_{ab}   \sum_{k_1+k_2=k-1}\sum_{\bm\alpha_1,\bm\alpha_2} \E\braket{\Delta^{ab}\partial_{\bm\alpha_1} G_{i1i}}\braket{\Delta^{ab}\partial_{\bm\alpha_2} G_{i1i}} \\
            &\quad + \sum_{k\ge 2}\cO\Bigl(\frac{1}{n^{(k+1)/2}}\Bigr)\sump_{ab}   \sum_{k_1+k_2=k}\sum_{\bm\alpha_1,\bm\alpha_2} \E\braket{\Delta^{ab}\partial_{\bm\alpha_1} G_{i1i}}\braket{\un{W\partial_{\bm\alpha_2} G_{i1i}}},
        \end{split}
    \end{equation}
    where \(\bm\alpha_i\) is understood to be summed over \(\bm\alpha_i\in\{ab,ba\}^{k_i}\). In~\eqref{G1Gii first} we only kept the scaling \(\abs{\kappa(ab,\bm\alpha)}\lesssim n^{-(k+1)/2}\) of the cumulants, and also absorb combinatorial factors as \(k! \) in \(\cO(\cdot)\). We first consider those terms in~\eqref{G1Gii first} which contain no self-renormalisations \(\un{Wf(W)}\) anymore since those do not have to be expanded further. For the very first term we obtain
    \begin{equation}\label{Gi1i single Gauss}
        \wt\E \braket{\wt W G_{i1i} } \braket{\wt W G_{i1i} } = \frac{\braket{G_{i1ii1i}}}{n^2} =\cO_\prec\Bigl( \frac{1}{n^2\eta_1^2\eta_i^3}\Bigr).
    \end{equation}
    To bound products of \(G_1\) and \(G_i\) we use Lemma~\ref{lemma general products}. For the second line on the rhs.\ of~\eqref{G1Gii first} we have to estimate
    \[
        \begin{split}
            \cO\Bigl(\frac{1}{n^{(k+1)/2+2}}\Bigr)\sum_{k\ge 2}\sump_{ab}   \sum_{k_1+k_2=k-1}\sum_{\bm\alpha_1,\bm\alpha_2} \E (\partial_{\bm\alpha_1} (G_{i1i})_{ba}) (\partial_{\bm\alpha_2} (G_{i1i})_{ba})
        \end{split}
    \]
    and we note that without derivatives we have the estimate \(\abs{(G_{i1i})}\prec (\eta_1\eta_i)^{-1}\). Additional derivatives do not affect this bound since if e.g.\ \(G_i\) is derived we obtain one additional \(G_i\) but also one additional product of \(G\)'s with \(G_i\) in the end, and one additional product with \(G_i\) in the beginning. Due to the structure of the estimate~\eqref{eq general iso bound} the bound thus remains invariant. For example \(\abs{(\partial_{ab} G_{i1i})_{ba}}=\abs{(G_{i})_{bb}(G_{i1i})_{aa}+\dots}\prec (\eta_1\eta_i)^{-1}\). Thus, by estimating the sum trivially we obtain
    \begin{equation}\label{Gi1i single cum}
        \frac{1}{n^{(k+1)/2}} \sum_{\substack{k_1+k_2=k-1\\ k\ge 2}}\sump_{ab}  \sum_{\bm\alpha_1,\bm\alpha_2} \E\braket{\Delta^{ab}\partial_{\bm\alpha_1} G_{i1i}}\braket{\Delta^{ab}\partial_{\bm\alpha_2} G_{i1i}}  = \cO_\prec\Bigl(\frac{1}{n^{3/2}\eta_1^2\eta_i^{2}}\Bigr)
    \end{equation}
    since \(k\ge 2\).

    It remains to consider the third line on the rhs.\ of~\eqref{G1Gii first} and the remaining terms from the first line. In both cases we perform a second cumulant expansion and again differentiate the Gaussian (i.e.\ the second order cumulant) term, and the terms from higher order cumulants. Since the two consecutive cumulant expansions commute it is clearly sufficient to consider the Gaussian term for the first line, and the full expansion for the third line. We begin with the latter and compute
    \begin{equation}\label{eq Gi1i cum second exp}
        \begin{split}
            &\E\braket{\Delta^{ab}\partial_{\bm\alpha_1} G_{i1i}}\braket{\un{W\partial_{\bm\alpha_2} G_{i1i}}} \\
            &\quad = \wt\E\E \braket{\Delta^{ab}\partial_{\bm\alpha_1}(G_i\wt W G_{i1i}+G_{i1}\wt W G_{1i}+G_{i1i}\wt W G_{i})} \braket{\wt W\partial_{\bm\alpha_2} G_{i1i}}  \\
            &\qquad + \sum_{l\ge 2}\sump_{cd}\sum_{\bm\beta_1,\bm\beta_2} \E \braket{\Delta^{ab}\partial_{\bm\alpha_1}\partial_{\bm\beta_1} G_{i1i}}\braket{\Delta^{cd}\partial_{\bm\alpha_2}\partial_{\bm\beta_2} G_{i1i}}\\
            &\quad = \frac{1}{n^2}\E \braket{\partial_{\bm\alpha_1}(G_{i1i}\Delta^{ab}G_i+G_{1i}\Delta^{ab}G_{i1}+G_{i}\Delta^{ab}G_{i1i})\partial_{\bm\alpha_2}(G_{i1i})}\\
            &\qquad + \sum_{l\ge 2}\sump_{cd}\sum_{\bm\beta_1,\bm\beta_2} \E \braket{\Delta^{ab}\partial_{\bm\alpha_1}\partial_{\bm\beta_1} G_{i1i}}\braket{\Delta^{cd}\partial_{\bm\alpha_2}\partial_{\bm\beta_2} G_{i1i}},
        \end{split}
    \end{equation}
    where \(\bm\beta_i\) are understood to be summed over \(\bm\beta_i\in\{cd,dc\}^{l_i}\) with \(l_1+l_2=l\). After inserting the first line of~\eqref{eq Gi1i cum second exp} back into~\eqref{G1Gii first} we obtain an overall factor of \(n^{-3-(k+1)/2}\) as well as the \(\sum_{ab}\)-summation over some \(\partial_{\bm\alpha} (\cG)_{ab}\), where \(\cG\) is a product of either \(2+5\) or \(3+4\) \(G_1\)'s and \(G_i\)'s respectively with \(G_i\) in beginning and end. We can bound \(\abs{\partial_{\bm\alpha} (\cG)_{ab}}\prec \eta_1^{-2}\eta_i^{-4}+\eta_1^{-3}\eta_i^{-3}\le\eta_1^{-2}\eta_i^{-2}\eta_\ast^{-2}\) and thus can estimate the sum by \(n^{-5/2}\eta_1^{-2}\eta_i^{-2}\eta_\ast^{-2}\) since \(k\ge 2\). Here we used~\eqref{eq general iso bound} to estimate all matrix elements  of the form \(\cG'_{ab}, \cG'_{aa},\dots\) emerging after performing the derivative \(\partial_{\bm \alpha} (\cG)_{ab}\).

    Now we turn to the second line of~\eqref{eq Gi1i cum second exp} when inserted back into~\eqref{G1Gii first}, where we obtain a total prefactor of \(n^{-(k+l)/2-3}\), a summation \(\sum_{abcd}\) over \((\partial_{\bm\alpha_1}\partial_{\bm\beta_1}G_{i1i})_{ab}(\partial_{\bm\alpha_2}\partial_{\bm\beta_2}G_{i1i})_{cd}\). In case \(k=l=2\), by parity, after performing the derivatives at least two factors are off-diagonal, while in case \(k+l=5\) at least one factor is off-diagonal. Thus we obtain a bound of \(n^{1-(k+l)/2}\eta_1^{-2}\eta_i^{-2}\) multiplied by a Ward-improvement of \((n\eta_\ast)^{-1}\) in the first, and \((n\eta_\ast)^{-1/2}\) in the second case. Thus we conclude
    \begin{equation}\label{eq Gi1i cum further exp}
        \frac{1}{n^{(k+1)/2}}  \sum_{\substack{k_1+k_2=k \\ k\ge 2}}\sump_{ab} \sum_{\bm\alpha_1,\bm\alpha_2} \E\braket{\Delta^{ab}\partial_{\bm\alpha_1} G_{i1i}}\braket{\un{W\partial_{\bm\alpha_2} G_{i1i}}} = \cO\Bigl(\frac{1}{n^2\eta_1^2\eta_i^2\eta_\ast^2}\Bigr).
    \end{equation}

    Finally, we consider the Gaussian part of the cumulant expansion of the remaining terms in the first line of~\eqref{G1Gii first}, for which we obtain
    \begin{equation}\label{eq gauss gauss}
        \begin{split}
            \frac{1}{n^2}\wt \E\braket{ (G_{i1i} \wt W G_i+G_{1i}\wt WG_{i1}+G_{i}\wt W G_{i1i})^2 } = O_\prec\Bigl(\frac{1}{n^2 \eta_1^2\eta_i^2\eta_\ast^2}\Bigr)
        \end{split}
    \end{equation}
    since
    \[
        \begin{split}
            \abs{\braket{G_i G_i}}&\prec \frac{1}{\eta_i},\quad \abs{\braket{G_i G_{i1}}}\prec \frac{1}{\eta_1\eta_i},\quad \abs{\braket{G_i G_{i1i}}}\prec\frac{1}{\eta_1\eta_i^2}, \\
            \abs{\braket{G_{1i}G_{1i}}}&\prec \frac{1}{\eta_1^2\eta_i}, \quad  \abs{\braket{G_{1i}G_{i1i}}}\prec \frac{1}{\eta_1^2\eta_i^2}, \quad \abs{\braket{G_{i1i}G_{i1i}}}\prec \frac{1}{\eta_1^2\eta_i^3}
        \end{split} \]
    due to~\eqref{eq general av bound}. By combining~\eqref{Gi1i single Gauss}--\eqref{eq gauss gauss} we conclude the proof of~\eqref{var trace bound} using~\eqref{G1Gii first}.
\end{proof}

\section{Independence of the small eigenvalues of \texorpdfstring{\(H^{z_1}\)}{Hz1} and \texorpdfstring{\(H^{z_2}\)}{Hz2}}\label{sec:IND}
Given an \(n\times n\) i.i.d.\ complex matrix \(X\), for any \(z\in\C \) we recall that the Hermitisation of \(X-z\) is given by
\begin{equation}\label{eq:her}
    H^z:= \left( \begin{matrix}
            0                & X-z \\
            X^*-\overline{z} & 0
        \end{matrix}\right).
\end{equation}
The block structure of \(H^z\) induces a symmetric spectrum with respect to zero, i.e.\ denoting by \(\{\lambda_{\pm i}^z\}_{i=1}^n\) the eigenvalues of \(H^z\), we have that \(\lambda_{-i}^z=-\lambda_i^z\) for any \(i\in[n]\). Denote the resolvent of \(H^z\) by \(G^z\), i.e.\ on the imaginary axis \(G^z\) is defined by \(G^z(\ii \eta):= (H^z-\ii\eta)^{-1}\), with \(\eta>0\).

\begin{convention}\label{rem:no0}
    We omitted the index \(i=0\) in the definition of the eigenvalues of \(H^z\). In the remainder of this section we always assume that all the indices are not zero, e.g we use the notation
    \[
        \sum_{j=-n}^n := \sum_{j=-n}^{-1}+ \sum_{j=1}^n.
    \]
    Similarly, by \(\abs{i}\le A\), for some \(A>0\), we mean \(0<\abs{i}\le A\), etc.
\end{convention}

The main result of this section is the proof of Proposition~\ref{prop:indmr} which follows by Proposition~\ref{prop:indeig} and rigidity estimates in Section~\ref{sec:ririri}.

\begin{proposition}\label{prop:indeig}
    Fix \(p\in\mathbf{N}\). For any \(\omega_d,\omega_f, \omega_h>0\) sufficiently small constants such that \(\omega_h\ll \omega_f\), there exits \(\omega, \widehat{\omega}, \delta_0,\delta_1>0\) with \(\omega_h\ll \delta_m\ll \widehat{\omega}\ll \omega\ll \omega_f\), for \(m=0,1\), such that for any fixed \(z_1,\dots,z_p\in \C \) such that \(\abs{z_l}\le 1-n^{-\omega_h}\), \(\abs{z_l-z_m}\ge n^{-\omega_d}\), with \(l,m\in [p]\), \(l\ne m\), it holds
    \begin{equation}\label{eq:indA}
        \begin{split}
            \E  &\prod_{l=1}^p \frac{1}{n}\sum_{\abs{i_l}\le n^{\widehat{\omega}}} \frac{\eta_l}{(\lambda_{i_l}^{z_l})^2+\eta_l^2}=\prod_{l=1}^p \E   \frac{1}{n}\sum_{\abs{i_l}\le n^{\widehat{\omega}}} \frac{\eta_l}{(\lambda_{i_l}^{z_l})^2+\eta_l^2} \\
            &\quad+\mathcal{O}\left(\frac{n^{\widehat{\omega}}}{n^{1+\omega}} \sum_{l=1}^p\frac{1}{\eta_l}\times\prod_{m=1}^p\left( 1+\frac{n^\xi}{n\eta_m}\right)+\frac{n^{p\xi+2\delta_0} n^{\omega_f}}{n^{3/2}}\sum_{l=1}^p \frac{1}{\eta_l}+\frac{n^{p\delta_0+\delta_1}}{n^{\widehat{\omega}}}\right),
        \end{split}
    \end{equation}
    for any \(\xi>0\), where \(\eta_1,\dots, \eta_p\in [n^{-1-\delta_0},n^{-1+\delta_1}]\) and the implicit constant in \(\mathcal{O}(\cdot)\) may depend on \(p\).
\end{proposition}
We recall that the eigenvalues of \(H^z\) are labelled by \(\lambda_{-n}\le \dots\le \lambda_{-1}\le\lambda_1\le\dots \le \lambda_n\), hence the summation over \(\abs{i_l}\le n^{\widehat{\omega}}\) in~\eqref{eq:indA} is over the smallest (in absolute value) eigenvalues of \(H^z\).

The remainder of Section~\ref{sec:IND} is divided as follows: in Section~\ref{sec:ririri} we state rigidity of the eigenvalues of the matrices \(H^{z_l}\) and a local law for \(\Tr G^{z_l}\), then using these results and Proposition~\ref{prop:indeig} we conclude the proof of Proposition~\ref{prop:indmr}. In Section~\ref{sec:PO} we state the main technical results needed to prove Proposition~\ref{prop:indeig} and conclude its proof. In Section~\ref{sec:BOUNDE} we estimate the overlaps of eigenvectors, corresponding to small indices, of \(H^{z_l}\), \(H^{z_m}\) for \(l\ne m\), this is the main input to prove the asymptotic independence in Proposition~\ref{prop:indeig}. In Section~\ref{sec:fixun} we present Proposition~\ref{pro:ciala} which is a modification of the pathwise coupling of DBMs from~\cite{MR3914908,MR3541852} (adapted to the \(2\times 2\) matrix model~\eqref{eq:her} in~\cite{MR3916329}) which is needed to deal with the (small) correlation of \({\bm \lambda}^{z_l}\), the eigenvalues of \(H^{z_l}\), for different \(l\)'s. In Section~\ref{sec:noncelpiu} we prove some technical lemmata used in Section~\ref{sec:PO}. Finally, in Section~\ref{sec:proofFEU} we prove Proposition~\ref{pro:ciala}.

\subsection{Rigidity of eigenvalues and proof of Proposition~\ref{prop:indmr}}\label{sec:ririri}

In this section, before proceeding with the actual proof of Proposition~\ref{prop:indeig}, we state the local law away from the imaginary axis,
proven in~\cite{MR4235475},  that will be used in the following sections. We remark that the averaged and entry-wise
version of this local law for \(\abs{z}\le 1-\epsilon\), for some small fixed \(\epsilon>0\), has already been established in~\cite[Theorem 3.4]{MR3230002}.

\begin{proposition}[Theorem 3.1 of~\cite{MR4235475}]\label{theo:trll}
    Let \(\omega_h>0\) be sufficiently small, and define \(\delta_l:= 1-\abs{z_l}^2\). Then with very high probability it holds
    \begin{equation}
        \label{eq:lll}
        \abs*{\frac{1}{2n}\sum_{1\le\abs{i}\le n} \frac{1}{\lambda_i^{z_l}-w}-m^{z_l}(w)}\le \frac{\delta_l^{-100}n^\xi}{n\Im w},
    \end{equation}
    uniformly in \(\abs{z_l}^2\le 1-n^{-\omega_h}\) and \(0< \Im w\le 10\). Here \(m^{z_l}\) denotes the solution of~\eqref{eq m}.
\end{proposition}

Note that \(\delta_l:= 1-\abs{z_l}^2\) introduced in Proposition~\ref{theo:trll} are not to be confused with the exponents \(\delta_0, \delta_1\) introduced in Proposition~\ref{prop:indeig}.

Let \(\{\lambda^z_{\pm i}\}_{i=1}^n\) denote the eigenvalues of \(H^z\), and  recall that \(\rho^z(E)=\pi^{-1} \Im m^z(E+\ii 0)\) is the limiting (self-consistent) density of states. Then by Proposition~\ref{theo:trll} the rigidity of \(\lambda^z_i\) follows by a standard application of Helffer-Sj\"ostrand formula (see e.g.~\cite[Lemma 7.1, Theorem 7.6]{MR3068390} or~\cite[Section 5]{MR2871147} for a detailed derivation):
\begin{equation}\label{eq:rigneed}
    \abs*{\lambda_i^{z}-\gamma_i^z }\le \frac{\delta^{-100} n^\xi}{n}, \qquad \abs{i}\le cn,
\end{equation}
with \(c>0\) a small constant and \(\delta:= 1-\abs{z}^2\), with very high probability, uniformly in \(\abs{z}\le 1-n^{-\omega_h}\). The quantiles \(\gamma_i^z\) are defined by
\begin{equation}\label{eq:defquant}
    \frac{i}{n}=\int_0^{\gamma_i^z} \rho^z(E)\dif E, \qquad 1\le i \le n,
\end{equation}
and \(\gamma_{-i}^z:= -\gamma_i^z\) for \(-n\le i \le -1\). Note that by~\eqref{eq:defquant} it follows that \(\gamma_i^z\sim i/(n\rho^z(0))\) for \(\abs{i}\le n^{1-10\omega_h}\), where \(\rho^z(0)=\Im m^z(0)=(1-\abs{z}^2)^{1/2}\) for \(\abs{z}< 1\) by~\eqref{eq:expm}.

Using the rigidity bound in~\eqref{eq:rigneed}, by Proposition~\ref{prop:indeig} we conclude the proof of Proposition~\ref{prop:indmr}.

\begin{proof}[Proof of Proposition~\ref{prop:indmr}]
    Let \(z_1,\dots, z_p\) such that \(\abs{z_l}\le 1-n^{-\omega_h}\) and \(\abs{z_l-z_m}\ge n^{-\omega_d}\), for any \(l,m\in [p]\), with \(\omega_d,\omega_h\) defined in Proposition~\ref{prop:indmr}. Let \(\omega,\widehat{\omega}, \delta_0,\delta_1\) be as in Proposition~\ref{prop:indeig}, i.e.
    \[
        \omega_h\ll \delta_m\ll \widehat{\omega}\ll \omega\ll \omega_f,
    \]
    for \(m=0,1\). For a detailed summary about all the different scales in the proof of Proposition~\ref{prop:indeig} and so of Proposition~\ref{prop:indmr} see Section~\ref{rem:s} later. Write
    \begin{equation}
        \label{eq:dedrez}
        \braket{ G^{z_l}(\ii\eta_l) }=\frac{\ii}{2n}\left[\sum_{\abs{i}\le \widehat{\omega}}+\sum_{\widehat{\omega}< \abs{i}\le n}\right] \frac{\eta_l}{(\lambda_i^{z_l})^2+\eta_l^2},
    \end{equation}
    for \(\eta_l\in [n^{-1-\delta_0},n^{-1+\delta_1}]\). As a consequence of Proposition~\ref{prop:indeig}, the summations over \(\abs{i}\le n^{\widehat{\omega}}\) are asymptotically independent for different \(l\)'s. We now prove that the sum over \(n^{\widehat{\omega}}<\abs{i}\le n\) in~\eqref{eq:dedrez} is much smaller \(n^{-c}\) for some small constant \(c>0\).

    Since \(\omega_h\ll \widehat{\omega}\) the rigidity of the eigenvalues in~\eqref{eq:rigneed} holds for \(n^{\widehat{\omega}}\le \abs{i}\le n^{1-10\omega_h}\), hence we conclude the following bound with very high probability:
    \begin{equation}
        \label{eq:rigbb}
        \frac{1}{n}\sum_{n^{\widehat{\omega}}\le \abs{i} \le n} \frac{\eta_l}{(\lambda_i^{z_l})^2+\eta_l^2}\lesssim n^{40\omega_h}\sum_{n^{\widehat{\omega}}\le\abs{i}\le n} \frac{n\eta_l }{i^2 (\rho^{z_l}(0))^2}\lesssim \frac{n^{\delta_1+40\omega_h}}{n^{\widehat{\omega}}},
    \end{equation}
    where we used that \((\lambda_i^z)^2+\eta^2\gtrsim n^{-40\omega_h}\) for \(n^{1-10\omega_h}\le \abs{i}\le n\), and that \(\eta_l\in[n^{-1-\delta_0},n^{-1+\delta_1}]\). In particular, in~\eqref{eq:rigbb} we used that by~\eqref{eq:defquant} it follows \(\gamma_i^{z_l}\sim i/(n\rho^{z_l}(0))\) for \(\abs{i}\le n^{1-10\omega_h}\), where \(\rho^{z_l}(0)=\Im m^{z_l}(0)=(1-\abs{z_l}^2)^{1/2}\) for \(\abs{z_l}^2\le 1\) by~\eqref{eq:expm}.

    Combining~\eqref{eq:dedrez}--\eqref{eq:rigbb} with Proposition~\ref{prop:indeig} we immediately conclude that
    \[
        \begin{split}
            \E \prod_{l=1}^p \braket{ G^{z_l}(\ii\eta_l) }&= \E \prod_{l=1}^p \frac{\ii}{2n}\sum_{\abs{i}\le n^{\widehat{\omega}}} \frac{\eta_l}{(\lambda_i^{z_l})^2+\eta_l^2}+\mathcal{O}\left( \frac{n^{\delta_1+40\omega_h}}{n^{\widehat{\omega}}} \right) \\
            &=\prod_{l=1}^p\E  \frac{\ii}{2n}\sum_{\abs{i}\le n^{\widehat{\omega}}} \frac{\eta_l}{(\lambda_i^{z_l})^2+\eta_l^2}+\mathcal{O}\left(\frac{n^{p\delta_0+\widehat{\omega}}}{n^{\omega}}+ \frac{n^{\delta_1+40\omega_h}}{n^{\widehat{\omega}}} \right) \\
            &=\prod_{l=1}^p\E  \braket{ G^{z_l}(\ii\eta_l)}+\mathcal{O}\left( \frac{n^{\delta_1+40\omega_h}}{n^{\widehat{\omega}}}+\frac{n^{p\delta_0+\widehat{\omega}}}{n^{\omega}} \right).
        \end{split}
    \]
    This concludes the proof of Proposition~\ref{prop:indmr} since \(\omega_h\ll \delta_m\ll \widehat{\omega}\ll\omega\), with \(m=0,1\).
\end{proof}

We conclude Section~\ref{sec:ririri} with some properties of \(m^z\), the unique solution of~\eqref{eq m}. Fix \(z\in \C \), and consider the \(2n \times 2n\) matrix \(A+F\), with \(F\) a Wigner matrix, whose entries are centred random variables of variance \((2n)^{-1}\), and \(A\) is a deterministic diagonal matrix \(A:= \diag(\abs{z},\dots,\abs{z},-\abs{z},\dots,-\abs{z})\). Then by~\cite[Eq.~(2.1)]{MR4026551},~\cite[Eq.~(2.2)]{MR4134946} it follows that the corresponding \emph{Dyson equation} is given by
\begin{equation}\label{eq:dyseq}
    \begin{cases}
        -\frac{1}{m_1}= w-\abs{z}+\frac{m_1+m_2}{2} \\
        -\frac{1}{m_2}= w+\abs{z}+\frac{m_1+m_2}{2},
    \end{cases}
\end{equation}
which has a unique solution under the assumption \(\Im m_1, \Im m_2>0\). By~\eqref{eq:dyseq} it readily follows that \(m^z\), the solution of~\eqref{eq m}, satisfies
\begin{equation}\label{eq:relwig}
    m^z(w)=\frac{m_1(w)+m_2(w)}{2}.
\end{equation}
In particular, this implies that all the regularity properties of \(m_1+m_2\) (see e.g.~\cite[Theorem 2.4, Lemma A.7]{MR4031100},~\cite[Proposition 2.3, Lemma A.1]{MR4164728}) hold for \(m^z\) as well, e.g.\ \(m^z\) is \(1/3\)-H\"older continuous for any \(z\in\C \).

\subsection{Overview of the proof of Proposition~\ref{prop:indeig}}\label{sec:PO}

The main result of this section is the proof of Proposition~\ref{prop:indeig}, which is divided into two further sub-sections. In Lemma~\ref{lem:GFTGFT}, we prove that we can add a common small Ginibre component to the matrices \(H^{z_l}\), with \(l\in [p]\), \(p\in\N \), without changing their joint eigenvalue distribution much. In Section~\ref{sec:COMPPRO}, we introduce comparison processes for the process defined in~\eqref{eq:DBMeA} below, with initial data \({\bm \lambda}^{z_l}=\{\lambda_{\pm i}^{z_l}\}_{i=1}^n\), where we recall that \(\{\lambda_i^{z_l}\}_{i=1}^n\) are the singular values of \(\check{X}_{t_f}-z_l\), and \(\lambda_{-i}^{z_l}=-\lambda_i^{z_l}\) (the matrix \(\check{X}_{t_f}\) is defined in~\eqref{eq:consOU} below). Finally, in Section~\ref{sec:INDFI} we conclude the proof of Proposition~\ref{prop:indeig}. Additionally, in Section~\ref{rem:s} we summarize the different scales used in the proof of Proposition~\ref{prop:indeig}.

Let \(X\) be an i.i.d.\ complex \(n\times n\) matrix, and run the Ornstein-Uhlenbeck (OU) flow
\begin{equation}\label{eq:OUflow}
    d\widehat{X}_t=-\frac{1}{2}\widehat{X}_t \dif t+\frac{\dif \widehat{B}_t}{\sqrt{n}}, \qquad \widehat{X}_0=X,
\end{equation}
for a time
\begin{equation}\label{eq:time1}
    t_f:= \frac{n^{\omega_f}}{n},
\end{equation}
with some small exponent \(\omega_f>0\) given in Proposition~\ref{prop:indeig}, in order to add a small Gaussian component to \(X\). \(\widehat{B}_t\) in~\eqref{eq:OUflow} is a standard matrix valued complex Brownian motion independent of \(\widehat{X}_0\), i.e.\ \(\sqrt{2}\Re \widehat{B}_{ab}\), \(\sqrt{2}\Im \widehat{B}_{ab}\) are independent standard real Brownian motions for any \(a,b\in [n]\). Then we construct an i.i.d.\ matrix \(\check{X}_{t_f}\) such that
\begin{equation}\label{eq:consOU}
    \widehat{X}_{t_f}\stackrel{d}{=}\check{X}_{t_f}+\sqrt{ct_f} U,
\end{equation}
for some constant \(c>0\) very close to \(1\), and \(U\) is a complex Ginibre matrix independent of \(\check{X}_{t_f}\).

Next, we define the matrix flow
\begin{equation}\label{eq:DBMmA}
    \dif X_t=\frac{\dif B_t}{\sqrt{n}}, \quad X_0=\check{X}_{t_f},
\end{equation}
where \(B_t\) is a standard matrix valued complex Brownian motion independent of \(X_0\) and \(\widehat{B}_t\). Note that by construction \(X_{ct_f}\) is such that
\begin{equation}\label{eq:impGFT}
    X_{ct_f}\stackrel{d}{=}\widehat{X}_{t_f}.
\end{equation}
Define the matrix \(H_t^{z_l}\) as in~\eqref{eq:her} replacing \(X-z\) by \(X_t-z_l\), for any \(l\in [p]\), then the flow in~\eqref{eq:DBMmA} induces the following DBM flow on the eigenvalues of \(H_t^{z_l}\) (cf.~\cite[Eq.~(5.8)]{MR2919197}):
\begin{equation}\label{eq:DBMeA}
    \dif \lambda_i^{z_l}(t)=\sqrt{\frac{1}{2 n}}\dif b_i^{z_l}+\frac{1}{2n}\sum_{j\ne  i} \frac{1}{\lambda_i^{z_l}(t)-\lambda_j^{z_l}(t)} \dif t, \qquad 1\le \abs{i}\le n,
\end{equation}
with initial data \(\{\lambda_{\pm i}^{z_l}(0)\}_{i=1}^n\), where \(\lambda_i^{z_l}(0)\), with \(i\in [n]\) and \(l\in [p]\), are the singular values of \(\check{X}_{t_f}-z_l\), and \(\lambda_{-i}^{z_l}=-\lambda_i^{z_l}\). The well-posedness of~\eqref{eq:DBMeA} follows by~\cite[Appendix A]{MR3916329}. It follows from this derivation that the Brownian motions \(\{b_i^{z_l}\}_{i=1}^n\), omitting the \(t\)-dependence, are defined as
\begin{equation}\label{eq:formbm}
    \dif b_i^{z_l}:= \sqrt{2}\left(\dif B_{ii}^{z_l}+\dif \overline{B_{ii}^{z_l}}\right), \qquad \dif B_{ij}^{z_l}:= \sum_{a,b=1}^n \overline{ u_i^{z_l}(a)} \dif B_{ab}v_j^{z_l}(b),
\end{equation}
where \(({\bm u}_i^{z_l},\pm {\bm v}_i^{z_l})\) are the orthonormal eigenvectors of \(H_t^{z_l}\) with corresponding eigenvalues \(\lambda_{\pm i}^{z_l}\), and \(B_{ab}\) are the entries of the Brownian motion defined in~\eqref{eq:DBMmA}. For negative indices we define \(b_{-i}^{z_l}:= -b_i^{z_l}\). It follows from~\eqref{eq:formbm} that for each fixed \(l\) the collection of Brownian motions \({\bm b}^{z_l}=\{b_i^{z_l}\}_{i=1}^n\) consists of i.i.d.\ Brownian motions, however the families \({\bm b}^{z_l}\) are not independent for different \(l\)'s, in fact their joint distribution is not necessarily Gaussian. The derivation of~\eqref{eq:DBMeA} follows standard steps, see e.g.~\cite[Section 12.2]{MR3699468}. For the
convenience of the reader we included this derivation in Appendix~\ref{sec:derdbm}.

\begin{remark}
    We point out that in the formula~\cite[Eq.~(3.9)]{MR3916329} analogous to~\eqref{eq:DBMeA} the term \(j=-i\) in~\eqref{eq:DBMeA} is apparently missing. This additional term does not influence the results in~\cite[Section 3]{MR3916329} (that are proven for the real DBM for which the term \(j=-i\) is actually not present).
\end{remark}

As a consequence of~\eqref{eq:impGFT} we conclude the following lemma.

\begin{lemma}\label{lem:GFTGFT}
    Let \({\bm \lambda}^{z_l}=\{\lambda_{\pm i}^{z_l}\}_{i=1}^n\) be the eigenvalues of \(H^{z_l}\) and let \({\bm \lambda}^{z_l}(t)\) be the solution of~\eqref{eq:DBMeA} with initial data \({\bm \lambda}^{z_l}\), then
    \begin{equation}
        \label{eq:stgft2}
        \begin{split}
            \E  \prod_{l=1}^p \frac{1}{n}\sum_{\abs{i_l}\le n^{\widehat{\omega}}} \frac{\eta_l}{(\lambda_{i_l}^{z_l})^2+\eta_l^2}&=\E  \prod_{l=1}^p \frac{1}{n}\sum_{\abs{i_l}\le n^{\widehat{\omega}}} \frac{\eta_l}{(\lambda_{i_l}^{z_l}(ct_f))^2+\eta_l^2} \\
            &\quad+\mathcal{O}\left(\frac{n^{p\xi+2\delta_0} t_f}{n^{1/2}}\sum_{l=1}^p \frac{1}{\eta_l}+\frac{n^{k\delta_0+\delta_1}}{n^{\widehat{\omega}}}\right),
        \end{split}
    \end{equation}
    for any sufficiently small \(\widehat{\omega}, \delta_0,\delta_1>0\) such that \(\delta_m\ll \widehat{\omega}\), where \(\eta_l\in [n^{-1-\delta_0},n^{-1+\delta_1}]\) and \(t_f\) defined in~\eqref{eq:time1}.
    \begin{proof}
        The equality in~\eqref{eq:stgft2} follows by a standard Green's function comparison (GFT) argument (e.g.\ see~\cite[Proposition 3.1]{MR4221653}) for the \(\braket{ G^{z_l}(\ii\eta_l)}\), combined with the same argument as in the proof of Proposition~\ref{prop:indmr}, using the local law~\cite[Theorem 5.1]{MR3770875} and~\eqref{eq:impGFT}, to show that the summation over \(n^{\widehat{\omega}}<\abs{i}\le n\) is negligible. We remark that the GFT used in this lemma is much easier than the one in~\cite[Proposition 3.1]{MR4221653} since here we used GFT only for a very short time \(t_f\sim n^{-1+\omega_f}\), for a very small \(\omega_f>0\), whilst in~\cite[Proposition 3.1]{MR4221653} the GFT is considered up to a time \(t=+\infty\). The scaling in the error term in~\cite[Proposition 3.1]{MR4221653} is different compared to the error term in~\eqref{eq:stgft2} since the scaling therein refers to the cusp-scaling.
    \end{proof}
\end{lemma}

\subsubsection{Definition of the comparison processes for \({\bm \lambda}^{z_l}(t)\)}\label{sec:COMPPRO}

The philosophy behind the proof of Proposition~\ref{prop:indeig} is to compare the distribution of  \({\bm \lambda}^{z_l}(t)=\{\lambda_{\pm i}^{z_l}(t)\}\),
the strong solutions of~\eqref{eq:DBMeA} for \(l\in [p]\), which are correlated for different \(l\)'s
and realized on a probability space \(\Omega_b\),
with carefully constructed independent processes \({\bm \mu}^{(l)}(t)=\{\mu^{(l)}_{\pm i}(t)\}_{i=1}^n\) on a different probability space \(\Omega_\beta\).
We choose \({\bm \mu}^{(l)}(t)\) to be the solution of
\begin{equation}\label{eq:ginev}
    \dif\mu_i^{(l)}(t)=\frac{\dif \beta_i^{(l)}}{\sqrt{2n}}+\frac{1}{2n}\sum_{j\ne  i} \frac{1}{\mu_i^{(l)}(t)-\mu_j^{(l)}(t)} \dif t, \quad \mu_i^{(l)}(0)=\mu_i^{(l)},
\end{equation}
for \(\abs{i}\le n\), with \(\mu_i^{(l)}\) the eigenvalues of the matrix
\[
    H^{(l)}:= \left(\begin{matrix}
            0           & X^{(l)} \\
            (X^{(l)})^* & 0
        \end{matrix}\right)
\]
where \(X^{(l)}\) are independent Ginibre matrices, \({\bm \beta}^{(l)}=\{\beta_i^{(l)}\}_{i=1}^n\) are independent vectors of i.i.d.\ standard real Brownian motions, and \(\beta_{-i}^{(l)}=-\beta_i^{(l)}\). We let  \(\mathcal{F}_{\beta,t}\)  denote the common filtration  of the Brownian motions \({\bm \beta}^{(l)}\) on \(\Omega_\beta\).

In the remainder of this section we define two processes \( \widetilde{{\bm \lambda}}^{(l)}\), \( \widetilde{{\bm \mu}}^{(l)}\) so that for a time \(t\ge 0\) large enough \(\widetilde{\lambda}_i^{(l)}(t)\), \(\widetilde{\mu}_i^{(l)}(t)\) for small indices \(i\) will be close to \(\lambda^{z_l}_i(t)\) and \(\mu_i^{(l)}(t)\), respectively, with very high probability. Additionally, the processes \(\widetilde{{\bm \lambda}}^{(l)}\), \( \widetilde{{\bm \mu}}^{(l)}\) will be such that they have the same joint distribution:
\begin{equation}\label{eq:nefdis}
    \left( \widetilde{{\bm \lambda}}^{(1)}(t),\dots, \widetilde{{\bm \lambda}}^{(p)}(t)\right)_{t\ge 0}\stackrel{d}{=}\left(\widetilde{{\bm \mu}}^{(1)}(t),\dots, \widetilde{{\bm \mu}}^{(p)}(t)\right)_{t\ge 0}.
\end{equation}

Fix \(\omega_A>0\) and define the process \(\widetilde{\bm \lambda}(t)\) to be the solution of
\begin{equation}\label{eq:nuproc}
    \dif \widetilde{\lambda}^{(l)}_i(t)=\frac{1}{2n}\sum_{j\ne  i} \frac{1}{\widetilde{\lambda}^{(l)}_i(t)-\widetilde{\lambda}^{(l)}_j(t)} \dif t+\begin{cases}
        \sqrt{\frac{1}{2 n}}\dif b_i^{z_l}             & \text{if} \quad \abs{i}\le n^{\omega_A}     \\
        \sqrt{\frac{1}{2 n}}\dif \widetilde{b}_i^{(l)} & \text{if} \quad n^{\omega_A}< \abs{i}\le n,
    \end{cases}
\end{equation}
with initial data \(\widetilde{\bm \lambda}^{(l)}(0)\) being the singular values, taken with positive and negative sign, of independent Ginibre matrices \(\widetilde{Y}^{(l)}\) independent of \({\bm \lambda}^{z_l}(0)\). Here \(\dif b_i^{z_l}\) is from~\eqref{eq:DBMeA}; this is used for small indices. For large indices we define the driving Brownian motions to be an independent collection \(\set{\{\widetilde{b}_i^{(l)}\}_{i=n^{\omega_A}+1}^n \given l\in [p]}\) of \(p\) vector-valued i.i.d.\ standard real Brownian motions which are also independent of \(\set{\{b_{\pm i}^{z_l}\}_{i=1}^n\given l\in [p]}\), and that \(\widetilde{b}_{-i}^{(l)}=-\widetilde{b}_i^{(l)}\). The Brownian motions \({\bm b}^{z_l}\), with \(l \in [p]\), and \(\set{\{\widetilde{b}_i^{(l)}\}_{i=n^{\omega_A}+1}^n \given l\in [p]}\) are defined on a common probability space  that we continue to denote by \(\Omega_b\) with the common filtration \(\mathcal{F}_{b,t}\).

We conclude this section by defining \(\widetilde{{\bm \mu}}^{(l)}(t)\), the comparison process of \({\bm \mu}^{(l)}(t)\). It is given as the solution of the following DBM\@:
\begin{equation}\label{eq:nuproc2}
    \dif \widetilde{\mu}^{(l)}_i(t)=\frac{1}{2n}\sum_{j\ne  i} \frac{1}{\widetilde{\mu}^{(l)}_i(t)-\widetilde{\mu}^{(l)}_j(t)} \dif t+\begin{cases}
        \sqrt{\frac{1}{2 n}}\dif \zeta_i^{z_l}             & \text{if} \quad \abs{i}\le n^{\omega_A}     \\
        \sqrt{\frac{1}{2 n}}\dif \widetilde{\zeta}_i^{(l)} & \text{if} \quad n^{\omega_A}< \abs{i}\le n,
    \end{cases}
\end{equation}
with initial data \(\widetilde{\bm \mu}^{(l)}(0)\) so that they are the singular values of independent Ginibre matrices \(Y^{(l)}\), which are also independent of \(\widetilde{Y}^{(l)}\). We now explain how to construct the driving Brownian motions in~\eqref{eq:nuproc2} so that~\eqref{eq:nefdis} is satisfied. We only consider positive indices, since the negative indices are defined by symmetry. For indices \(n^{\omega_A}< i\le n\) we choose \(\{\widetilde{\zeta}_{\pm i}^{(l)}\}_{n^{\omega_A}+1}^n\) to be independent families (for different \(l\)'s) of i.i.d.\ Brownian motions, defined on the same probability space of \(\{{\bm \beta}^{(l)}: l\in [p]\}\), that are independent of the Brownian motions \(\{ \beta^{(l)}_{\pm i}\}_{i=1}^n\) used in~\eqref{eq:ginev}. For indices \(1\le i \le n^{\omega_A}\) the families \(\set{\{\zeta_i^{z_l}\}_{i=1}^{n^{\omega_A}}\given l\in [p]}\) will be constructed from the independent families \(\set{\{\beta_i^{(l)}\}_{i=1}^{n^{\omega_A}} \given l\in [p]}\) as follows.

Arranging \(\set{\{\beta_i^{(l)}\}_{i=1}^{n^{\omega_A}}\given l\in [p]}\) into a single vector, we define the \(pn^{\omega_A}\)-dimensional vector
\begin{equation}\label{eq:vecla}
    \underline{\beta}:=(\beta_1^{(1)},\dots,\beta_{n^{\omega_A}}^{(1)},\dots, \beta_1^{(p)}, \dots, \beta_{n^{\omega_A}}^{(p)}).
\end{equation}
Similarly we define the \(pn^{\omega_A}\)-dimensional vector
\begin{equation}\label{eq:vecla1}
    \underline{b}:=(b_1^{z_1},\dots,b_{n^{\omega_A}}^{z_1},\dots, b_1^{z_p}, \dots, b_{n^{\omega_A}}^{z_p})
\end{equation}
which is a continuous martingale. To make our notation easier, in the following we assume that \(n^{\omega_A}\in\N \). For any \(i,j\in [pn^{\omega_A}]\), we use the notation
\begin{equation}\label{eq:frakind}
    i=(l-1) n^{\omega_A}+\mathfrak{i}, \qquad j=(m-1) n^{\omega_A}+\mathfrak{j},
\end{equation}
with \(l,m \in [p]\) and \(\mathfrak{i}, \mathfrak{j}\in [n^{\omega_A}]\). Note that in the definitions in~\eqref{eq:frakind} we used \((l-1), (m-1)\) instead of \(l,m\) so that \(l\) and \(m\) exactly indicate in which block of the matrix \(C(t)\) in~\eqref{eq:matC} the indices \(i,j\) are. With this notation,  the covariance matrix of the increments of \( \underline{b}\) is
the  matrix \(C(t)\) consisting of \(p^2\) blocks of size \(n^{\omega_A}\) is defined as
\begin{equation}\label{eq:matC}
    C_{ij}(t) \dif t:= \Exp*{\dif b_{\mathfrak{i}}^{z_l} \dif b_{\mathfrak{j}}^{z_m}\given\mathcal{F}_{b,t}} =\begin{cases}
        \Theta_{\mathfrak{i}\mathfrak{j}}^{z_l,z_m}(t) \dif t & \text{if} \quad l\ne m, \\
        \delta_{\mathfrak{i}\mathfrak{j}} \dif t              & \text{if} \quad l=m.
    \end{cases}
\end{equation}
Here
\begin{equation}
    \label{eq:ovcorr}
    \Theta_{\mathfrak{i}\mathfrak{j}}^{z_l,z_m}(t):= 4\Re\bigl[\braket{ {\bm u}_{\mathfrak{i}}^{z_l}(t),{\bm u}_{\mathfrak{j}}^{z_m}(t)}\braket{ {\bm v}_{\mathfrak{i}}^{z_m}(t),{\bm v}_{\mathfrak{j}}^{z_l}(t)} \bigr],
\end{equation}
with \(\{{\bm w}_{\pm i} \}_{i\in [n]}=\{({\bm u}_i^{z_l}(t), \pm {\bm v}_i^{z_l}(t))\}_{i\in [n]}\) the orthonormal eigenvectors of \(H_t^{z_l}\). Note that \(\{{\bm w}_i \}_{\abs{i}\le n}\) are not well-defined if \(H_t^{z_l}\) has multiple eigenvalues. However, without loss of generality, we can assume that almost surely \(H_t^{z_l}\) does not have multiple eigenvalues for any \(l\in [p]\), as a consequence of~\cite[Lemma 6.2]{MR4242226} (which is the adaptation of~\cite[Proposition 2.3]{MR4009717} to the \(2\times 2\) block structure of \(H_t^{z_l}\)).

By Doob's martingale representation theorem~\cite[Theorem 18.12]{MR1876169} there exists a standard Brownian motion \( {\bm \theta}_t \in \R^{pN^{\omega_A}} \)
realized on an extension   \( (\widetilde\Omega_b,  \widetilde{\mathcal{F}}_{b,t} )\)  of the original filtrated
probability space \( (\Omega_b, \mathcal{F}_{b,t}) \)
such that \( \dif \underline{\bm b}= \sqrt{C} \dif {\bm \theta}\). Here \( {\bm \theta}_t \)
and \( C(t)\) are adapted to the filtration \( \widetilde{\mathcal{F}}_{b,t} \) and
note that \( C=C(t)\) is a  positive semi-definite matrix
and \(  \sqrt{C}\) denotes its positive semi-definite matrix square root.

For the clarity of the presentation the original processes \( {\bm \lambda}^{z_l}\) and the
comparison  processes \({\bm \mu}^{(l)}\) will be realized on completely different probability spaces. We thus construct
another copy \( (\Omega_\beta, \mathcal{F}_{\beta,t} )\) of the filtrated probability space  \( (\widetilde\Omega_b,  \widetilde{\mathcal{F}}_{b,t} )\)
and we
construct a  matrix valued process \(C^\#(t)\)  and a  Brownian motion \( \underline{\beta} \) on  \(( \Omega_\beta, \mathcal{F}_{\beta,t}) \)
such that  \( (C^\#(t),  \underline{\beta}(t) ) \) are adapted to the filtration  \(\mathcal{F}_{\beta,t}\)
and they have the same  joint distribution as \( (C(t), {\bm \theta}(t)) \). The Brownian motion \( \underline{\beta} \) is used
in~\eqref{eq:ginev} for small indices.

Define the process
\begin{equation}\label{eq:defzeta}
    \underline{\zeta}(t):=\int_0^t\sqrt{C^\#(s)} \dif \underline{\beta}(s),\quad \underline{\zeta}=(\zeta_1^{z_1},\dots,\zeta_{n^{\omega_A}}^{z_1},\dots, \zeta_1^{z_p}, \dots, \zeta_{n^{\omega_A}}^{z_p}),
\end{equation}
on the probability space \(\Omega_\beta\)
and define \(\zeta_{-i}^{z_l}:=-\zeta_i^{z_l}\) for any \(1\le i\le n^{\omega_A}\), \(l\in[p]\).
Since \(\underline{\beta}\) are i.i.d.\ Brownian motions, we clearly have
\begin{equation}\label{eq:newocov}
    \Exp*{\dif \zeta_{\mathfrak{i}}^{z_l}(t)\dif \zeta_{\mathfrak{j}}^{z_m}(t)\given \mathcal{F}_{\beta,t}}= C^\#(t)_{ij} \dif t,
    \qquad \abs{\mathfrak{i}},\abs{\mathfrak{j}}\le n^{\omega_A}.
\end{equation}
By construction
we see that  the processes \( ( \{b_{\pm i}^{z_l}\}_{i=1}^{n^{\omega_A}})_{l=1}^k \) and
\( ( \{\zeta_{\pm i}^{z_l}\}_{i=1}^{n^{\omega_A}} )_{l=1}^k \) have the same distribution. Furthermore,
since by definition the two collections
\[\set*{\{\widetilde{b}_{\pm i}^{(l)}\}_{i=n^{\omega_A}+1}^n, \{\widetilde{\zeta}^{(l)}_{\pm i}\}_{i=n^{\omega_A}+1}^n \given l\in [k]}\]
are
independent of \[\set*{\{b_{\pm i}^{z_l}\}_{i=1}^{n^{\omega_A}}, \{\beta_{\pm i}^{(l)}\}_{i=1}^{n^{\omega_A}}\given l\in [k]}\]
and among each other, we have
\begin{equation}\label{eq:BMsost}
    \left( \{b_{\pm i}^{z_l}\}_{i=1}^{n^{\omega_A}}, \{\widetilde{b}_{\pm i}^{(l)}\}_{i=n^{\omega_A}+1}^n\right)_{l=1}^p\stackrel{d}{=}\left( \{\zeta_{\pm i}^{z_l}\}_{i=1}^{n^{\omega_A}}, \{\widetilde{\zeta}_{\pm i}^{(l)}\}_{i=n^{\omega_A}+1}^n\right)_{l=1}^p.
\end{equation}

Finally, by the definitions in~\eqref{eq:nuproc},~\eqref{eq:nuproc2}, and~\eqref{eq:BMsost}, it follows that the Dyson Brownian motions \(\widetilde{\bm \lambda}^{(l)}\) and \(\widetilde{\bm \mu}^{(l)}\) have the same distribution, i.e.
\begin{equation}\label{eq:samedistr}
    \left(\widetilde{\bm \lambda}^{(1)}(t), \dots, \widetilde{\bm \lambda}^{(p)}(t)\right)\stackrel{d}{=} \left(\widetilde{\bm \mu}^{(1)}(t), \dots, \widetilde{\bm \mu}^{(p)}(t)\right)
\end{equation}
since their initial conditions, as well as their driving processes~\eqref{eq:BMsost},  agree in distribution. Note that these processes are Brownian motions for each fixed \( l\) since  \(  C_{ij}(t)=\delta_{\mathfrak{i}\mathfrak{j}} \) if  \( l=m\), but jointly they are not necessarily Gaussian due to the non-trivial  correlation \( \Theta_{\mathfrak{i}\mathfrak{j}}^{z_l,z_m} \) in~\eqref{eq:matC}.

\subsubsection{Proof of Proposition~\ref{prop:indeig}}\label{sec:INDFI}

In this section we conclude the proof of Proposition~\ref{prop:indeig} using the comparison processes defined in Section~\ref{sec:COMPPRO}. More precisely, we use that the processes \({\bm \lambda}^{z_l}(t)\), \(\widetilde{\bm \lambda}^{(l)}(t)\) and \({\bm \mu}^{(l)}(t)\), \(\widetilde{\bm \mu}^{(l)}(t)\) are close
pathwise at time \(t_f\), as stated below in
Lemma~\ref{lem:firststepmason} and Lemma~\ref{lem:secondstepmason}, respectively. The proofs of these lemmas are postponed to Section~\ref{sec:noncelpiu}. They will be a consequence of Proposition~\ref{pro:ciala}, which is an adaptation to our case of the main technical estimate of~\cite{MR3914908}. The main input  is the bound on the eigenvector overlap in Lemma~\ref{lem:overb}, since it gives an upper bound on the correlation structure in~\eqref{eq:newocov}. Let \(\rho_{sc}(E) =\frac{1}{2\pi}\sqrt{4-E^2}\) denote the semicircle density.

\begin{lemma}\label{lem:firststepmason}
    Fix \(p\in \N \), and let \({\bm \lambda}^{z_l}(t)\), \(\widetilde{\bm \lambda}^{(l)}(t)\), with \(l\in [p]\), be the processes defined in~\eqref{eq:DBMeA} and~\eqref{eq:nuproc}, respectively. For any small \(\omega_h,\omega_f>0\) such that \(\omega_h\ll \omega_f\) there exist \(\omega, \widehat{\omega}>0\) with \(\omega_h\ll \widehat{\omega}\ll \omega\ll \omega_f\), such that for any \(\abs{z_l}\le 1-n^{-\omega_h}\) it holds
    \begin{equation}
        \label{eq:firshpb}
        \abs*{\rho^{z_l}(0)\lambda_i^{z_l}(ct_f)-\rho_{sc}(0) \widetilde{\lambda}_i^{(l)}(ct_f) }\le n^{-1-\omega}, \qquad \abs{i}\le n^{\widehat{\omega}},
    \end{equation}
    with very high probability, where \(t_f:= n^{-1+\omega_f}\) and \(c>0\) is defined in~\eqref{eq:impGFT}.
\end{lemma}

\begin{lemma}\label{lem:secondstepmason}
    Fix \(p\in \N \), and let \({\bm \mu}^{(l)}(t)\), \(\widetilde{\bm \mu}^{(l)}(t)\), with \(l\in [p]\), be the processes defined in~\eqref{eq:ginev} and~\eqref{eq:nuproc2}, respectively. For any small \(\omega_h,\omega_f, \omega_d>0\) such that \(\omega_h\ll \omega_f\) there exist \(\omega, \widehat{\omega}>0\) with \(\omega_h\ll \widehat{\omega}\ll \omega\ll \omega_f\), such that for any  \(\abs{z_l}\le 1-n^{-\omega_h}\), \(\abs{z_l-z_m}\ge n^{-\omega_d}\), with \(l\ne m\), it holds
    \begin{equation}
        \label{eq:firshpb2}
        \abs*{\mu_i^{(l)}(ct_f)- \widetilde{\mu}_i^{(l)}(ct_f) }\le n^{-1-\omega}, \qquad \abs{i}\le n^{\widehat{\omega}},
    \end{equation}
    with very high probability, where \(t_f:= n^{-1+\omega_f}\) and \(c>0\) is defined in~\eqref{eq:impGFT}.
\end{lemma}

\begin{proof}[Proof of Proposition~\ref{prop:indeig}]
    In the following we omit the trivial scaling factors \(\rho^{z_l}(0)\), \(\rho_{sc}(0)\) in the second term in the lhs.\ of~\eqref{eq:firshpb} to make our notation easier.
    We recall that by Lemma~\ref{lem:GFTGFT} we have
    \begin{equation}\label{eq:stgft}
        \begin{split}
            \E  \prod_{l=1}^p \frac{1}{n}\sum_{\abs{i_l}\le n^{\widehat{\omega}}} \frac{\eta_l}{(\lambda_{i_l}^{z_l})^2+\eta_l^2}&=\E  \prod_{l=1}^p \frac{1}{n}\sum_{\abs{i_l}\le n^{\widehat{\omega}}} \frac{\eta_l}{(\lambda_{i_l}^{z_l}(ct_f))^2+\eta_l^2} \\
            &\quad +\mathcal{O}\left(\frac{n^{p\xi+2\delta_0} t_f}{n^{1/2}}\sum_{l=1}^p \frac{1}{\eta_l}+\frac{n^{p\delta_0+\delta_1}}{n^{\widehat{\omega}}}\right),
        \end{split}
    \end{equation}
    where \(\lambda_i^{z_l}(t)\) is the solution of~\eqref{eq:DBMeA} with initial data \(\lambda_i^{z_l}\).
    Next we replace \(\lambda_i^{z_l}(t)\) with \(\widetilde \lambda_i^{z_l}(t)\) for small indices  by using Lemma~\ref{lem:firststepmason}; this is
    formulated in the following lemma whose detailed proof is postponed to the end of this section.

    \begin{lemma}\label{lem:stanc}
        Fix \(p\in\N \), and let \(\lambda_i^{z_l}(t)\), \(\widetilde{\lambda}_i^{(l)}(t)\), with \(l\in [p]\), be the solution of~\eqref{eq:DBMeA} and~\eqref{eq:nuproc}, respectively. Then
        \begin{equation}\label{eq:hhh1}
            \E  \prod_{l=1}^p \frac{1}{n}\sum_{\abs{i_l}\le n^{\widehat{\omega}}} \frac{\eta_l}{(\lambda_{i_l}^{z_l})^2+\eta_l^2}=\E  \prod_{l=1}^p \frac{1}{n}\sum_{\abs{i_l}\le n^{\widehat{\omega}}} \frac{\eta_l}{(\widetilde{\lambda}_{i_l}^{(l)}(ct_f))^2+\eta_l^2}+\mathcal{O}(\Psi),
        \end{equation}
        where \(\lambda_{i_l}^{z_l}=\lambda_{i_l}^{z_l}(0)\), \(t_f=n^{-1+\omega_f}\), and the error term is given by
        \[
            \Psi:= \frac{n^{\widehat{\omega}}}{n^{1+\omega}}\left(\sum_{l=1}^p \frac{1}{\eta_l}\right)\cdot \prod_{l=1}^p \left(1+\frac{n^\xi}{n\eta_l}\right)+\frac{n^{p\xi+2\delta_0} t_f}{n^{1/2}}\sum_{l=1}^p \frac{1}{\eta_l}+\frac{n^{p\delta_0+\delta_1}}{n^{\widehat{\omega}}}.
        \]
    \end{lemma}

    By~\eqref{eq:samedistr} it readily follows that
    \begin{equation}
        \label{eq:hhh2}
        \E  \prod_{l=1}^p \frac{1}{n}\sum_{\abs{i_l}\le n^{\widehat{\omega}}} \frac{\eta_l}{(\widetilde{\lambda}_{i_l}^{(l)}(ct_f))^2+\eta_l^2}=\E  \prod_{l=1}^p \frac{1}{n}\sum_{\abs{i_l}\le n^{\widehat{\omega}}} \frac{\eta_l}{(\widetilde{\mu}_{i_l}^{(l)}(ct_f))^2+\eta_l^2}.
    \end{equation}
    Moreover, by~\eqref{eq:firshpb2}, similarly to Lemma~\ref{lem:stanc}, we conclude
    \begin{equation}
        \label{eq:hhh3}
        \E  \prod_{l=1}^p \frac{1}{n}\sum_{\abs{i_l}\le n^{\widehat{\omega}}} \frac{\eta_l}{(\widetilde{\mu}_{i_l}^{(l)}(ct_f))^2+\eta_l^2}=\E  \prod_{l=1}^p \frac{1}{n}\sum_{\abs{i_l}\le n^{\widehat{\omega}}} \frac{\eta_l}{(\mu_{i_l}^{(l)}(ct_f))^2+\eta_l^2}+\mathcal{O}(\Psi).
    \end{equation}
    Additionally, by the definition of the processes \({\bm \mu}^{(l)}(t)\) in~\eqref{eq:ginev} it follows that \({\bm \mu}^{(l)}(t)\), \({\bm \mu}^{(m)}(t)\) are independent for \(l\ne m\) and so that
    \begin{equation}
        \label{eq:hhh6}
        \E  \prod_{l=1}^p \frac{1}{n}\sum_{\abs{i_l}\le n^{\widehat{\omega}}} \frac{\eta_l}{(\mu_{i_l}^{(l)}(ct_f))^2+\eta_l^2}= \prod_{l=1}^p \E \frac{1}{n}\sum_{\abs{i_l}\le n^{\widehat{\omega}}} \frac{\eta_l}{(\mu_{i_l}^{(l)}(ct_f))^2+\eta_l^2}.
    \end{equation}

    Combining~\eqref{eq:hhh1}--\eqref{eq:hhh6}, we get
    \begin{equation}
        \label{eq:hhh4}
        \E  \prod_{l=1}^p \frac{1}{n}\sum_{\abs{i_l}\le n^{\widehat{\omega}}} \frac{\eta_l}{(\lambda_{i_l}^{z_l})^2+\eta_l^2}= \prod_{l=1}^p \E \frac{1}{n}\sum_{\abs{i_l}\le n^{\widehat{\omega}}} \frac{\eta_l}{(\mu_{i_l}^{(l)}(ct_f))^2+\eta_l^2}+\mathcal{O}(\Psi).
    \end{equation}

    Then, by similar computation to the ones in~\eqref{eq:stgft}--\eqref{eq:hhh4} we conclude that
    \begin{equation}
        \label{eq:hhh5}
        \prod_{l=1}^p \E \frac{1}{n}\sum_{\abs{i_l}\le n^{\widehat{\omega}}} \frac{\eta_l}{(\lambda_{i_l}^{z_l})^2+\eta_l^2}= \prod_{l=1}^p \E \frac{1}{n}\sum_{\abs{i_l}\le n^{\widehat{\omega}}} \frac{\eta_l}{(\mu_{i_l}^{(l)}(ct_f))^2+\eta_l^2}+\mathcal{O}(\Psi).
    \end{equation}
    We remark that in order to prove~\eqref{eq:hhh5} it would not be necessary to introduce the additional comparison processes \(\widetilde{\bm \lambda}^{(l)}\) and \(\widetilde{\bm \mu}^{(l)}\) of Section~\ref{sec:COMPPRO}, since in~\eqref{eq:hhh5} the product is outside the expectation, so one can compare the expectations one by one; the correlation between these processes for different \(l\)'s plays no role. Hence, already the usual coupling (see e.g.~\cite{MR3541852,MR3916329, MR3914908}) between the processes \({\bm \lambda}^{z_l}(t)\), \({\bm \mu}^{(l)}(t)\) defined in~\eqref{eq:DBMeA} and~\eqref{eq:ginev}, respectively, would be sufficient to prove~\eqref{eq:hhh5}.

    Finally, combining~\eqref{eq:hhh4}--\eqref{eq:hhh5} we conclude the proof of Proposition~\ref{prop:indeig}.
\end{proof}

\begin{proof}[Proof of Lemma~\ref{lem:stanc}]
    We show the proof  for \(p=2\) in order to make our presentation easier. The case \(p\ge 3\) proceeds exactly in the same way. In order to make our notation shorter, for \(l\in \{1,2\}\), we define
    \[
        T_{i_l}^{(l)} := \frac{\eta_l}{ (\lambda_{i_l}^{z_l}(ct_f))^2 + \eta_l^2}.
    \]
    Similarly, replacing \(\lambda_{i_l}^{z_l}(ct_f)\) with \(\widetilde{\lambda}_{i_l}^{(l)}(ct_f)\), we define \(\widetilde{T}_l\). Then, by telescopic sum, we have
    \begin{gather}
        \begin{aligned}
             & \abs*{\E  \prod_{l=1}^{2}\frac{1}{n}\sum_{\abs{i_l}\le n^{\widehat{\omega}}} T_{i_l}^{(l)}-\E  \prod_{l=1}^{2}\frac{1}{n}\sum_{\abs{i_l}\le n^{\widehat{\omega}}} \widetilde{T}_{i_l}^{(l)} }                                                                                                     \\
             & \quad= \frac{1}{n^2}\abs*{\E  \sum_{\abs{i_1}, \abs{i_2}\le n^{\widehat{\omega}}} \left[T_{i_1}^{(1)}-\widetilde{T}_{i_1}^{(1)}\right]   T_{i_2}^{(2)} -\E  \sum_{\abs{i_1},\abs{i_2}\le n^{\widehat{\omega}}} \left[T_{i_2}^{(2)}-\widetilde{T}_{i_2}^{(2)}\right]   \widetilde{T}_{i_1}^{(1)} } \\
             & \quad \lesssim \sum_{\substack{l,m=1                                                                                                                                                                                                                                                              \\ l\ne m}}^2\left(1+\frac{n^\xi}{n\eta_l}\right)  \E \frac{1}{n}\sum_{\abs{i_m}\le n^{\widehat{\omega}}} \frac{T_{i_m}^{(m)}\widetilde{T}_{i_m}^{(m)}}{\eta_m}\abs*{(\widetilde{\lambda}_{i_m}^{(m)}(ct_f))^2-(\lambda_{i_m}^{z_m}(ct_f))^2} \\
             & \quad \lesssim \frac{n^{\widehat{\omega}}}{n^{1+\omega}}\left(\frac{1}{\eta_1}+\frac{1}{\eta_2}\right)\cdot \prod_{l=1}^2\left(1+\frac{n^\xi}{n\eta_l}\right),
        \end{aligned}\label{eq:thirdstemason}\raisetag{-8em}
    \end{gather}
    where we used the local law~\eqref{theo:Gll} in the first inequality and~\eqref{eq:firshpb} in the last step.
    Combining~\eqref{eq:thirdstemason} with~\eqref{eq:stgft} we conclude the proof of Lemma~\ref{lem:stanc}.
\end{proof}

Before we continue, we summarize the scales  used
in the entire Section~\ref{sec:IND}.

\subsubsection{Relations among the scales in the proof of Proposition~\ref{prop:indeig}}\label{rem:s}

Scales in the proof of Proposition~\ref{prop:indeig} are characterized by various exponents \(\omega\)'s of \(n\) that we will also refer to scales, for simplicity.
The basic input scales
in the proof of Proposition~\ref{prop:indeig} are \(0<\omega_d,\omega_h,\omega_f\ll 1\), the
others will depend on them.  The exponents \(\omega_h,\omega_d\) are chosen
within the assumptions of Lemma~\ref{lem:overb} to control the location of \(z\)'s as
\(\abs{z_l}\le 1-n^{-\omega_h}\), \(\abs{z_l-z_m}\ge n^{-\omega_p}\), with \(l\ne m\). The exponent \(\omega_f\) defines
the time \(t_f=n^{-1+\omega_f}\)  so that the local equilibrium of the DBM is reached after \(t_f\). This will provide
the asymptotic independence of \(\lambda_i^{z_l}\), \(\lambda_j^{z_m}\) for small indices and for \(l\ne m\).

The primary  scales created along  the proof of Proposition~\ref{prop:indeig}  are
\(\omega\), \(\widehat{\omega}\), \(\delta_0\), \(\delta_1\), \(\omega_E\), \(\omega_B\).
The scales \(\omega_E, \omega_B\) are given in Lemma~\ref{lem:overb}: \(n^{-\omega_E}\) measures the size of the
eigenvector overlaps from~\eqref{eq:ovcorr} while the exponent \(\omega_B\)  describes the range of indices for which
these overlap estimates hold. Recall that the overlaps determine the correlations among the
driving Brownian motions.   The scale \(\omega\) quantifies the \(n^{-1-\omega}\) precision of the coupling between various processes.
These couplings are effective only for small indices \(i\), their range is given by
\(\widehat{\omega}\) as \(\abs{i}\le n^{\widehat{\omega}}\).
Both  these scales are much bigger than \(\omega_h\)  but much smaller than \(\omega_f\).  They are
determined in Lemma~\ref{lem:firststepmason}, Lemma~\ref{lem:secondstepmason}, in fact both lemmas give only a
necessary upper bound on the scales \(\omega, \widehat{\omega}\), so we can pick the smaller of them.
The exponents \(\delta_0, \delta_1\) determine the range of \(\eta\in [n^{-1-\delta_0}, n^{-1+\delta_1}]\) for which
Proposition~\ref{prop:indeig} holds; these are  determined in  Lemma~\ref{lem:GFTGFT} after \(\omega, \widehat{\omega}\)
have already been fixed. These steps yield the scales
\(\omega, \widehat{\omega}, \delta_0, \delta_1\) claimed in
Proposition~\ref{prop:indeig} and hence also in Proposition~\ref{prop:indmr}. We summarize order relation among all these scales as
\begin{equation}
    \label{eq:chain}
    \omega_h\ll \delta_m \ll \widehat{\omega}\ll \omega\ll \omega_B\ll\omega_f\ll \omega_E\ll1, \qquad m=0,1.
\end{equation}

We mention that three further auxiliary scales emerge along the proof but they play only a local, secondary role. For completeness
we also list them here; they are
\(\omega_1,\omega_A,\omega_l\). Their meanings are the following: \(t_1:= n^{-1+\omega_1}\), with \(\omega_1\ll \omega_f\),
is the time needed for the DBM process \(x_i(t,\alpha)\), defined in~\eqref{eq:intflowA}, to reach local equilibrium, hence to prove its universality; \(t_0:=t_f-t_1\) is the initial time we run the DBM before starting with the actual proof of universality so that the solution \({\bm \lambda}^{z_l}(t_0)\) of~\eqref{eq:DBMeA} at time \(t_0\) and the density \(\dif \rho(E,t,\alpha)\) (which we will define in Section~\ref{sec:PR}) satisfy certain technical regularity conditions~\cite[Lemma 3.3-3.5]{MR3916329},~\cite[Lemma 3.3-3.5]{MR3914908}. Note that \(t_0\sim t_f\), in fact they are almost the same.
The other two scales are technical: \(\omega_l\) is the scale of the short range interaction, and \(\omega_A\) is a cut-off scale such that \(x_i(t,\alpha)\) is basically independent of \(\alpha\) for \(\abs{i}\le n^{\omega_A}\). These scales are inserted in the above chain
of inequalities~\eqref{eq:chain} between \(\omega, \omega_B\) as follows
\[
    \omega_h\ll\delta_m\ll\widehat{\omega}\ll \omega\ll \omega_1\ll \omega_l\ll \omega_A\le \omega_B\ll \omega_f \ll \omega_E\ll 1, \quad m=0,1.
\]
In particular, the relation \(\omega_A\ll \omega_E\) ensures that the effect of the correlation is small, see the bound in~\eqref{eq:imperrest} later.

We remark that introducing the additional initial time layer  \(t_0\) is not really necessary for our proof of Proposition~\ref{prop:indeig} since the initial data \({\bm \lambda}^z(0)\) of the DBM in~\eqref{eq:DBMeA} and their deterministic density \(\rho^z\) already satisfy~\cite[Lemma 3.3-3.5]{MR3916329},~\cite[Lemma 3.3-3.5]{MR3914908} as a consequence of~\eqref{eq:lll} (see Remark~\ref{rem:t0diff} and Remark~\ref{rem:ne} for more details).
We keep it only to facilitate the comparison with
~\cite{MR3916329, MR3914908}.

\subsection{Bound on the eigenvector overlap for large \texorpdfstring{\(\abs{z_1-z_2}\)}{|z1-z2|}}\label{sec:BOUNDE}
For any \(z\in\C \), let \(\{{\bm w}^z_{\pm i}\}_{i=1}^n\) be the eigenvectors of the matrix \(H^z\). They are of the form \({\bm w}^z_{\pm i}=({\bm u}_i^z,\pm {\bm v}_i^z)\), with \({\bm u}_i^z, {\bm v}_i^z\in\C ^n\), as a consequence of the symmetry of the spectrum of \(H^z\) induced by its block structure. The main input to prove Lemma~\ref{lem:firststepmason}--\ref{lem:secondstepmason} is the
following high probability bound on the almost orthogonality of the eigenvectors belonging to distant \(z_l\), \(z_m\) parameters and eigenvalues close to zero.
With the help of the Dyson Brownian motion (DBM), this information will then be used to establish almost independence of these eigenvalues.

\begin{lemma}\label{lem:overb}
    Let \(\{{\bm w}^{z_l}_{\pm i}\}_{i=1}^n=\{({\bm u}_i^{z_l},\pm {\bm v}_i^{z_l})\}_{i=1}^n\), for \(l=1,2\), be the eigenvectors of matrices \(H^{z_l}\)  of the form~\eqref{eq:her} with i.i.d.\ entries. Then for any sufficiently small \(\omega_d, \omega_h>0\) there exist \(\omega_B, \omega_E>0\) such that if \(\abs{z_1-z_2}\ge n^{-\omega_d}\), \(\abs{z_l}\le 1-n^{-\omega_h}\) then
    \begin{equation}
        \label{eq:bbev}
        \abs*{\braket{ {\bm u}_i^{z_1}, {\bm u}_j^{z_2}}}+\abs*{\braket{ {\bm v}_i^{z_1}, {\bm v}_j^{z_2}}}\le n^{-\omega_E}, \quad 1\le i,j\le n^{\omega_B},
    \end{equation}
    with very high probability.
\end{lemma}
\begin{proof}
    Using the spectral symmetry of \(H^z\), for any \(z\in\C \) we write \(G^z\) in spectral decomposition as
    \[
        G^z(\ii\eta)=\sum_{j>0} \frac{2}{(\lambda_j^z)^2+\eta^2}\left( \begin{matrix}
                \ii \eta {\bm u}_j^z ({\bm u}_j^z)^*   & \lambda_j^z {\bm u}_j^z ({\bm v}_j^z)^* \\
                \lambda_j^z {\bm v}_j^z({\bm u}_j^z)^* & \ii \eta {\bm v}_j^z ({\bm v}_j^z)^*
            \end{matrix}\right).
    \]
    Let \(\eta \ge n^{-1}\), then by rigidity of the eigenvalues in~\eqref{eq:rigneed}, for any \(i_0, j_0\ge 1\) such that \(\lambda_{i_0}^{z_l},\lambda_{j_0}^{z_l}\lesssim \eta\), with \(l=1,2\), and any \(z_1, z_2\) such that \(n^{-\omega_d} \lesssim \abs{z_1-z_2}\lesssim 1\), for some \(\omega_d>0\) we will choose shortly, it follows that
    \begin{gather}
        \begin{aligned}
             & \abs*{\braket{ {\bm u}_{i_0}^{z_1}, {\bm u}_{j_0}^{z_2}}}^2+\abs*{\braket{ {\bm v}_{i_0}^{z_1}, {\bm v}_{j_0}^{z_2}}}^2                                                                                                      \\
             & \qquad\lesssim \sum_{i,j=1}^n \frac{4\eta^4}{((\lambda_i^{z_1})^2+\eta^2)((\lambda_j^{z_2})^2+\eta^2)} \left(\abs*{\braket{ {\bm u}_i^{z_1}, {\bm u}_j^{z_2}}}^2+\abs*{\braket{ {\bm v}_i^{z_1}, {\bm v}_j^{z_2}}}^2 \right) \\
             & \qquad=\eta^2\Tr  (\Im G^{z_1})(\Im G^{z_2}) \lesssim \frac{n^{8\omega_d/3}}{(n\eta)^{1/4}}+(\eta^{1/12}+n\eta^2) n^{2\omega_d}                                                                                              \\
             & \qquad\lesssim \frac{n^{2\omega_d+100\omega_h}}{n^{1/23}}.
        \end{aligned}\label{eq:ooo}\raisetag{-6em}
    \end{gather}
    The first inequality in the second line of~\eqref{eq:ooo} is from Theorem~\ref{thm local law G2} and the lower bound on \(\abs{\widehat{\beta}_*}\) from~\eqref{beta ast bound}. In the last inequality we choose \(\eta=n^{-12/23}\),
    under the assumption that \(\omega_d \le 1/100\) and that \(i_0,j_0\le n^{1/5}\) (in order to make sure that the first inequality in~\eqref{eq:ooo} hold). We also used that the first term in the lhs.\ of the last inequality is always smaller than the other two for \(\eta\ge n^{-5/9}\), and in the second line of~\eqref{eq:ooo} we used that \(M_{12}\), the deterministic approximation of \(\Tr \Im G^{z_1}\Im G^{z_2}\) in Theorem~\ref{thm local law G2}, is bounded by \(\norm{ M_{12}}\lesssim \abs{z_1-z_2}^{-2}\).

    This concludes the proof by choosing \(\omega_B\le 1/5\) and \(\omega_d= 1/100\), which implies a choice of \(\omega_E=-(2\omega_d+100\omega_h-1/23)\).
\end{proof}

\subsection{Pathwise coupling of DBM close to zero}\label{sec:fixun}

This section is the main technical result used in the proof of Lemma~\ref{lem:firststepmason} and Lemma~\ref{lem:secondstepmason}. We compare the evolution of two DBMs whose driving Brownian motions are nearly the same for small indices and are independent for large indices. In Proposition~\ref{pro:ciala} we will show that the points with small indices in the two processes become very close to each other on a certain time scale \(t_f\). This time scale is chosen to be larger than the local equilibration time, but not too large so that the independence of the driving Brownian motions for large indices do not yet have an effect on particles with small indices.

\begin{remark}\label{rem:t0diff}
    The main result of this section (Proposition~\ref{pro:ciala}) is stated for general deterministic initial data \({\bm s}(0)\) satisfying Definition~\ref{eq:defregpro} even if for its applications in the proof of Proposition~\ref{prop:indeig} we only consider initial data which are eigenvalues of i.i.d.\ random matrices.
\end{remark}

The proof of Proposition~\ref{pro:ciala} follows the proof of fixed energy universality in~\cite{MR3541852,MR3916329,MR3914908}, adapted to the block structure~\eqref{eq:her} in~\cite{MR3916329} (see also~\cite{MR4009717,MR4242226} for further adaptations of~\cite{MR3541852,MR3914908} to different matrix models). The main novelty in our DBM analysis compared to~\cite{MR3541852,MR3916329, MR3914908} is that we analyse a process for which we allow not (fully) coupled driving Brownian motions (see Assumption~\ref{ass:close}).

Define the processes \(s_i(t)\), \(r_i(t)\) to be the solution of
\begin{equation}\label{eq:lambdapr}
    \dif s_i(t)=\sqrt{\frac{1}{2 n}}\dif \mathfrak{b}^s_i(t)+\frac{1}{2n}\sum_{j\ne  i} \frac{1}{s_i(t)-s_j(t)} \dif t, \qquad 1\le \abs{i}\le n,
\end{equation}
and
\begin{equation}\label{eq:mupr}
    \dif r_i(t)=\sqrt{\frac{1}{2 n}}\dif \mathfrak{b}^r_i(t)+\frac{1}{2n}\sum_{j\ne  i} \frac{1}{r_i(t)-r_j(t)} \dif t, \qquad 1\le \abs{i}\le n,
\end{equation}
with initial data \(s_i(0)=s_i\), \(r_i(0)=r_i\), where \({\bm s}=\{s_{\pm i}\}_{i=1}^n\) and \({\bm r}=\{r_{\pm i}\}_{i=1}^n\) are two independent sets of particles such that \(s_{-i}=-s_i\) and \(r_{-i}=-r_i\) for \(i\in [n]\). The driving standard real Brownian motions \(\{\mathfrak{b}^s_i\}_{i=1}^n\), \(\{\mathfrak{b}^r_i\}_{i=1}^n\) in~\eqref{eq:lambdapr}--\eqref{eq:mupr} are two i.i.d.\ families and they are such that \(\mathfrak{b}^s_{-i}=-\mathfrak{b}^s_i\), \(\mathfrak{b}^r_{-i}=-\mathfrak{b}^r_i\) for \(i\in [n]\). For convenience we also assume that \(\{r_{\pm i}\}_{i=1}^n\) are the singular values of \(\widetilde{X}\), with \(\widetilde{X}\) a Ginibre matrix. This is not a restriction; indeed, once a process with general initial data \({\bm s}\) is shown to be close to the reference process with Ginibre initial data, then processes with any two initial data will be close.

Fix an \(n\)-dependent parameter \(K=K_n=n^{\omega_K}\), for some \(\omega_K>0\).
On the correlation structure between the two families of i.i.d.\ Brownian motions \(\{\mathfrak{b}^s_i\}_{i=1}^n\), \(\{\mathfrak{b}^r_i\}_{i=1}^n\) we make the following assumptions:

\begin{assumption}\label{ass:close}
    Suppose that the families \(\{\mathfrak{b}^s_{\pm i}\}_{i=1}^n\), \(\{\mathfrak{b}^r_{\pm i}\}_{i=1}^n\) in~\eqref{eq:lambdapr} and\eqref{eq:mupr} are realised on a common probability space with a common filtration \(\mathcal{F}_t\). Let
    \begin{equation}
        \label{eq:defL}
        L_{ij}(t) \dif t:= \Exp*{\bigl(\dif \mathfrak{b}^s_i(t)-\dif \mathfrak{b}^r_i(t)\bigr) \bigl(\dif \mathfrak{b}^s_j(t)-\dif \mathfrak{b}^r_j(t)\bigr)\given \mathcal{F}_t}
    \end{equation}
    denote the covariance of the increments conditioned on \(\mathcal{F}_t\). The processes satisfy the following assumptions:
    \begin{enumerate}[label=(\alph*)]
        \item\label{close1} \(\{ \mathfrak{b}^s_i\}_{i=1}^n\), \(\{ \mathfrak{b}^r_i\}_{i=1}^n\) are two families of i.i.d.\ standard real Brownian motions.
        \item\label{close2} \(\{ \mathfrak{b}^r_{\pm i}\}_{i=K+1}^n\) is independent of \(\{\mathfrak{b}^s_{\pm i}\}_{i=1}^n\),
        and \(\{ \mathfrak{b}^s_{\pm i}\}_{i=K+1}^n\) is independent of \(\{\mathfrak{b}^r_{\pm i}\}_{i=1}^n\).
        \item\label{close3} Fix \(\omega_Q>0\) so that \(\omega_K\ll \omega_Q\). We assume that the subfamilies \(\{\mathfrak{b}^s_{\pm i}\}_{i=1}^K\), \(\{\mathfrak{b}^r_{\pm i}\}_{i=1}^K\) are very strongly dependent  in the sense that for any \(\abs{i}, \abs{j}\le K\) it holds
        \begin{equation}
            \label{eq:assbqv}
            \abs{L_{ij}(t)}\le n^{-\omega_Q}
        \end{equation}
        with very high probability for any fixed \(t\ge 0\).
    \end{enumerate}
\end{assumption}

Furthermore we assume that the initial data \(\{s_{\pm i}\}_{i=1}^n\) is regular in the following sense (cf.~\cite[Definition 3.1]{MR3916329},~\cite[Definition 2.1]{MR3914908}, motivated by~\cite[Definition 2.1]{MR3687212}).

\begin{definition}[\((g,G)\)-regular points]\label{eq:defregpro}
    Fix a very small \(\nu>0\), and choose \(g\) and \(G\) such that
    \[
        n^{-1+\nu}\le g\le n^{-2\nu}, \qquad G\le n^{-\nu}.
    \]
    A set of \(2n\)-points \({\bm s}=\{s_i\}_{i=1}^{2n}\) on \(\R \) is called \((g,G)\)-\emph{regular} if there exist constants \(c_\nu,C_\nu>0\) such that
    \begin{equation}
        \label{eq:upbv}
        c_\nu \le \frac{1}{2n}\Im \sum_{i=-n}^n \frac{1}{s_i-(E+\ii \eta)}\le C_\nu,
    \end{equation}
    for any \(\abs{E}\le G\), \(\eta \in [g, 10]\), and if there is a constant \(C_s\) large enough such that \(\norm{ {\bm s}}_\infty\le n^{C_s}\). Moreover, \(c_\nu,C_\nu\sim 1\) if \(\eta \in [g, n^{-2\nu}]\) and \(c_\nu\ge n^{-100\nu}\), \(C_\nu\le n^{100\nu}\) if \(\eta\in [n^{-2\nu},10]\).
\end{definition}
\begin{remark}
    We point out that in~\cite[Definition 3.1]{MR3916329} and~\cite[Definition 2.1]{MR3914908} the constants \(c_\nu, C_\nu\) do not depend on \(\nu>0\), but this change does not play any role since \(\nu\) will always be the smallest exponent of scale involved in the analysis of the DBMs~\eqref{eq:lambdapr}--\eqref{eq:mupr}, hence negligible.
\end{remark}

Let \(\rho_{\mathrm{fc},t}(E)\) be the deterministic approximation of the density of the particles \(\{s_{\pm i}(t)\}_{i=1}^n\) that is obtained from the semicircular flow acting on the empirical density of the initial data \(\{s_{\pm i}(0)\}_{i=1}^n\), see~\cite[Eq.~(2.5)--(2.6)]{MR3914908}.
Recall that \(\rho_{sc}(E)\) denotes the semicircular density.

\begin{proposition}\label{pro:ciala}
    Let the processes \({\bm s}(t)=\{s_{\pm i}(t)\}_{i=1}^n\), \({\bm r}(t)=\{r_{\pm i}(t)\}_{i=1}^n\) be the solutions of~\eqref{eq:lambdapr} and~\eqref{eq:mupr}, respectively, and assume that the driving Brownian motions in~\eqref{eq:lambdapr}--\eqref{eq:mupr} satisfy Assumption~\ref{ass:close}. Additionally, assume that \({\bm s}(0)\) is \((g,G)\)-regular in the sense of Definition~\ref{eq:defregpro} and that \({\bm r}(0)\) are
    the singular values of a Ginibre matrix. Then for any small \(\nu,\omega_f>0\) such that  \(\nu\ll \omega_K\ll \omega_f\ll \omega_Q\) and that \(g n^\nu\le t_f\le n^{-\nu}G^2\), there exist \(\omega,\widehat{\omega}>0\) with \(\nu\ll\widehat{\omega}\ll \omega\ll \omega_f\), and such that it holds
    \begin{equation}
        \label{eq:hihihi}
        \abs*{ \rho_{\mathrm{fc},t_f}(0) s_i(t_f)- \rho_{\mathrm{sc}}(0) r_i(t_f)}\le n^{-1-\omega}, \qquad \abs{i}\le n^{\widehat{\omega}},
    \end{equation}
    with very high probability, where \(t_f:= n^{-1+\omega_f}\).
\end{proposition}

The proof of Proposition~\ref{pro:ciala} is postponed to Section~\ref{sec:proofFEU}.

\begin{remark}\label{ren:res}
    Note that, without loss of generality, it is enough to prove Proposition~\ref{pro:ciala} only for the case \(\rho_{\mathrm{fc},t_f}(0)=\rho_{sc}(0)\), since we can always rescale the time: we may define \(\widetilde{s}_i:= (\rho_{\mathrm{fc},t_f}(0)s_i/\rho_{sc}(0))\) and notice that \(\widetilde{s}_i(t)\) is a solution of the DBM~\eqref{eq:lambdapr} after rescaling as \(t'=(\rho_{\mathrm{fc},t_f}(0)/\rho_{sc}(0))^2t\).
\end{remark}

\subsection{Proof of Lemma~\ref{lem:firststepmason} and  Lemma~\ref{lem:secondstepmason}}\label{sec:noncelpiu}

In this section we prove that by Lemma~\ref{lem:overb} and Proposition~\ref{pro:ciala} Lemmas~\ref{lem:firststepmason}--\ref{lem:secondstepmason} follow.

\subsubsection{Application of Proposition~\ref{pro:ciala} to \({\bm{\lambda}}^{z_l}(t)\) and \(\tilde{\bm{\lambda}}^{(l)}(t)\)}\label{sec:lamtillam}

In this section we prove that for any fixed \(l\) the processes \({\bm \lambda}^{z_l}(t)\) and \(\widetilde{\bm \lambda}^{(l)}(t)\) satisfy Assumption~\ref{ass:close}, Definition~\ref{eq:defregpro} and so that by Proposition~\ref{pro:ciala} we conclude the lemma.

\begin{proof}[Proof of Lemma~\ref{lem:firststepmason}]
    For any fix \(l\in[p]\), by the definition of the driving Brownian motions of the processes~\eqref{eq:DBMeA} and~\eqref{eq:nuproc} it is clear that they satisfy Assumption~\ref{ass:close} choosing \({\bm s}(t)={\bm \lambda}^{z_l}(t)\), \({\bm r}(t)=\widetilde{\bm \lambda}^{(l)}(t)\), and \(K=n^{\omega_A}\), since \(L_{ij}(t)\equiv 0\) for \(\abs{i}, \abs{j}\le K\).

    We now show that the set of points \(\{\lambda_{\pm i}^{z_l}\}_{i=1}^n\), rescaled by \(\rho^{z_l}(0)/\rho_{sc}(0)\), is \((g,G)\)-\emph{regular} for
    \begin{equation}
        \label{eq:defgG}
        g=n^{-1+\omega_h}\delta_l^{-100}, \qquad G=n^{-\omega_h}\delta_l^{10}, \qquad \nu=\omega_h.
    \end{equation}
    with \(\delta_l:= 1-\abs{z_l}^2\), for any \(l\in [p]\). By the local law~\eqref{eq:lll}, together with the regularity properties of \(m^{z_l}\) which follow by~\eqref{eq:relwig}, namely that \(m^{z_l}\) is \(1/3\)-H\"older continuous, we conclude that there exist constants \(c_{\omega_h},C_{\omega_h}>0\) such that
    \begin{equation}
        \label{eq:checkass}
        c_{\omega_h}\le \Im \frac{1}{2n} \sum_{i=-n}^n \frac{1}{[\rho^{z_l}(0)\lambda^{z_l}_i/\rho_{sc}(0)]-(E+\ii \eta)}\le C_{\omega_h},
    \end{equation}
    for any \(\abs{E}\le n^{-\omega_h}\delta_l^{10}\), \(n^{-1}\delta_l^{-100}\le \eta \le 10\). In particular, \(c_{\omega_h},C_{\omega_h}\sim 1\) for \(\eta\in [g, n^{-2\omega_h}]\), and \(c_{\omega_h}\gtrsim n^{-100\omega_h}\), \(C_{\omega_h}\lesssim n^{100\omega_h}\) for \(\eta\in [n^{-2\omega_h},10]\). This implies that the set \({\bm \lambda}^{z_l}=\{\lambda_{\pm i}^{z_l}\}_{i=1}^n\) satisfies Definition~\ref{eq:defregpro} and it concludes the proof of this lemma.
\end{proof}

\subsubsection{Application of Proposition~\ref{pro:ciala} to \({\bm \mu}^{(l)}(t)\) and \(\tilde{\bm \mu}^{(l)}(t)\)}\label{sec:mutilmu}

We now prove that for any fixed \(l\) the processes \({\bm \mu}^{(l)}(t)\) and \(\widetilde{\bm \mu}^{(l)}(t)\) satisfy Assumption~\ref{ass:close}, Definition~\ref{eq:defregpro} and so that by Proposition~\ref{pro:ciala} we conclude the lemma.

\begin{proof}[Proof of Lemma~\ref{lem:secondstepmason}]
    For any fixed \(l\in [p]\), we will apply  Proposition~\ref{pro:ciala}
    with the choice \({\bm s}(t)={\bm \mu}^{(l)}(t)\), \({\bm r}(t)=\widetilde{\bm \mu}^{(l)}(t)\) and \(K=n^{\omega_A}\).
    Since the initial data \(s_i(0)=\mu_i^{(l)}(0)\) are the singular values of a Ginibre matrix \(X^{(l)}\), it is clear that the assumption in Definition~\ref{eq:defregpro} holds choosing \(g= n^{-1+\delta}\) and \(G=n^{-\delta}\), and \(\nu=0\), for any small \(\delta>0\) (see e.g.\ the local law in~\eqref{eq:lll}).

    We now check  Assumption~\ref{ass:close}. By the definition of the families of i.i.d.\ Brownian motions
    \begin{equation}
        \label{eq:twofam}
        \left( \{\zeta_{\pm i}^{z_l}\}_{i=1}^{n^{\omega_A}}, \{\widetilde{\zeta}_{\pm i}^{(l)}\}_{i=n^{\omega_A}+1}^n\right)_{l=1}^p, \qquad \left(\{\beta_{\pm i}^{(l)}\}_{i=1}^n\right)_{l=1}^p,
    \end{equation}
    defined in~\eqref{eq:nuproc2} and~\eqref{eq:ginev}, respectively, it immediately follows that they satisfy~\ref{close1} and~\ref{close2} of Assumption~\ref{ass:close}, since
    \(\{\widetilde{\zeta}^{(l)}_{\pm i}\}_{i=n^{\omega_A}+1}^n\) are independent of \( \{\beta_{\pm i}^{(l)}\}_{i=1}^n\) as well as \( \{\beta_{\pm i}^{(l)}\}_{i=n^{\omega_A}+1}^n\)  are independent of  \(\{\widetilde{\zeta}^{(l)}_{\pm i}\}_{i=1}^n\) by construction. Recall that \(\mathcal{F}_{\beta,t}\) denotes the common filtration of all the Brownian motions \({\bm \beta}^{(m)}=\{ \beta_i^{(m)}\}_{i=1}^n\), \(m\in[p]\).

    Finally, we prove that also~\ref{close3} of Assumption~\ref{ass:close} is satisfied.  We recall the
    relations \(i=\mathfrak{i}+(l-1)n^{\omega_A}\) and \(j=\mathfrak{j}+(l-1)n^{\omega_A}\) from~\eqref{eq:frakind}
    which, for any fixed \(l\), establish a one to one relation between a pair \(\mathfrak{i}, \mathfrak{j}\in [n^{\omega_B}]\)
    and a pair \(i,j\) with  \((l-1)n^{\omega_A}+1\le i,j \le l n^{\omega_A}\).
    By the definition of \(\{\zeta_{\pm i}^{z_l}\}_{i=1}^{n^{\omega_A}}\) it follows that
    \begin{equation}
        \label{eq:qvexw}
        \dif \zeta^{z_l}_{\mathfrak{i}}-\dif \beta_{\mathfrak{i}}^{(l)}=\sum_{m=1}^{p n^{\omega_A}} \left(\sqrt{ C^\#(t)}-I\right)_{im} \dif (\underline{\beta})_m, \qquad 1\le \mathfrak{i} \le n^{\omega_A},
    \end{equation}
    with \(\underline{\beta}\) defined in~\eqref{eq:vecla}, and so that for any \(1\le \mathfrak{i}, \mathfrak{j} \le n^{\omega_A}\)
    and fixed \(l\) we have
    \[
        \begin{split}
            &\Exp*{\bigl(\dif \zeta^{z_l}_{\mathfrak{i}}-\dif \beta_{\mathfrak{i}}^{(l)}\bigr)\bigl(\dif \zeta^{z_l}_{\mathfrak{j}}-\dif \beta_{\mathfrak{j}}^{(l)}\bigr)\given \mathcal{F}_{\beta,t}} \\
            &\quad=\sum_{m_1,m_2=1}^{p n^{\omega_A}} \left(\sqrt{C^\#(t)}-I\right)_{im_1}\left(\sqrt{C^\#(t)}-I\right)_{jm_2} \Exp*{\dif (\underline{\beta})_{m_1} \dif (\underline{\beta})_{m_2}\given \mathcal{F}_{\beta,t}} \\
            &\quad = \left[\left(\sqrt{C^\#(t)}-I\right)^2\right]_{ij}\dif t,
        \end{split}
    \]
    since \(\sqrt{C^\#(t)}\) is real symmetric. Hence, \(L_{ij}(t)\) defined in~\eqref{eq:defL} in this case is given by
    \[
        L_{ij}(t)= \left[\left(\sqrt{C^\#(t)}-I\right)^2\right]_{ij}.
    \]
    Then, by Cauchy-Schwarz inequality, we have that
    \begin{equation}
        \label{eq:imperrest}
        \begin{split}
            \abs{L_{ij}(t)}&%
            \le \left[\left(\sqrt{C^\#(t)}-I\right)^2\right]^{1/2}_{ii} \left[\left(\sqrt{C^\#(t)}-I\right)^2\right]^{1/2}_{jj}   \\ & \le \Tr  \left[(\sqrt{C^\#(t)}-I)^2 \right]   \le \Tr  \left[(C^\#(t)-I)^2 \right]
            \lesssim \frac{p^2n^{2\omega_A}}{n^{4\omega_E}},
        \end{split}
    \end{equation}
    with very high probability,
    where in the last inequality we used that \(C^\#(t)\) and \(C(t)\) have the same distribution
    and the bound~\eqref{eq:bbev} of Lemma~\ref{lem:overb} holds for \(C(t)\) hence for \(C^\#(t)\) as well.
    This implies that for any fixed \(l\in [p]\) the two families of Brownian motions \(\{\beta^{(l)}_{\pm i}\}_{i=1}^n\)
    and \(( \{\zeta_{\pm i}^{z_l}\}_{i=1}^{n^{\omega_A}}, \{\widetilde{\zeta}_{\pm i}^{(l)}\}_{i=n^{\omega_A}+1}^n)\) satisfy Assumption~\ref{ass:close} with \(K=n^{\omega_A}\) and \(\omega_Q=4\omega_E-2\omega_A\). Applying Proposition~\ref{pro:ciala} this concludes the proof of Lemma~\ref{lem:secondstepmason}.
\end{proof}

\subsection{Proof of Proposition~\ref{pro:ciala}}\label{sec:proofFEU}

We divide the proof of Proposition~\ref{pro:ciala} into four  sub-sections. In Section~\ref{sec:DIP} we introduce an interpolating process \({\bm x}(t,\alpha)\) between the processes \({\bm s}(t)\) and \({\bm r}(t)\) defined in~\eqref{eq:lambdapr}--\eqref{eq:mupr}, and in Section~\ref{sec:PR} we introduce a measure which approximates the particles \({\bm x}(t,\alpha)\) and  prove their rigidity. In Section~\ref{sec:SR} we introduce a cut-off near zero (this scale will be denoted by \(\omega_A\) later) such that we only couple the dynamics of the particles \(\abs{i}\le n^{\omega_A}\), as defined in~\ref{close3} of Assumption~\ref{ass:close}, i.e.\ we will choose \(\omega_A=\omega_K\). Additionally, we also localise the dynamics on a scale \(\omega_l\) (see Section~\ref{rem:s}) since the main contribution to the dynamics comes from the nearby particles. We will refer to the new process \(\widehat{\bm x}(t,\alpha)\) (see~\eqref{eq:intflowshortA} later) as the \emph{short range approximation} of the process \({\bm x}(t,\alpha)\). Finally, in Section~\ref{sec:endsec} we conclude the proof of Proposition~\ref{pro:ciala}.

Large parts of our proof closely follow~\cite{MR3916329, MR3914908} and for brevity we will focus on the differences. We use~\cite{MR3916329, MR3914908} as our main references since the \(2\times 2\) block matrix setup of~\cite{MR3916329} is very close to the current one and~\cite{MR3916329} itself closely follows~\cite{MR3914908}. However, we
point out that many key ideas of this technique have been introduced in earlier papers on universality; e.g.\ short range cut-off and finite speed of propagation in~\cite{MR3372074, MR3606475}, coupling and homogenisation in~\cite{MR3541852}; for more historical references, see~\cite{MR3914908}. The main novelty of~\cite{MR3914908} itself is a mesoscopic analysis of the fundamental solution
\(p_t(x,y)\) of~\eqref{eq:conteqA} which enables the authors to prove short time universality for general deterministic
initial data. They also proved the result with very high probability unlike~\cite{MR3541852} that relied on level repulsion estimates.
We also mention a related but different more recent technique to prove universality~\cite{1812.10376}, which has been recently adapted to the singular values setup, or equivalently to the \(2\times 2\) block matrix structure, in~\cite{1912.05473}.

\subsubsection{Definition of the interpolated process}\label{sec:DIP}
For \(\alpha\in[0,1]\) we introduce the continuous interpolation process \({\bm x}(t,\alpha)\), between the processes \({\bm s}(t)\) and \({\bm r}(t)\) in~\eqref{eq:lambdapr}--\eqref{eq:mupr}, defined as the solution of the flow
\begin{equation}\label{eq:intflowA}
    \dif x_i(t,\alpha)= \alpha \frac{\dif \mathfrak{b}_i^s}{\sqrt{2n}}+(1-\alpha)\frac{\dif \mathfrak{b}_i^r}{\sqrt{2n}}+\frac{1}{2n}\sum_{j\ne i} \frac{1}{x_i(t,\alpha)-x_j(t,\alpha)}\dif t,
\end{equation}
with initial data
\begin{equation}\label{eq:indatA}
    {\bm x}(0,\alpha)=\alpha{\bm s}(t_0)+(1-\alpha){\bm r}(t_0),
\end{equation}
with some \(t_0\) that is a slightly smaller than \(t_f\). In fact  we will write \(t_0+t_1= t_f\) with \(t_1\ll t_f\), where \(t_1\) is the
time scale for the equilibration of the DBM with initial condition~\eqref{eq:indatA} (see~\eqref{eq:oldt1}).
To make our notation consistent with~\cite{MR3916329, MR3914908} in the remainder of this section we assume that
\(t_0=n^{-1+\omega_0}\), for some small \(\omega_0>0\), such that \(\omega_K\ll \omega_0\ll\omega_Q\).  The reader can think of \(\omega_0=\omega_f\).
Note that the strong solution of~\eqref{eq:intflowA} is well defined since the variance of its driving Brownian motion
is smaller than \(\frac{1}{2n}(1-2\alpha(1-\alpha) n^{-\omega_Q})\) by~\eqref{eq:assbqv},
which is below the critical variance  for well-posedness of the DBM since we are
in the complex symmetry class (see e.g.~\cite[Lemma 4.3.3]{MR2760897}).

By~\eqref{eq:intflowA} it clearly follows that \({\bm x}(t,0)={\bm r}(t+t_0)\) and \({\bm x}(t,1)={\bm s}(t+t_0)\), for any \(t\ge 0\).  Note that the process~\eqref{eq:intflowA} is almost the same as~\cite[Eq.~(3.13)]{MR3914908},~\cite[Eq.~(3.13)]{MR3916329}, except for the stochastic term, which in our case depends on \(\alpha\). Also, to make the notation clearer, we remark that in~\cite{MR3916329, MR3914908} the interpolating process is denoted by \({\bm z}(t,\alpha)\). We  changed this notation to  \({\bm x}(t,\alpha)\) to avoid confusions with the \(z_l\)-parameters introduced in the previous sections where we apply Proposition~\ref{pro:ciala} to the processes defined in Section~\ref{sec:COMPPRO}.

\begin{remark}\label{rem:ne}
    Even if all processes \({\bm \lambda}(t)\), \(\widetilde{\bm \lambda}(t)\), \(\widetilde{\bm \mu}(t)\), \({\bm \mu}(t)\) introduced in Section~\ref{sec:COMPPRO} already satisfy~\cite[Lemma 3.3-3.5]{MR3916329},~\cite[Lemma 3.3-3.5]{MR3914908} as a consequence of the local law~\eqref{eq:lll} and the rigidity estimates~\eqref{eq:rigneed}, we decided to present the proof of Proposition~\ref{pro:ciala} for general deterministic initial data \({\bf s}(0)\) satisfying Definition~\ref{eq:defregpro} (see Remark~\ref{rem:t0diff}). Hence, an additional time \(t_0\) is needed to ensure the validity of~\cite[Lemma 3.3-3.5]{MR3916329},~\cite[Lemma 3.3-3.5]{MR3914908}. More precisely, we first let the DBMs~\eqref{eq:lambdapr}--\eqref{eq:mupr} evolve for a time \(t_0:= n^{-1+\omega_0}\), and then we consider the process~\eqref{eq:intflowA} whose initial data in~\eqref{eq:indatA} is given by a linear interpolation of the solutions of~\eqref{eq:lambdapr}--\eqref{eq:mupr} at time \(t_0\).
\end{remark}

Before proceeding with the analysis of~\eqref{eq:intflowA} we give some definitions and state some preliminary results necessary for its analysis.

\subsubsection{Interpolating measures and particle rigidity}\label{sec:PR}
Using the convention of~\cite[Eq.~(3.10)--(3.11)]{MR3916329}, given a probability measure \(\dif \rho(E)\),  we define the \(2n\)-quantiles \(\gamma_i\) by
\begin{equation}\label{eq:quantA}
    \begin{split}
        \gamma_i &:= \inf\set*{x\given \int_{-\infty}^x \dif \rho(E)\ge \frac{n+i-1}{2n} }, \quad 1\le i \le n, \\
        \gamma_i &:= \inf\set*{x\given\int_{-\infty}^x \dif \rho(E)\ge \frac{n+i}{2n} }, \qquad -n\le i \le -1,
    \end{split}
\end{equation}
Note that \(\gamma_1=0\) if \(\dif \rho(E)\) is symmetric with respect to \(0\).

Let \(\rho_{fc,t}(E)\) be defined above Proposition~\ref{pro:ciala} (see e.g.~\cite[Eq.~(2.5)--(2.6)]{MR3914908} for more details), and let \(\rho_{sc}(E)\) denote the semicircular density, then by \(\gamma_i(t)\), \(\gamma_i^{sc}\) we denote the \(2n\)-quantiles, defined as in~\eqref{eq:quantA}, of \(\rho_{fc,t}\) and \(\rho_{sc}\), respectively.

Following the construction of~\cite[Lemma 3.3-3.4, Appendix A]{MR3914908},~\cite[Section 3.2.1]{MR3916329}, we define the interpolating (random) measure \(\dif \rho(E,t,\alpha)\) for any \(\alpha\in [0,1]\). More precisely, the measure \(\dif \rho(E,t,\alpha)\) is deterministic close to zero, and it consists of delta functions of the position of the particles \(x_i(t,\alpha)\) away from zero.

Denote by \(\gamma_i(t,\alpha)\) the quantiles of \(\dif\rho(E,\alpha,t)\), and by \(m(w,t,\alpha)\), with \(w\in\HC \), its Stieltjes transform. Fix \(q_*\in (0,1)\) throughout this section, and let \(k_0=k_0(q_*)\in\mathbf{N}\) be the largest index such that
\begin{equation}\label{eq:ko}
    \abs{\gamma_{\pm k_0}(t_0)}, \abs{\gamma_{\pm k_0}^{sc}}\le q_*G,
\end{equation}
with \(G\) defined in~\eqref{eq:defgG}, then the measure \(\dif \rho(E,t,\alpha)\) has a deterministic density (denoted by \(\rho(E,\alpha,t)\) with a slight abuse of notation) on the interval
\begin{equation}\label{eq:ga}
    \mathcal{G}_\alpha:= [\alpha\gamma_{-k_0}(t_0)+(1-\alpha)\gamma_{-k_0}^{sc}, \alpha\gamma_{k_0}(t_0)+(1-\alpha)\gamma_{k_0}^{sc}].
\end{equation}
Outside \(\mathcal{G}_\alpha\) the measure \(\dif \rho(E,t,\alpha)\) consists of \(1/(2n)\) times delta functions of the particle locations \(\delta_{x_i(t,\alpha)}\).

\begin{remark}
    By the construction \(\dif\rho(E,t,\alpha)\) as in~\cite[Lemma 3.3-3.4, Appendix A]{MR3914908},~\cite[Section 3.2.1]{MR3916329} all the regularity properties of \(\dif\rho(E,\alpha,t)\), its quantiles \(\gamma_i(t,\alpha)\), and its Stieltjes transform \(m(E+\ii \eta,t,\alpha)\) in~\cite[Lemma 3.3-3.4]{MR3914908},~\cite[Lemma 3.3-3.4]{MR3916329} hold without any change. In particular, it follows that
    \begin{equation}
        \abs{\gamma_i(t,\alpha)-\gamma_j(t,\alpha)}\sim \frac{\abs{i-j}}{n}, \qquad \abs{i},\abs{j}\le q_*G,
    \end{equation}
    with \(q_*\) defined above~\eqref{eq:ko}, and \(G\) in~\eqref{eq:defgG}.
\end{remark}

Define the Stieltjes transform of the empirical measure of the particle configuration \(\{x_{\pm i}(t,\alpha)\}_{i=1}^n\) by
\begin{equation}\label{eq:empmA}
    m_n(w,t,\alpha):= \frac{1}{2n}\sum_{i=-n}^n\frac{1}{x_i(t,\alpha)-w}, \quad w\in\HC .
\end{equation}
We recall that the summation does not include the term \(i=0\) (see Remark~\ref{rem:no0}). The local law for $m_n(w,t,\alpha)$, hence rigidity for the interpolated particles \eqref{eq:intflowA}, easily follows similarly to \cite[Section 3.2]{MR4009708}; the only minor difference is that the driving martingales in \eqref{eq:intflowA} are not independent for different indices (see also \cite[Lemma 7.12]{MR4235475} when a similar proof has been presented with more details).

\begin{lemma}\label{lem:local}
    Fix \(q\in (0,1)\) and  \(\tilde{\epsilon}>0\). Define \(\widehat{C}_q:= \{j:\abs{j}\le qk_0\}\), with \(k_0\) defined in~\eqref{eq:ko}. Then for any \(\xi>0\), with very high probability we have the optimal rigidity
    \begin{equation}
        \label{eq:rigA}
        \sup_{0\le t\le t_0 n^{-\tilde{\epsilon}}}\sup_{i\in \widehat{C}_q}\sup_{0\le\alpha\le 1}\abs{x_i(t,\alpha)-\gamma_i(t,\alpha)}\le \frac{n^{\xi+100\nu}}{n},
    \end{equation}
    and the local law
    \begin{equation}
        \label{eq:trllA}
        \sup_{n^{-1+\tilde{\epsilon}}\le \eta\le 10} \sup_{0\le t\le t_0 n^{-\tilde{\epsilon}}}\sup_{0\le \alpha\le 1} \sup_{E\in q\mathcal{G}_\alpha}\abs{m_n(E+\ii\eta,t,\alpha)-m(E+\ii\eta,t,\alpha)}\le \frac{n^{\xi+100\nu}}{n\eta},
    \end{equation}
    for sufficiently large \(n\), with \(\nu>0\) from  Definition~\ref{eq:defregpro}.
\end{lemma}

Without loss of generality in Lemma~\ref{lem:local} we assumed \(k_1=k_0\) in~\cite[Eq.~(3.25)--(3.26)]{MR3916329}.

\subsubsection{Short range analysis}\label{sec:SR}

In the following of this section we perform a local analysis of~\eqref{eq:intflowA} adapting the analysis of~\cite{MR3916329, MR3914908} and explaining the minor changes needed for the analysis of the flow~\eqref{eq:intflowA}, for which the driving Brownian motions \(\mathfrak{\bm b}^s\), \(\mathfrak{\bm b}^r\) satisfy Assumption~\ref{ass:close}, compared to the analysis of~\cite[Eq.~(3.13)]{MR3916329},~\cite[Eq.~(3.13)]{MR3914908}. More precisely, we run the DBM~\eqref{eq:intflowA} for a time
\begin{equation}\label{eq:oldt1}
    t_1:= \frac{n^{\omega_1}}{n},
\end{equation}
for any \(\omega_1>0\) such that \(\nu\ll \omega_1\ll \omega_K\), with \(\nu,\omega_K\) defined in Definition~\ref{eq:defregpro} and above Assumption~\ref{ass:close}, respectively, so that~\eqref{eq:intflowA} reaches its local equilibrium (see Section~\ref{rem:s} for a summary on the different scales). Moreover, since the dynamics of \(x_i(t,\alpha)\) is mostly influenced by the particles close to it, in the following we define a short range approximation of the process \({\bm x}(t,\alpha)\) (see~\eqref{eq:intflowshortA} later), denoted by \(\widehat{\bm x}(t,\alpha)\), and use the homogenisation theory developed in~\cite{MR3914908}, adapted in~\cite{MR3916329} for the singular values flow, for the short range kernel.

\begin{remark}
    We do not need to define the shifted process \(\widetilde{\bm x}(t,\alpha)\) as in~\cite[Eq.~(3.29)--(3.32)]{MR3916329} and~\cite[Eq.~(3.36)--(3.40)]{MR3914908}, since in our case the measure \(\dif \rho(E,t,\alpha)\) is symmetric with respect to \(0\) by assumption, hence, using the notation in~\cite[Eq.~(3.29)--(3.32)]{MR3916329}, we have \(\widetilde{\bm x}(t,\alpha)={\bm x}(t,\alpha)-\gamma_1(t,\alpha)={\bm x}(t,\alpha)\). Hence, from now on we only use \({\bm x}(t,\alpha)\) and the reader can think \(\widetilde{\bm x}(t,\alpha)\equiv {\bm x}(t,\alpha)\) for a direct analogy with~\cite{MR3916329, MR3914908}.
\end{remark}

Our analysis will be completely local, hence we introduce a short range cut-off. Fix \(\omega_l, \omega_A>0\) so that
\begin{equation}\label{eq:relscalA}
    0< \omega_1 \ll \omega_l\ll \omega_A\ll \omega_0\ll \omega_Q,
\end{equation}
with \(\omega_1\) defined in~\eqref{eq:oldt1}, \(\omega_0\) defined below~\eqref{eq:indatA}, and \(\omega_Q\) in~\ref{close3} of Assumption~\ref{ass:close}.
Moreover, we assume that \(\omega_A\) is such that
\begin{equation}\label{eq:choosekn}
    K_n=n^{\omega_A},
\end{equation}
with \(K_n=n^{\omega_K}\) in Assumption~\ref{ass:close}, i.e.\ \(\omega_A=\omega_K\). We remark that it is enough to choose \(\omega_A\ll\omega_K\), but to avoid further splitting in~\eqref{eq:intflowshortA} we assumed \(\omega_K=\omega_A\).

For any \(q\in(0,1)\), define the set
\begin{equation}\label{eq:shortscsetA}
    A_q:= \set*{(i,j)\given \abs{i-j}\le n^{\omega_l} \; \text{or} \; ij>0, i\notin \widehat{C}_q, j\notin \widehat{C}_q},
\end{equation}
and denote \(A_{q,(i)}:= \set{j\given(i,j)\in A_q }\). In the remainder of this section we will often use the notations
\[
    \sum_j^{A_{q,(i)}}:=  \sum_{j\in A_{q,(i)}}, \qquad \sum_j^{A_{q,(i)}^c}:=  \sum_{j\notin A_{q,(i)}}.
\]

Let \(q_*\in (0,1)\) be defined above~\eqref{eq:ko}, then we define the short range process \(\widehat{\bm x}(t,\alpha)\) (cf.~\cite[Eq.~(3.35)--(3.36)]{MR3916329},~\cite[Eq.~(3.45)--(3.46)]{MR3914908}) as follows
\begin{equation}\label{eq:intflowshortA}
    \begin{split}
        \dif \widehat{x}_i(t,\alpha)&= \frac{1}{2n}\sum_j^{A_{q_*,(i)}} \frac{1}{\widehat{x}_i(t,\alpha)-\widehat{x}_j(t,\alpha)} \dif t \\
        &\quad +\begin{cases}
            \alpha \frac{\dif \mathfrak{b}^s}{\sqrt{2n}}+(1-\alpha)\frac{\dif \mathfrak{b}^r}{\sqrt{2n}}                      & \text{if} \quad \abs{i}\le n^{\omega_A},   \\
            \alpha \frac{\dif \mathfrak{b}^s}{\sqrt{2n}}+(1-\alpha)\frac{\dif \mathfrak{b}^r}{\sqrt{2n}}+J_i(\alpha,t) \dif t & \text{if} \quad n^{\omega_A}<\abs{i}\le n,
        \end{cases}
    \end{split}
\end{equation}
where
\begin{equation}\label{eq:defJA}
    J_i(\alpha,t):= \frac{1}{2n}\sum_j^{A_{q_*,(i)}^c} \frac{1}{x_i(t,\alpha)-x_j(t,\alpha)},
\end{equation}
and initial data \(\widehat{\bm x}(0,\alpha)={\bm x}(0,\alpha)\). Note that
\begin{equation}\label{eq:boundJA}
    \sup_{0\le t\le t_1}\sup_{0\le\alpha\le 1}\abs{J_1(\alpha,t)}\le \log n,
\end{equation}
with very high probability.

\begin{remark}
    Note that the SDE defined in~\eqref{eq:intflowshortA}  has the same form as in~\cite[Eq.~(3.70)]{MR3914908}, with \(F_i=0\) in our case, except for the stochastic term in~\eqref{eq:intflowshortA} that looks slightly different, in particular it depends on \(\alpha\). Nevertheless, by Assumption~\ref{ass:close}, the quadratic variation of the driving Brownian motions in~\eqref{eq:intflowshortA} is also bounded by one uniformly in \(\alpha\in[0,1]\). Moreover, the process defined in~\eqref{eq:intflowshortA} and the measure \(\dif \rho(E,t,\alpha)\) satisfy~\cite[Eq.~(3.71)--(3.77)]{MR3914908}.
\end{remark}

Since when we consider the difference process \(\widehat{\bm x}(t,\alpha)-{\bm x}(t,\alpha)\) the stochastic differential disappears, by~\cite[Lemma 3.8]{MR3914908}, without any modification, it follows that
\begin{equation}\label{eq:shortlongA}
    \sup_{0\le t\le t_1}\sup_{0\le \alpha\le 1}\sup_{\abs{i}\le n}\abs{\widehat{x}_i(t,\alpha)-x_i(t,\alpha)}\le n^{\xi+100\nu} t_1\left( \frac{1}{n^{\omega_l}}+\frac{n^{\omega_A}}{n^{\omega_0}}+\frac{1}{\sqrt{nG}}\right),
\end{equation}
for any \(\xi>0\) with very high probability, with \(G\) defined in~\eqref{eq:defgG}. In particular,~\eqref{eq:shortlongA} implies that the short range process \(\widehat{\bm x}(t,\alpha)\), defined in~\eqref{eq:intflowshortA}, approximates very well (i.e.\ they are closer than the fluctuation scale) the process \({\bm x}(t,\alpha)\) defined in~\eqref{eq:intflowA}.

Next, in order to use the smallness of~\eqref{eq:defL}--\eqref{eq:assbqv} in Assumption~\ref{ass:close} for \(\abs{i}\le n^{\omega_A}\), we define \({\bm u}(t,\alpha):= \partial_\alpha \widehat{\bm x}(t,\alpha)\), which is the solution of the following discrete SPDE (cf.~\cite[Eq.~(3.38)]{MR3916329},~\cite[Eq.~(3.63)]{MR3914908}):
\begin{equation}\label{eq:pareqA}
    \dif {\bm u}=\sum_j^{A_{q_*,(i)}} B_{ij}(u_j-u_i) \dif t+\dif {\bm \xi}_1+{\bm\xi}_2 \dif t=-\mathcal{B}{\bm u} \dif t+\dif {\bm \xi}_1+{\bm\xi}_2 \dif t,
\end{equation}
where
\begin{equation}\label{eq:shortrankerA}
    \begin{split}
        B_{ij}&:= \frac{\bm1_{j\ne \pm i}}{2n(\widehat{x}_i-\widehat{x}_j)^2}, \quad \dif \xi_{1,i}:=  \frac{\dif \mathfrak{b}_i^s}{\sqrt{2n}}-\frac{\dif \mathfrak{b}_i^r}{\sqrt{2n}} \\
        \xi_{2,i}&:=  \begin{cases}
            0                             & \text{if} \quad \abs{i}\le n^{\omega_A},   \\
            \partial_\alpha J_i(\alpha,t) & \text{if} \quad n^{\omega_A}<\abs{i}\le n,
        \end{cases}
    \end{split}
\end{equation}
with \(J_i(\alpha,t)\) defined in~\eqref{eq:defJA}. We remark that the operator\footnote{The operator \(\cB\) defined here is not to be confused with the completely unrelated one in~\eqref{cB def}.} \(\cB\) defined via the kernel in~\eqref{eq:shortrankerA} depends on \(\alpha\) and \(t\). It is not hard to see (e.g.\ see~\cite[Eq.~(3.65), Eq.~(3.68)--(3.69)]{MR3914908}) that the forcing term \({\bm \xi}_2\) is bounded with very high probability by \(n^C\), for some \(C>0\), for \(n^{\omega_A}<\abs{i}\le n\). Note that the only difference in~\eqref{eq:pareqA} compared to~\cite[Eq.~(3.38)]{MR3916329},~\cite[Eq.~(3.63)]{MR3914908} is the additional term \(\dif {\bm \xi}_1\) which will be negligible for our analysis.

Let \(\mathcal{U}\) be the semigroup associated to \(\mathcal{B}\), i.e.\ if \(\partial_t {\bm v}=-\mathcal{B}{\bm v}\), then for any \(0\le s\le t\) we have that
\[
    v_i(t)=\sum_{j=-n}^n \mathcal{U}_{ij}(s,t,\alpha) v_j(s), \qquad \abs{i}\le n.
\]
The first step to analyse the equation in~\eqref{eq:pareqA} is the following finite speed of propagation estimate (cf.~\cite[Lemma 3.9]{MR3916329},~\cite[Lemma 3.7]{MR3914908}).

\begin{lemma}\label{lem:finspeedA}
    Let \(0\le s\le t\le t_1\). Fix \(0<q_1<q_2<q_*\), with \(q_*\in (0,1)\) defined in~\eqref{eq:ko}, and \(\epsilon_1>0\) such that \(\epsilon_1\ll\omega_A\). Then for any \(\alpha\in[0,1]\) we have
    \begin{equation}
        \label{eq:finspestA}
        \abs{U_{ji}(s,t,\alpha)}+\abs{U_{ij}(s,t,\alpha)}\le n^{-D},
    \end{equation}
    for any \(D>0\) with very high probability, if either \(i\in \widehat{C}_{q_2}\) and \(\abs{i-j}> n^{\omega_l+\epsilon_1}\), or if \(i\notin \widehat{C}_{q_2}\) and \(j\in\widehat{C}_{q_1}\).
    \begin{proof}
        The proof of this lemma follows the same lines as~\cite[Lemma 3.7]{MR3914908}. There are only two differences that we point out. The first one is that~\cite[Eq.~(4.15)]{MR3914908}, using the notation therein, has to be replaced by
        \begin{equation}
            \label{eq:cpl}
            \sum_k v_k^2(\nu^2(\psi_k')^2+\nu\psi_k'') \Exp*{\dif C_k(\alpha,t)\dif C_k(\alpha,t)\given \mathcal{F}_t},
        \end{equation}
        where \(\mathcal{F}_t\) is the filtration defined in Assumption~\ref{ass:close}, and \(C_k(\alpha,t)\) is defined as
        \begin{equation}
            \label{eq:defcka}
            C_k(\alpha,t):=  \alpha \frac{\mathfrak{b}_k^s(t)}{\sqrt{2n}}+(1-\alpha)\frac{\mathfrak{b}_k^r(t)}{\sqrt{2n}}.
        \end{equation}
        We remark that \(\nu\) in~\eqref{eq:cpl} should not to be confused with \(\nu\) in Definition~\ref{eq:defregpro}.
        Then, by Kunita-Watanabe inequality, it is clear that
        \begin{equation}
            \label{eq:mindt}
            \Exp*{\dif C_k(\alpha,t)\dif C_k(\alpha,t)\given \mathcal{F}_t}\lesssim \frac{\dif t}{n},
        \end{equation}
        uniformly in \(\abs{k}\le n\), \(t\ge 0\), and \(\alpha\in [0,1]\). The fact that~\eqref{eq:mindt} holds is the only input needed to bound~\cite[Eq.~(4.21)]{MR3914908}.

        The second difference is that the stochastic differential \((\sqrt{2}\dif B_k)/\sqrt{n}\) in~\cite[Eq.~(4.21)]{MR3914908} has to be replaced by \(\dif C_k(\alpha,t)\) defined in~\eqref{eq:defcka}. This change is inconsequential in the bound~\cite[Eq.~(4.26)]{MR3914908}, since \(\E \dif C_k(\alpha,t)=0\).
    \end{proof}
\end{lemma}

Moreover, the result in~\cite[Lemma 3.8]{MR3916329},~\cite[Lemma 3.10]{MR3914908} hold without any change, since its proof is completely deterministic and the stochastic differential in the definition of the process \(\widehat{\bm x}(t,\alpha)\) does not play any role.

In the remainder of this section, before completing the proof of Proposition~\ref{pro:ciala}, we describe the homogenisation argument to approximate the \(t\)-dependent kernel of \(\mathcal{B}\) with a continuous kernel (denoted by \(p_t(x,y)\) below). We follow verbatim~\cite[Section 3-4]{MR3914908} and its adaptation to the singular value flow of~\cite[Section 3.4]{MR3916329}, except for the bound of the rhs.\ of~\eqref{eq:step4A}, where we handle the additional term \(\dif {\bm \xi}_1\) in~\eqref{eq:shortrankerA}.

Fix a constant \(\epsilon_B>0\) such that \(\omega_A-\epsilon_B>\omega_l\), and let \(a\in\Z \) be such that \(0<\abs{a}\le n^{\omega_A-\epsilon_B}\). Define also the equidistant points \(\gamma_j^f:= j (2n\rho_{sc}(0))^{-1}\), which approximate the quantiles \(\gamma_j(t,\alpha)\) very well for small \(j\), i.e.\ \(\abs{\gamma_j^f-\gamma_j(t,\alpha)}\lesssim n^{-1}\) for \(\abs{j}\le n^{\omega_0/2}\) (see~\cite[Eq.~(3.91)]{MR3914908}). Consider the solution of
\begin{equation}\label{eq:wsolA}
    \partial_t w_i=-(\mathcal{B}w)_i, \quad w_i(0)=2n\delta_{ia},
\end{equation}
and define the cut-off \(\eta_l:= n^{\omega_l} (2n\rho_{sc}(0))^{-1}\). Let \(p_t(x,y)\) be the fundamental solution of the equation
\begin{equation}\label{eq:conteqA}
    \partial_t f(x)=\int_{\abs{x-y}\le \eta_l}\frac{f(y)-f(x)}{(x-y)^2}\rho_{sc}(0)\dif y.
\end{equation}
The idea of the homogenisation argument is that the deterministic solution \(f\) of~\eqref{eq:conteqA} approximates very well the random solution of~\eqref{eq:wsolA}. This is formulated in terms of the solution kernels of the two equations in Proposition~\ref{prop:homA}. Following~\cite[Lemma 3.9-3.13, Corollary 3.14, Theorem 3.15-3.17]{MR3916329}, which are obtained adapting the proof of~\cite[Section 3.6]{MR3914908}, we will conclude the following proposition.

\begin{proposition}\label{prop:homA}
    Let \(a,i\in\Z \) such that \(\abs{a}\le n^{\omega_A-\epsilon_B}\) and \(\abs{i-a}\le n^{\omega_l}/10\). Fix \(\epsilon_c>0\) such that \(\omega_1-\epsilon_c>0\), let \(t_1:=  n^{-1+\omega_1}\) and \(t_2:=  n^{-\epsilon_c}t_1\), then for any \(\alpha\in [0,1]\) and for any \(\abs{u}\le t_2\) we have
    \begin{equation}
        \label{eq:homapprA}
        \abs*{\mathcal{U}_{ia}(0,t_1+u,\alpha)-\frac{p_{t_1}(\gamma_i^f,\gamma_a^f)}{n}}\le \frac{n^{100\nu+\epsilon_c}}{nt_1}\left(\frac{(nt_1)^2}{n^{\omega_l}}+\frac{1}{(nt_1)^{1/10}}+\frac{1}{n^{3\epsilon_c/2}}\right),
    \end{equation}
    with very high probability.
    \begin{proof}
        The proof of this proposition relies on~\cite[Section 3.6]{MR3914908}, which has been adapted to the \(2\times 2\) block structure in~\cite[Lemma 3.9--3.13, Corollary 3.14, Theorem 3.15--3.17]{MR3916329}. We thus present only the differences compared to\cite{MR3916329, MR3914908}; for a complete proof we defer the reader to these works.

        The only difference in the proof of this proposition compared to the proof of~\cite[Theorem 3.17]{MR3916329},~\cite[Theorem 3.11]{MR3914908} is in~\cite[Eq.~(3.121) of Lemma 3.14]{MR3914908} and~\cite[Eq.~(3.148) of Lemma 3.14]{MR3914908}. The main goal of~\cite[Lemma 3.14]{MR3914908} and~\cite[Lemma 3.14]{MR3914908} is to prove that
        \begin{equation}
            \label{eq:maindiff}
            \dif \frac{1}{2n}\sum_{1\le \abs{i}\le n}(w_i-f_i)^2=-\braket{{\bm w}(t)-{\bm f}(t),\mathcal{B}({\bm w}(t)-{\bm f}(t))}+\text{Lower order},
        \end{equation}
        where \(f_i:=f(\widehat{x}_i(t,\alpha),t)\), with \(\widehat{x}_i(t,\alpha)\) being the solution of~\eqref{eq:intflowshortA}, and \({\bm w}(t)\), \({\bm f}(t)\) being the solutions of~\eqref{eq:wsolA} and~\eqref{eq:conteqA} with \(x=\widehat{x}_i(t,\alpha)\), respectively. In order to prove~\eqref{eq:maindiff}, following~\cite[Eq.~(3.121)]{MR3914908} and using the notation therein (with \(N=2n\) and replacing \(\widehat{z}_i\) by \(\widehat{x}_i\)), we compute
        \begin{equation}
            \label{eq:replace1}
            \begin{split}
                &\dif \frac{1}{2n}\sum_{1\le \abs{i}\le n}(w_i-f_i)^2 \\
                &\,\,= \frac{1}{n}\sum_{1\le \abs{i}\le n}(w_i-f_i)\left[\partial_t w_i \dif t-(\partial_t f)(t,\widehat{x}_i)\dif t-f'(t,\widehat{x}_i)\dif \widehat{x}_i\right]\\
                &\,\,\quad +\frac{1}{n}\sum_{1\le \abs{i}\le n} \left(-(w_i-f_i)f''(t, \widehat{x}_i)+(f'(t,\widehat{x}_i))^2\right) \Exp*{\dif C_i(\alpha,t)\dif C_i(\alpha,t)\given\mathcal{F}_t},
            \end{split}
        \end{equation}
        where
        \[
            C_i(\alpha,t):=  \alpha \frac{\mathfrak{b}_i^s(t)}{\sqrt{2n}}+(1-\alpha)\frac{ \mathfrak{b}_i^r(t)}{\sqrt{2n}}.
        \]
        As a consequence of the slight difference in definition of \(\dif \widehat{x}_i\) in~\eqref{eq:intflowshortA}, compared to the definition of \(\dif \widehat{z}_i\) in~\cite[Eq. (3.70)]{MR3914908}, the martingale term in~\eqref{eq:replace1} is given by (cf.~\cite[Eq.~(3.148)]{MR3914908})
        \begin{equation}
            \label{eq:replace2}
            \dif M_t=\frac{1}{2n}\sum_{1\le \abs{i}\le n} (w_i-f_i) f_i' \dif C_i(\alpha,t).
        \end{equation}
        The terms in the first line of the rhs.\ of~\eqref{eq:replace1} are bounded exactly as in~\cite[Eq.~(3.124)--(3.146), (3.149)--(3.154)]{MR3914908}. It remains to estimate the second line in the rhs.\ of~\eqref{eq:replace1}.

        The expectation of the second line of~\eqref{eq:replace1} is bounded by a constant times \(n^{-1}\dif t\), exactly as in~\eqref{eq:mindt}. This is the only input needed to bound the terms~\eqref{eq:replace1} in~\cite[Eq. (3.122)-(3.123)]{MR3914908}. Hence, in order to conclude the proof of this proposition we are left with the term in~\eqref{eq:replace2}.

        The quadratic variation of the term in~\eqref{eq:replace2} is given by
        \[
            \dif \braket{ M}_t=\frac{1}{2n}\sum_{1\le \abs{i}, \abs{j}\le n} (w_i-f_i)(w_j-f_j) f_i' f_j' \Exp*{\dif C_i(\alpha,t)\dif C_j(\alpha,t)\given\mathcal{F}_t},
        \]
        using the notation in~\cite[Eq.~(3.155)--(3.157)]{MR3914908} is used.
        By~\ref{close2} of Assumption~\ref{ass:close} it follows that
        \begin{gather}
            \begin{aligned}
                \dif \braket{ M}_t & =\frac{1}{4n^2}\sum_{1\le \abs{i}, \abs{j}\le n^{\omega_A}} (w_i-f_i)(w_j-f_j) f_i' f_j' \Exp*{\dif C_i(\alpha,t)\dif C_j(\alpha,t)\given\mathcal{F}_t} \\
                                   & \quad +\frac{\alpha^2+(1-\alpha)^2}{8n^3} \sum_{n^{\omega_A}< \abs{i}\le n} (w_i-f_i)^2 (f_i')^2 \dif t.
            \end{aligned}\label{eq:qM1}\raisetag{-5em}
        \end{gather}

        Then, by~\ref{close3} of Assumption~\ref{ass:close}, for \(\abs{i}, \abs{j}\le n^{\omega_A}\) we have
        \begin{gather}
            \begin{aligned}
                \Exp*{\dif C_i(\alpha,t)\dif C_j(\alpha,t)\given \mathcal{F}_t} & =\bigl[\alpha^2+(1-\alpha)^2 \bigr] \frac{\delta_{ij}}{2n} \dif t                                                                                             \\
                                                                                & \quad+\frac{\alpha(1-\alpha)}{2n}\Exp*{\bigl(\dif \mathfrak{b}_i^s\dif \mathfrak{b}_j^r+\dif \mathfrak{b}_i^r\dif \mathfrak{b}_j^s\bigr)\given\mathcal{F}_t},
            \end{aligned}\label{eq:qM2}\raisetag{-4em}
        \end{gather}
        and that
        \begin{equation}
            \label{eq:qM3}
            \Exp*{ \dif \mathfrak{b}_i^s\dif \mathfrak{b}_j^r\given\mathcal{F}_t}=\Exp*{ (\dif \mathfrak{b}_i^s-\dif \mathfrak{b}_i^r)\dif \mathfrak{b}_j^r\given\mathcal{F}_t}+\delta_{ij} \dif t \lesssim (\abs{L_{ii}(t)}^{1/2}+\delta_{ij} )\dif t,
        \end{equation}
        where in the last inequality we used Kunita-Watanabe inequality.

        Combining~\eqref{eq:qM1}--\eqref{eq:qM3} we finally conclude that
        \begin{gather}
            \begin{aligned}
                \dif \braket{ M}_t & \le \frac{1}{8n^3} \sum_{1\le \abs{i}\le n} (w_i-f_i)^2 (f_i')^2 \dif t                                                                               \\
                                   & \quad + \frac{\alpha(1-\alpha)}{4n^3}\sum_{1\le \abs{i}, \abs{j}\le n^{\omega_A}} \abs{L_{ii}(t)}^{1/2} \abs*{ (w_i-f_i)(w_j-f_j) f_i' f_j' } \dif t.
            \end{aligned}\label{eq:qM4}\raisetag{-4em}
        \end{gather}
        Since \(\alpha\in [0,1]\), \(\abs{L_{ii}(t)}\le n^{-\omega_Q}\) and \(\omega_A\ll \omega_Q\) by~\eqref{eq:assbqv} and~\eqref{eq:relscalA}--\eqref{eq:choosekn}, using Cauchy-Schwarz in~\eqref{eq:qM4}, we conclude that
        \begin{equation}
            \label{eq:qM5}
            \dif \braket{ M}_t\lesssim\frac{1}{n^3} \sum_{1\le \abs{i}\le n} (w_i-f_i)^2 (f_i')^2 \dif t,
        \end{equation}
        which is exactly the lhs.\ in~\cite[Eq.~(3.155)]{MR3914908}, hence the high probability bound in~\cite[Eq.~(3.155)]{MR3914908} follows. Then the remainder of the proof of~\cite[Lemma 3.14]{MR3914908} proceeds exactly in the same way.

        Given~\eqref{eq:replace1} as an input, the proof of~\eqref{eq:homapprA} is concluded following the proof of~\cite[Theorems 3.16-3.17]{MR3914908} line by line.

    \end{proof}
\end{proposition}

\subsubsection{Proof of Proposition~\ref{pro:ciala}}\label{sec:endsec}

We conclude this section with the proof of Proposition~\ref{pro:ciala} following~\cite[Section 3.6]{MR3916329}. We remark that all the estimates above hold uniformly in \(\alpha\in[0,1]\) when bounding an integrand by~\cite[Appendix E]{MR3914908}.
\begin{proof}[Proof of Proposition~\ref{pro:ciala}]
    For any \(\abs{i}\le n\), by~\eqref{eq:shortlongA}, it follows that
    \begin{equation}
        \label{eq:step1A}
        s_i(t_0+t_1)-r_i(t_0+t_1)=x_i(t_1,1)-x_i(t_1,0)=\widehat{x}_i(t_1,1)-\widehat{x}_i(t_1,0)+\mathcal{O}\left(\frac{n^\xi t_1}{n^{\omega_l}} \right).
    \end{equation}
    We remark that in~\eqref{eq:step1A} we ignored the scaling~\eqref{eq:hihihi} since it can be removed by a simple time-rescaling (see Remark~\ref{ren:res} for more details). Then, using that \(u_i=\partial_\alpha \widehat{x}_i\) we have that
    \begin{equation}
        \label{eq:step2A}
        \widehat{x}_i(t_1,1)-\widehat{x}_i(t_1,0)=\int_0^1 u_i(t_1,\alpha)\dif \alpha.
    \end{equation}
    We recall that \(u\) is a solution of
    \begin{equation}
        \label{difeq}
        \dif {\bm u}=\mathcal{B}{\bm u} \dif t+\dif {\bm \xi}_1+{\bm \xi}_2 \dif t,
    \end{equation}
    as defined in~\eqref{eq:pareqA}--\eqref{eq:shortrankerA}, with
    \begin{equation}
        \label{eq:step3A}
        \abs{\xi_{2,i}(t)}\le \bm1_{\{\abs{i}> n^{\omega_A}\}}n^C,
    \end{equation}
    with very high probability for some constant \(C>0\) and any \(0\le t\le t_1\). The estimates to remove the forcing term ${\bm \xi}_2$ in \eqref{eq:step2A} are completely analogous to \cite[Section 3.7]{MR3916329}, and so omitted; we thus focus on the estimate of the effect
    of $\dif {\bm \xi}_1$ as if  ${\bm \xi}_2$ were not present in~\eqref{difeq}.
    We now split the estimate of this term into short and long range part.
    We start with the short range part, and then explain the relatively minor differences in the estimates for the long range part. Define \({\bm v}={\bm v}(t)\) as the solution of
    \begin{equation}
        \label{long}
        \dif {\bm v}=\mathcal{B}{\bm v}\dif t+\bm1(|\cdot| > n^{\omega_A})\dif {\bm \xi}_1, \quad {\bm v}(0)={\bm u}(0),
    \end{equation}
    and let ${\bm w}:={\bm u}-{\bm v}$. We thus get
    \begin{equation}
        \dif {\bm w}=\mathcal{B}{\bm w}\dif t+\bm1(|\cdot | \le n^{\omega_A})\dif {\bm \xi}_1, \qquad {\bm w}(0)=0.
    \end{equation}
    To show that $w_i(t)\le n^{-1-\omega}$ holds with very high probability
    for any $t\le t_f$,
    for some $\omega>0$, we do a standard $\ell^2$--estimate. Define $F(t):=\sum_i |w_i(t)|^2$, then we have
    \begin{equation}
        \label{eq:flowF}
        \dif F=-\frac{1}{2}\sum_{i,j}\mathcal{B}_{ij} (w_i-w_j)^2\dif t
        +\sum_{|i|\le n^{\omega_A}} w_i\dif {\bm \xi}_1+\frac{1}{n}\sum_{|i|\le n^{\omega_A}}L_{ii}\dif t.
    \end{equation}

    Define the stopping time
    \[
        \tau:=\inf\{t>0 : F(t)> n^{\omega_A+\omega_f-\omega_Q/2-2}\}\wedge t_f.
    \]
    Using  \eqref{eq:assbqv} and that
    $|w_i(t)|\lesssim n^{-1+\xi}$ by rigidity, the quadratic variation of the stochastic term in \eqref{eq:flowF} is estimate by
    \[
        \E\left[\sum_{|i|\le n^{\omega_A}} w_i\dif {\bm \xi}_1,\sum_{|i|\le n^{\omega_A}} w_i\dif {\bm \xi}_1\Big|\mathcal{F}_t\right]=\frac{1}{2n}\sum_{|i|,|j|\le n^{\omega_A}} w_i w_j L_{ij}\dif t\le n^{2\xi+2\omega_A-2\omega_Q-3}\dif t.
    \]
    By the Burkholder--Davis--Gundy (BDG) inequality, we thus conclude that
    \begin{equation}
        \sup_{0\le t\le t_f}\left|\int_0^t \sum_{|i|\le n^{\omega_A}} w_i(t)\dif {\bm \xi}_1(t)\right|\lesssim n^{2\xi+\omega_A-\omega_Q-3/2}\sqrt{t_f}
    \end{equation}
    with very high probability. Then, using that $|L_{ii}|\le n^{-\omega_Q}$
    to estimate the last term in~\eqref{eq:flowF}
    and that $F(0)=0$, we obtain
    \begin{equation}
        \sup_{0\le t<\tau} F(t)\le \frac{n^{\xi+\omega_A+\omega_f}}{n^{2+\omega_Q}},
    \end{equation}
    for any arbitrary small $\xi>0$ with very high probability. This shows that $\tau=t_f$ and so that $|w_i(t_f)|\le n^{-1-\omega}$, for any $i\in [n]$, for some sufficiently small $\omega>0$. This follows from the relation among the various $\omega$'s in Section~\ref{rem:s}, as in our application $\omega_Q=\omega_E$.

    We now turn to the estimate of the long range part, i.e.
    we look at~\eqref{long}. We fix an index $|p|\ge n^{\omega_A}$, and we will
    study the effect on the solution of each single $\dif \xi_{1,p}$ separately.
    We thus  consider ${\bm w}={\bm w}^{(p)}:={\bm v}-{\bm v}^{(p)}$,  with
    \[
        \dif {\bm v}^{(p)}=\mathcal{B}{\bm v}^{(p)}\dif t+ (1-\bm1(\cdot=p))\dif {\bm \xi}_1, \qquad {\bm v}^{(p)}(0)={\bm v}(0),
    \]
    and so obtain
    \begin{equation}
        \dif {\bm w}=\mathcal{B}{\bm w}\dif t+\bm1(\cdot = p)\dif {\bm \xi}_1, \qquad {\bm w}(0)=0,
    \end{equation}
    i.e. we just removed a single stochastic forcing term from ${\bm u}$. To show that for any $|p|\ge n^{\omega_A}$ we have $|w_i^{(p)}|\le \exp(-n^{\omega_l/10})$ for $|i|\le n^{\omega_l}$, with $\omega_l$ the short--range scale from \eqref{eq:shortscsetA}, it is enough to use a "modified" $\ell^2$--method  as in the proof of finite speed of propagation in \cite[Lemma 3.9]{MR3916329}, \cite[Lemma 3.7]{MR3914908}, together with the changes presented in the proof of Lemma~\ref{lem:finspeedA} to deal with the new additional stochastic term; we thus omit the details. Since the bound $|w_i^{(p)}|\le \exp(-n^{\omega_l/10})$ holds for any fixed $p$, we then conclude that for $|i|\le n^{\omega_l}$ can remove all the stochastic forcing terms in \eqref{long} at the price of an error which is smaller than $n\exp(-n^{\omega_l/10})$. This concludes the bound of the new term \(\dif {\bm \xi}_1\). The remainder of the proof of Proposition~\ref{pro:ciala} proceeds exactly in the same way of~\cite[Eq.~(3.86)--(3.99)]{MR3916329}, hence we omit it. Since  \(t_f=t_0+t_1\), choosing \(\omega=\omega_1/10\), \(\widehat{\omega}\le \omega/10\), the above computations conclude the proof of Proposition~\ref{pro:ciala}.
\end{proof}

\appendix

\section{Proof of Lemma~\ref{lem:intbp}}\label{sec:INTBP}
In order to prove Lemma~\ref{lem:intbp} we have to compute
\begin{equation}\label{eq:mainh}
    \frac{2}{\pi^2}\int_\C  \dif^2 z_1\int_\C  \dif^2 z_2 \partial_1\overline{\partial}_1 f(z_1)\partial_2\overline{\partial}_2 \overline{g(z_2)} \Theta(z_1,z_2)
\end{equation}
for compactly supported smooth functions \(f, g\).
We recall that
\begin{equation}\label{eq:thetsplit}
    \begin{split}
        \Theta(z_1,z_2)&=\Xi(z_1,z_2)+\Lambda(z_1,z_2), \qquad \Lambda(z_1,z_2):= -\frac{1}{2}\log \abs{1-z_1\overline{z}_2}^2 \bm1(\abs{z_1},\abs{z_2}> 1), \\
        \Xi(z_1,z_2)&:= -\frac{1}{2}\log \abs{z_1-z_2}^2\bigl[1-\bm1(\abs{z_1},\abs{z_2}>1)\bigr]+\frac{1}{2}\log\abs{z_1}^{2} \bm1(\abs{z_1}\ge 1) \\
        &\qquad +\frac{1}{2}\log\abs{z_2}^{2}\bm1(\abs{z_2}\ge 1).
    \end{split}
\end{equation}
In order to compute~\eqref{eq:mainh} we will perform integration by parts twice. For this purpose we split the integral in~\eqref{eq:mainh} for \(\Xi(z_1,z_2)\) into the regimes \(\abs{z_1-z_2}\ge \epsilon\) and its complement, and the integral of \(\Lambda(z_1,z_2)\) into the regimes \(\abs{1-z_1\overline{z}_2}\ge \epsilon\) and its complement. We decided to perform two different cut-offs for \(\Xi\) and \(\Lambda\) as a consequence of the different kind of singularity of the logarithms in their definition. By the explicit definitions in~\eqref{eq:thetsplit}, it is easy to see that the integrals in the regimes \(\abs{z_1-z_2}\le \epsilon\), \(\abs{1-z_1\overline{z}_2}\le \epsilon\) go to zero as \(\epsilon\to 0\), hence we have
\begin{equation}\label{eq:epslimitA}
    \begin{split}
        2\mathcal{I}&:= \frac{2}{\pi^2}\int_\C  \dif^2 z_1\int_\C  \dif^2 z_2 \partial_1\overline{\partial}_1 f(z_1)\partial_2\overline{\partial}_2 \overline{g(z_2)} \Theta(z_1,z_2) \\
        & =\lim_{\epsilon\to 0} \frac{2}{\pi^2}\int_\C  \dif^2 z_1\int_\C  \dif^2 z_2 \partial_1\overline{\partial}_1 f(z_1)\partial_2\overline{\partial}_2 \overline{g(z_2)} \\
        &\qquad\qquad\qquad\qquad\qquad\times\Big[\Xi(z_1,z_2)\bm1(\abs{z_1-z_2}\ge \epsilon)+\Lambda(z_1,z_2)\bm1(\abs{1-z_1\overline{z}_2}\ge \epsilon)\Big].
    \end{split}
\end{equation}
In order to prove Lemma~\ref{lem:intbp} we write the l.h.s.\ of~\eqref{eq:epslimitA} as \(\mathcal{I}+\mathcal{I}\) so that in the first integral we perform integration by parts with respect to \(\partial_1,\overline{\partial}_2\) and in the second one with respect to \(\overline{\partial}_1,\partial_2\). This split is motivated by the fact that
\[
    \overline{\partial} \overline{g}\partial f+ \partial \overline{g}\overline{\partial} f=\frac{1}{2}\braket{ \nabla g, \nabla f },
\]
which is the first term in the l.h.s.\ of~\eqref{eq:intbp} in Lemma~\ref{lem:intbp}. From now on we focus only on the integral for which we perform integration by parts with respect to \(\partial_1,\overline{\partial}_2\). The computations for the other integral are exactly the same. It is well known that the distributional Laplacian of \(\log\abs{z_1-z_2}\) is \(2\pi\) the delta function in \(z_1=z_2\), more precisely, we have that
\begin{equation}\label{eq:deltaderfun}
    -\partial_1\partial_2\log\abs{z_1-z_2} \dif^2 z_1\dif^2 z_2=\frac{\pi}{2}\delta(z_1-z_2),
\end{equation}
in the sense of distributions. Hence, in the remainder of this section we focus on the computation of the integral of \(\Lambda(z_1,z_2)\) and omit the \(\epsilon\)-regularisation in the integral of \(\Xi\).

Performing integration by parts in \(\mathcal{I}\), which is defined in~\eqref{eq:epsintA}, with respect to \(\partial_1,\overline{\partial}_2\) we get
\begin{equation}\label{eq:epsintA}
    \begin{split}
        &\lim_{\epsilon\to 0} \frac{1}{\pi^2}\int_\C  \dif^2 z_1\int_\C  \dif^2 z_2 \partial_1\overline{\partial}_1 f(z_1)\partial_2\overline{\partial}_2 \overline{g(z_2)}\Big[\Xi(z_1,z_2)+\Lambda(z_1,z_2)\bm1(\abs{1-z_1\overline{z}_2}\ge \epsilon)\Big] \\
        &\quad=\lim_{\epsilon\to 0} \frac{1}{\pi^2} \int_\C  \dif^2 z_1 \int_\C  \dif^2 z_2 \overline{\partial}_1 f(z_1) \partial_2 \overline{g(z_2)} \Big[\partial_1\overline{\partial}_2\Xi(z_1,z_2)+\partial_1\overline{\partial}_2\Lambda(z_1,z_2)\bm1(\abs{1-z_1\overline{z}_2}\ge \epsilon) \Big] \\
        &\qquad+ \lim_{\epsilon\to 0}-\frac{\ii}{2\pi^2} \int_\C  \int_\C \dif^2 z_2 \overline{\partial}_1 f \Big[ \overline{\partial}_2\partial_2 \overline{g} \Lambda\bm1(\abs{1-z_1\overline{z}_2}=\epsilon)\dif \overline{z}_1 -\partial_2 \overline{g} \partial_1\Lambda\bm1(\abs{1-z_1\overline{z}_2}=\epsilon)\dif z_2 \Big] \\
        &\quad =:  \lim_{\epsilon\to 0}\bigl[J_{1,\epsilon}+J_{2,\epsilon}\bigr].
    \end{split}
\end{equation}
where in the fourth line we used  Stokes theorem written symbolically  in the form
\begin{equation}\label{eq:eqdis}
    \partial_z \bm1(\abs{z-z_2}\ge \epsilon) \dif^2 z=\frac{\ii}{2}\bm1(\abs{z-z_2}=\epsilon) \dif \overline{z}
\end{equation}
for any fixed \(z_2\).
We remark that~\eqref{eq:eqdis} is understood in the sense of distributions, i.e.\ the equality holds when tested again smooth compactly supported test functions \(f\), i.e.
\[
    -\int_\C \partial_z f(z) \bm1(\abs{z-z_2}\ge \epsilon) \dif^2 z=\frac{\ii}{2} \int_{\abs{z-z_2}=\epsilon} f(z) \dif \overline{z}.
\]
Moreover, with a slight abuse of notation in~\eqref{eq:epsintA}--\eqref{eq:eqdis} by \(\bm1(\abs{z-z_2}=\epsilon)\dif \overline{z}\) we denoted
the clock-wise contour integral over the circle of radius \(\epsilon\) around \(z_2\). We use the notation above in the remainder of this section.

The second derivative (in the sense of the distributions) of \(\Xi(z_1,z_2)\) in~\eqref{eq:epsintA}, using~\eqref{eq:deltaderfun}, is given by
\begin{equation}\label{eq:xider}
    \begin{split}
        &\partial_1\overline{\partial}_2\Xi \dif^2 z_1 \dif^2 z_2\\
        &=\frac{\pi}{2}\delta(z_1-z_2)\bigl[1-\bm1(\abs{z_1},\abs{z_2}>1)\bigr]\dif^2 z_1 \dif^2 z_2-\frac{1}{8}\log\abs{z_1-z_2}^2
        \bm1(\abs{z_1}=1) \dif \overline{z}_1 \bm1(\abs{z_2}=1)\dif z_2\\
        &\quad +\frac{\ii}{4}\frac{1}{z_1-z_2}\bm1(\abs{z_1}>1) \dif^2 z_1 \bm1(\abs{z_2}=1) \dif z_2
        -\frac{\ii}{4}\frac{1}{\overline{z}_1-\overline{z}_2}\bm1(\abs{z_2}>1) \dif^2 z_2 \bm1(\abs{z_1}=1) \dif \overline{z}_1,
    \end{split}
\end{equation}
whilst the second derivative of \(\Lambda(z_1,z_2)\) by
\begin{equation}\label{eq:lader}
    \begin{split}
        &\partial_1\overline{\partial}_2\Lambda \dif^2 z_1 \dif^2 z_2\\
        &= \frac{1}{2(1-z_1\overline{z}_2)^2}\bm1(\abs{z_1},\abs{z_2}>1)\dif^2 z_1 \dif^2 z_2+\frac{1}{8}\log\abs{1-z_1\overline{z}_2}
        \bm1(\abs{z_1}=1) \dif \overline{z}_1 \bm1(\abs{z_2}=1)\dif z_2\\
        &\quad +\frac{\ii}{4}\frac{\overline{z_2}}{1-z_1\overline{z}_2}
        \bm1(\abs{z_1}>1) \dif^2 z_1 \bm1(\abs{z_2}=1) \dif z_2   +\frac{\ii}{4}\frac{z_1}{1-z_1\overline{z}_2}
        \bm1(\abs{z_2}>1) \dif^2 z_2 \bm1(\abs{z_1}=1) \dif \overline{z}_1.
    \end{split}
\end{equation}
Note that
\[
    \begin{split}
        \partial_1\overline{\partial}_2(\Xi+\Lambda)\dif^2 z_1 \dif^2 z_2&=\frac{\pi}{2}\delta(z_1-z_2)\bm1(\abs{z_1},\abs{z_2}\le 1)\dif^2 z_1 \dif^2 z_2\\
        &\quad +\frac{1}{2(1-z_1\overline{z}_2)^2}\bm1(\abs{z_1},\abs{z_2}>1)\dif^2 z_1 \dif^2 z_2,
    \end{split}
\]
hence, by~\eqref{eq:xider}--\eqref{eq:lader} we conclude that
\begin{equation}\label{eq:us}
    \lim_{\epsilon\to 0} J_{1,\epsilon}=\frac{1}{2\pi}\int_\DD  \overline{\partial} f \partial\overline{g}\dif^2 z+\lim_{\epsilon\to 0} \frac{1}{2\pi}\int_{\abs{z_1}\ge 1}\dif^2 z_1 \int_{\abs{z_2}\ge 1} \dif^2 z_2 \frac{\overline{\partial}_1 f(z_1)\partial_2 g(z_2) }{(1-z_1\overline{z}_2)^2}\bm1(\abs{1-z_1\overline{z}_2}\ge \epsilon).
\end{equation}
On the other hand, the integration by parts with respect to \(\overline{\partial}_1,\partial_2\) gives
\begin{equation}\label{eq:uss}
    \frac{1}{2\pi}\int_\DD  \overline{\partial} f \partial\overline{g}\dif^2 z+\lim_{\epsilon\to 0} \frac{1}{2\pi}\int_{\abs{z_1}\ge 1}\dif^2 z_1 \int_{\abs{z_2}\ge 1} \dif^2 z_2 \frac{\overline{\partial}_1 f(z_1)\partial_2 g(z_2) }{(1-z_1\overline{z}_2)^2}\bm1(\abs{1-z_1\overline{z}_2}\ge \epsilon).
\end{equation}
Hence, summing~\eqref{eq:us}--\eqref{eq:uss} we get exactly the r.h.s.\ of~\eqref{eq:intbp} using that
\[
    \frac{1}{2\pi}\int_\DD  \bigl[ \overline{\partial}\overline{g}\partial f+\partial \overline{g}\overline{\partial}f\bigr]\dif^2 z=\frac{1}{4\pi}\int_\DD  \braket{ \nabla g, \nabla f}\dif^2 z.
\]
In order to conclude the proof of Lemma~\ref{lem:intbp} we prove that \(\abs{J_{2,\epsilon}}\to 0\) as \(\epsilon\to 0\) in Lemma~\ref{lem:epslimit0} and that the limit in the r.h.s.\ of~\eqref{eq:us} exists in Lemma~\ref{lem:fin0e}.

\begin{lemma}\label{lem:epslimit0}
    Let \(J_{2,\epsilon}\) be defined in~\eqref{eq:epsintA}, then
    \begin{equation}
        \label{eq:zerolim}
        \lim_{\epsilon\to 0}\abs{J_{2,\epsilon}}=0.
    \end{equation}
\end{lemma}
\begin{proof}
    For the first integral in \(J_{2,\epsilon}\), using the parametrization \(z_2=r_2 e^{\ii\theta_2}\) and \(z_1=(1+\epsilon e^{\ii\theta_1})/\overline{z}_2\), for any fixed \(z_2\), we get
    \begin{equation}
        \label{eq:firstjint}
        \abs*{\int_1^\infty\dif r_2 \int_0^{2\pi}\dif \theta_1 \int_0^{2\pi}\dif \theta_2 \, \epsilon e^{\ii(\theta_1+\theta_2)} \overline{\partial}_1 f\left(r_2^{-1}e^{\ii\theta_2} [1+\epsilon e^{\ii\theta_1}]\right) \overline{\partial}_2\partial_2 g(r_2e^{\ii\theta_2}) \log \epsilon }\lesssim \epsilon\log \epsilon,
    \end{equation}
    where we used that \(\norm{ \overline{\partial}_1 f}_{L^\infty(\C )}, \norm{ \overline{\partial}_2\partial_2 g}_{L^1(\C )}\lesssim 1\) as a consequence of \(f,g\in H_0^{2+\delta}(\Omega)\), for an open set \(\Omega\subset \C \) such that \(\overline{\DD }\subset \Omega\).

    Furthermore, using the parametrizations \(z_1=r_1 e^{\ii\theta_1}\) and \(z_2=(1+\epsilon e^{\ii\theta_2})/\overline{z}_1\) for the second integral in \(J_{2,\epsilon}\), we have that
    \begin{equation}
        \label{eq:secjint1}
        \begin{split}
            J_{2,\epsilon}&=\Bigg[\int_1^\infty\dif r_1 \int_0^{2\pi}\dif \theta_1 \int_0^{2\pi}\dif \theta_2 \,\, \epsilon e^{\ii(\theta_1+\theta_2)} \overline{\partial}_1 f(r_1e^{\ii\theta_1})\partial_2 \overline{g}\left(r_1^{-1}e^{\ii\theta_1}[1+\epsilon e^{\ii\theta_2}]\right) \\
            &\qquad\qquad\qquad\qquad\qquad\times \frac{1+\epsilon e^{-\ii\theta_2}}{\epsilon r_1 e^{-\ii\theta_2} e^{\ii\theta_1}}\bm1(\abs{1+\epsilon e^{\ii\theta_2}}> r_1)\Bigg]+\mathcal{O}(\epsilon\log \epsilon),
        \end{split}
    \end{equation}
    where the error term comes from the integral of \(\partial_1 \bm1(\abs{z_1},\abs{z_2}>1)\) and the bound in~\eqref{eq:firstjint}. Note that \(\bm1(\abs{1+\epsilon e^{\ii\theta_1}}> r_1)=0\) if \(r_1\ge 1+2\epsilon\), hence we can bound the first term in \(J_2\) by
    \begin{equation}
        \int_1^{1+2\epsilon}\dif r_1 \int_0^{2\pi}\dif \theta_1 \int_0^{2\pi}\dif \theta_2 \, \abs*{\overline{\partial}_1 f(r_1e^{\ii\theta_1})\partial_2 \overline{g}\left(r_1^{-1} e^{\ii\theta_1}[1+\epsilon e^{\ii\theta_2}]\right)} \lesssim \epsilon,
    \end{equation}
    since \(\norm{ \partial_2 \overline{g}}_{L^\infty(\C )}, \norm{ \overline{\partial}_1 f}_{L^1(\C )}\lesssim 1\). Hence, we conclude that
    \[
        J_{2,\epsilon}=\mathcal{O}\left(\epsilon+\epsilon \log \epsilon \right).
    \]
    This concludes the proof of~\eqref{eq:zerolim}.
\end{proof}

We conclude this section proving the existence of the limit of \(J_{1,\epsilon}\) as \(\epsilon\to 0\). More precisely, in Lemma~\ref{lem:fin0e} we prove that \(J_{1,\epsilon}\) is a Cauchy sequence.

\begin{lemma}\label{lem:fin0e}
    Let \(J_{1,\epsilon}\) be defined in~\eqref{eq:epsintA}, then for any \(0<\epsilon'\le \epsilon\) we have that
    \begin{equation}
        \label{eq:cauch}
        \abs{J_{1,\epsilon}-J_{1,\epsilon'}}\lesssim \epsilon^\delta,
    \end{equation}
    for some \(\delta>0\).
\end{lemma}
\begin{proof}
    We only consider the integral with the second derivative of \(\Lambda\). We dealt with the integral of the second derivative of \(\Xi(z_1,z_2)\) already in~\eqref{eq:deltaderfun}. Define
    \begin{equation}
        \label{eq:hplas}
        I_\epsilon:= \frac{1}{\pi^2} \int_\C  \dif^2 z_1 \int_\C  \dif^2 z_2 F(z_1,z_2)\Big[\partial_2\overline{\partial}_1\Lambda(z_1,z_2)\bm1(\abs{1-z_1\overline{z}_2}\ge \epsilon) \Big],
    \end{equation}
    where \(F(z_1,z_2):= \overline{\partial}_1 f(z_1) \partial_2 \overline{g(z_2)}\) is a \(\delta\)-H\"older continuous function. Then, for any \(0<\epsilon'<\epsilon\), using the change of variables \(z_2=r_2 e^{\ii\theta_2}\) and \(z_1=(1+r_1e^{\ii \theta_1})/\overline{z}_2\), we write
    \begin{equation}
        \label{eq:finexlim}
        \begin{split}
            I_{\epsilon'}-I_\epsilon&=\frac{1}{\pi^2} \int_\C  \dif^2 z_1 \int_\C  \dif^2 z_2 \bigl(F(z_1,z_2)-F(\overline{z}_2^{-1},z_2)\bigr)\Big[\partial_2\overline{\partial}_1\Lambda(z_1,z_2)\bm1(\epsilon\ge \abs{1-z_1\overline{z}_2}\ge \epsilon') \Big] \\
            &\quad +\frac{1}{\pi}\int_1^\infty \dif r_2\int_0^{2\pi} \dif \theta_2 \int_0^{2\pi} \dif \theta_1 \int_{\epsilon'}^\epsilon\dif r_1 \, F(r_2^{-1} e^{-\ii\theta_2},r_2 e^{\ii\theta_2}) \frac{e^{2\ii \theta_1}}{r_1 r_2}.
        \end{split}
    \end{equation}
    Note that the integral in the second line of~\eqref{eq:finexlim} is exactly zero since \(e^{2\ii\theta_1}\) the only term which depends on \(\theta_1\). On the other hand, we can bound the first integral in~\eqref{eq:finexlim} by \(\epsilon^{2\delta}\), with \(\delta\) the H\"older exponent of \(F\), using the fact that
    \[
        \abs*{F(z_1,z_2)-F(\overline{z}_2^{-1},z_2)} \le \abs*{\frac{1}{\overline{z}_2}+\frac{r_1 e^{\ii\theta_1}}{\overline{z}_2}-\frac{1}{\overline{z}_2}}^{2\delta}\lesssim \left(\frac{r_1}{r_2}\right)^{2\delta}.
    \]
    This concludes the proof of this lemma.
\end{proof}

\section{Derivation of the DBM for the eigenvalues of \texorpdfstring{\(H^z\)}{Hz}}\label{sec:derdbm}
Let \(X\) be an \(n\times n\) complex random matrix, let \(H^z\) be the Hermitisation of \(X-z\) defined in~\eqref{eq:her}, and define \(Y^z:= X-z\). We recall that \(\{\lambda^z_i,-\lambda^z_i\}_{i=1}^n\) are the eigenvalues of \(H^z\), and \(\{{\bm w}^z_i,{\bm w}^z_{-i}\}_{i=1}^n\) are the corresponding orthonormal eigenvectors, i.e.\ for any \(i,j\in[n]\) we have
\begin{equation}\label{eigeq}
    H^z{\bm w}^z_{\pm i}=\pm\lambda^z_i, \qquad ({\bm w}_i^z)^* {\bm w}^z_j=\delta_{i,j},\qquad ({\bm w}^z_i)^* {\bm w}^z_{-j}=0,
\end{equation}
for any \(i,j\in [n]\). For simplicity in the following derivation we assume that the eigenvalues are all distinct. In particular, for any \(i\in [n]\), by the block structure of \(H^z\) it follows that
\begin{equation}\label{singval}
    {\bm w}^z_{\pm i}=( {\bm u}^z_i, \pm {\bm v}^z_i), \qquad Y^z{\bm v}^z_i=\lambda^z_i {\bm u}^z_i, \qquad (Y^z)^* {\bm u}^z_i=\lambda^z_i {\bm v}^z_i.
\end{equation}
Moreover, since \(\{{\bm w}^z_{\pm i}\}_{i=1}^n\) is an orthonormal base, we conclude that
\begin{equation}\label{hnorm}
    ({\bm u}^z_i)^* {\bm u}^z_i=({\bm v}^z_i)^* {\bm v}^z_i=\frac{1}{2}.
\end{equation}

In the following, for any fixed entry \(x_{ab}\) of \(X\), we will use the notation
\begin{equation}
    \dot{f}=\frac{\partial f}{\partial x_{ab}} \quad \text{or} \quad  \dot{f}=\frac{\partial f}{\partial \overline{x_{ab}}},
\end{equation}
where \(f=f(X)\) is a function of the matrix \(X\). Then, we consider the flow
\begin{equation}\label{OU}
    \dif X_t=\frac{\dif B_t}{\sqrt{n}}, \quad X_0=X,
\end{equation}
where \(B_t\) is a matrix valued complex standard Brownian motion.

From now on we only consider positive indices \(1\le i\le n\). We may also drop the \(z\) and \(t\) dependence to make our notation easier. For any \(i,j\in [n]\), differentiating~\eqref{eigeq} we get
\begin{align}\label{findl1}
    \dot{H} {\bm w}_i+H\dot{\bm w}_i=\dot{\lambda}_i {\bm w}_i+\lambda_i \dot{\bm w}_i, \\\label{findl2}
    \dot{\bm w}_i^*{\bm w}_j+{\bm w}_i^*\dot{\bm w}_j=0,                                \\\label{findl3}
    {\bm w}_i^* \dot{\bm w}_i+\dot{\bm w}_i^*{\bm w}_i=0.
\end{align}
Note that~\eqref{findl3} implies that \(\Re[{\bm w}_i^* \dot{\bm w}_i]=0\). Hence, since the eigenvectors are defined modulo a phase, we can choose eigenvectors such that \(\Im[{\bm w}_i^* \dot{\bm w}_i]=0\) for any \(t\ge 0\). Then, multiplying~\eqref{findl1} by \({\bm w}_i^*\) we conclude that
\begin{equation}\label{dereig}
    \dot{\lambda}_i={\bm u}_i^*\dot{Y}{\bm v}_i+{\bm v}_i^*\dot{Y}^*{\bm u}_i.
\end{equation}
Moreover, multiplying~\eqref{findl1} by \({\bm w}_j^*\), with \(j\ne i\), and by \({\bm w}_{-j}*\), we get
\begin{equation}\label{diffl}
    (\lambda_i-\lambda_j){\bm w}_j^*\dot{\bm w}_i={\bm w}_j^* \dot{H} {\bm w}_i, \qquad (\lambda_i+\lambda_j){\bm w}_{-j}^*\dot{\bm w}_i={\bm w}_{-j}^* \dot{H} {\bm w}_i,
\end{equation}
respectively.
By~\eqref{findl2}--\eqref{findl3} it follows that
\begin{equation}\label{dereigv1}
    \dot{\bm w}_i=\sum_{j\ne i} ({\bm w}_j^*\dot{\bm w}_i) {\bm w}_j+\sum_j ({\bm w}_{-j}^* \dot{\bm w}_i){\bm w}_{-j},
\end{equation}
hence by~\eqref{diffl} we conclude
\begin{equation}\label{dereigv2}
    \dot{\bm w}_i=\sum_{j\ne i} \frac{{\bm v}_j^*\dot{Y}^*{\bm u}_i+{\bm u}_j^*\dot{Y}{\bm v}_i}{\lambda_i-\lambda_j} {\bm w}_j+\sum_j \frac{{\bm u}_j^*\dot{Y}{\bm v}_i-{\bm v}_j^*\dot{Y}^* {\bm u}_i}{\lambda_i+\lambda_j} {\bm w}_{-j}.
\end{equation}

By Ito's formula we have that
\begin{equation}\label{Itoegc}
    \dif \lambda_i=\sum_{ab} \frac{\partial \lambda_i}{\partial x_{ab}} \dif x_{ab}+\frac{\partial \lambda_i}{\partial \overline{x_{ab}}} \dif \overline{x_{ab}} +\frac{1}{2}\sum_{ab}\sum_{kl} \frac{\partial^2 \lambda_i}{\partial x_{ab}\partial \overline{x_{kl}}} \dif x_{ab} \dif \overline{x_{kl}}+\frac{\partial^2 \lambda_i}{\partial \overline{x_{ab}}\partial x_{kl}} \dif\overline{x_{ab}} \dif x_{kl}.
\end{equation}
Note that in~\eqref{Itoegc} we used that \(\dif x_{ab} \dif x_{ab}=\dif \overline{x_{kl}}\dif \overline{x_{kl}}=0\). Then by~\eqref{dereig}--\eqref{dereigv2} it follows that
\begin{equation}\label{compleg}
    \frac{\partial \lambda_i}{\partial x_{ab}}=u_i(a)^* v_i(b), \qquad \frac{\partial \lambda_i}{\partial \overline{x_{ab}}}=v_i(b)^* u_i(a),
\end{equation}
and that
\begin{align}\label{complev1}
     & \frac{\partial w_i}{\partial x_{ab}}(k)=\sum_{j\ne i} \left[ \frac{u_j^*(a)v_i(b)}{\lambda_i-\lambda_j}w_j(k)+  \frac{u_j^*(a)v_i(b)}{\lambda_i+\lambda_j} w_{-j}(k)\right] +\frac{u_i(a)^*v_i(b)}{2\lambda_i}w_{-i}(k),            \\\label{complev2}
     & \frac{\partial w_i}{\partial \overline{x_{ab}}}(k)=\sum_{j\ne i} \left[ \frac{v_j^*(b)u_i(a)}{\lambda_i-\lambda_j} w_j(k)-  \frac{v_j^*(b)u_i(a)}{\lambda_i+\lambda_j}w_{-j}(k)\right] -\frac{v_i(b)^*u_i(a)}{2\lambda_i}w_{-i}(k).
\end{align}

Next, we compute
\begin{equation}\label{secderlc}
    \begin{split}
        \frac{\partial^2 \lambda_i}{\partial x_{ab}\partial \overline{x_{kl}}}&=\frac{\partial v_i^*}{\partial x_{ab}}(l)u_i(k)+v_i(l)^*\frac{\partial u_i}{\partial x_{ab}}(k) \\
        &=\sum_{j\ne i} \left[ \frac{v_j(b)u_i^*(a)}{\lambda_i-\lambda_j}v_j(l)^* u_i(k)+  \frac{v_j(b)u_i(a)^*}{\lambda_i+\lambda_j}v_j(l)^* u_i(k)\right] +\frac{v_i(b)u_i(a)^*}{2\lambda_i}v_i(l)^*u_i(k)  \\
        &\quad +\sum_{j\ne i} \left[ \frac{u_i^*(a)v_i(b)}{\lambda_i-\lambda_j}v_i(l)^* u_j(k)+  \frac{u_j^*(a)v_i(b)}{\lambda_i+\lambda_j}v_i(l)^* u_j(k)\right] +\frac{u_i(a)^*  v_i(b)}{2\lambda_i}v_\alpha(l)^* u_i(k).
    \end{split}
\end{equation}
Finally, combining~\eqref{OU},~\eqref{compleg},~\eqref{Itoegc} and~\eqref{secderlc}, we conclude (cf.~\cite[Eq.~(5.8)]{MR2919197})
\begin{equation}\label{egcflow}
    \dif \lambda^z_i= \frac{\dif b^z_i}{\sqrt{2n}}+\frac{1}{2n}\sum_{j\ne i} \left[\frac{1}{\lambda^z_i-\lambda^z_j}+\frac{1}{\lambda^z_i+\lambda^z_j}\right] \dif t +\frac{\dif t}{4n\lambda_i},
\end{equation}
where we defined
\begin{equation}\label{brmotc}
    \dif b^z_i:= \sqrt{2}(\dif B^z_{ii}+\dif \overline{B^z_{ii}}),\quad \dif B^z_{ij}:= \sum_{ab} \overline{u^z_i(a)}\dif B_{ab} v^z_j(b),
\end{equation}
where \(B_t\) is the matrix values Brownian motion in~\eqref{OU}. In particular, \(b^z_i\) is a standard real Brownian motion, indeed
\[
    \begin{split}
        \E (B^z_{ii}+\overline{B^z_{ii}})(B^z_{ii}+\overline{B^z_{ii}})^*&=\E \left(\sum_{ab} \overline{u^z_i(a)}B_{ab} v^z_i(b) +u^z_i(a)\overline{B}_{ab} \overline{v^z_i(b)}\right)^2 \\
        &= 2 \sum_{abcd} \overline{u^z_i(a)}B_{ab} v^z_i(b) u^z_i(c)\overline{B}_{cd} \overline{v^z_i(d)} \\
        &=2 \sum_{abcd} \delta_{ac}\delta_{bd}\overline{u^z_i(a)} v^z_i(b) u^z_i(c) \overline{v^z_i(d)} =\frac{1}{2}.
    \end{split}
\]

\printbibliography%
\end{document}